\documentclass[11pt]{amsart}
\usepackage[margin=0.8in]{geometry} 
\usepackage{amsthm, amsrefs, amsmath,amsfonts,amssymb,euscript,hyperref,graphics,color,slashed,mathrsfs}
\usepackage{enumerate}
\usepackage{graphicx}
\usepackage[toc]{appendix}
\usepackage{comment}
\usepackage{import}
\usepackage{tikz}
\usepackage{latexsym}
 
\usepackage[utf8]{inputenc}
\usepackage{enumitem}

\def\Xh{{\widehat{X}}}
\def\Th{{\widehat{T}}}
\def\Lh{{\widehat{L}}}
\def\Ntop{{N_{_{\rm top}}}}
\def\Ninf{N_{_{\infty}}}

\def\Xr{\mathring{X}}
\def\Tr{\mathring{T}}
\def\XR{\mathring{X}}
\def\TR{\mathring{T}}
\def\Trh{\widehat{\mathring{T}}}

\def\Zr{{\mathring{Z}}}
\def\Lr{{\mathring{L}}}
\def\ZR{{\mathring{Z}}}
\def\LR{{\mathring{L}}}
\def\Lbr{{\mathring{\underline{L}}}}
\def\Psir{\mathring{\Psi}}

\def\zr{{\mathring{z}}}
\def\yr{{\mathring{y}}}

\def\varrhor{{\mathring{\varrho}}}

\def\kappar{{\mathring{\kappa}}}
\def\mur{{\mathring{\mu}}}
\def\etar{{\mathring{\eta}}}
\def\deltar{{\mathring{\delta}}}
\def\deltasr{{\mathring{\slashed{\delta}}}}

\def\zetar{{\mathring{\zeta}}}
\def\chir{{\mathring{\chi}}}
\def\chibr{{\mathring{\underline{\chi}}}}

\def\E{{\mathcal E}}
\def\Eb{{\underline{\mathcal E}}}
\def\EB{{\underline{\mathscr{E}}}}

\def\F{{\mathcal F}}
\def\Fb{{\underline{\mathcal F}}}

\def\wb{\underline{w}}

\def\chib{\underline{\chi}}

\def\Lb{\underline{L}}

\def\tr{\text{tr}}

\def\Lh{\widehat{{L}}}

\newtheorem*{MainTheorem}{Main Theorem}

\newtheorem{theorem}{Theorem}[section]
\newtheorem{lemma}[theorem]{Lemma}
\newtheorem{proposition}[theorem]{Proposition}
\newtheorem{corollary}[theorem]{Corollary}

\newtheorem{remark}[theorem]{Remark}

\numberwithin{equation}{section}

\begin{document}
\title {On the stability of multi-dimensional rarefaction waves I: the energy estimates}

\author{Tian-Wen LUO and Pin YU}

\address{School of Mathematical Sciences, South China Normal University, Guangzhou, China}
\email{twluo@m.scnu.edu.cn}

\address{Department of Mathematical Sciences, Tsinghua University\\ Beijing, China}
\email{yupin@mail.tsinghua.edu.cn}


\maketitle
\begin{abstract}
We study the resolution of discontinuous singularities in gas dynamics via rarefaction waves. The mechanism is well-understood in the one dimensional case. We will prove the {\color{black}nonlinear} stability of the Riemann problem for multi-dimensional isentropic Euler equations in the regime of rarefaction waves. The proof relies on the new energy estimates \emph{without loss of derivatives}. We also give a detailed geometric description of the rarefaction wave fronts. This is the first paper in the series which provides the \emph{a priori} energy bounds.

\end{abstract}
\tableofcontents

\section{Introduction}

{\color{black}In the first paragraph} of Courant and Friedrichs's {\color{black}classic monograph \cite{CourantFriedrichs} on shocks}, the following observation is made to describe {\color{black}one of most distinctive nonlinear features} of compressible flow: \textit{{\color{black}``}Even when the start of the motion is perfectly continuous, shock discontinuities may later arise automatically. Yet, under other conditions, just the opposite may happen; initial discontinuities may be smoothed out immediately{\color{black}''}}. The first situation refers to the formation of shocks. Inspired by the seminal work \cite{ChristodoulouShockFormation} of Christodoulou, much progress has been made {\color{black}on the formation} and propagation of shocks in multi-dimension (see a more detailed account in Section \ref{Sec:history shock}). The second situation refers to the resolution of discontinuities through rarefaction waves. However, much less is known on multi-dimensional rarefaction waves, apart from {\color{black}the pioneer works of Alinhac}  \cite{AlinhacWaveRare1,AlinhacWaveRare2}. This work is devoted to study the resolution of discontinuous singularities in gas dynamics.

We consider the isentropic motion of a polytropic gas, described by the isentropic compressible Euler system in dimension two,
\begin{equation}\label{eq: Euler in Euler coordinates}
	\begin{cases}
		&(\partial_t + v \cdot \nabla) \rho = -\rho \nabla \cdot v,\def\E{{\mathcal E}}
		\\
		&(\partial_t + v \cdot \nabla)v = -\rho^{-1} \nabla p,
	\end{cases}
\end{equation} 
where $\rho$, $p$ and $v$ are the density, pressure, and velocity of the gas, respectively. The equation of state is given by $p(\rho) = k_0 \rho^{\gamma}$ with constants $\gamma\in (1,3)$ and $k_0 >0$.
The sound speed $c$ is then given by $c=\sqrt{\frac{dp}{d\rho}}=k_0^{\frac{1}{2}}\gamma^{\frac{1}{2}}\rho^{\frac{\gamma-1}{2}}$. {\color{black}For an irrotational motion,}  there exists a velocity potential $\phi$ which satisfies a quasi-linear wave equation
\begin{equation}\label{eq:wave-eqn}
	g^{\mu \nu}(D \phi) \frac{\partial^2 \phi}{\partial x^\mu \partial x^\nu}  = 0,
\end{equation}
where $g = {\color{black} - c^2 dt^2} + \sum_{i=1}^2 (dx^i - v^idt)^2$ is the acoustical metric. Our goal is to study a family of singular solutions called rarefaction waves. The region of rarefaction wave is foliated by characteristic hypersurfaces called rarefaction wave fronts. These rarefaction wave fronts all emanate from an initial surface ({\color{black}a curve in the two-dimensional case}). The expansion of the characteristic hypersurfaces provides the physical mechanism to resolve the discontinuous singularities at the initial surface.

\medskip

The aim of this paper is to establish a stable nonlinear energy estimates of rarefaction waves for ideal polytropic gas,  \emph{without loss of derivatives}. 
In particular, we provide a detailed geometric description of the rarefaction wave fronts. 



\subsection{Review on the problem in one dimension}\label{subsection:1D-Riemann-problem}
In this subsection, we give a brief review of the problem in one spatial dimension. It serves as illustration and motivation of our work. We focus on the Riemann problem and its solutions consisting of elementary waves. The Riemann problem is one of the most fundamental problem in the entire field of non-linear hyperbolic conservation laws. It remains a great challenge to understand the structure of the problem in {\color{black}higher dimensions.}

The early study of nonlinear wave phenomena goes back to Poisson {\color{black}in the 1800s,} who discovered a solution to \eqref{eq: Euler in Euler coordinates} of the form $\partial_x \phi = f(x + (c-v)t)$ for an arbitrary smooth function $f$. Forty years later, Stokes \cite{Stokes} studied extensively the finite time blow-up phenomena implicated in Poisson's solution, recognizing it as waveform breaking. Stokes computed the time of singularity formation, and speculated that the solution can be continued along a surface of discontinuity, but he abandoned this idea in later years in flavor of {\color{black}the viscosity smoothing effect from the Navier-Stokes equations.} 

It was Riemann that first gave a definite and rigorous treatment of nonlinear wave phenomena in one spatial dimension, from a surprisingly modern PDE viewpoint. His monumental work \cite{Riemann} introduces most important basic concepts such as shocks and Riemann invariants, and initiates shock wave theory. In particular, Riemann proposed the Riemann problem and solved it for isentropic gas in terms of shocks and rarefaction waves. Riemann's work eventually became the foundation of the theory of conservation laws in one-dimension developed in the 20th century.

We consider the isentropic motion of a compressible gas where the motion takes place along the $x^1$ direction. The governing equations \eqref{eq: Euler in Euler coordinates} reduce to
\begin{align}\label{eq:Euler-1D}
	\begin{cases}
		\partial_t \rho + v \partial_x \rho = -\rho \partial_x v{\color{black},}\\
		\rho(\partial_t v + v \partial_x v) = -\partial_x p(\rho){\color{black},} 
	\end{cases}
\end{align}
where we denote $v = v^1$ and $x = x^1$.
Riemann introduced the following functions, known as the Riemann invariants:
\[
\begin{cases}
	&\wb = \frac{1}{2}\big(\int^{\rho} \frac{c(\rho')}{\rho'}d\rho ' + v\big) = \frac{1}{2}\big(\frac{2}{\gamma-1}c + v\big), \\
	&w = \frac{1}{2}\big(\int^{\rho} \frac{c(\rho')}{\rho'}d\rho ' - v\big) = \frac{1}{2}\big(\frac{2}{\gamma-1}c - v\big){\color{black}.}
\end{cases}
\]
In terms of the Riemann invariants,  the Euler system \eqref{eq:Euler-1D} takes the diagonal form
\begin{equation}\label{eq:Euler-1D-diagonal}
	\begin{cases}
	L_+(\wb) &:= \partial_t \wb + (v+c)\partial_x\wb = 0, \\
	 L_-(w) &:= \partial_t w + (v-c)\partial_x w = 0.
	 \end{cases}
\end{equation}
More generally,  if we regard \eqref{eq:Euler-1D} as a quasilinear hyperbolic system $\partial_t U + A(U)\partial_x U = 0$ where $U = \begin{pmatrix} \rho \\ v \end{pmatrix}$, the Riemann invariants $r_1(U) = w$ and $r_2(U) = \wb$ constitute a complete set of right eigenvectors with respect to the corresponding eigenvalues $\lambda_1(U) = v - c$ and $\lambda_2(U) = v + c$.

As a hyperbolic system, \eqref{eq:Euler-1D} has a finite speed of propagation. The solutions adjacent to constant states are called {\bf simple waves}. They are characterized by the constancy of one of the Riemann invariants. Consider forward-facing simple waves where $w = \text{const}$. By the first equation of \eqref{eq:Euler-1D-diagonal} the solution stays constant on integral curves of $L_+$. These characteristic curves then must be straight lines. They are categorized into two types: {\bf expansion waves} and {\bf compression waves}. 

\begin{center}
	\includegraphics[width=3in]{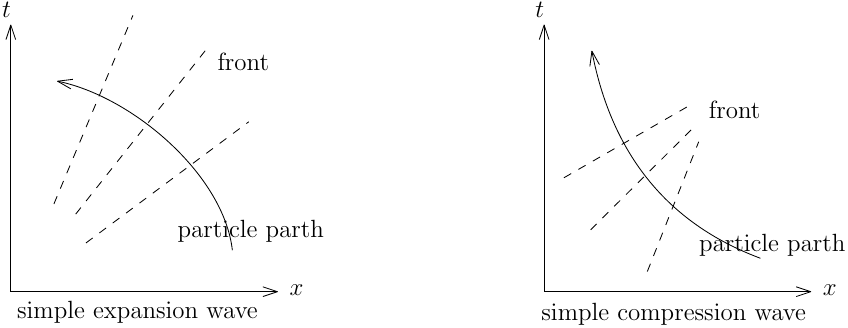}
\end{center}

It is clear that a {\color{black}simple compression wave} must form a singularity in a finite time. As Riemann observed in \cite{Riemann}, this happens for generic smooth data. Therefore, it is imperative to study initial data with discontinuities.

The Riemann problem is the study of the initial value problem connecting two piecewise constant states:
\begin{align}\label{eq:Riemann-problem-1D}
	U(t=0,x) = \begin{cases}
		U_l = \begin{pmatrix} \rho_l \\ v_l \end{pmatrix}, & x < 0;\\
		U_r = \begin{pmatrix} \rho_r \\ v_r \end{pmatrix}, & x > 0.
	\end{cases}
\end{align}
For the system \eqref{eq:Euler-1D}, the Riemann problem can be solved in terms of shocks and rarefaction waves.  

Shock fronts are piecewise continuous solutions that propagate the initial discontinuities \eqref{eq:Riemann-problem-1D}. The conservation of mass and momentum impose the {\bf jump conditions} across the shock front:
\[(v_l - v_r)^2 = (\nu_r - \nu_l)(p(\nu_l) - p(\nu_r)), \]
where $\nu = \rho^{-1}$ is the specific volume. However, such discontinuous solutions are manifestly non-unique. The physical shock waves must satisfy certain stability condition, found by Riemann in \cite{Riemann} and generalized by Lax \cite{Lax1957} as the Lax entropy condition for general hyperbolic conservation laws. Physically, it means the flow velocity relative to the shock front is supersonic at the front side where the gas particle flows into the shock front, and subsonic at the back side. In particular, the shock fronts are non-characteristic hypersurfaces.
\begin{center}
\includegraphics[width=2.8in]{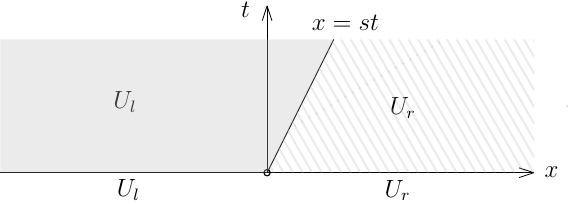}
\end{center}

The (centered) rarefaction waves are solutions that immediately smooth out the initial discontinuities. For the piecewise constant Riemann initial data \eqref{eq:Riemann-problem-1D}, they can be constructed as simple expansion waves where all the forward-facing characteristic lines emanate from the initial discontinuity (the center).
To {\color{black}motivate the multi-dimensional case}  in this paper, we record explicit expressions for the one dimensional rarefaction wave. On the positive axis $x_1=x>0$, we pose constant data $(v,c)\big|_{t=0}=(v_0,c_0)$. We then have a unique family of forward-facing centered rarefaction waves connected to the given data.

\begin{center}
	\includegraphics[width=3in]{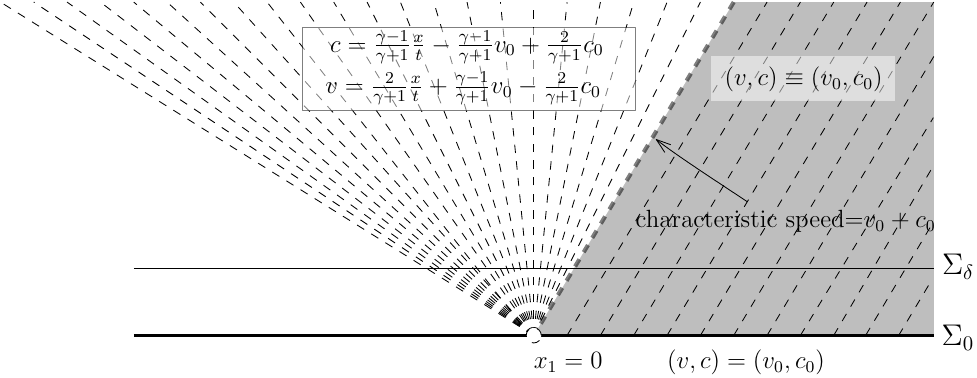}
\end{center}

The dashed lines in the picture denote the characteristics lines of the system. It corresponds to the null hypersurfaces in higher dimensions. The unshaded region is the rarefaction wave zone, where the solution is given by
\begin{align}\label{eq:1D-rarefaction-wave}
	\begin{cases}
		v&=\frac{2}{\gamma+1}\frac{x}{t}+\big(\frac{\gamma-1}{\gamma+1}v_0-\frac{2}{\gamma+1}c_0\big),\\
		c&=\frac{\gamma-1}{\gamma+1}\frac{x}{t}-\big(\frac{\gamma-1}{\gamma+1}v_0-\frac{2}{\gamma+1}c_0\big).
	\end{cases}
\end{align}

In terms of shocks and rarefaction waves, the Riemann problem for \eqref{eq:Euler-1D} is solved explicitly. We refer to Riemann's original paper \cite{Riemann} or the textbooks \cite{CourantFriedrichs, Smoller} for detailed computations. Riemann's work on gas dynamics was generalized by Lax to general hyperbolic conservation laws in his seminal paper \cite{Lax1957}.
Since then, the study of compressible fluids in one spatial dimension has {\color{black}evolved} into a fruitful field of research and it is known nowadays as the theory of one dimensional conservation laws.  In the one dimensional case, the space of functions with bounded variations (BV space) is a suitable functional space to study the evolution problem for compressible Euler equations. With the help of BV space, the theory is fairly complete: we can prove the well-posedness for initial data problem and existence of global unique weak solutions; we can also treat the formation of singularities and the interactions of elementary waves such as shocks and rarefaction waves. The reader may consult the encyclopedic book \cite{Dafermos} of Dafermos and {\color{black}the references therein} for a detailed account.

\subsection{Prior results on multi-dimensional rarefaction waves}\label{subsection:prior-results-rarefaction-waves}
The multi-dimensional theory of compressible Euler equations is much less developed. {\color{black}One of the major} technical obstacles is the breakdown of the BV space approach in a multi-dimensional setting, see \cite{Rauch}. The only effective way to control multi-dimensional systems is through the $L^2$-based energy method. The evolution of hyperbolic systems in one spacial dimension are captured by characteristic curves, which are well adapted to BV spaces. In contrast, the multi-dimensional theory are deeply tied to the characteristic hypersurfaces. The associated spacetime geometry is much more complicated and it requires new insights.

The study of multi-dimensional elementary waves was initiated by the pioneering works of Majda \cite{MajdaShock2,MajdaShock3}. It is known as the shock front problem where the initial data are perturbations of the plane shock \eqref{eq:Riemann-problem-1D}. {\color{black}For an ideal} isentropic gas with $\gamma > 1$, Majda observed the linearized shock front equations satisfy a uniform stability condition and the shock fronts can be obtained in $L^2$-based iteration via Kreiss's symmetrization, \emph{without losing derivatives}. Surprisingly, Majda also showed that the multi-dimensional shock fronts in gas dynamics have stronger stability than the counterparts in multi-dimensional scalar conservation laws (in the latter case the uniform stability assumption is {\emph not} valid).  Majda's work on shock fronts has been extended in multiple directions; see the survey \cite{Metivier2001book} by M\'{e}tivier  and the book \cite{Benzoni-Gavage-Serre2007book} by Benzoni-Gavage-Serre for these developments. We remark that shock fronts are non-characteristic hypersurfaces. 

At the end of his book on compressible flows \cite{MajdaBook}, Majda proposed a few open problems.  The first one is {\bf {\color{black}``}the existence and structure of rarefaction fronts{\color{black}''}}: \textit{{\color{black}``}Discuss the rigorous existence of rarefaction fronts for the physical equations and elucidate the differences in multi-D rarefaction phenomena when compared with the 1-D case{\color{black}''}.} The existing techniques for multi-dimensional shocks fronts can not be applied.  One of the main technical obstacles in constructing rarefaction waves is, according to Majda on page 154 of \cite{MajdaBook}, {\color{black}``}the dominant signals in rarefaction fronts move at characteristic wave speeds{\color{black}''}, i.e., the surfaces bounding the rarefaction wave regions are {\bf characteristic} hypersurfaces. As a matter of fact, rarefaction fronts could not satisfy the uniform stability condition, and the linearized equations would suffer loss of derivatives. These difficulties are coupled with the strong initial singularity at the center, further complicating the analysis.

The first known results on the construction of multi-dimensional rarefaction waves were due to Alinhac in the late 1980's. He proved the local existence and uniqueness of multi-dimensional rarefaction waves for a general hyperbolic system in his seminal papers \cite{AlinhacWaveRare1} and \cite{AlinhacWaveRare2}, which include scalar conservation laws and compressible Euler equations as special examples. Alinhac has introduced several innovative techniques to deal with the singularity of rarefaction waves. He designed an ingenious Nash–Moser type scheme based on non-isotropic Littlewood–Paley decomposition to overcome the derivative loss. He reformulated the problem in an approximate characteristic coordinate system which blows up at the initial discontinuity. He also introduced the celebrated {\color{black}``}good unknown{\color{black}''} for the linearized equations.  A key part of his proof was finding an approximate ansatz for rarefaction waves up to sufficiently large order near the singularity. The treatment of the characteristic boundary was also crucial to the Nash-Moser scheme. 

However, Alinhac's scheme \cite{AlinhacWaveRare1} suffer from loss of normal derivatives, persisting even for one-space-dimensional rarefaction waves and even {\color{black}at the linear level.} In addition, the estimates  were obtained in weighted spacetime norms which are degenerate near the rarefaction fronts.

Alinhac's approach \cite{AlinhacWaveRare1} was employed to study the combinations of shocks and rarefaction waves in \cite{Li1991}. Wang and Yin in \cite{WangYin} adapted Alinhac's scheme to rarefaction waves  in  steady supersonic flow around a sharp corner. Other elementary wave patterns such as contact discontinuities were studied in \cite{Coulombel-Secchi2008ASENR, Coulombel-Secchi2004IUMJ} by Nash-Moser schemes. We also mention the recent paper of Wang and Xin \cite{WangXin} which proves the existence of contact discontinuities for ideal compressible MHD in Sobolev spaces, utilizing the boundary regularizing effect of the transversal magnetic field to avoid loss of derivatives.

\subsection{A rough version of the main results}\label{subsection:rough-version-main-results}
\subsubsection{The setting}

We consider the two dimensional Euler flow \eqref{eq: Euler in Euler coordinates}. The initial data is a small perturbation of the plane rarefaction data. More precisely, $\{x^1=0\} \subset \Sigma_0$ is the flat initial curve {\color{black}(we assume that the data is periodic in $x^2$ and identify $\{x^1=0\}$ with a circle)}. On the half plane $\{x^1 > 0\}$ the initial motion is assumed to be irrotational and isentropic. We assume the data on {\color{black}$\{x^1>0\}$} is a small perturbation of constant states away from vacuum of order $O(\varepsilon)$. We remark that $\varepsilon = 0$ corresponds precisely to {\color{black}the one dimensional} constant case \eqref{eq:1D-rarefaction-wave}. 

The initial data on $\{x^1 > 0\}$ determines a region $\mathcal{D}_0$ (its development)  with a characteristic hypersurface denoted by $C_0$ as its boundary. On the region adjacent to $C_0$ we shall construct a family of multi-dimensional rarefaction waves that converge to the 1D picture \eqref{eq:1D-rarefaction-wave} as the perturbation $\varepsilon \to 0$. It takes two steps to complete this goal. In the current paper, we establish a stable nonlinear energy estimates in Sobolev spaces. We will prove the existence and convergence in a {\color{black}follow-up} paper \cite{LuoYu2}.

The rarefaction region will be studied in the acoustical coordinate $(t,u,\vartheta)$. The level sets of $u$, denoted by $C_u$, correspond to rarefaction fronts emanating from the initial curve and foliate the rarefaction wave region with {\color{black}foliation ``density'' approximately of size $\frac{1}{t}$.} For an arbitrary small constant $\delta>0$, we study the energy propagation on the spacetime domain $\mathcal{D}$ bounded by $C_0, C_u$ and $\Sigma_{\delta}, \Sigma_t$. The picture is depicted as follows:  

\begin{center}
	\includegraphics[width=4in]{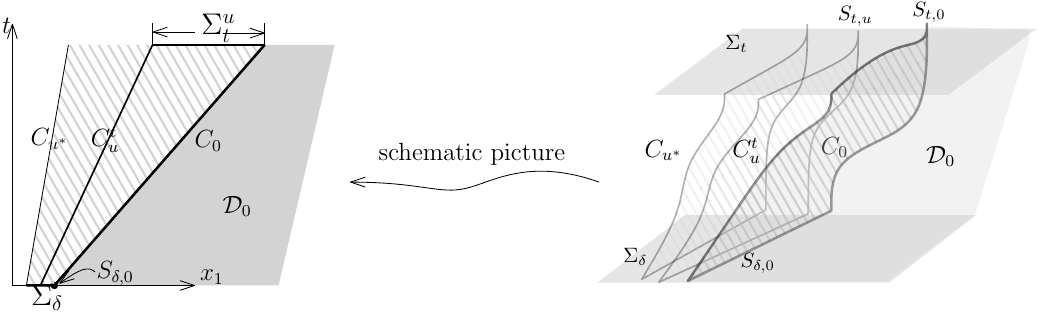}
\end{center}

The data on $C_0$ is determined a priori by the data {\color{black}on the half space $\{x^1 > 0\}$.} However, the data on $\Sigma_{\delta}$ is not known in advance. In fact, for rarefaction waves the domain $\Sigma_{\delta}$ shrinks to the initial curve as $\delta \to 0$. The data on $\Sigma_{\delta}$ has to be carefully chosen and it is indeed determined asymptotically by those on $C_0$.

\subsubsection{A rough version of the main a priori energy estimates}
\begin{MainTheorem}
	{\color{black}There exist} a small positive constant $\varepsilon_0$ and a positive integer $n$ so that, for all $\varepsilon < \varepsilon_0$, for data of size $O(\varepsilon)$ satisfying the initial ansatz \eqref{initial Iinfty} and \eqref{initial irrotational} specified in Section \ref{subsection:energy-ansatz-initial-data} (the data will be constructed in the second paper \cite{LuoYu2}), for $(t,u)\in[\delta,t^*]\times [0,u^*]$,  we have the following energy estimates:
	\begin{align*}
		\mathscr{E}_{\leqslant n}|_{\Sigma_{t}} + \mathscr{F}_{\leqslant n}|_{C_u} \leqslant \mathscr{E}_{\leqslant n}|_{\Sigma_{\delta}} + \mathscr{F}_{\leqslant n}|_{C_0} + \mathbf{error}{\color{black}.}
	\end{align*}
The error term $\mathbf{error}$ is bounded by $C \varepsilon$ where the universal constant $C$ is independent of $\varepsilon$. The notations $\mathscr{E}_{\leqslant n}|_{\Sigma_{t}}$ and $\mathscr{F}_{\leqslant n}|_{C_u}$ denote the higher order energy (up to $n$-th order) and flux through $\Sigma_t$ and $C_u$ respectively. 
\end{MainTheorem}

\subsubsection{Remarks on the main theorem}

\begin{remark}
	Our work provides a rather complete answer to Majda's open question on multi-dimensional rarefaction waves \cite{MajdaBook} (see also Section \ref{subsection:prior-results-rarefaction-waves}) in the case of two dimensional ideal gas dynamics. The three dimensional ideal gas dynamics can be handled exactly in the same way.
\end{remark}

\begin{remark}[Linear estimates]

We provide energy bounds for linearized acoustical waves in rarefaction wave regions without loss of derivatives. We use energy and flux norms in standard Sobolev spaces so that the estimates do not degenerate even at the boundaries of the rarefaction wave regions. In contrast, Alinhac's works on multi-dimensional rarefaction \cite{AlinhacWaveRare1, AlinhacWaveRare2} and the subsequent {\color{black}follow-up} papers \cite{Li1991,ChenLi,WangYin} rely on linear estimates in spacetime co-normal spaces that lose derivatives and degenerate near boundaries. 
\end{remark}

\begin{remark}[Nonlinear estimates]
	We provide  nonlinear energy bounds {\color{black}which are} uniform with respect to $\varepsilon$ and $\delta$. There are no loss of derivatives in our nonlinear energy estimates.  The previous work  \cite{AlinhacWaveRare1,AlinhacWaveRare2,Li1991,ChenLi,WangYin} employ Nash-Moser iteration scheme with loss of derivatives {\color{black}at the nonlinear level.} 
\end{remark}

\begin{remark}[The geometry of hypersurfaces and the stability]

We give a complete description of the geometry of the rarefaction wave fronts $C_u$. Roughly speaking, it is completely captured by the second fundamental form $\chi$. If $\chi$ vanishes, the problem reduces to one-dimensional rarefaction waves.  

We also provide a detailed description of the following stability picture which is quantified by the parameter $\varepsilon$: as $\varepsilon\rightarrow 0$, the multi-dimensional rarefaction waves constructed in the paper converge to the classical centered rarefaction waves in one spatial dimension.
\end{remark}

\begin{remark}
	We focus on compressible Euler equations for an ideal gas, in contrast to Alinhac's work \cite{AlinhacWaveRare1,AlinhacWaveRare2} for a general hyperbolic system. The picture of acoustic waves, especially the acoustical geometry, is indispensable for the linear and nonlinear estimates in the current paper. This indicates that multi-dimensional rarefaction waves in gas dynamics exhibit stronger stability than those for a general hyperbolic system.  
\end{remark}

\subsubsection{Remarks on the new ingredients of the proof}

The proof is done in the geometric framework, pioneered by Christodoulou and Klainerman \cite{CK} on the nonlinear stability of Minkowski spacetime and developed by Christodoulou \cite{ChristodoulouShockFormation} on shock formation for Euler equations. 

Let $\mu$ be the inverse density of characteristic hypersurfaces. The monotonicity of $L(\mu) < 0$ is essential to the stability mechanisms in shock formation, while in rarefaction wave regions we have $L(\mu) > 0${\color{black}.}  This reflects the following fundamental physical picture: characteristic hypersurfaces converge in shock formation and diverge from the singularity in rarefaction waves. This new picture poses new obstacles. We find several new mechanisms for rarefaction waves: 
\begin{remark}
	We obtain energy estimates for linearized wave equations in rarefaction regions, which is completely different from the coercive control of angular derivative first discovered in Christodoulou's work on shock formation \cite{ChristodoulouShockFormation} (based on $L(\mu) < 0$) and the subsequent works \cite{ChristodoulouMiao,LukSpeck2D,LukSpeck3D,AbbresciaSpeck2,HKSW,HLSW,Miao,MiaoYu,Speck1,Speck2,Speck3}.
\end{remark}

\begin{remark}
	We develop a new {\color{black}approach to nonlinear estimates based on a new null frame as commutators for rarefaction waves,} in contrast to the descent schemes in shock formation invented by Christodoulou \cite{ChristodoulouShockFormation} and employed in \cite{ChristodoulouMiao,LukSpeck2D,LukSpeck3D,AbbresciaSpeck2,HKSW,HLSW,Miao,MiaoYu,Speck1,Speck2,Speck3}.
\end{remark}

We will further discuss the above remarks in Section \ref{section: difficulty}.

{\color{black}

\subsection{Applications to the nonlinear stability of the Riemann problem: existence and uniqueness}\label{Sec:Riemann-Problem}

We recall the  solutions of two families of rarefaction waves to the classical Riemann problem. Let $U_l=\begin{pmatrix}
  v_l \\
  \rho_l
   \end{pmatrix}$ and $U_r=\begin{pmatrix}
  v_r \\
  \rho_r
   \end{pmatrix}$ be two constant states for the velocity $v$ and the density $\rho$. 
If we take the following initial data for the Euler equations \eqref{eq: Euler in Euler coordinates}
\[
\begin{pmatrix}
  v \\
  \rho
   \end{pmatrix}\Big|_{t=0}=\begin{cases}
&U_l, \ \ x^1<0, \\
&U_r, \ \ x^1>0,
\end{cases}
\]
with specifically chosen $U_l$ and $U_r$,  the solution for $t>0$ consists of a back rarefaction wave and a front rarefaction wave. The rarefaction waves are illustrated as follows in the second picture:
\begin{center}
	\includegraphics[width=4.5in]{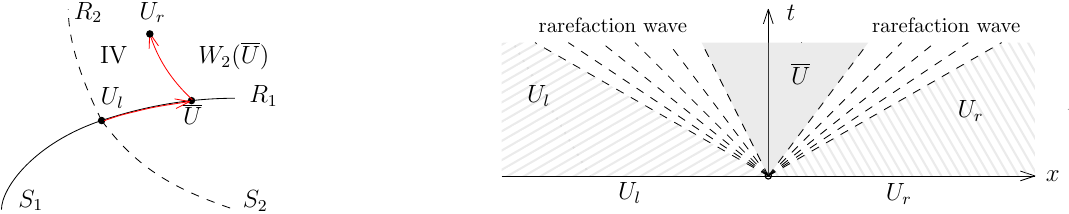}
\end{center}
The shape of the two families of rarefaction wave fronts is like a \emph{fan}. The first picture illustrates the way of choosing the $U_l$ and $U_r$. We refer readers to Chapter 17 of Smoller's textbook \cite{Smoller} for details. 

By virtue of the energy estimates in {\bf Main Theorem}, we will show in \cite{LuoYu2} that, for sufficiently small smooth perturbation of $U_l$ and $U_r$ at $t=0$ of size $O(\varepsilon)$, there still exists a solution to \eqref{eq: Euler in Euler coordinates} defined for $t\in (0,1]$ which asymptotically converges to the above 1D solution as $\varepsilon\rightarrow 0$. In fact,  the shape of the rarefaction fronts becomes an \emph{opened book} and the structure is the same as in one dimension, see the following picture and \cite[Theorem 3]{LuoYu2} for a detailed description of the rarefaction front geometry:
\begin{center}
\includegraphics[width=2in]{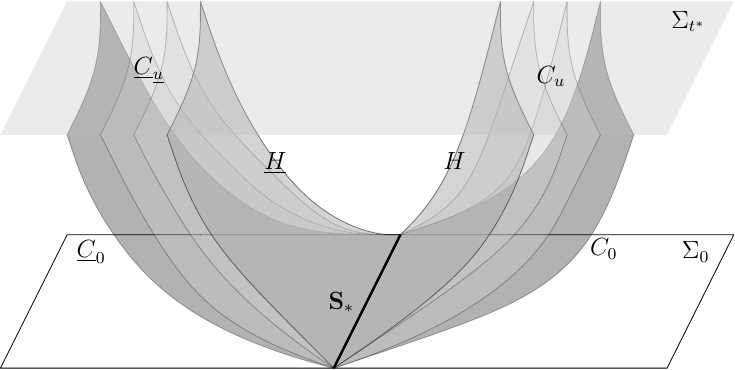}
\end{center}

Furthermore, we will show that the solution constructed in the above picture is indeed unique among all the measurable bounded functions satisfying the entropy inequality, see \cite[Proposition 2.11]{LuoYu2}.  This is among the largest possible classes of functions in the 1D conservation laws that one expects uniqueness.

\begin{remark}
	We will also show that the family of front rarefaction waves that can connected to the initial characteristic hypersurface $C_0$ is unique, see \cite[Proposition 2.14]{LuoYu2}. Note that the solution generated by the initial data on $\Sigma_{\delta}$ on the left-hand-side of $C_0$ is not unique, due to the non-uniqueness of the extension of data from $C_0$ to $\Sigma_{\delta}$. Nevertheless, uniqueness is retrieved in the limit as $\delta \to 0$. 
\end{remark}

\begin{remark}
	We have made the following assumption for the sake of simplicity: the initial discontinuity is across a \emph{straight} line (a circle) on $\Sigma_0$. To go beyond this limitation, i.e., extending the theorems to the general case when the initial discontinuity is an arbitrary smooth curve, we believe that one should make the following modifications: the Riemann invariants  should be chosen adapted to the curve of singularity:
	\[w=\frac{1}{2}\big(\frac{2}{\gamma-1}c+(\widehat{T'})^i\psi_i\big), \ \ \wb=\frac{1}{2}\big(\frac{2}{\gamma-1}c-(\widehat{T'})^i\psi_i\big), \ \ \psi_2= (\widehat{X'})^i\psi_i,\]
	where $X'$ and $\widehat{T'}$ are the unit tangential vector field and the unit normal vector field of the separating curves; {\color{black}see Section \ref{section_acoustical geometry} for the notations $\psi_i$ and compare with the Riemann invariants defined in \eqref{def: Riemann invariants}}. We should also  choose $X'$ and $T'=\kappar \widehat{T'}$ as commutator vector fields. The construction of the initial data can be derived in the same manner. However, the proof of the \emph{a priori} energy estimates would be much longer since the equations for the new Riemann invariants and the commutators of $X'$ and $T'$ will be more complicated. We plan to construct centered rarefaction waves for data across a curved surfaces with vorticity and entropy in three dimensions in future work.
\end{remark}

}

\subsection{Recent progress on shock formation and shock development problem}

\subsubsection{Multi-dimensional shock and singularity formation}\label{Sec:history shock}

 As we mentioned before, in multi-dimensional cases, without the framework of BV spaces, it requires new insights to understand the characteristic hypersurfaces of the Euler equations. {\color{black}One of the major breakthroughs} in this direction is the work \cite{ChristodoulouShockFormation} of Christodoulou on the formation of shocks for an irrotational and isentropic fluids on three dimensions. {\color{black}Some of} his ideas to understand the geometry of the acoustical waves can be traced back to the monumental work \cite{CK} of  Christodoulou and Klainerman on the proof of the nonlinear stability of Minkowski spacetime. We will discuss this insight in details later on.

Nevertheless, Sideris contributed the first blow up result for the compressible Euler equations in three dimensions. His work \cite{Sideris} exhibits stable blow-up for the classical solutions associated to an open set of initial data. However, since the approach is based on the proof by contradiction, it provides no description on the nature of the singularity. In \cite{AlinhacWaveShock1}, \cite{AlinhacWaveShock2} and \cite{AlinhacWaveShock3}, Alinhac has contributed a series of work  on the formation of  singularities for two dimensional compressible Euler equations. He treated the radially symmetric solutions and obtained precise estimates on the time parameter for the first blow-up point. Later on, Alinhac in \cite{AlinhacWaveBlowUp1} and \cite{AlinhacWaveBlowUp2} has exhibited stable blow-up for a class of quasilinear wave equations without any symmetry assumptions on the data. The blow-up mechanism is due to the collapse of the characteristic hypersurface foliations. Though Alinhac did not prove shock formation for Euler equations, his results can be in principle extended to the fluid case since compressible Euler equations in the irrotational case can be reduced to a quasi-linear wave equations similar to the type of equations in \cite{AlinhacWaveBlowUp1} and \cite{AlinhacWaveBlowUp2}. We also remark that Alinhac's estimates suffer derivative losses on the top order quantities of the characteristic hypersurfaces. Hence, his framework is based on a Nash–Moser iteration scheme.

In the monograph \cite{ChristodoulouShockFormation} published in 2007, Christodoulou made a breakthrough and he proved stable shock formation for irrotational relativistic Euler equations in $3+1$ dimensions. Moreover, his work also described the geometry of the boundary of the maximal development of the data. As his work inspired most of the recent developments on shock formation, it is worthy of giving a more detailed account on several of the original ideas appeared in \cite{ChristodoulouShockFormation}.
\begin{itemize}[noitemsep,wide=0pt, leftmargin=\dimexpr\labelwidth + 2\labelsep\relax]
\item Geometrization via the acoustical metric.

The acoustical metric defined on the maximal development of the initial data offers a new Lorentzian spacetime viewpoint to study the Euler equations. Under this set-up, the entire picture becomes an analogue to the theory of general relativity where one studies the Einstein equations. Therefore, the techniques developed in the proof of the stability of the Minkowski space \cite{CK} by Christodoulou and Klainerman can be borrowed to study Euler equations. Indeed,  \cite{CK} offers an insightful paradigm to study quasilinear partial differential equations: assuming the underlying geometry, the quasilinear systems behave very much like a linear system. This new idea leads to a detailed description of the system from multiple perspectives:
\begin{itemize}

\item The characteristic hypersurfaces become the null hypersurfaces with respect to the acoustical metric.  We can mimic the study of null hypersurfaces in general relativity to study the characteristic hypersurfaces for compressible Euler equations.

\item The formation of shocks can be captured by the inverse density $\mu$ of the characteristic hypersurfaces. This quantity can be represented in a geometric way and it also enjoys a geometric transport equation.

\item The formation of shocks is characterized by the non-equivalence of acoustical coordinates and standard Cartesian coordinates. In particular, the solution behaves in a smooth way up to shocks in acoustical coordinates . 
\end{itemize}

\item A coercive mechanism tied to the shock formation.

Compared to the usual case on non-singular spacetime, even the energy estimates for linear wave equations can degenerate near shocks. This degeneration is the most challenging obstacle to the energy method. Christodoulou found an elegant mechanism to overcome the degeneration. He showed that near shocks the inverse density $\mu$ satisfies a monotonicity condition. Therefore, the uncontrolled terms due to the degeneration is coercive in the sense it has a favorable sign. This is a unexpected discovery and it is the key to the entire proof. 

\item A descent scheme to close the top order energy estimates.

\cite{CK} also uses a descent scheme to study the top order estimates. Together with the previous coercive mechanism, the descent scheme can close the energy estimates in finite order Sobolev norms without using the Nash-Moser schemes.

\end{itemize}

The work of Christodoulou has great impacts in the field and it has stimulated several important progress on shock formation for Euler equations and in other settings. In \cite{ChristodoulouMiao}, Christodoulou and Miao proved the shock formation for the non-relativistic compressible Euler equations. Luk and Speck \cite{LukSpeck2D} proved the shock formation for two dimensional  barotropic compressible Euler equations and later on in \cite{LukSpeckNull} and {\color{black}\cite{LukSpeck3D}}  they  have extended their work to the three dimensional compressible Euler equations  with vorticity and entropy. The most recent work \cite{AbbresciaSpeck1} \cite{AbbresciaSpeck2} of Abbrescia and Speck further studies the structure of the singular boundary of the maximal developments of the data. For the new developments on the shock formation in other hyperbolic equations under the geometric frame work of Christodoulou, we refer the readers to \cite{HKSW}, \cite{HLSW}, \cite{Miao}, \cite{MiaoYu}, \cite{Speck1}, \cite{Speck2} and \cite{Speck3}. The work of Christodoulou also inspired research on the low regularity theory on Euler equations, see the series of work \cite{DCLMS}, \cite{DS} and \cite{Speck4}  and also a sharper result \cite{Wang} of Q. Wang.

We also mention the new progress on the blow-up of compressible Euler equations in multi-dimensions  that are not built upon Christodoulou’s framework. In \cite{BuckmasterShkollerVicol2DShock}, \cite{BuckmasterShkollerVicol3DShock} and \cite{BuckmasterShkollerVicol3DVorticityCreation}, Buckmaster, Shkoller and Vicol used different approaches to construct shock formation with vorticity and entropy. The approach is  based on the perturbation of a Burgers shock and it works all the way to the time of first blowup and provided isolated  singularities. See also \cite{BuckmasterIyer} for a result on the unstable behavior of the singularity.

In the recent breakthrough \cite{MRRS1} and \cite{MRRS2}, Merle, Rapha\"{e}l, Rodnianski and Szeftel constructed the implosion type singularity for the compressible three-dimensional Navier-Stokes and Euler equations in a suitable regime of barotropic laws. This is a new family of blow-up solutions for compressible fluids and the density becomes infinity at the blow-up point. See also \cite{Biasi2021} for numerical investigation.

\subsubsection{Multi-dimensional shock development problem}\label{Sec:history development}

The shock development problem is aiming at a more complete picture: to understand how the smooth solution to the Euler equations forms shocks and then develops a shock surface. The work \cite{ChristodoulouShockFormation} is the first step towards the shock development problem. Christodoulou has made another breakthrough \cite{ChristodoulouShockDevelopment} towards the resolution of the shock development problem. Starting with the shock from the work \cite{ChristodoulouShockFormation}, he constructed the shock surface in the restricted regime (there is no jump in entropy and vorticity across shocks) without any symmetry assumptions. The theorems were proved for relativistic Euler equations and they can be translated to the non-relativistic compressible Euler equations {\color{black}by letting the speed of light go to infinity.}

{\color{black}Under symmetry assumptions,} the problem has many features similar to the one dimensional case. {\color{black}There are a few works} that solved the shock development problem in this set-up. In \cite{Yin}, Yin first studied the problem for the three dimensional Euler equations in spherical symmetry. It has been revisited by Christodoulou and Lisibach using different methods in \cite{ChristodoulouLisibach}. In \cite{BDSV}, Buckmaster, Drivas, Shkoller and Vicol solved the shock development problem for solutions to two dimensional Euler equations with vorticity and entropy in azimuthal symmetry. Very recently, using the same method as in \cite{ChristodoulouLisibach}, Lisibach in \cite{ Lisibach 1} and \cite{ Lisibach 2} also studied the shock reflection problem and  interactions of two shocks in plane-symmetry.

\subsection{Technical remarks on  \cite{ChristodoulouShockFormation}, \cite{CK}  and  \cite{AlinhacWaveRare1} }

\subsubsection{Remarks on Christodoulou \cite{ChristodoulouShockFormation} and Christodoulou-Klainerman \cite{CK}}\label{section: Shock and Minkowski}
We briefly describe two fundamental ideas from  Christodoulou \cite{ChristodoulouShockFormation} and Christodoulou-Klainerman \cite{CK} respectively. They will play a central role in the current work.
\begin{itemize}[noitemsep,wide=0pt, leftmargin=\dimexpr\labelwidth + 2\labelsep\relax]
\item The coercivity of energy norms of angular directions near shocks, see \cite{ChristodoulouShockFormation}.

In the near-shock region, i.e., the inverse density of the characteristic hypersurfaces $\mu$ close to $0$, the energy estimate encounters a fundamental difficulty: the energy integrals for rotational directions look like $\displaystyle\int_{\mathcal{D}} \mu |\slashed{\nabla}\psi|^2$ where $\psi$ denotes a component for the acoustical wave, while the error integrals have $\slashed{\nabla}\psi$ components \textit{without} any $\mu$ factor. Thus, when $\mu\rightarrow 0$ near shocks, the disparity in $\mu$ shows that the error integrals can not be bounded by the energy integrals. This even {\color{black}happens at the linear level.}

Christodoulou had made the following remarkable discovery:  although the aforementioned degeneration in the rotational directions is due to the formation of shocks, it is also resolved by the mechanism of shock formation. Since the initial value of $\mu$ is almost $1$ and near shocks $\mu$ is close to $0$, the value of $\mu$ should decrease along the direction $L$ which is towards the shock. Using a transport equation of $\mu$ as well as the acoustical wave equations, he showed that $L(\mu)<0$. He also showed that main contribution of the error integrals without factor $\mu$ for $\slashed{\nabla}{\psi}$ must be in the form $\displaystyle \int_{\mathcal{D}} L(\mu)\cdot |\slashed{\nabla} \psi|^2$. The negative sign of $L(\mu)$ manifests a miraculous coercivity of the energy estimates.  With the help of the sign of $L(\mu)$, this enables one to control all the error terms involving the rotational directions of $\psi$.

We remark that the sign of $L(\mu)$ in the rarefaction wave region is positive so that it is not favorable to the energy estimates near singularities. Therefore, we need completely new mechanism in the current work. Please see the next section for some technical remarks on this point. 

\item The last slice argument from Christodoulou-Klainerman \cite{CK}.

We have mentioned the basic ideas of the stability of the Minkowski space \cite{CK}, such as constructions of null hypersurfaces and energy identities in the spacetime etc, are indispensable to study the acoustical geometry defined by solutions to the compressible Euler equations. The work \cite{CK} also contributes another important idea: the so-called last slice argument. 

We give a schematic review on the last slice argument. In \cite{CK}, the authors ran a bootstrap argument to solve vacuum Einstein equations on a spacetime region $\mathcal{D}_T$ which can be regarded as $[0,T]\times \mathbb{R}^3$. We use $\Sigma_t$ to denote the spacelike hypersurface $\{t\}\times \mathbb{R}^3$ for $t\in [0,T]$. The initial data were given on $\Sigma_0$. In order to construct the null cone foliations of $\mathcal{D}_T$, the usual procedure is as follows: we first choose a sphere foliation on $\Sigma_0$, say the geodesic spheres with respect to a fixed point on $\Sigma_0$. Next, for each sphere in the foliation, it emanates an out-going null cone. The collection of these null cones give the desired foliation of $\mathcal{D}_T$. If one uses this foliation in the proof of stability of Minkowski spacetime, it is very likely that one can not close the top order estimates on the underlying geometry.

{\color{black}Instead of choosing} sphere foliation from the initial slice $\Sigma_0$, Christodoulou and Klainerman's last slice argument has chosen the initial sphere foliation from the last slice $\Sigma_T$. The incoming null cones emanating from these spheres at the last slice give the foliation of the spacetime. Rather than a technical trick, the last slice argument is indeed deeply related to the nature of the problem.  Since the problem is about the asymptotic stability, the larger the time parameter $t$ is, the better the Minkowski spacetime approximates $\Sigma_t$. Therefore, the construction of the geodesic spheres should be more precise on $\Sigma_T$ than $\Sigma_0$.

In the current work, we will construct approximate data close to the singularity. The na\"ive way of construct initial foliation of the null hypersurfaces also suffers a similar loss as above. We will use ideas reminiscent of the last slice argument to get the correct initial foliation by tracing back the data from singularity. This is done in the second paper \cite{LuoYu2} of the series. Please see the next section for some technical remarks on this point, e.g. the fourth remark in Section \ref{Section:schematic description} and d) of Section \ref{section:nonlinear estimates}.
\end{itemize}

\subsubsection{Remarks on the work \cite{AlinhacWaveRare1} of Alinhac}
We summarize the main results of \cite{AlinhacWaveRare1}.  The author studied a general quasilinear symmetrizable hyperbolic system
 \begin{equation}\label{eq:hyper system}
 \partial_t v+ A_1(v)\partial_x v+ A_2(v)\partial_y v =0
 \end{equation}
 where $v(t,x,y) \in \mathbb{R}^N$, $(x,y)\in \mathbb{R} \times \mathbb{R}^{n-2}$, $N\geqslant 1$, $n\geqslant 2$, $t\in \mathbb{R}$ and the coefficient matrices $A_1$ and $A_2$ are smooth in $v$. It is assumed that for all $\eta\in \mathbb{R}^{n-2}$, $A_1(v)+\eta \cdot A_2(v)$ has a simple real eigenvalue $\lambda(v,\eta)$ which is genuinely non-linear.  Let $x=\varphi_0(y)$ be {\color{black}a smooth hypersurface} on $\mathbb{R}^{n-1}$ so that $\varphi_0(0)=\nabla \varphi_0(0)=0$. We pose $v_+(x,y)$ on $x> \varphi_0(y)$ and $v_-(x,y)$ on $x<\varphi_0(y)$ as the initial data.

The data $(v_+,v_-)$ is assumed to satisfy \emph{the compatibility condition}: 

\medskip

(Compatibility). \ For all $y\in \mathbb{R}^{n-2}$, there exists a one-dimensional centered rarefaction wave in the direction $\eta = -\nabla\varphi_0(y)$ joining $v_-(\varphi_0(y),y)$ and $v_+(\varphi_0(y),y)$, corresponding to the simple real eigenvalue $\lambda(\cdot,-\nabla\varphi_0(y))$.

\medskip

To define the rarefaction waves, we consider a domain $\mathcal{R}=\{(t, u,\vartheta) \in \mathbb{R}^n|t>0,u\in(0,1)\}$. Let $\Psi:\overline{\mathcal{R}}\rightarrow \mathbb{R}^n$ be a continuous map where we use the standard Cartesian coordinates $(t,x,y)$ on the target. We assume that $\Psi \in C^\infty(\mathcal{R})$. It is given by
\[\Psi: (t,u,\vartheta)\mapsto (t, \psi(t,u,\vartheta),\vartheta).\]
We also assume a key linear expansion condition $\psi_u(t,u,\vartheta) = t\overline{\psi}(t,u,\vartheta)$ where $\overline{\psi}$ is positive on $\overline{R}$. The image of {\color{black}$\Psi$} is the dihedral angle region $\mathcal{S}$ defined by
\[\mathcal{S}=\{(t, x,y) \in \mathbb{R}^n|t>0, \psi(t,0,y)<x<\psi(t,1,y)\}.\]
\begin{center}
\includegraphics[width=4in]{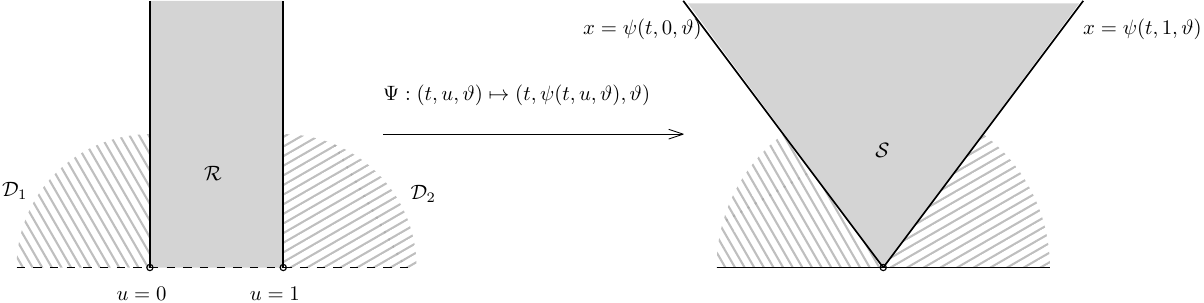}
\end{center}
Assume that $v(t,x,y)$ solves \eqref{eq:hyper system} on $\mathcal{S}$. Then, on $w(t,u,\vartheta)=v\circ \Psi$ solves the following equation
 \begin{equation}\label{eq:hyper system 2}
 L(w,\psi)w=\partial_t w+ \frac{1}{\psi_u}\left(A_1(w)-\psi_t \cdot \mathbf{I}-\psi_y\cdot A_2(w)\right)\partial_u w+ A_2(w)\partial_\vartheta w =0,
 \end{equation}
 where $\mathbf{I}$ is the $N\times N$ identity matrix.
 
 A rarefaction wave is defined as a juxtaposition of three smooth solutions to \eqref{eq:hyper system} defined on three regions $x<\varphi(t,0,\vartheta)$, $\varphi(t,0,\vartheta)<x<\varphi(t,1,\vartheta)$ and $x>\varphi(t,1,\vartheta)$ with $t>0$ so that on $t=0$ they agree with $v_-$ on $x< \varphi_0(y)$ and with $v_+$ on $x> \varphi_0(y)$. 

The main theorem proved in \cite{AlinhacWaveRare1} can be stated as follows: 
there exists a smooth rarefaction wave verifying the above conditions, for $t$ sufficiently small.

\medskip

We now list several key aspects of the proof in \cite{AlinhacWaveRare1} and we also compare them with the current work. 
\begin{itemize}[noitemsep,wide=0pt, leftmargin=\dimexpr\labelwidth + 2\labelsep\relax]
	\item Alinhac's seminal work \cite{AlinhacWaveRare1} used the Nash-Moser iteration scheme to construct multi-dimensional rarefaction {\color{black}waves for a general hyperbolic system.} The Nash-Moser technique was necessary due to the loss of regularity (even in the linear estimates). 
	
	In this work, we establish energy estimates for rarefaction waves in compressible isentropic Euler equations with the ideal gas equation of state. We do not lose derivatives and we can close the energy {\color{black}estimates in standard Sobolev spaces} $H^s$ with $s\geqslant 6$.
	
	\item \cite{AlinhacWaveRare1} used an approximately characteristic coordinate system $(t,u,\vartheta)$ on the region  $\mathcal{R}$ to blow up the rarefaction wave region $\mathcal{S}$ so that the estimates in $\mathcal{R}$ become regular, see the above picture. One of {\color{black}the main technical constraints} in the proof is to require the hypersurfaces defined by $u=0$ and $u=1$ to be characteristic in the process of iteration.  The hypersurface $u=a$ with $a\in (0,1)$ may not be characteristic. The boundary conditions at $u=0$ should be very carefully chosen in each step and this is one of the main difficulties solved in \cite{AlinhacWaveRare1}. Furthermore, \cite{AlinhacWaveRare1} requires the compatibility condition and the three solutions on $\mathcal{D}_1, \mathcal{D}_2, \mathcal{R}$ must be iterated simultaneously to correct the boundaries. 
	
	We construct the acoustical coordinate system $(t,u,\vartheta)$ by using the acoustical (Lorentzian) metric $g$ defined by the solution. The hypersurfaces $C_u$ are inherently characteristic (null) and correspond to rarefaction wave fronts emanating from the initial discontinuity curve.  In particular, we do not pose any boundary condition on the left boundary $u=u^*$ (counterpart of $u=0$ in \cite{AlinhacWaveRare1}) and do not require compatibility conditions. Instead, we describe all rarefaction waves which can be connected to the initial characteristic hypersurface $C_0$, similar to the one dimensional picture.
	
	\item \cite{AlinhacWaveRare1} introduced the celebrated {\color{black}``}good unknown{\color{black}''} for the linearized  equations in the blow-up variables. However, the linearized equations are still singular and lose derivatives in higher order estimates due to the characteristic nature of rarefaction wave. A crucial step in \cite{AlinhacWaveRare1} was the construction of higher order approximate solutions near the singularity via Taylor expansions in time. The corrections to approximate solutions of sufficiently high order satisfy linear estimates in weighted spacetime norms {\color{black}which degenerate} near the boundary $u=0$ and $u=1$.
	
	Our work relies on the physical mechanism of acoustic wave propagation. The wave equations avoid the loss of derivatives in linearized first order system. Based on rarefaction wave energy ansatz, we derive a new linear energy estimates in Sobolev spaces. In particular, our estimates do not degenerate on boundaries of rarefaction wave regions.

	\item To implement the Nash-Moser iteration schemes,  \cite{AlinhacWaveRare1} introduced a chain of weighted Sobolev type spaces (based on anisotropic Littlewood-Paley decomposition) to handle the normal derivatives. The scheme and estimates in \cite{AlinhacWaveRare1} indeed suffer from loss of normal derivatives due the the degeneration of weight functions. This loss persists even for one dimensional rarefaction waves. In particular, since the smallness in \cite{AlinhacWaveRare1} is posed on the time interval, it does not provide error estimates which measures the closeness of the solution to the one dimensional rarefaction waves.
	
	{\color{black}We obtain top order estimates which quantify the perturbations relative to one dimensional case in terms of the small parameter $\varepsilon$. In particular, we can characterize the geometry of the rarefaction front $C_u$ by the second fundamental form $\chi$. The vanishing of $\chi$ indicates that solution reduces to 1-D rarefaction waves. See Section \ref{Sec:Riemann-Problem} for the picture of the rarefaction front geometry. }
\end{itemize}

\subsection{Comments on the proof: difficulties, ideas, and novelties}\label{section: difficulty}

We address the major difficulties in the construction of rarefaction waves, and briefly describe the ideas to overcome them.

\subsubsection{A schematic description}\label{Section:schematic description}
\begin{itemize}[noitemsep,wide=0pt, leftmargin=\dimexpr\labelwidth + 2\labelsep\relax]
	\item[1)] Characteristic propagation speed and loss of derivatives. 
	
	In contrast to shock fronts which are non-characteristic hypersurfaces, rarefaction waves are inherently hyperbolic characteristic problems. Because of the characteristic nature, the linearized rarefaction wave equations could not satisfy the uniform stability condition according to Majda \cite{MajdaBook}, and would suffer loss of normal derivatives. According to Alinhac \cite[Section 3.3]{AlinhacWaveRare0}, the linearized rarefaction wave equations for a general hyperbolic system lose $\frac{s}{2}$ derivatives in $H^s$-norm estimates (see Majda and Osher \cite{MajdaOsher} for detailed analysis). This motivated Alinhac's Nash-Moser schemes in a chain of weighted co-normal spaces.

	To overcome the loss of derivatives in linearized equations, we rely crucially on the following facts for sound waves in gas dynamics: they satisfy wave equations. This is not true for general hyperbolic systems. In particular, the characteristic component $\wb$ (this is one of the Riemann invariants defined in \eqref{def: Riemann invariants}) satisfies a wave equation, and could be used to recover the normal derivatives. This is the basis for linear estimates in Sobolev spaces.
	
	\item[2)] A difficulty in linear energy estimates absent in shock formation.
	
	Owing to the initial discontinuities, the linearized wave equations are singular in rarefaction wave regions. This leads to the degeneracy of angular derivatives estimates, in analogue of the shock formation mentioned in Section \ref{section: Shock and Minkowski}. Unfortunately, on account of the reverse sign of $L(\mu)$, the crucial coercive mechanism in shock formation fails to work for rarefaction waves. A new perspective is needed to understand linear estimates in rarefaction wave regions.
	
	We will provide a detailed asymptotic analysis of rarefaction waves near singularities. The formulation in terms of Riemann invariant variables $\{\wb,w,\psi_2\}$ plays a key role. We derive precise hierarchical energy ansatz not only on the initial Cauchy hypersurface $\Sigma_\delta$ but also on the characteristic hypersurface $C_0$. The hierarchical ansatz, primarily in the form of vanishing of normal derivatives $T(\psi), \psi \in \{w,\psi_2\}$, forms the basis of a new mechanism. This provides linear estimates for acoustic rarefaction waves in Sobolev spaces. The key {\color{black}technical tool for the linear estimates} is a refined Gronwall type inequality. It relies crucially on the positive energy flux through the characteristic hypersurfaces $C_u$ (rarefaction fronts). Furthermore, the energy flux estimates also provide a means to directly control the geometry of rarefaction fronts, which is missing in previous works \cite{AlinhacWaveRare1,AlinhacWaveRare2}.

	\item[3)] A difficulty in nonlinear estimates.
	
	The nonlinear energy estimates are also coupled with the bounds on acoustical geometry. The key geometric quantity is ${\rm tr}(\chi)$, i.e., the mean curvature of the rarefaction fronts.  The standard method to estimate top derivatives of ${\rm tr}(\chi)$, due to Christodoulou \cite{ChristodoulouShockFormation}, is to renormalize the propagation equation $L\big({\rm tr}(\chi)\big)$ which retrieves the loss of one derivative. To handle the singular renormalized equation near singularity, for shock formation the key idea is to make use of the minus sign of $L(\mu)$ which eliminates the leading singular term. Unfortunately, the idea breaks down due to the positive sign of $L(\mu)$ in rarefaction waves. The blow-up of top order derivatives of ${\rm tr}(\chi)$ seems to be inevitable near singularities, rather than a technical issue. This is by far \textit{the most challenging part} of this work.
	
	The top order derivatives of geometric quantities such as ${\rm tr}(\chi)$ are indeed coming from deformation tensors of commutator vector fields.  The strategy is to avoid derivatives on geometric quantities by commuting with a new null frame. We introduce a new \textit{non-integrable} null frame $\{\Lr,\Lbr, \Xr\}$ adapted to the Riemann invariants $\{\wb,w,\psi_2\}$. The covariant nature of the Euler equations allows us to express the associated deformation tensors $\ ^{(\Zr)} \pi$ in terms of Riemann invariants.  The Riemann invariants and the new frame allow us to use the null structure of the solutions to control most of the error terms. Meanwhile, there is still a price to pay due to commutation with the new frame. The worst possible error terms are related to $^{(\Zr)}\pi_{\Lr\LR}$. It can not be controlled by the energy and becomes the primary threats to the energy estimates.  Its resolution relies on the following observation: due to the expansion nature of rarefaction waves, the density of the gas decreases across the rarefaction fronts.  This shows that the worst top order error term has a favorable sign so that it is coercive.
	
	\item[4)] The control of geometry and a hidden vanishing.
	
	Even though the energy estimates can be closed in the second null frame, the non-integrability of the frame, i.e., it is not tangential to the rarefaction fronts, creates new {\color{black}obstacles.} There are error terms similar to $^{(\Zr)}\pi_{\Lr\LR}$. Only this time we no longer have a favorable sign to control them. 
	
	The last ingredient to control the acoustical geometry is the following discovery: there is an extra \textit{vanishing of angular derivatives} for the maximal characteristic speed $v^1+c$. It is a hidden structure of the multi-dimensional rarefaction waves without an analogue in {\color{black}the one dimensional theory,} and it can not be directly predicted from the energy estimates. To capture this extra vanishing, we must trace back the data from singularity, reminiscent of the last slice argument mentioned in Section \ref{section: Shock and Minkowski}. Furthermore, we show that {\color{black}the extra vanishing indeed propagates} by a key commutation formula.
\end{itemize} 

\medskip

In the following, we outline the proof and explain the ideas in more details.

\subsubsection{Linear estimates}

We use the acoustical coordinates $(t,u,\vartheta)$ and we foliate the spacetime by the level sets $C_u$ of $u$ with {\color{black}density $\frac{1}{\mu}$ of order $O(t^{-1})$ at time $t$.} We also use null frame $\{L,\Lb,\Xh\}$ where $L, \Xh$ are tangent to $C_u$. See the figure in Section \ref{subsection:rough-version-main-results} and the precise definitions in Section \ref{section_acoustical geometry}.

We study the following linear wave equation defined on $\mathcal{D}(\delta) = \{t \in [\delta,t], u \in [0,u^*]\}$:
\begin{equation*}
	\Box_{g}\psi = \rho,
	\end{equation*}
	where {\color{black}$\Box_g \sim \Xh^2(f) - \mu^{-1}L\big(\underline{L}(f)\big) + \cdots$}  in the null frame, see \eqref{eq:wave operator in null frame}. The goal is to obtain energy estimates of $\psi$ independent of $\delta \to 0$ (so that $\mu \to 0$ approaching $\Sigma_{\delta}$).

As mentioned in Section \ref{section: Shock and Minkowski}, as $\mu$ is close to $0$, the degeneracy of angular derivative estimates is the main difficulty. In previous works on shock formation, the favorable negative sign of $L(\mu)$ provides a (negative) coercive term in the form  $\displaystyle \int_{\mathcal{D}} L(\mu)\cdot |\slashed{\nabla} \psi|^2$ in the error integral. For rarefaction waves, the degeneration still presents while  $L(\mu)$ become positive. Hence, the coercivity is lost in the energy estimates. Therefore, we have to handle a non-integrable factor of size $\frac{1}{t}$ coming from the degeneration of $\mu$. 

We make the following comparison to illustrate the difficulty. Schematically, let $E(t)$ be the energy at time $t>0$ and $t=\delta$ is the initial time. In the worst scenario, $E(t)$ satisfies the following estimate:
\[E(t)\leqslant E(\delta)+\int_{\delta}^t\frac{C_0}{\tau}E(\tau)d\tau.\]
We may compare this with the energy inequalities often appeared in small-data-global-existence {\color{black}problems} for nonlinear wave equations:
\[E(t)\leqslant E(0)+\int_{0}^t\frac{C}{\tau}E(\tau)d\tau.\]
In the second case, under the ansatz that $E(t)$ is bounded, we can use Gronwall's inequality to show that $E(t)=O\left(\log(t)\right)$. There is a $\log(t)$-loss but the estimates is at least useful to construct long time solutions with lifespan at least of size $O(e^\frac{1}{\varepsilon})$. In the first case, the Gronwall's inequality gives
\[E(t) \leqslant \left(\frac{t}{\delta}\right)^{C_0}E(\delta).\]
When $\delta\rightarrow 0$, unless the initial energy $E(\delta)$ decays in the correct way, the above estimate blows up for arbitrary small time $t$. The loss comes directly from the data and it is the main technical obstacle even for constructing local solutions (regardless the regularity issue). In fact, the analysis indicates that the linearized wave equations in rarefaction wave region {\color{black}are ill-posed for generic data} in Sobolev spaces.

We solve this problem by introducing the correct energy ansatz and the Riemann invariants. On the technical level, we also need a refined Gronwall's inequality.
\begin{itemize}[noitemsep,wide=0pt, leftmargin=\dimexpr\labelwidth + 2\labelsep\relax]
\item[1)]  The energy ansatz and the Riemann invariant variables.

Suggested by the asymptotic analysis of rarefaction waves near the initial singularity, we introduce the Riemann invariant variables $\{\wb,w,\psi_2\}$. It not only allows us to approximately diagonalize the Euler equations in the null direction, but also reveals a hierarchy of energy ansatz that plays a dominant role throughout the proof. 

We define energy norms on a constant $t$-slice $\Sigma_t$ associated with outgoing and incoming null directions $L$ and $\Lb$. For different derivatives and different Riemann invariants, we have a hierarchy on the associated energies. The essence of the energy bounds for rarefaction waves can be reflected in the following manner: 
\begin{itemize}
	\item If $\psi \neq \wb$ or $k\geqslant 1$ ({\color{black}$k$ is the number} of derivatives applied on $\psi$), for all possible commutation vector fields $Z$, the $L^2$-norms of  the outgoing derivatives $L(Z^k \psi)$ and rotational derivatives $\Xh(Z^k \psi)$ are of size $\varepsilon^2$; The $L^2$-norms of  the incoming derivatives $\Lb(Z^k \psi)$ are of size $t^2\varepsilon^2$.
	
	\item The $\Lb \wb$ is of size $1$ {\color{black}and it} will generate most of the linear terms in the energy estimates. These linear terms will be the main enemies in the proof.
\end{itemize}  
We believe that it is the unique energy ansatz which can be proved for {\color{black}the linear wave equation} $\Box_{g}\psi=0$ in rarefaction wave region. See Section \ref{subsection:heuristics-energy-ansatz} for a heuristic derivation of the energy ansatz.  We will construct initial data on $\Sigma_{\delta}$ satisfying such ansatz in the {\color{black}forthcoming paper} \cite{LuoYu2}.

We note that this part is similar in spirit to Alinhac's construction of approximation solution in \cite{AlinhacWaveRare1, AlinhacWaveRare2}. The difference is that we have to derive much more precise hierarchical ansatz not only on the initial Cauchy hypersurface $\Sigma_\delta$ but also on the characteristic hypersurface $C_0$. Furthermore, instead of a diagonalization method depending on the characteristic hypersurfaces, we use the decomposition of Riemann invariant variables and it avoids the loss of derivatives.

\item[2)] The refined Gronwall type inequality and the positive energy flux.

As we mentioned above, we have to use the following bounds:
\[E(t)\leqslant E(\delta)+\int_{\delta}^t\frac{C_0}{\tau}E(\tau)d\tau \ \ {\implies} \ \ E(t) \leqslant \left(\frac{t}{\delta}\right)^{C_0}E(\delta).\]
The correct ansatz gives the decay of the form $E(\delta) \lesssim \delta^2 \varepsilon^2$. In order to get a bound independent of $\delta$, it requires $C_0 < 2$. This restriction does not seem to be realistic, since for higher order estimates we encounter many error terms generated from commutations and sources.

The existence for a positive energy flux through the characteristic hypersurfaces $C_u$ provides a way to implement the above idea. It turns out that most of the errors can always be bounded by $\displaystyle a_0^{-1}\int_0^u F (t,u') du' + a_0 \int_{\delta}^{t}\frac{E(t',u)}{t'}dt'$. We remark that $a_0$ is a small constant at disposal and this retrieves the smallness. We also note that there is a big constant $a_0^{-1}$ for the flux term, but it is not harmful; see Section \ref{section: bilinear error integrals}. 

In fact, we have the following Gronwall type  inequality for the energy $E(t,u)$ and flux $F(t,u)$:
\begin{align*}
	{E}(t,u)+ {F}(t,u)\leqslant At^2+ B \int_0^u F (t,u') du'   + C \int_{\delta}^{t}\frac{E(t',u)}{t'}dt'.
\end{align*}
See Lemma \ref{refined Gronwall} for the proof. We remark that the $At^2$ in the above inequality is consistent with the energy ansatz. We have
\begin{equation*}
	{E}(t,u)+F(t,u)\leqslant  3Ae^{Bu} t^{2}
\end{equation*}
provided $e^{B u^*}C\leqslant 1$. This Gronwall type inequality enables us to obtain linear energy estimates for rarefaction waves. We emphasize that the estimates are in Sobolev spaces, and in particular do not degenerate at the boundaries of the rarefaction wave region, in contrast to previous work \cite{AlinhacWaveRare1,AlinhacWaveRare2}. Furthermore, the energy flux also controls the geometry of rarefaction fronts. 
\end{itemize}

\subsubsection{Nonlinear estimates}\label{section:nonlinear estimates}
The nonlinear energy estimates are always coupled to the control of the underlying geometry. The acoustical geometry is indeed controlled by two functions:  the mean curvature ${\rm tr}(\chi)$ of $C_u$ and the inverse density $\mu$ of the foliation by $C_u$. This is also the case for shock formation, see \cite{ChristodoulouShockFormation,ChristodoulouMiao}.

As we mentioned, the reverse sign of $L(\mu)$ compared to the case of shock formation is not only an obstacle for linear estimate but also is tied to the loss of derivatives on the top order derivatives of ${\rm tr}(\chi)$ and  $\mu$. This loss might prevent us from closing the nonlinear energy estimates in finitely many derivatives.

{\color{black}This scenario is illustrated} as follows. Schematically, the highest order term $Z^N({\rm tr}(\chi))$ satisfies the following equation:
\[	L\left(Z^N({\rm tr}(\chi))\right)=Z^{N+2}(\psi)+\cdots.
\]
The terms in the $\cdots$ are of lower orders and $Z^{N+2}(\psi)$ is one order higher in derivatives than $Z^N({\rm tr}(\chi))$. Thus, a direct integration along $L$ would cause a loss of one derivative. In \cite{ChristodoulouShockFormation}, using the wave equation satisfied by $\psi$, Christodoulou finds a neat algebraic expression of $Z^{N+2}(\psi)$ as $Z^{N+2}(\psi)=L\big(Z^{N+1}(\psi')\big)$ up to lower order terms. Therefore, we can move the top order term to the lefthand side  to derive
\begin{equation}\label{schematic equation for chi 1}
	{\color{black}L\left(Z^N({\rm tr} \chi)-Z^{N+1}(\psi')\right)}=\frac{L(\mu)}{\mu}Z^N({\rm tr} \chi)+\cdots.
\end{equation}
The terms on the righthand side of \eqref{schematic equation for chi 1} are of lower order. Thus, the above trick avoids the loss of derivatives. On the other hand, if we convert \eqref{schematic equation for chi 1} in $L^2$ norms, we arrive at
\begin{equation}\label{schematic equation for chi 2}
	{\color{black}L\big\|Z^N({\rm tr} \chi)-Z^{N+1}(\psi')\big\|_{L^2}}=\frac{L(\mu)}{\mu}\big\|Z^N({\rm tr} \chi)-\left(Z^{N+1}(\psi')\right)\big\|_{L^2}+\cdots.
\end{equation}
This is another place where the sign of $L(\mu)$ plays a crucial role. For shock formation $L\mu<0$, the first term on the righthand can be dropped. This crucial step avoids the unacceptable loss in $\mu$. For rarefaction waves, $L(\mu)$ becomes positive. Integrating \eqref{schematic equation for chi 2} leads to a loss in $\mu$. The loss is even worse since we integrate from the singularity so that we can not bound this term even for very short time. This is the second  difficulty tied to the sign of $L(\mu)$ and it prevents us from closing the top order derivative estimates for $\mu$ or ${\rm tr}\chi$.

The difficulty is resolved by the combination of the following observations:
\begin{itemize}[noitemsep,wide=0pt, leftmargin=\dimexpr\labelwidth + 2\labelsep\relax]
\item[a)]  A non-integrable null frame adapted to the Riemann invariants.

The motivation for introducing the new null frame is to avoid the higher order derivatives of $\chi$ and $\mu$.  For $\psi \in \{\wb,w,\psi_2\}$, by commuting derivatives $Z^N$ with $\Box_g \psi$, we have 
\[ \Box_g (Z^N \psi)=\rho_N. \] 
If we use $Z\in \{L,\Lb,\Xh\}$ as commutators, the source term $\rho_N$ contains $Z^N(\chi)$ coming from the deformation tensors of $Z$. As we explained, this term can not be controlled.

Since the wave equations for the Riemann invariants are covariant, {\color{black}we are free to choose any frame.} The new null frame $\{\Lr,\Lbr, \Xr\}$ is determined by the initial discontinuity surface (a flat curve in our setting) and the acoustical metric $g$ (given directly by the Riemann invariants). In particular, the new frame can be explicitly expressed in terms of the Riemann invariants $\{\wb,w,\psi_2\}$. In contrast, the first null frame is implicitly defined, i.e., we have to solve $\mu$ by integrating along $L$. Since $g$ and $\Zr\in \{\Lr,\Lbr, \Xr\}$ can all be explicitly written in  $\psi\in \{\wb,w,\psi_2\}$, commuting with $\Zr^N$ can only contribute terms of the form $\Zr^{k}(\psi)$ in $\rho_N$. These terms have a better chance to be directly controlled by the energy norms via Gronwall type inequalities.

The new null frame also brings in additional difficulties. They generate new error terms, see 4) of Section \ref{Section:schematic description} and c) below. Furthermore, since the standard null frame $\{L,\Lb,\Xh\}$ adapts naturally to the hypersurfaces $C_u$ and the energy estimates, we have to handle the transformation between two frames. We give the following example to show the challenges related to the change of frames. We have a transport equation $L (\chi) = - \frac{\gamma+1}{\gamma-1}\Xh^2(c^2) + \cdots$ to bound $\chi$. To use the energy ansatz,  we have to change to the new frame $\{\Lr,\Lbr, \Xr\}$. The difference $\Xr - \Xh$ leads to
\begin{equation}\label{remark: a L chi} L (\chi) = -\frac{\Xh(\Xh^1)}{c^{-1}\mu}\Tr(c)   - \frac{\gamma+1}{\gamma-1}\Xr^2(c^2) + \cdots,
\end{equation}
where $\TR = \frac{1}{2}(\Lbr - c^{-2}t \Lr)$.
Unless $\chi\equiv 0$ at the initial singularity which is the one dimensional case, the first term on righthand side still suffers a loss of $\mu$. {\color{black}However, in general we have $\chi = O(\varepsilon)$; see \eqref{initial Iinfty}. This is one of the main difficulty in controlling the acoustical geometry; see the following Point d) for its resolution with the `extra vanishing' of $\Xh\Xr(v^1+c)$.}

\item[b)] The null structures with respect to the Riemann invariants. 

The source terms of the wave equations for the Riemann invariants are all in the covariant form $g^{\alpha\beta}\partial_\alpha \psi\partial_\beta \psi'$. Since $\Lbr$ is null, the contraction with the acoustical metric $g$ guarantees at most one $\Lbr \wb$ term appearing in each of the source terms, i.e., no terms of the type $\Lbr \psi \cdot \Lbr \psi'$.  We notice that there is no smallness in $\Lbr \wb$. Therefore, the worst contribution in the energy estimates from the source terms are at least linear hence borderline terms. See also Remark \ref{remark:null structure}. The deformation tensors associated with the commutators also exhibit similar null structures. In view of the fact that the flux term on the characteristic hypersurfaces $C_u$ contains no $\Lbr$-derivative components, these null structures allow us to deal with most of the error terms, by reducing them to one of the bilinear error integrals in Section \ref{section: bilinear error integrals}. These bilinear error integrals can be bounded by the energy flux through rarefaction fronts. 

\item[c)] The favorable sign from the ‘rarefaction’ effect.

One of top order error terms can be computed as 
\begin{align*}
&\int_{\mathcal{D}}\frac{\Lbr(\psi)}{4t}\Zr^{\beta}\Lbr\big(c^{-1}\,^{(\Zr_0)}\pi_{\Lr \Lr}\big)\cdot \Lbr(\Zr^\alpha(\psi)) =-\int_{\mathcal{D}}\frac{1}{2t}\Zr^{\beta}\Lbr\Zr_0\big(\frac{\gamma+1}{2}\TR(\wb) + \frac{\gamma-3}{2}\TR(w)\big)\cdot\Lbr(\Zr^\alpha(\psi)){\color{black}.}
\end{align*}
The worst case happens for $\psi = \wb$ where we have $\Lbr(\psi) = \Lbr(\wb) \approx -1$. Furthermore, if all the commutators in $\ZR^{\beta}$ are the transversal direction $\Tr$,  it violates the null structures so that it is even not in the scope of the refined Gronwall type inequality. Fortunately, we can use the fact that $\Lbr(\wb) < 0$  so that this term can be ignored. This is due to the {\color{black}expansive nature} of rarefaction waves and it reflects the fact that along the transversal direction the density of the gas is decreasing. This is essential for top order energy estimates in the new null frame. {\color{black}See the estimates of the major error term $\mathbf{I}_{1,1}$ in Section \ref{section sigma34}.}

\item[d)]  A hidden extra vanishing and the new null frame.

Observe that the component $^{(Z)}\pi_{LL}$ for the commutators $Z$ from the first null frame vanishes, while $^{(\Zr)}\pi_{\Lr\LR}\neq 0$ for the new null frame. {\color{black}It originates from the commutation} of the null generators of $C_u$ with the new null frame. In energy estimates, we will encounter the following terms:
 \begin{equation}\label{eq: extra vanishing}
 t^{-1} \,^{(\Xr)}\pi_{\Lr \Lr} = \frac{\Xr(v^1+c)}{t},\ \ t^{-1}\,^{(\Tr)}\pi_{\Lr \Lr} 
 = t^{-1}(1 + \frac{\gamma+1}{2}\TR(\wb) + \frac{\gamma-3}{2}\TR(w)).
 \end{equation}
 We refer to Section \ref{section:Energy identities for higher order terms} for details. The energy ansatz suggests these terms are of size $O(t^{-1}\varepsilon)$ and $O(t^{-1})$. The $t^{-1}$ factor is out of reach for the energy estimates.

In fact, these two terms are of size $O(\varepsilon)$ and $O(1)$. It comes from the delicate choice of the initial data near singularity.  It turns out that the geometry of initial rarefaction wave fronts must be matched in an exact way on $\Sigma_{\delta}$. Even a slight deviation would result in uncontrollable errors. For example, we must have the exact constant $-\frac{\gamma+1}{2}$ in front of $\wb$ {\color{black}in \eqref{eq: extra vanishing}. The} wave fronts are defined by tracing back the data from singularity, reminiscent of Christodoulou and Klainerman's last slice argument in \cite{CK}.

This extra vanishing is also key to retrieve the loss in \eqref{remark: a L chi} of a). We can derive the following equation for $\Xh(\Xh^1)$:
\[ L\big(\Xh(\Xh^1)\big) =  \frac{\Xh(\Xh^1)}{t} + \Xh\Xr(v^1+c) + \cdots.\]
The extra vanishing of $\Xh\Xr(v^1+c)$ provides enough $t$-factors so that we can bound $\Xh(\Xh^1)$ by Gronwall's inequality. We can then come back to \eqref{remark: a L chi} to control $\chi$.

Given the aforementioned importance, we define $\yr = \frac{\Xr(v^1+c)}{t}$ and we expect $\yr$ to have size $O(\varepsilon)$. However, the size of $\yr$ can not be obtained directly from the energy estimates. It turns out that the behavior of $\yr$ can be captured  by a commutation formula which is again related to the nature of rarefaction waves. We observe that $\yr$ appears {\color{black}through} the commutator $[\Lr,\Xr] = \yr \cdot \Tr + \cdots$. We apply this formula to $\wb$ to derive
\[ \yr \cdot \Tr \wb = \Lr \Xr (\wb) - \Xr \Lr(\wb) + \cdots = \Lr \Xr (\wb) + \Xr^2(\psi_2) + \cdots{\color{black},} \]
where we used the Euler equation to substitute $\Lr(\wb)$. The terms on the righthand side are bounded by the energy estimates. This formula  encodes the key information of the rarefaction waves which is $\Tr\wb \approx -1$; {\color{black}see Section \ref{section: yr zr}}. {\color{black}The bounds on $\yr$} play a dominant role in treating the error terms violating the null structures and also in the comparisons of two different null frames; {\color{black}see Remark \ref{remark:estimates-Lambda}, Section \ref{section sigma23} and Section \ref{section: closing B2 for alpha2 psi=wb}}.

\smallskip

This extra vanishing seems hard to be detected using the standard null frame $(L,\Lb,\Xh)$.  The asymptotic analysis shows that $\Xh(v^1+c)$ is of size $O(\varepsilon)$. To our best knowledge, this unexpected vanishing has not appeared in physical or mathematical literature.

\end{itemize}

\subsection{Future work}

In the one dimensional case, given any data on $x_1>0$, we can connect its development by a rarefaction wave in a unique way on the left. This is shown in the first one of the following pictures. For the Riemann problem with an open set of data given in \eqref{eq:Riemann-problem-1D}, as shown in the second one of the following pictures, $U_l$ {\color{black}is first connected to} $\overline{U}$ by a back rarefaction wave and then connected to $U_r$ by a front rarefaction wave. Therefore, the initial discontinuity is resolved by two families of rarefaction waves. \begin{center}
\includegraphics[width=5.5in]{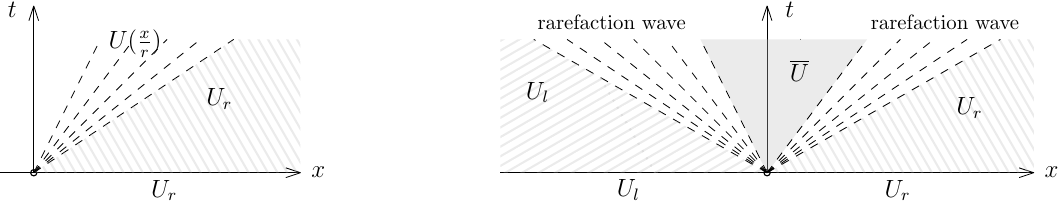}
\end{center}

In the second paper \cite{LuoYu2} of the series,  we will construct initial data on $\Sigma_\delta$ so that the assumptions in Section \ref{subsection:energy-ansatz-initial-data} are satisfied. We also show that, when $\delta\rightarrow 0$, the solutions corresponding to the given data on $\Sigma_a$ converge to a multi-dimensional centered rarefaction wave connecting to the given data given on $x_1>0$. This proves the existence of centered rarefaction wave and exhibits the first picture in multi-dimensional case. As applications, we also prove that small perturbations of data in \eqref{eq:Riemann-problem-1D} leads to the second picture. This proves the non-linear stability of the Riemann problem for two families of rarefaction waves for higher dimensional compressible Euler equations. 

The current work and the second paper \cite{LuoYu2} focus on the irrotational flow because sound waves are the core problems in rarefaction waves and they already reveal the nature of the subject. We will study general Euler flows with vorticity and entropy in three dimensions in the third paper of the series.

\subsection{Organization of the paper}
In Section \ref{section_acoustical geometry}, we recall the acoustical geometry and introduce two sets of null frames. We also introduce Riemann invariants and diagonalize the Euler equations. In Section \ref{section:energy-method-main-theorem}, we introduce the energy identities and the bootstrap ansatz. We also state the main theorem. In Section \ref{section:preparations-for-energy-estimates}, we control the the acoustical geometry and we obtain pointwise bounds for the Riemann invariants. In Section \ref{section:linear-energy-estimates}, we establish the energy estimates for linear equations which are applied to the lowest order energy estimates in Section \ref{section:lowest-order-energy-estimates}. In Section \ref{section:lower-order-estimates}, we derive lower commutator estimates including the bounds on $\yr$ and $\zr$. In Section \ref{section:higher-order-energy-estimates}, we close the energy estimates. The last section is devoted to close the pointwise bootstrap assumptions.

\section{Rarefaction waves and acoustical geometry}\label{section_acoustical geometry}
In terms of the enthalpy $h = e + pV$ ($V$ and $e$ are the specific volume and specific energy, respectively), the Euler system \eqref{eq: Euler in Euler coordinates} is equivalent to
\[\begin{cases}
&(\partial_t + v \cdot \nabla)v = -\nabla h,\\
&c^{-2}(\partial_t + v \cdot \nabla) h + \nabla \cdot v = 0.
\end{cases}
\]
For {\color{black}an isentropic ideal gas,} $h$ can be represented in terms of the sound speed, i.e., $h = \frac{1}{\gamma-1} c^2$. 
We consider the case where there exists a velocity potential function $\phi$ so that $v = -\nabla \phi$. Therefore, the fluid is irrotational. The enthalpy $h$ can be expressed as $h = \partial_t \phi - \frac{1}{2}|\nabla \phi|^2$. The Euler system is then equivalent to the following quasi-linear wave equation in Galilean coordinates $(t,x_1,x_2)$
\begin{equation}\label{eq:Euler:isentropic:irrotational}
g^{\mu \nu} \frac{\partial^2 \phi}{\partial x^\mu \partial x^\nu}  = 0.
\end{equation}
where we have used the Einstein summation convention and the acoustical metric $g$ is defined by
\[{\color{black}g = - c^2 dt^2 + }\sum_{i=1}^2 (dx^i - v^idt)^2.\]
The equation \eqref{eq:Euler:isentropic:irrotational} is the Euler-Lagrange equation corresponding to {\color{black}the Lagrangian density} $L = p(h)$.

Let $\{\phi_{\lambda}: \lambda \in (-1,1)\}$ be a family of solutions of \eqref{eq:Euler:isentropic:irrotational} such that $\phi_0 = \phi$. We call $\psi =\frac{d \phi_{\lambda}}{d \lambda}\big|_{\lambda = 0}$ a variation of $\phi$ through solutions.  Such families of solutions often arises from the symmetry of the spacetime and of the equations, e.g., we may take $\phi(t+\lambda, x_1,x_2)$, $\phi(t, x_1+\lambda,x_2)$ or $\phi(t, x_1,x_2+\lambda)$. We use the following notation to denote the corresponding variation through solution in the rest of the paper:
\[\psi_0= \frac{\partial \phi}{\partial t}, \ \psi_1= \frac{\partial \phi}{\partial x_1}=-v^1, \psi_2= \frac{\partial \phi}{\partial x_2}=-v^2.\]
By differentiating \eqref{eq:Euler:isentropic:irrotational} in $\lambda$, we derive that the variation $\psi$ satisfies a linear  wave equation corresponding to a metric $\widetilde{g}$: 
\[\Box_{\widetilde{g}} \psi = 0,\]
where $\widetilde{g}$ is a conformal change of the acoustical metric $\widetilde{g} = \Omega g$ and $\Omega = \frac{\rho}{c}$. In terms of the original acoustical metric $g$, it is equivalent to
\begin{equation}\label{eq: conformal Euler}
\Box_{g} \psi = -\frac{1}{2}g(D\log(\Omega), D\psi),
\end{equation}
where $D$ is the gradient define with respect to the acoustical metric $g$.

We assume that the fluid flows on the 2-dimensional tube $\Sigma_0=\left\{(t,x_1,x_2)\big| t=0, x_1\in \mathbb{R}, 0\leqslant x_2\leqslant 2\pi \right\}$. We identify $(t,x_1,0)$ and $(t,x_1,2\pi)$  so that we only consider the problem with {\color{black}periodic conditions in $x_2$,}  i.e., $\Sigma_0 =\mathbb{R}\times \mathbb{R}\slash 2\pi\mathbb{Z}$. The initial data of the system are posed on {\color{black}$x_1\geqslant 0$} (the grey region) by
	\[(v,c)\big|_{t=0}=\big(v^1_0(x_1,x_2),v_0^2(x_1,x_2), c_0(x_1,x_2)\big),  \ \ x_1\geqslant 0.\]
\begin{center}
\includegraphics[width=3in]{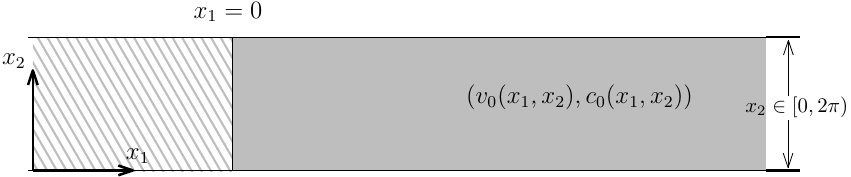}
\end{center}
If  $v\big|_{t=0}=(\overline{v_0},0)$ and $c\big|_{t=0}=\overline{c_0}$, where $\overline{v_0}$ and $\overline{c_0}>0$ are constants, the problem reduces to the classical one-dimensional centered rarefaction wave; see Section \ref{subsection:1D-Riemann-problem}. In this paper, we consider {\bf the perturbed data} where $v\big|_{t=0}-(\overline{v_{0}},0)$ and $c\big|_{t=0}-\overline{c_0}$ are small in Sobolev norms near $x_1=0$. Let $\mathcal{D}_0$ be the future domain of dependence of the solutions to \eqref{eq:Euler:isentropic:irrotational} with respect to the perturbed data. We use $C_0$ to denote its characteristic boundary.
\begin{center}
\includegraphics[width=3in]{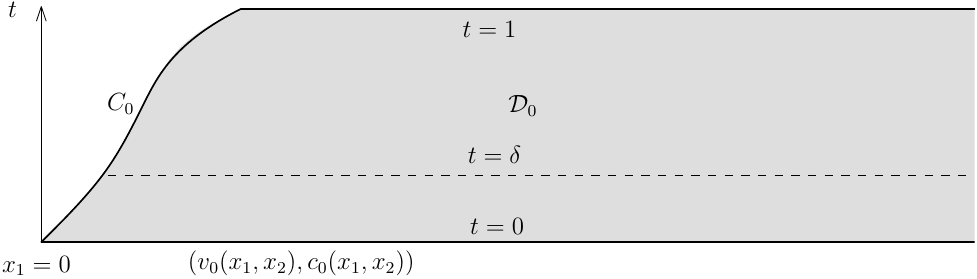}
\end{center}
For small perturbation, we may assume that $\mathcal{D}_0$ at least covers up to $t=1$.

\medskip

Throughout the paper, we use $(x_0,x_1,x_2)=(t,x_1,x_2)$ to denote the Cartesian coordinates on the Galilean spacetime. We use $\Sigma_{t_0}$ to denote the spatial hypersurface $\{(t,x_1,x_2) |t=t_0\}$. We will use a limiting process to construct centered rarefaction waves. We fix a positive parameter $\delta$ (which will be sent to $0$ in the limiting process).  We draw $\mathcal{D}\cap \Sigma_\delta$ as follows:
\begin{center}
\includegraphics[width=3in]{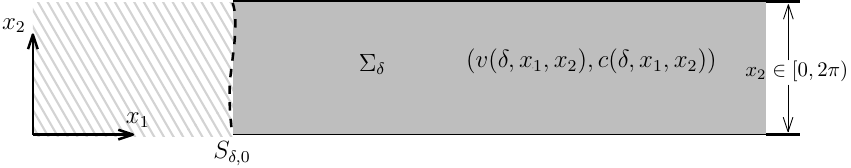}
\end{center}
We define $S_{\delta,0}=\Sigma_\delta\cap C_0=\partial \Sigma_\delta$. It is no longer a straight curve defined by {\color{black}$x_1=\text{constant}$.} The solution $(v,c)$ restricted to $t=\delta$ and on the righthand side of $S_{\delta,0}$ is given by $\big(v^1(\delta, x_1,x_2),v^2(\delta,x_1,x_2), c(\delta,x_1,x_2)\big)$. The data in the rarefaction wave region will be given on $\Sigma_\delta$ on the lefthand of $S_{\delta,0}$. To start with, we choose a smooth function $u$ on $\Sigma_\delta$ so that $S_{\delta,0}$ is given by $u=0$. The lefthand side of $S_{\delta,0}$ on $\Sigma_\delta$ are given by $u>0$.   We will specify data for the Euler equations  for $u \in [0,u^*]$ on $\Sigma_\delta$. The parameter $u^*$, which represents the width of the rarefaction wave, will be determined later on in the proof. It depends on {\color{black}the sound speed} on $C_0$. Once the data is prescribed for $u \in [0,u^*]$ on $\Sigma_\delta$, together with the data on $C_0$, it evolves to the development $\mathcal{D}(\delta)$ according to the Euler equations. In the rest of the paper, since we mainly work in $\mathcal{D}(\delta)$, we use $\mathcal{D}$ to denote $\mathcal{D}(\delta)$. See the shaded region depicted in the following picture:  
\begin{center}
	\includegraphics[width=4in]{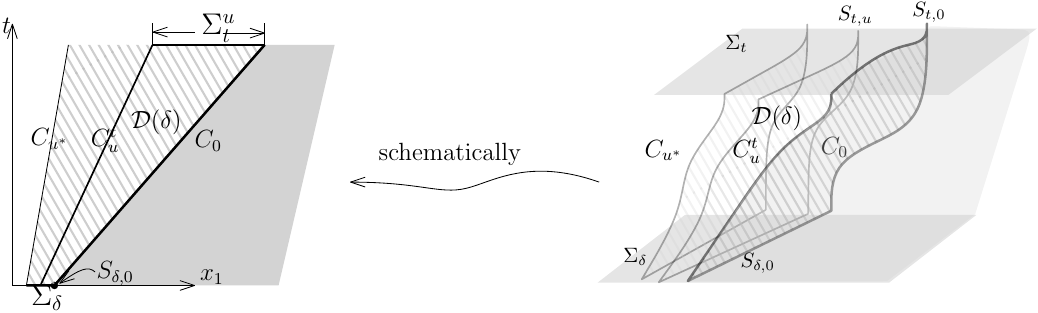}
\end{center}

\subsection{The acoustical coordinate system}

We refer to \cite{ChristodoulouShockFormation} and \cite{ChristodoulouMiao} for details of the construction of the acoustical coordinates. The acoustical coordinate system on $\mathcal{D}$ consists of three smooth functions $t$, $u$ and $\vartheta$. The function $t$ is defined as $x_0$ restricted to $\mathcal{D}$. 

The acoustical function $u$ is {\color{black}already given on} $\Sigma_{\delta}$. In fact, $u$ will be defined in a specific way and it will be given in the course of the construction of the data on $\Sigma_\delta$, see the sequel \cite{LuoYu2}. We define $C_u$ to be the null hypersurfaces consisting of null (future right-going) geodesics emanating from each level set of $u$ on $\Sigma_\delta$. We require $C_u$ to be the level sets of $u$ and this defines $u$ on $\mathcal{D}$. We define $\mathcal{D}(t^*,u^*) =\bigcup_{(t,u) \in [\delta,t^*]\times [0,u^*]}S_{t,u}$. In the rest of the paper, since we will deal with a priori estimates, we assume that $\mathcal{D}=\mathcal{D}(t^*,u^*)$ where $t^*= 1$ and $u^*>0$ are given. We will also use the notation $\mathcal{D}(t,u) =\bigcup_{(t',u') \in [\delta,t]\times [0,u]}S_{t',u'}$, $\Sigma_t^u=\bigcup_{u' \in [0,u]}S_{t,u'}$ and $C_u^{t}=\bigcup_{t' \in [\delta,t]}S_{t',u}$. We also use $\Sigma_t$ to denote $\Sigma_t^{u^*}$.

 We choose {\color{black}the future-pointed vector field $L$} to be the generators of the null geodesics on $C_u$ in such a way that $L(t) = 1$. The inverse density function $\mu$ measures the temporal density of the foliations $\{C_u\}_{u\geqslant 0}$ and it is defined as
\[{\color{black}\mu^{-1} = -g(D t,D u).}\]
Let $S_{t,u} = \Sigma_t \cap C_u$. Therefore, we have $\Sigma_t =\displaystyle \bigcup_{u}S_{t,u}$. The normal vector field $T$ is uniquely defined by the following three conditions:
\[(1) \ T \  \text{is tangent to} \ \Sigma_t; \ \ (2) \ {\color{black}T \ \text{is $g$-perpendicular to}} \ S_{t,u}; \ \ (3) \ Tu = 1.\]


To define the angular function $\vartheta$, we first solve the following system on $C_0$ with data given on $S_{0,0}$:
\[
L(\slashed{\vartheta})=0, \ \ \ \slashed{\vartheta}\big|_{S_{0,0}}=x_2\big|_{S_{0,0}}.
\]
Hence, $\slashed{\vartheta}(\delta)$  is a smooth parametrization of the circle $S_{\delta,0}$. The next step is to define $\vartheta$ on $\Sigma_\delta$ by extending $\slashed{\vartheta}(\delta)$ through the following equation on $\Sigma_\delta$:
\[
T( \vartheta)=0,\ \ \ \vartheta \big|_{S_{\delta,0}}= \slashed{\vartheta}(\delta).
\]
Finally, we use $L(\vartheta)=0$ to extend it to the entire spacetime $\mathcal{D}$ with $\vartheta$ prescribed on $\Sigma_\delta$. This gives the construction of $\vartheta$. Therefore, we obtain the acoustical coordinate system $(t,u,\vartheta)$.

In the acoustical coordinates $(t,u,\vartheta)$, we have
\begin{equation}\label{eq: L T in terms of coordinates}
L = \frac{\partial}{\partial t}, \quad T = \frac{\partial}{\partial u} - \Xi \frac{\partial}{\partial \vartheta},
\end{equation}
where $\Xi$ is a smooth function. In view of the construction, we observe that $T\big|_{\Sigma_\delta} = \frac{\partial}{\partial u}$.

We also define $X=\frac{\partial}{\partial \vartheta}$, $\slashed{g}=g(X,X)$ and {\color{black}the unit vector field} $\widehat{X}=\slashed{g}^{-\frac{1}{2}}X$. 
Therefore, we have
\[g(L,T) = -{\mu}, \quad g(L,L) = g(L,\Xh) = g(T,\Xh)=0, \quad g(\Xh,\Xh) = 1.
\]
We also introduce the vector field $B$ which is uniquely defined by requiring $B(t)=1$ and $B$ is $g$-perpendicular to $\Sigma_t$. It is straightforward to show that  $B$ is the material vector field $B=\frac{\partial}{\partial t}+v$. In particular, we have $g(B,B) = - c^2$. Let $\kappa^2=g(T,T)$, we can also compute that $\mu = c \kappa$. We also define the unit vector $\widehat{T} = \kappa^{-1}T$. The null vector field $L$ can be represented as $L = \frac{\partial}{\partial t} +v-c\widehat{T}$.
\subsection{The geometry of the first null frame}
We refer to \cite{ChristodoulouShockFormation} and \cite{ChristodoulouMiao} for details of computations in this subsection.

We have three kinds of {\color{black}embeddings} $\Sigma_t\hookrightarrow \mathcal{D}$,  $S_{t,u} \hookrightarrow C_u $ and $S_{t,u} \hookrightarrow \Sigma_t$. We use $k$, $\chi$ and $\theta$ to denote the second fundamental forms of these {\color{black}embeddings} respectively:
\[2 c k = \overline{\mathcal{L}}_B g, \ \ 2 \chi = \slashed{\mathcal{L}}_L g, \ \ 2 \kappa \theta = \slashed{\mathcal{L}}_T g.\]
We define the torsion 1-forms $\zeta$ and $\eta$ on $S_{t,u}$ as
\[\zeta(Y) =  g(D_{Y} L, T), \ \ \eta(Y) =  -g(D_{Y} T, L),\]
where $Y$ is any vector field tangent to $S_{t,u}$. We also define the $1$-form $\slashed{\varepsilon}$ as $\kappa \slashed{\varepsilon}(Y)=k(Y,T)$.

Since the $S_{t,u}$'s are 1-dimensional circles, we can represent the tensors by functions. For the sake of simplicity, we use the same symbol to denote the following scalar functions:
\[\chi = \chi(\Xh,\Xh), \  \theta =  \theta(\Xh,\Xh), \ \slashed{k} = k(\Xh,\Xh), \ \zeta = \zeta(\Xh),\ \eta = \eta(\Xh), \ \slashed{\varepsilon}=\slashed{\varepsilon}(\Xh).\]
We also write $\slashed{g}=g(\frac{\partial}{\partial \vartheta},\frac{\partial}{\partial \vartheta})$ and we have $\chi=\frac{1}{2}\slashed{g}^{-1}L(\slashed{g})$ or equivalently $
L(\slashed{g})=2 \slashed{g} \cdot \chi$. These quantities are related by
\[\chi =  c(\slashed{k} - \theta), \ \ \eta = \zeta + \Xh(\mu), \ \ \zeta=\kappa\big(c\slashed{\varepsilon}-\Xh(c)\big).\]

We have the following propagation equation for $\kappa$:
\begin{equation}\label{structure eq 1: L kappa}
L \kappa = m' + e' \kappa 
\end{equation}
where 
\begin{equation}\label{eq: m' e'}
m' = -\frac{\gamma+1}{\gamma-1}Tc, \ \ e'= c^{-1}\widehat{T}^i\cdot L (\psi_i).
\end{equation} 
The repeated indices indicate the summation over $i=1,2$ and $\Th^i$ is the $i$-th component of $\Th$ in the Cartesian coordinates, i.e., $\Th=\sum_{i=1}^2\Th^i \frac{\partial}{\partial x_i}$. There is another way to write $L\kappa$ as
\begin{equation}\label{structure eq 1: L kappa 2nd form}
L \kappa = -Tc -\Th^jT(\psi_j) = -T(v^1 + c) - (\Th^1 + 1)T(\psi_1) - \Th^2T(\psi_2){\color{black}.}
\end{equation}

Since $\Th(\slashed{g}) = 2\slashed{g}\theta$, we have
\begin{equation}\label{defining eq of theta and chi}
	\theta = \Xh^2\Xh(\Xh^1) - \Xh^1\Xh(\Xh^2), \ \ \chi  = -\Xh^i\Xh(\psi_i) - c\Xh^2\Xh(\Xh^1) + c\Xh^1\Xh(\Xh^2).
\end{equation}

We then introduce the left-going null vector field $\underline{L} = c^{-1} \kappa L + 2T$. Hence, we obtain {\bf the first null frame} $(L,\Lb,\Xh)$. This also leads to the second fundamental form $\underline{\chi}$ which is defined by $2 \underline{\chi} = \slashed{\mathcal{L}}_{\underline{L}} g$. We will also work with its scalar version $\underline{\chi} =\underline{\chi}(\Xh,\Xh)$. It can also be computed by $\underline{\chi} =  \kappa(\slashed{k} + \theta)$.

The above geometric quantities can be computed in terms of $\mu$, $\chi$ and $\psi_i$'s as follows:
\begin{equation}\label{structure quantities: connection coefficients}
\begin{cases}
&c k_{ij} = \frac{1}{2}\big(\partial_j v^i + \partial_i v^j\big) = - \partial_i \psi_j = - \partial_j \psi_i, \ \ \slashed{\varepsilon}=-\mu^{-1}\Xh^i T^j \partial_i \psi_j,\\
&\zeta =-\kappa\big(\widehat{T}^j\cdot \Xh(\psi_j) + \Xh(c)\big), \ \ \eta=-\kappa \widehat{T}^j\cdot \Xh(\psi_j)+c\Xh(\kappa),\\
&\chib=2\kappa \slashed{k}-\kappa\alpha^{-1}\chi=c^{-1}\kappa\big(-2\Xh^j\cdot \Xh(\psi_j)-\chi\big).
\end{cases}
\end{equation}

In the first null frame, the Levi-Civita connection $D$ of $g$ can be expressed as:
\begin{equation}\label{structure eq: connection D}
\begin{cases}
D_L L &= \mu^{-1}(L\mu)\cdot L, \ \ D_{\underline{L}} L = -L(c^{-1}\kappa) \cdot L + 2 \eta \cdot \Xh, \ \ D_{L} \underline{L} = -2  \zeta \cdot \Xh,\\
D_{\underline{L}} \underline{L} &= \big(\mu^{-1}\underline{L}\mu + L(c^{-1}\kappa)\big)\underline{L} - 2\mu  \Xh(c^{-1}\kappa)\cdot \Xh,\\
D_\Xh L &= -\mu^{-1} \zeta \cdot  L + \chi \cdot \Xh, \ \ D_\Xh \underline{L} = \mu^{-1} \eta \cdot \underline{L} +  \underline{\chi} \cdot \Xh, \ \ D_L \Xh = -\mu^{-1} \zeta \cdot  L , \\
D_\Xh \Xh &= \frac{1}{2}\mu^{-1} \underline{\chi} \cdot L +  \frac{1}{2}\mu^{-1}\chi \cdot \underline{L} .
\end{cases}
\end{equation}
We also collect the following formulas of the Lie brackets for future uses:
\begin{equation}\label{eq:commutator formulas}
\begin{cases}
&[L,\Xh]=-\chi\cdot \Xh,  \ \ [L,\Lb]=- 2(\zeta + \eta)\Xh+L(c^{-1}\kappa)L, \\
&[L,T]=- (\zeta + \eta)\Xh=-\left(\kappa\left(2c\Xh^i\cdot T(\psi_i)+2\Xh(c)\right)-\Xh(\mu)\right)\Xh,\\
&[T,\Xh]=-\kappa\theta\cdot \Xh, \ \ [\Lb,\Xh]=-\chib\cdot\Xh-\Xh(c^{-2}\mu)L.
\end{cases}
\end{equation}
The wave operator $\Box_g$ can also be decomposed with respect to the first null frame:
\begin{equation}\label{eq:wave operator in null frame}
 \Box_{g} (f) = {\color{black}\Xh^2(f)} - \mu^{-1}L\big(\underline{L}(f)\big) - \mu^{-1}\big(\frac{1}{2}\chi\cdot \underline{L}(f) +\frac{1}{2}\chib\cdot L(f)\big)- 2 \mu^{-1}\zeta \cdot \Xh(f){\color{black}.}
\end{equation}
The null second fundamental form $\chi$ satisfies the following propagation equation
\begin{equation*}
L (\chi) = \mu^{-1} (L\mu) \chi -\chi^2 + R(\Xh,L,\Xh,L),
\end{equation*}
where $R$ is the curvature tensor of $g$. We define the two tensor $w_{\mu\nu}=\partial_{\mu}\psi_{\nu}$ in Cartesian coordinates. The above equation can be expressed explicitly as
\begin{equation}\label{structure eq 2: L chi}
\begin{split}
L (\chi)&=-\frac{\gamma+1}{2}\Xh^2(h)+e\chi -\chi^2 +c^{-2}\left(\frac{\gamma+1}{2}\right)^2 \Xh(h)^2 \\
&\ \ -c^{-2}\left(w(\Xh,\Xh)w(L,L)-w(\Xh,L)^2\right)-(\gamma+1)c^{-2}\left(\Xh(h)w(L,\Xh)-\frac{1}{2}L(h)w(\Xh,\Xh)\right),
\end{split}
\end{equation}
{\color{black} where the function $e$ is defined as $e = \frac{\gamma-1}{2}c^{-2}L(h)+c^{-1}\Th^i \cdot L(\psi_i)$.}

In Cartesian coordinates, we have $\Xh=\Xh^i \partial_i$, $\Th=\Th^i \partial_i$ and $L=\partial_0+L^i \partial_i$. Since $\Xh$ is perpendicular to $\Th$, we know that $\Th^1=-\Xh^2$ and $\Th^2=\Xh^1$. For  $k=1,2$, we have
\begin{equation}\label{structure eq 3: L T on Ti Xi Li}
\begin{cases}
&L(L^{k}) =-\left(L(c) +\widehat{T}^i\cdot L(\psi_i)\right) \widehat{T}^k -\frac{\gamma+1}{2}\Xh(h) \Xh^k,\\
&L(\widehat{T}^{k})= -\kappa^{-1}\zeta \cdot \Xh^k=\left(\widehat{T}^j\cdot \Xh(\psi_j) + \Xh(c)\right)\Xh^k,\\
&T(L^{i}) = L(\kappa) \widehat{T}^i + \eta\cdot \Xh^i=L(\kappa) \widehat{T}^i + \left(-\kappa\left(\widehat{T}^j\cdot \Xh(\psi_j) + \Xh(c)\right)+\Xh(\mu)\right) \Xh^i,\\
&T(\widehat{T}^{i})= -\Xh(\kappa) \Xh^i.
\end{cases}
\end{equation}

\subsection{The geometry of the second null frame}
Using the Cartesian coordinates, we define
\[\Xr=\partial_2, \ \Trh=-\partial_1,  \ \Lr=\partial_t+v-c\Trh=\partial_t+(v^1+c)\partial_1+v^2\partial_2.\]
We also introduce 
\[\kappar=t, \ \ \Tr=\kappar \Trh, \ \ \mur = c\kappar.\]
It is straightforward to check that
\[g(\Lr,\Tr)=-\mur, \ g(\Lr,\Lr)=g(\Lr,\Xr)=0,\  g(\Xr,\Xr)=1, \ g(\Tr,\Tr)=\kappar^2,  {\color{black}\ g(\Tr,\Xr)=0.}\]
We define $\Lbr= c^{-1} \kappar \Lr + 2\Tr$. Hence, we obtain {\bf the second null frame}
$(\Lr, \Lbr, \Xr)$. One can check that
\[g(\Lr,\Lbr)=-2\mur, \ g(\Lr,\Lr)=g(\Lbr,\Lbr)=g(\Lbr,\Xr)=g(\Lr,\Xr)=0,\  g(\Xr,\Xr)=1.\]
We introduce functions $y$, $\yr$, $z$ and $\zr$ as follows:
\[y=\Xr(v^1+c), \ \ \yr = \frac{y}{\kappar},  \ \ z=1+\Tr(v^1+c), \ \ \zr = \frac{z}{\kappar}.\]
These functions play a central role {\color{black}in the characterization of} the rarefaction waves at the initial singularity. The connection coefficients with respect to the new frame can be computed in terms of these functions. We list the definitions and formulas as follows:
{\color{black}
\begin{equation*}
\begin{cases}
&\chir:=g(D_{\Xr} \Lr,\Xr)=-\Xr(\psi_2), \ \ \chibr :=  g(D_{\Xr} \Lbr, \Xr)=c^{-1}\kappar \chir=-c^{-1}\kappar\Xr(\psi_2), \\ 
&\zetar:=  g(D_{\Xr} \Lr, \Tr)=-\kappar y, \ \ \etar:=  -g(D_{\Xr} \Tr, \Lr)=\zetar+\Xr(\mur)=ck(\Tr,\Xr)=-\Tr(\psi_2), \\
&\deltasr:=g(D_{\Lr} \Lr,\Xr)=cy,\ \ \deltar:=g(D_{\Lr} \Lr,\Tr)=-\Lr(\mur)+cz.
\end{cases}
\end{equation*}
}
We can express the Levi-Civita connection in the second null frame as follows:
\begin{equation}
\begin{cases}
D_{\Lr} {\Lr} &= -\mur^{-1} \deltar\cdot \Lr+\deltasr \cdot\Xr
, \ \ D_{\Lbr} \Lr = c^{-2}\left(\deltar+2\Lr(c) \kappar \right)\cdot \Lr+\left(c^{-1}\kappar\deltasr +2\etar \right)\cdot \Xr,
\\
D_{\Lr} \Lbr &= -\zetar \cdot \Xr+\zr\cdot \Lbr,\ \ D_{\Lbr } \Lbr  =  c^{-1}\kappar(2\etar-\zetar)  \cdot \Xr+\left(\Lr(c^{-1}\kappar)-c^{-1}z+\mur^{-1}\Lbr(\mur)\right)\Lbr,\\
D_{\Xr} \Lr & = -\mur^{-1} \zetar \cdot  \Lr +  \chir \cdot \Xr, \ \ 
D_{\Lr} \Xr  =-\frac{1}{2}\mur^{-1} \zetar \cdot  \Lr+\frac{1}{2}\yr \cdot \Lbr,\ \
D_{\Xr} \Lbr  =  \chibr \cdot \Xr+\mur^{-1} \etar \cdot \Lbr , \\ 
D_{\Lbr} \Xr &= -\left[\frac{1}{2}c^{-2}\kappar y+\Xr(c^{-1}\kappar)\right]\Lr+\left(\frac{1}{2}c^{-1}y+\mur^{-1} \etar\right)\Lbr, \ \
D_{\Xr} \Xr = \frac{1}{2}\mur^{-1} \chibr \cdot \Lr +  \frac{1}{2}\mur^{-1}\chir \cdot \Lbr .
\end{cases}
\end{equation}
We also compute the commutators as follows:
\begin{equation*}
\begin{cases}
&[\Tr, \Xr]=0, \ \ [\Lr, \Xr]=\yr\cdot \Tr-\chir\cdot\Xr, \ \ [\Lr, \Tr]=\zr\cdot \Tr-\etar\Xr,\\
&[\Lbr, \Xr]=-\left(\frac{1}{2}c^{-2}\kappar y+\Xr(c^{-1}\kappar)\right)\Lr-\chibr\cdot\Xr+\frac{1}{2}c^{-1}y\cdot\Lbr, \\
& [\Lr, \Lbr]=\left(\Xr(c^{-1}\kappar)-c^{-1}z\right)\Lr-2\etar\cdot \Xr+\zr\cdot \Lbr.
\end{cases}
\end{equation*}
Finally, we define the set $\Lambda=\{\yr, \zr, \chir,\etar\}$. The bounds on the objects of $\Lambda$ will be the key ingredients in the energy estimates.



\subsection{Riemann invariants and Euler equations in the diagonal form}
The acoustical geometry allows one to diagonalize the Euler equations \eqref{eq: Euler in Euler coordinates} in a very concise way. Indeed, it is straightforward to show that the Euler equations are equivalent to
\begin{equation}\label{Euler equations:form 1}
\begin{cases}
L (\frac{2}{\gamma-1}c) &= -c \widehat{T}(\frac{2}{\gamma-1}c)+c \widehat{T}(\psi_k)\widehat{T}^k +c\Xh(\psi_k)\Xh^k,\\
L (\psi_1) &= -c \widehat{T}(\psi_1)+\frac{2}{\gamma-1}c \widehat{T}(c)\widehat{T}^1+\frac{2}{\gamma-1}c\Xh(c)\Xh^1,\\
L (\psi_2) &= -c \widehat{T}(\psi_2)+\frac{2}{\gamma-1}c \widehat{T}(c)\widehat{T}^2+\frac{2}{\gamma-1}c\Xh(c)\Xh^2.
\end{cases}
\end{equation}
Following Riemann \cite{Riemann}, we define the Riemann invariants with respect to the flat initial curve:
\begin{equation}\label{def: Riemann invariants}
w=\frac{1}{2}\left(\frac{2}{\gamma-1}c+\psi_1\right),  \ \ \wb=\frac{1}{2}\left(\frac{2}{\gamma-1}c-\psi_1\right).
\end{equation}
Therefore, we have
\begin{equation}\label{Euler equations:form 2}
\begin{cases}
L (\wb) &= -c \widehat{T}(\wb)(\widehat{T}^1+1)+\frac{1}{2}c \widehat{T}(\psi_2)\widehat{T}^2 +\frac{1}{2}c \Xh(\psi_2)\Xh^2-c\Xh(\wb)\Xh^1,\\
L (w) &= 	c \widehat{T}(w)(\widehat{T}^1-1)	+\frac{1}{2}c \widehat{T}(\psi_2)\widehat{T}^2 +c\Xh(w)\Xh^1+\frac{1}{2}c\Xh(\psi_2)\Xh^2,\\
L (\psi_2) &= -c \widehat{T}(\psi_2)+c\widehat{T}( w+\wb)\widehat{T}^2+c\Xh( w+\wb)
\Xh^2.
\end{cases}
\end{equation}
 Let $A=\left( {\begin{array}{ccc}
-( \widehat{T}^1+1)& 0&   \frac{1}{2}\widehat{T}^2 \\
   0 &   \widehat{T}^1-1& \frac{1}{2}\widehat{T}^2\\
     \widehat{T}^2 &  \widehat{T}^2&-1
  \end{array} } \right)$, $B=\left( {\begin{array}{ccc}
-\Xh^1& 0&   \frac{1}{2}\Xh^2 \\
   0 &   \Xh^1 & \frac{1}{2}\Xh^2\\
    \Xh^2 & \Xh^2&0
  \end{array} } \right)$ and $V=\left( {\begin{array}{c}
  \wb \\
     w\\
     \psi_2
  \end{array} } \right)$, {\color{black}\eqref{Euler equations:form 2}} is equivalent to
  \[L(V)=c A \cdot \widehat{T}(V)+c B\cdot \Xh(V).\]
There is a remarkable feature of the  matrix $A$: since $(\widehat{T}^1)^2+(\widehat{T}^2)^2=1$, $A$ has three eigenvalues $0$, $-1$ and $-2$ regardless the values of $\widehat{T}^1$ and $\widehat{T}^2$. This can be proved by a straightforward computation. We choose  three eigenvectors $\frac{1}{2}\left( {\begin{array}{c}
  1-\widehat{T}^1\\
   1+\widehat{T}^1 \\
    2 \widehat{T}^2
  \end{array} } \right)$, $\frac{1}{2}\left( {\begin{array}{c}
  \widehat{T}^2\\
   -\widehat{T}^2 \\
     2\widehat{T}^1
  \end{array} } \right)$ and $\frac{1}{2} \left( {\begin{array}{c}
 1+ \widehat{T}^1\\
  1 -\widehat{T}^1 \\
     -2\widehat{T}^2
  \end{array} } \right)$ corresponding to the eigenvalues $0$, $-1$ and $-2$ respectively. 
 Using these eigenvectors as columns, we can construct  $P=\left( {\begin{array}{ccc}
\frac{1- \widehat{T}^1}{2}& \frac{\widehat{T}^2}{2}&  \frac{1+ \widehat{T}^1}{2} \\
    \frac{1+ \widehat{T}^1}{2}&   - \frac{\widehat{T}^2}{2}& \frac{1- \widehat{T}^1}{2}\\
  \widehat{T}^2 &\widehat{T}^1&-\widehat{T}^2
  \end{array} } \right)$. 
To diagonalize \eqref{Euler equations:form 2} in the $L$-direction, we define $U=P^{-1} \cdot V$ and we have
 \[LU=c\Lambda\cdot \widehat{T}(U)+cP^{-1} B P \cdot \Xh(U)+\left(c\Lambda P^{-1}\widehat{T}(P)-P^{-1}L(P)+cP^{-1}B \Xh(P)\right)\cdot U,\]
where $\Lambda$ is the diagonal matrix with $0,-1,-2$ on the diagonals. Since $\widehat{T}=\kappa T$, we finally obtain:
 \begin{equation}\label{eq:Euler in diagonalized form}
 LU=\frac{c}{\kappa}\Lambda\cdot {T}(U)+cP^{-1} B P \cdot \Xh(U)+\left( \frac{c}{\kappa}\Lambda P^{-1} {T}(P)-P^{-1}L(P)+cP^{-1}B \Xh(P)\right)\cdot U.
\end{equation}
In an explicit manner, we can represent $U$ as
\begin{equation}\label{eq:explicit formula for U}
\left( {\begin{array}{c}
  U^{(0)} \\
     U^{(-1)}\\
     U^{(-2)}
  \end{array} } \right)=\left( {\begin{array}{c}
\frac{1-\widehat{T}^1}{2}\wb+\frac{1+\widehat{T}^1}{2}w+  \frac{\widehat{T}^2}{2}\psi_2 \\
   \widehat{T}^2 \wb - \widehat{T}^2 w+\widehat{T}^1\psi_2\\
  \frac{1+\widehat{T}^1}{2}\wb+ \frac{1-\widehat{T}^1}{2}w-\frac{\widehat{T}^2}{2}\psi_2
  \end{array} } \right) \ \Leftrightarrow  \  \begin{cases}
   \wb&=\frac{1-\Th^1}{2}U^{(0)}+  \frac{\Th^2}{2} U^{(-1)}+ \frac{1+\Th^1}{2}U^{(-2)},\\
   w&=  \frac{1-\Th^1}{2}U^{(-2)}+\frac{1+\Th^1}{2}U^{(0)}-  \frac{\Th^2}{2} U^{(-1)},\\
   \psi_2&= \Th^1 U^{(-1)}+\Th^2U^{(0)}-\Th^2 U^{(-2)}{\color{black}.}
   \end{cases}
\end{equation}
{\color{black}where $U^{(\lambda)}$ is the corresponding component for the eigenvalue $\lambda$.}

 We can also diagonalize the Euler equations using the second null frame. In fact, similar to \eqref{Euler equations:form 1}, we have
\begin{equation}\label{Euler equations:form 1 ringed}
\begin{cases}
\Lr (\frac{2}{\gamma-1}c) &= -c \Trh(\frac{2}{\gamma-1}c)-c \Trh(\psi_1) +c\Xr(\psi_2),\\
\Lr (\psi_1) &= -c \Trh(\psi_1)-c \Trh\left(\frac{2}{\gamma-1}c\right),\\
\Lr (\psi_2) &= -c \Trh(\psi_2)+c\Xr\left(\frac{2}{\gamma-1}c\right).
\end{cases}
\end{equation}
In terms of Riemann invariants, \eqref{Euler equations:form 1 ringed} reduces to a simple form
\begin{equation}\label{Euler equations:form 3}
\begin{cases}
\Lr (\wb) &= \frac{1}{2}c \Xr(\psi_2),\\
\Lr (w) &= 	-2c \Trh(w)+\frac{1}{2}c\Xr(\psi_2),\\
\Lr (\psi_2) &= -c \Trh(\psi_2)+c\Xr( w+\wb).
\end{cases}
\end{equation}
Therefore, for $A=\left( {\begin{array}{ccc}
0& 0&   0\\
   0 &   -2& 0\\
    0 &  0&-1
  \end{array} } \right)$, $B=\left( {\begin{array}{ccc}
0& 0&   \frac{1}{2}  \\
   0 &   0 & \frac{1}{2} \\
    1 & 1&0
  \end{array} } \right)$ and $V=\left( {\begin{array}{c}
  \wb \\
     w\\
     \psi_2
  \end{array} } \right)$, \eqref{Euler equations:form 3} is equivalent to
    \[\Lr(V)=c A \cdot \Trh(V)+c B\cdot \Xr(V).\]
We then take $P=\left( {\begin{array}{ccc}
1 & 0 &  0 \\
   0&   0& 1\\
  0 &-1&0
  \end{array} } \right)$
 and   $U=P^{-1} \cdot V$. Hence, we diagonalize the Euler equations with respect to the $\Lr$-direction as follows:
  \[\Lr U=c\Lambda\cdot \Trh(U)+cP^{-1} B P \cdot \Xr(U).\]
In terms of the Riemann invariants, we have
\begin{equation*}
U=\left( {\begin{array}{c}
  U^{(0)} \\
     U^{(-1)}\\
     U^{(-2)}
  \end{array} } \right)=\left( {\begin{array}{c}
 \wb \\
  -\psi_2\\
   w 
  \end{array} } \right) \ \ \Leftrightarrow \ \ \begin{cases}
   \wb&= U^{(0)},\\
   w&=   U^{(-2)}{\color{black},} \\
   \psi_2&= - U^{(-1)}{\color{black}.}
   \end{cases}
  \end{equation*}

\subsection{The classical 1-D rarefaction waves in geometric formulation}\label{section: 1D case}

We apply the previous geometric considerations to the 1-D rarefaction waves reviewed in \ref{subsection:1D-Riemann-problem}. The problem considered in this paper will be a multi-dimensional perturbation of this classical 1D picture. 

On the positive axis $x_1=x\geqslant 0$, we pose constant data $(v,c)\big|_{t=0}=(v_0,c_0)$. There exists a unique family of forward-facing centered rarefaction waves connected to the given data, with the explicit solution in \eqref{eq:1D-rarefaction-wave}.
Thus, the acoustical coordinate function $u$ and the null vector field $L$ are given by
{\color{black}
\[u-u_0=-(v+c)=-\frac{x}{t},   \ \ L=\partial_t+(v+c)\partial_x.\]
where $u_0 := -(v_0+c_0)$, ensuring $u=0$ on $C_0$.}
We also have
\begin{align*}
\kappa=t, \ \ \mu=c t,   \ \ T=-t\partial_x,\ \  L\mu=c, \ \ L c =0, \ \ L v=0, \ \ T u =1.
\end{align*}
In particular, on the time slice $\Sigma_\delta$, we have	
\begin{align*}
{\color{black}u-u_0}=-(v+c)=-\frac{x}{\delta}, \ \ \kappa=\delta, \ \ \mu\big|_{t=\delta}&=c \delta,   \ \ T\big|_{t=\delta}=-\delta\partial_x.
\end{align*}

The solution $(v,c)$ is {\color{black}piece-wise smooth for $t > 0$.} It is merely continuous across the line defined by $u=-(v_0+c_0)$ and $t>0$. We emphasize that the solution is not continuous at the {\color{black}singularity} $(t,x)=(0,0)$.  We also notice that on the time slice $\Sigma_\delta$, although the solution is not smooth at $x=\delta(v_0+c_0)$, all possible $L$-derivatives of $(u,c)$ are the same (in fact vanish) for $x<\delta(v_0+c_0)$ and $x>\delta(v_0+c_0)$ at this point.

In terms of  $U^{(0)}$, $U^{(-1)}$ and $U^{(-2)}$, we have
 \begin{align*}
 &U^{(0)}=\frac{1}{2}\left[\frac{4}{\gamma+1}\frac{x}{t}+\frac{\gamma-3}{\gamma-1}\left(\frac{\gamma-1}{\gamma+1}v_0-\frac{2}{\gamma+1}c_0\right)\right],\  U^{(-1)}=0, \ U^{(-2)}=\frac{\gamma+1}{2(\gamma-1)}\left(\frac{\gamma-1}{\gamma+1}v_0-\frac{2}{\gamma+1}c_0\right).
 \end{align*}
 In particular, we have $T\left(U^{(0)}\right)=-\frac{2}{\gamma+1}$. These computations are illuminating for the construction of initial data in higher dimensional situations.

\section{Energy methods and the main theorem}\label{section:energy-method-main-theorem}

\subsection{Multipliers, commutators and their  deformation tensors}\label{section:commutators and their  deformation tensors}

Given a vector field $Z$ on $\mathcal{D}$, its deformation tensor with respect to $g$ is defined as $\,{}^{(Z)} {\pi}_{\mu\nu}= {D}_\mu Z_\nu+ {D}_\nu Z_\mu$.  We will use two types of vector fields. The first set $\mathscr{J}$ is call the set of \emph{multiplier vector fields}; The second type of sets $\mathscr{Z}$  and $\mathring{\mathscr{Z}}$ are called sets of \emph{commutation vector fields}. They are defined as follows:
\[\mathscr{J}=\{\Lh, \Lb\}, \ \ \mathscr{Z}=\{T, \Xh\}, \ \ \mathring{\mathscr{Z}}=\{\mathring{T}, \mathring{X}\},\]
where $\Lh = c^{-1}\kappa L$.
The null components of the deformation tenors of the vectors from $\mathscr{J}$ and $\mathscr{Z}$ are listed in the following tables:
\begin{table}[ht]\label{table: deformation tensors 1}
\centering
\begin{tabular}{c c c c c}
  & $\Lh$ & $\Lb$ & $\Xh$ & $T$\\ [0.5ex] 
\hline
${\pi}_{LL}$     &     $0$ &$0$   &$0$ & $0$\\
$ {\pi}_{\Lb\Lb}$ &     $-8\mu T\left(c^{-1}\kappa\right)$ & $0$&$4\mu \Xh(c^{-1}\kappa)$ & $4\mu T(c^{-1}\kappa)$\\
$ {\pi}_{L\Lb}$ & $-4\kappa L(\kappa )$ & $-4\left(\kappa L(\kappa) + T(\mu)\right)$ & $2(\zeta-\eta)$ & $-2T\left(\mu\right)$\\
$ {\pi}_{L\Xh}$ & $0$ & $-2(\zeta+\eta)$ & $-\chi$ & $-(\zeta+\eta)$\\
$ {\pi}_{\Lb\Xh}$ & $2 \big(c^{-1}\kappa(\zeta+\eta)- \mu \Xh\left(c^{-1}\kappa\right)\big)$ & $-2\mu  \Xh(c^{-1}\kappa)$ & $-\chib$ & $-c^{-1}\kappa(\zeta+\eta)$ \\ [1ex]
$ {\pi}_{\Xh\Xh}$ & $2c^{-1}\kappa \chi$ &$2\chib$ & $0$ & $2\kappa\theta$\\ [1ex]
\hline
\end{tabular}
\end{table}

\begin{table}[ht]\label{table: deformation tensors for Zr}
\centering
\begin{tabular}{c c c c c}
  & $\Xr$ & $\Tr$ \\ [0.5ex] 
\hline
$ {\pi}_{\Lr\Lr}$     &    $-2 cy$ &$-2cz  $    \\
$ {\pi}_{\Lbr\Lbr}$ &     $2  c^{-1}\kappar^2 \big(y-2\Xr(c)\big)$ & $2  c^{-1}\kappar^2\big(z-2\Tr(c)\big)$ \\
$ {\pi}_{\Lr\Lbr}$ & $-2 \kappar \Xr(c) $ & $-2 \kappar \Tr (c) $ \\
${\pi}_{\Lr\Xr}$ & $ -\chir$ & $ -\etar $  \\
$ {\pi}_{\Lbr\Xr}$ & $  -c^{-1}\kappar\chir$ & $ -c^{-1}\kappar\etar$   \\ [1ex]
$ {\pi}_{\Xr\Xr}$ & $0$ &$0$  \\ [1ex]
\hline
\end{tabular}
\end{table}

A \emph{multi-index} $\alpha$ is a string of numbers $\alpha=(i_1,i_2,\cdots,i_n)$ with {\color{black}$i_j=0$ or $1$} for $1\leqslant j\leqslant n$. The \emph{length} of the multi-index $\alpha$ is defined as $|\alpha|=n$. {\color{black}Given a multi-index $\alpha$ and a smooth function $\psi$,} the shorthand notation $Z^\alpha(\psi)$ and  $\Zr^\alpha(\psi)$ denote the following functions:
\[Z^\alpha(\psi)=Z_{(i_N)}\left(Z_{(i_{N-1})}\left(\cdots \left(Z_{(i_1)}(\psi)\right)\cdots\right)\right), \ \Zr^\alpha(\psi)=Z_{(i_N)}\big(\cdots \big(\Zr_{(i_1)}(\psi)\big)\cdots \big),\]
where $Z_{(0)}=\Xh$, $Z_{(1)}=T$, $\Zr_{(0)}=\Xr$ and $\Zr_{(1)}=\Tr$. If $\psi\in \{w,\wb,\psi_2\}$ and $|\alpha|=n$, we also use  $\Psi_n$ to denote $Z^\alpha(\psi)$ and use  $\Psir_n$ to denote $\Zr^\alpha(\psi)$. We also use the notation $Y(\psi)$ where $Y\in \mathscr{Y}$ and $\mathscr{Y}=\{L,\Lb, \Xh\}$.

We introduce the notion of {\bf order} which counts the number of derivatives.  For $U$ from the set $\{w,\wb,\psi_2, c,c^{-1}, \mu, \kappa\}$, we require that the order of $U$ is zero, denoted by ${\rm ord}(U)=0$. For $V$ from the set $\{\eta,\zeta,\chi, \chib, k,\theta,\etar,\zetar,\chir, \chibr, \deltasr,\deltar, y,z, \yr,\zr\}$, we require that ${\rm ord}(V)=1$. For all {\color{black}$Z\in \mathscr{Y}\cup \mathscr{Z}\cup\mathring{\mathscr{Z}}$,} for all $U$ with a well-defined order, we require that ${\rm ord}\left(Z(U)\right)={\rm ord}\left(U\right)+1$. We also define that ${\rm ord}\left(U\cdot V\right)={\rm ord}\left(U\pm V\right)=\max\left({\rm ord}\left(U\right),{\rm ord}\left(U\right)\right)$.

\subsection{Energy identities}
We also refer to \cite{ChristodoulouShockFormation} and \cite{ChristodoulouMiao} for details of computations in this subsection.

\subsubsection{Energy identities for linear waves}
Let $\varrho$ be a source function. We derive energy identities for the linear wave equation:
\begin{equation}\label{model linear wave equation}
 \Box_{g}\psi = \varrho.
\end{equation}
The energy momentum tensor associated to $\psi$ is defined as $\mathbb{T}=d\psi \otimes d\psi-\frac{1}{2}g(D\psi, D\psi)g$.
In the first {\color{black}null frame $(L, \Lb, \Xh)$,} the components of ${\mathbb{T}}_{\mu\nu}$ are listed as follows:
\begin{equation}\label{energy momentum tensor in null frame}
 \begin{split}
  \mathbb{T}_{LL} &= (L\psi)^2,\ \mathbb{T}_{\Lb\Lb} = (\Lb\psi)^2, \ \mathbb{T}_{\Lb L} = \mu (\Xh\psi)^2, \ \mathbb{T}_{L\Xh} = L\psi \cdot \Xh(\psi),\\
\mathbb{T}_{\Lb \Xh} &= \Lb \psi \cdot \Xh (\psi),\ \mathbb{T}_{\Xh\Xh} = \frac{1}{2}(\Xh \psi)^2+\frac{1}{2\mu}L\psi \Lb \psi .
 \end{split}
\end{equation}

The divergence of the energy momentum tensor ${\mathbb{T}}_{\mu\nu}$ is $D^\mu \mathbb{T}_{\mu\nu} = \varrho\cdot \partial_\nu \psi$. For a vector multiplier vector field $J\in \mathscr{J}$, its energy current field is defined as $
P{}^\mu = - {\mathbb{T}}^{\mu}{}_{\nu}J{}^\nu$.
Therefore,
\begin{equation}\label{eq:Divergence of current}
 {D}_\mu  {P}{}^\mu= {Q}= -\varrho\cdot J( \psi)-\frac{1}{2} {\mathbb{T}}^{\mu\nu} \,^{(J)}{\pi}_{\mu\nu}.
\end{equation}
For $(t,u)\in [\delta,t^*]\times [0,u^*]$ and a smooth function $f$ defined on $\mathcal{D}(t,u)$, we use the following notations to {\color{black}denote the integrals:}
\begin{align*}
&\int_{\Sigma_t^u} f=\int_{0}^{u}\!\!\!\int_{0}^{2\pi}\!\!\! f(t,u',\vartheta')  \sqrt{\slashed{g}}du'd\vartheta', \ \ \int_{C_u^t} f= \int_{0}^{t}\!\!\! \int_{0}^{2\pi} \!\!\!f(t',u,\vartheta') \sqrt{\slashed{g}}dt'd\vartheta', \\
&\int_{\mathcal{D}(t,u)} f=\int_{0}^{u}\!\!\!\int_{0}^{t}\!\!\!\int_{0}^{2\pi}\!\!\! f(t',u',\vartheta')  \sqrt{\slashed{g}}dt'du'd\vartheta'.
\end{align*}
The $L^2$ norms are defined using these integrals, i.e., $\displaystyle\|f\|_{L^2(\Sigma_t^u)}=\sqrt{\int_{\Sigma_t^u} |f|^2}$ and $\displaystyle\|f\|_{L^2(C_u^t)}= \sqrt{\int_{C_u^t} |f|^2}$.

We have two choices for $J\in \mathscr{J}$. This leads to the following two energy identities:

\begin{itemize}

\item[Case 1)] $J= \Lh$. We define
\begin{align*}\mathcal{E}(\psi)(t,u)&=\frac{1}{2}\int_{\Sigma_t^u}c^{-1}\kappa \left(c^{-1}\kappa (L\psi)^2+\mu (\Xh\psi)^2\right),\ \ \mathcal{F}(\psi)(t,u)=\int_{C_u^t}  c^{-1}\kappa(L\psi)^2.
\end{align*}
We integrate  \eqref{eq:Divergence of current} over $\mathcal{D}(t,u)$ to derive
\begin{equation}\label{identity: energy with Lh}
\mathcal{E}(\psi)(t,u)+\mathcal{F}(\psi)(t,u)={\color{black}\mathcal{E}(\psi)(0,u) + \mathcal{F}(\psi)(t,0)}+\int_{\mathcal{D}(t,u)} Q,
\end{equation}
where
\begin{align*}
\int_{\mathcal{D}(t,u)} Q&=\underbrace{-\int_{\mathcal{D}(t,u)} \mu\varrho\cdot \Lh\psi}_{Q_0}   + \underbrace{\int_{\mathcal{D}(t,u)}T (c^{-1}\kappa)  (L\psi)^2}_{Q_1}  +\underbrace{\int_{\mathcal{D}(t,u)} \frac{1}{2}L(\kappa^2)(\Xh\psi)^2}_{Q_2} \\
&\ \ +\underbrace{\int_{\mathcal{D}(t,u)} \left(c^{-1}\kappa(\zeta+\eta)-\mu \Xh(c^{-1}\kappa)\right)L\psi\cdot \Xh\psi}_{Q_3} \underbrace{ -\int_{\mathcal{D}(t,u)}  \frac{\kappa^2\chi}{2}(\Xh \psi)^2 +\frac{c^{-1}\kappa \chi}{2}L\psi\cdot \Lb\psi}_{Q_4}.
\end{align*}

\item[Case 2)] $J=\Lb$. We define
\[\Eb(\psi)(t,u)=\frac{1}{2}\int_{\Sigma_t^u} (\Lb\psi)^2+\kappa^2 (\Xh \psi)^2, \ \ \Fb(\psi)(t,u)=\int_{C_u^t}c\kappa  (\Xh \psi)^2.\]
We integrate  \eqref{eq:Divergence of current} over $\mathcal{D}(t,u)$ to derive
\begin{equation}\label{identity: energy with Lb}
\Eb(\psi)(t,u)+\Fb(\psi)(t,u)={\color{black}\Eb(\psi)(0,u)+\Fb(\psi)(t,0)}+\int_{\mathcal{D}(t,u)} \underline{Q},
\end{equation}
where
\begin{align*}
\int_{\mathcal{D}(t,u)} \underline{Q}&=\underbrace{-\int_{\mathcal{D}(t,u)} \mu \varrho\cdot\Lb\psi}_{\underline{Q}_0}  +\underbrace{\int_{\mathcal{D}(t,u)} \frac{1}{2} \left( \mu  L(c^{-1}\kappa)+ \Lb(c^{-1}\kappa ) \right)(\Xh\psi)^2}_{\underline{Q}_1}  \underbrace{-\int_{\mathcal{D}(t,u)} (\zeta+\eta) \Lb\psi\cdot \Xh\psi}_{\underline{Q}_2}  \\
&
\underbrace{-\int_{\mathcal{D}(t,u)} \mu \Xh(c^{-1}\kappa) L\psi \cdot \Xh\psi}_{\underline{Q}_3} \underbrace{-\int_{\mathcal{D}(t,u)}\frac{1}{2}{\mu}\chib \left( (\Xh \psi)^2 + \frac{1}{\mu}L\psi\cdot \Lb\psi \right)}_{\underline{Q}_4}.
\end{align*}

\end{itemize}

\subsubsection{Energy identities for higher order terms}\label{section:Energy identities for higher order terms}
We shall commute derivatives with $\Box_{g}$ to derive higher order energy estimates.  Let $\psi$ be a smooth solution of $\Box \psi=\varrho$ and $Z$ be a vector field on $\mathcal{D}$. We have
\begin{equation}\label{commute a vector field with wave equation}
\Box \left( Z \psi \right) = Z (\varrho) + \frac{1}{2}\tr {}^{(Z)} {\pi} \cdot \varrho +   {\rm div}_{g} \left( {}^{(Z)} J\right)
\end{equation}
where the vector field ${}^{(Z)}J$ is defined by ${}^{(Z)}J^{\mu} = \left( {}^{(Z)} {\pi}^{\mu\nu} - \frac{1}{2}g^{\mu\nu} \tr_{g}{}^{(Z)} {\pi} \right)\partial_\nu \psi$ and the trace $\tr$ is taken with respect to $g$.

In view of \eqref{eq: conformal Euler}, we have the following equations for the Riemann invarints:
\begin{equation}\label{Main Wave equation: order 0}
\begin{cases}
\Box_{g}\wb&=-c^{-1}\left(g(D \wb,D \wb)+\frac{\gamma-3}{4}g(D \wb,D w)+\frac{\gamma+1}{4}g(D w,D w)+\frac{1}{2}g( D\psi_2, D\psi_2)\right){\color{black},}\\
\Box_{g} w&=-c^{-1}\left(\frac{\gamma+1}{4}g(D \wb,D \wb)+\frac{\gamma-3}{4}g(D \wb,D w)+g(D w,D w)+\frac{1}{2}g( D\psi_2, D\psi_2)\right){\color{black},}\\
\Box_{g}\psi_2&=-c^{-1}\left(\frac{3-\gamma}{4}g(D\wb, D \psi_2)+\frac{3-\gamma}{4}g(Dw, D \psi_2)\right).
\end{cases}
\end{equation}
where we use $D\log\Omega=\frac{3-\gamma}{2}c^{-1}(Dw + D\wb)$.
Let $\Psi_0=\Psir_0 \in \{\wb,w,\psi_2\}$ and $Z=\Zr\in\mathring{\mathscr{Z}}$, we then have the following recursion relations:
\begin{equation}\label{commute vector fields with equations}
\begin{split}
\Box_{ {g}} \Psir_n &= \varrho_n, \ \ \Psir_n = \Zr (\Psi_{n-1}),\ \ \varrho_n = \Zr (\varrho_{n-1}) + \frac{1}{2}\tr {}^{(\Zr)} {\pi} \cdot \varrho_{n-1} +  {\rm div}_{g} \left(	 {}^{(\Zr)} J_{n-1}\right){\color{black},}\\
{}^{(\Zr)}J_{n-1}^{\mu}&= \left( {}^{(\Zr)} {\pi}^{\mu\nu} - \frac{1}{2}g^{\mu\nu} \tr_{g}{}^{(\Zr)} {\pi} \right)\partial_\nu \Psir_{n-1}.
\end{split}
\end{equation}
We use $N_{_{\rm top}}$ to denote the total number of $\Zr$'s {\color{black}commuted with the equation.} Therefore, the sub-index of $\Psir_n$ satisfies $0\leqslant n\leqslant N_{_{\rm top}}$. {\color{black}We} also define $\Ninf=\Ntop-1$.

\begin{remark}\label{remark:null structure}
By using the above notations, we rewrite \eqref{Main Wave equation: order 0} as $\Box_g \Psir_0 =\varrho_0$ where $\Psir_0 \in \{\wb,w,\psi_2\}$. The source term $\varrho_0$ is a linear combination of the following terms
\[\big\{c^{-1}g(D f_1,Df_2)\big| f_1,f_2\in \{\wb,w,\psi_2\}\big\},\]
where
\begin{equation}\label{eq: decompose rho0 using Lr Lbr Xr}
g(D f_1,Df_2)=-\frac{1}{2\mur}\Lr (f_1)\Lbr (f_2)-\frac{1}{2\mur}\Lbr (f_1) \Lr (f_2)+\Xr(f_1)\Xr(f_2).
\end{equation}
We notice that the term $\Lb(\wb)\cdot \Lb(\wb)$ is absent in all possible $g(D f_1,Df_2)$'s in \eqref{eq: decompose rho0 using Lr Lbr Xr}. This is the null structure {\color{black}mentioned in b) of Section \ref{section:nonlinear estimates}}.
\end{remark}

We can apply the energy identities for $\Psir_n$. Thus, the integrands of the source terms, i.e., $Q_0$ and $\underline{Q}_0$, are given by
\[Q_0=-\int_{D(t,u)}  \frac{\mu}{\mur} \cdot \varrhor_n \cdot \Lh\Psi_n,  \ \ \underline{Q}_0=-\int_{D(t,u)} \frac{\mu}{\mur} \cdot \varrhor_n \cdot \Lb\Psi_n.\]
where $\varrhor_n=  \mur \varrho_n$. In view of \eqref{commute vector fields with equations}, we have the following recursion relations:
\begin{equation*}
\begin{split}
\varrhor_n&= \Zr ( \varrhor_{n-1} ) +{}^{(\Zr)} \delta \cdot \varrhor_{n-1} + {}^{(\Zr)} \sigma_{n-1}, \
 {}^{(\Zr)} \sigma_{n-1} = \mu\cdot {\rm div}_{g} \left(\, {}^{(\Zr)} J_{n-1}\right), \ {}^{(\Zr)} \delta =\frac{1}{2}\tr {}^{(\Zr)} {\pi}-\mu^{-1}\Zr(\mu).
\end{split}
\end{equation*}
We notice that, for $\Zr=\Xr$ or $\Tr$,  we have ${}^{(\Zr)} \delta =0$. Thus, $\varrhor_n = \Zr ( \varrhor_{n-1} ) + {}^{(\Zr)} \sigma_{n-1}$. According to Section 7.2 of  \cite{ChristodoulouMiao}, we decompose ${}^{(\Zr)} \sigma_{n-1}$ as follows:
\[
{}^{(\Zr)} \sigma_{n-1}= {}^{(\Zr)} \sigma'_{n-1,1}+{}^{(\Zr)} \sigma'_{n-1,2}+{}^{(\Zr)} \sigma_{n-1,3},
\]
where
\begin{equation*}
\begin{cases}
{}^{(\Zr)} \sigma'_{n-1,1}&=-\frac{1}{2}\left(\Lr(c^{-1}\kappar)+\chibr-c^{-1}z\right)\left(\pi_{\Lr\Xr}\Xr(\Psi_{n-1})-\frac{1}{2\mur}\pi_{\Lr\Lr}\Lbr(\Psi_{n-1})\right)\\
&\ \ -\frac{1}{2}(\chir-\zr)\left(\pi_{\Lbr\Xr}\Xr(\Psi_{n-1})-\frac{1}{2\mur}\pi_{\Lbr\Lbr}\Lr(\Psi_{n-1})\right),\\
{}^{(\Zr)} \sigma'_{n-1,2}&=-\frac{1}{2} \pi_{\Lbr \Xr} \cdot \Lr \Xr(\Psi_{n-1})-\frac{1}{2} \pi_{\Lr \Xr} \cdot \Lbr \Xr(\Psi_{n-1})+\frac{1}{4\mur}\left( \pi_{\Lbr \Lbr} \Lr \Lr(\Psi_{n-1})+\pi_{\Lr \Lr} \Lbr \Lbr(\Psi_{n-1})\right)\\
&\ \ +\frac{1}{2} \pi_{\Lr \Lbr} \cdot \Xr \Xr(\Psi_{n-1})- \frac{1}{2} \pi_{\Lbr \Xr} \cdot \Xr \Lr(\Psi_{n-1})- \frac{1}{2} \pi_{\Lr \Xr} \cdot \Xr \Lbr(\Psi_{n-1}),\\
{}^{(\Zr)} \sigma'_{n-1,3}&=-\frac{1}{2} \Lr \left(\pi_{\Lbr \Xr} \right) \Xr(\Psi_{n-1})+\Lr\left(\frac{1}{4\mur} \pi_{\Lbr \Lbr}\right)  \Lr(\Psi_{n-1}) -\frac{1}{2} \Lbr \left(\pi_{\Lr \Xr} \right) \Xr(\Psi_{n-1})+\Lbr\left(\frac{1}{4\mur} \pi_{\Lr \Lr}\right)  \Lbr(\Psi_{n-1})\\
&\ \ +\frac{1}{2}  \Xr(\pi_{\Lr \Lbr})  \Xr(\Psi_{n-1})- \frac{1}{2}  \Xr (\pi_{\Lbr \Xr} )\cdot\Lr(\Psi_{n-1})- \frac{1}{2} \Xr ( \pi_{\Lr \Xr} )\cdot \Lbr(\Psi_{n-1}).
\end{cases}
\end{equation*}
In the above formulas, we use $\pi$ to denote {\color{black}${}^{(\Zr)}\pi$.} In {\color{black}${}^{(\Zr)} \sigma'_{n-1,3}$,} we expand the term $\Lbr\left(\frac{1}{4\mur} \pi_{\Lr \Lr}\right)  \Lbr(\Psi_{n-1})$ as $\Lbr\left(\frac{1}{4\kappar}\right) c^{-1}\pi_{\Lr \Lr}  \Lbr(\Psi_{n-1})+\frac{1}{4\kappar} \Lbr\left(c^{-1}\pi_{\Lr \Lr}\right)  \Lbr(\Psi_{n-1})$.
We move the first term from $\sigma'_{n-1,3}$ to $\sigma'_{n-1,1}$. (This operation leads to a cancellation in the energy estimates and it will provide a gain in $t$). Therefore, we have
\begin{equation}\label{eq: decompose sigma}
{}^{(\Zr)} \sigma_{n-1}= {}^{(\Zr)} \sigma_{n-1,1}+{}^{(\Zr)} \sigma_{n-1,2}+{}^{(\Zr)} \sigma_{n-1,3},
\end{equation}
with (we use $\pi$ to denote {\color{black}${}^{(\Zr)}\pi$})
\begin{equation}\label{eq: sigma 1 circle}
\begin{split}
{}^{(\Zr)} \sigma_{n-1,1}&=-\frac{1}{2}\left(\Lr(c^{-1}\kappar)+\chibr-c^{-1}z\right)\left(\pi_{\Lr\Xr}\Xr(\Psi_{n-1})-\frac{1}{2\mur}\pi_{\Lr\Lr}\Lbr(\Psi_{n-1})\right)\\
&\ \ -\frac{1}{2}(\chir-\zr)\left(\pi_{\Lbr\Xr}\Xr(\Psi_{n-1})-\frac{1}{2\mur}\pi_{\Lbr\Lbr}\Lr(\Psi_{n-1})\right)+\Lbr\left(\frac{1}{4\kappar}\right) c^{-1}\pi_{\Lr \Lr}  \Lbr(\Psi_{n-1}),
\end{split}
\end{equation}
\begin{equation}\label{eq: sigma 2 circle}
\begin{split}
{}^{(\Zr)} \sigma_{n-1,2}&=-\frac{1}{2} \pi_{\Lbr \Xr} \cdot \Lr \Xr(\Psi_{n-1})-\frac{1}{2} \pi_{\Lr \Xr} \cdot \Lbr \Xr(\Psi_{n-1})+\frac{1}{4\mur}\left( \pi_{\Lbr \Lbr} \Lr \Lr(\Psi_{n-1})+\pi_{\Lr \Lr} \Lbr \Lbr(\Psi_{n-1})\right)\\
&\ \ +\frac{1}{2} \pi_{\Lr \Lbr} \cdot \Xr \Xr(\Psi_{n-1})- \frac{1}{2} \pi_{\Lbr \Xr} \cdot \Xr \Lr(\Psi_{n-1})- \frac{1}{2} \pi_{\Lr \Xr} \cdot \Xr \Lbr(\Psi_{n-1}){\color{black},}
\end{split}
\end{equation}
\begin{equation}\label{eq: sigma 3 circle} 
\begin{split}
{}^{(\Zr)} \sigma_{n-1,3}&=-\frac{1}{2} \Lr \left(\pi_{\Lbr \Xr} \right)\cdot \Xr(\Psi_{n-1})+\Lr\left(\frac{1}{4\mur} \pi_{\Lbr \Lbr}\right)  \Lr(\Psi_{n-1})\\
& \ \ -\frac{1}{2} \Lbr \left(\pi_{\Lr \Xr} \right)\cdot \Xr(\Psi_{n-1})+\boxed{\frac{1}{4\kappar}\Lbr\left( c^{-1}\pi_{\Lr \Lr}\right)  \Lbr(\Psi_{n-1})}\\
&\ \ +\frac{1}{2}  \Xr(\pi_{\Lr \Lbr}) \cdot \Xr(\Psi_{n-1})- \frac{1}{2}  \Xr (\pi_{\Lbr \Xr} )\cdot\Lr(\Psi_{n-1})- \frac{1}{2} \Xr ( \pi_{\Lr \Xr} )\cdot \Lbr(\Psi_{n-1}).
\end{split}
\end{equation}
\begin{remark}
	The boxed term is the most dangerous new error terms associated with the second null frame, violating the null structures mentioned in Section \ref{section: difficulty}. Notice that $^{(Z)}\pi_{LL}$ vanishes identically for $Z \in \{\Lh,\Lb,\Xh,T\}$ from the first null frame; see the tables in \ref{section:commutators and their  deformation tensors}.
\end{remark}
Since $\varrhor_n = \Zr ( \varrhor_{n-1} ) + {}^{(\Zr)} \sigma_{n-1}$, for $\Psir_{n}:=\Zr_{n}\Big(\Zr_{n-1}\big(\cdots \big(\Zr_{1}(\Psir_0)\big)\cdots\big)\Big)$, we have
\begin{equation}\label{eq: varphor_n expression}
\varrhor_{n}=\Zr_{n}\big(\cdots \big(\Zr_{1}(\varrhor_{0})\big)\cdots\big)+\sum_{i=0}^{n-1} \Zr_{n} \big( \cdots \big( \Zr_{i+2}\big(\,^{(\Zr_{i+1})}{\sigma}_{i}\big)\big)\cdots \big){\color{black}.}
\end{equation}
We remark that, if $i=n-1$ in the above sum, the corresponding term is $\,^{(\Zr_{n})}{\sigma}_{n-1}${\color{black}.}

\subsection{The energy ansatz and the main theorem of the paper}
Throughout the paper, we use the notations $F\lesssim_s G$ to denote $F\leqslant C\cdot G$ where $C$ is a constant {\color{black}depending only on $s$.} The notation $F\lesssim G$ {\color{black}means that $C$} is a universal constant.

\subsubsection{The small parameter $\varepsilon$}
We recall that on the righthand side of $S_{0,0}$ on $\Sigma_0$, i.e., the region $t=0$ and $x_1\geqslant 0$, we have already posed data  $(v,c)\big|_{t=0}=\big(v^1_0(x_1,x_2),v_0^2(x_1,x_2), c_0(x_1,x_2)\big)$. Let $\overline{v_0}$ and $\overline{c_0}>0$ be fixed constants. We assume that the data is a small irrotational perturbation of the one dimensional data, i.e.,  there is a constant $\varepsilon>0$, so that for all positive integer $k>0$, we have 
\[\|v^1_0(x_1,x_2)-\overline{v_{0}}\|_{H^k}+\|v^2_0(x_1,x_2)\|_{H^k}+\|c_0(x_1,x_2)-\overline{c_0}\|_{H^k}\lesssim_k \varepsilon,\]
where the $H^k$-norms are taken on $\Sigma_0$ with $x_1\geqslant 0$. In addition, we have $\frac{\partial v^2_0}{\partial x^1} = \frac{\partial v^1_0}{\partial x^2}$.

Since the classical solutions to the Euler equations depend continuously on the initial data, we conclude that for any positive integer $k$, for $\psi \in \{w,\wb,\psi_2\}$, for all $1\leqslant |\alpha|\leqslant k$, for $Z\in \mathscr{Z}$,we have
\[\|w-\overline{w_0}\|_{L^\infty(C_0)}+\|\wb-\overline{\wb_0}\|_{L^\infty(C_0)}+\|\psi_2\|_{L^\infty(C_0)}+\|Z^\alpha(\psi)\|_{L^\infty(C_0)} \lesssim_k \varepsilon,\]
where $\overline{w_0}=\frac{1}{2}\left(\frac{2}{\gamma-1}\overline{c_0}-\overline{v_0}\right)$ and $\overline{\wb_0}=\frac{1}{2}\left(\frac{2}{\gamma-1}\overline{c_0}+\overline{v_0}\right)$.

\begin{remark}
We may remove the smallness of $\varepsilon$ by shrinking the time interval $[0,t^*]$. Since we are mainly interested in the stability problem of 1-dimensional rarefaction waves, we will focus on the case where $\varepsilon$ is sufficiently small.
\end{remark}

\subsubsection{The assumptions on the initial data in the rarefaction wave region}\label{subsection:energy-ansatz-initial-data}
Given a smooth function on $\mathcal{D}(t^*,u^*)$, for a multi-index $\alpha$, for all $(t,u)\in [\delta,t^*]\times [0,u^*]$, we define the total energy and the total flux associated to $\Zr^\alpha(\psi)$ as follows:
\[\begin{cases}
&\mathscr{E}_{\alpha}(\psi)(t,u)= \mathcal{E}\big(\Zr^\alpha(\psi)\big)(t,u)+\Eb\big(\Zr^\alpha(\psi)\big)(t,u),\\
&\mathscr{F}_{\alpha}(\psi)(t,u) = \F\big(\Zr^\alpha(\psi)\big)(t,u)+\Fb\big(\Zr^\alpha(\psi)\big)(t,u).
\end{cases}
\]
For all $n\leqslant N_{_{\rm top}}$, we define
\[\mathscr{E}_{n}(\psi)(t,u)= \sum_{|\alpha|= n} \mathscr{E}_{\alpha}(\psi)(t,u), \ \ \mathscr{F}_{n}(\psi)(t,u)= \sum_{|\alpha|= n} \mathscr{F}_{\alpha}(\psi)(t,u).
\]
For $\psi \in \{w,\psi_2\}$, we also define
	\[\mathscr{E}_{\leqslant n}(\psi)(t,u)= \sum_{|\alpha|\leqslant n} \mathscr{E}_{\alpha}(\psi)(t,u), \ \ \mathscr{F}_{\leqslant n}(\psi)(t,u)= \sum_{|\alpha|\leqslant n} \mathscr{F}_{\alpha}(\psi)(t,u){\color{black},}\]
	{\color{black}while} for $\psi = \wb$, we define
	\[\mathscr{E}_{\leqslant n}(\wb)(t,u)= \mathring{\mathscr{E}}_0(\wb)(t,u) + \sum_{1 \leqslant|\alpha|\leqslant n}\mathscr{E}_{\alpha}(\wb)(t,u), \ \ \mathscr{F}_{\leqslant n}(\psi)(t,u)= \mathring{\mathscr{F}}_0(\wb)(t,u) + \sum_{1 \leqslant|\alpha|\leqslant n}\mathscr{F}_{\alpha}(\wb)(t,u),\]
	where 
	\begin{equation}\label{def:Er_0-Fr_0}
		\mathring{\mathscr{E}}_0(\wb)(t,u) = \frac{1}{2}\int_{\Sigma_t^u}c^{-2}\kappa^2(L\wb)^2 + \kappa^2(\XR\wb)^2, \ \ \mathring{\mathscr{F}}_0(\wb)(t,u) =\int_{C_u^t}c^{-1}\kappa(L\wb)^2 + c\kappa(\XR\wb)^2.
	\end{equation}
In order to state the main theorem of the paper, we need precise estimates on the initial data posed on $\Sigma_\delta^{u^*}$ and $C_0^1$. It consists of three sets of assumptions $\mathbf{(I_0)}$, $\mathbf{(I_2)}$ and $\mathbf{(I_\infty)}$. We remark that, for the one dimensional Riemann problem, $u^*=\frac{\gamma+1}{\gamma-1}\overline{c_0}$ corresponds to the vacuum state, see Section \ref{section: 1D case}. The assumptions are listed as follows:
\begin{equation}\label{initial size u*}
\mathbf{(I_0)} \ \ \ u^*=\frac{1}{2}\cdot\frac{\gamma+1}{\gamma-1}\overline{c_0}. 
\end{equation}
{\color{black}
\begin{equation}\label{initial I2}
	\mathbf{(I_2)}\begin{cases}
		&\mathscr{E}_{n}(\psi)(\delta,u^*)+\mathscr{F}_{n}(\psi)(t,0)\leqslant C_0 \varepsilon^2 t^2, \ \ \psi \in \{w,\wb,\psi_2\},  \  1\leqslant n\leqslant \Ntop, \   t \in [\delta,t^*];\\
		&\mathcal{E}(\psi)(\delta,u^*)+\underline{\mathcal{E}}(\psi)(\delta,u^*)+\mathcal{F}(\psi)(t,0)+\underline{\mathcal{F}}(\psi)(t,0)\leqslant C_0 \varepsilon^2 t^2, \  \psi \in \{w,\psi_2\}, \   t \in [\delta,t^*].
	\end{cases}
\end{equation}
}

\begin{equation}\label{initial Iinfty}
\mathbf{(I_\infty)}
\begin{cases}
	&\|L\psi\|_{L^\infty(\Sigma^{u}_\delta)} + \|\Xh\psi\|_{L^\infty(\Sigma^{u}_\delta)} \lesssim \varepsilon, \ \ \psi \in \{w,\wb,\psi_2\};\\
	&\|T(w)\|_{L^\infty(\Sigma^{u}_\delta)} + \|T(\psi_2)\|_{L^\infty(\Sigma^{u}_\delta)} +  \|T\wb + \frac{2}{\gamma+1}\|_{L^\infty(\Sigma^{u}_\delta)} \lesssim \varepsilon \delta; \\
	&\|LZ^{\alpha}\psi\|_{L^\infty(\Sigma^{u}_\delta)}+\|\Xh Z^{\alpha}\psi\|_{L^\infty(\Sigma^{u}_\delta)}+\delta^{-1}\|TZ^{\alpha}\psi\|_{L^\infty(\Sigma^{u}_\delta)} \lesssim \varepsilon, \ \  Z \in \{\Xh, T\}, \ 1 \leqslant |\alpha| \leqslant 2;\\
	& \|\slashed{g} - 1\|_{L^\infty(\Sigma^{u}_\delta)} +  \|\frac{\kappa}{\delta}-1\|_{L^\infty(\Sigma^{u}_\delta)} + \|\Th^2\|_{L^\infty(\Sigma^{u}_\delta)} \lesssim  \varepsilon \delta, \ \|\Th^1+1\|_{L^\infty(\Sigma^{u}_\delta)} \lesssim \varepsilon^2 \delta^2;\\
	& \|Z(\slashed{g})\|_{L^\infty(\Sigma^{u}_\delta)}  \lesssim \varepsilon \delta, \ \|Z^\alpha(\kappa)\|_{L^\infty(\Sigma^{u}_\delta)} \lesssim \varepsilon \delta^2,  \ \ Z \in \{\Xh, T\}, \ 1 \leqslant |\alpha| \leqslant 2;\\
	& \|Z^\alpha(\Th^1)\|_{L^\infty(\Sigma^{u}_\delta)}\leqslant \varepsilon^2 \delta^2, \ \|Z^\alpha(\Th^2)\|_{L^\infty(\Sigma^{u}_\delta)}\leqslant \varepsilon \delta,  \ \  Z \in \{\Xh, T\}, \ 1 \leqslant |\alpha| \leqslant 2.
\end{cases}
\end{equation}
In addition, we also {\color{black}assume that the initial motion} is irrotational:
\begin{equation}\label{initial irrotational}
	\mathbf{(I_{irrotational})} \ \ \ \frac{\partial v^2}{\partial x^1}\Big|_{\Sigma^{u}_\delta} = \frac{\partial v^1}{\partial x^2}\Big|_{\Sigma^{u}_\delta}.
\end{equation}

\begin{remark}
By the scaling of the Euler equations, we may assume that $\overline{c_0}=1$.
Notice that by $\mathbf{(I_\infty)}$ and \eqref{def: Riemann invariants} we have $\|T(c)-\frac{\gamma-1}{\gamma+1}\|_{L^{\infty}(\Sigma_{\delta}^u)} \lesssim \varepsilon \delta$.  In view of $\mathbf{(I_0)}$, $\|c-\overline{c_0}\|_{L^{\infty}(C_0)} \lesssim \varepsilon$, and  $T\big|_{\Sigma_\delta} = \frac{\partial}{\partial u}$,  we may assume that $\frac{1}{4} \leqslant c \leqslant 2$ on $\Sigma_{\delta}$.
\end{remark}
In the second paper \cite{LuoYu2} of this series, we will construct initial data on $\Sigma_\delta$ so that all the above assumptions are verified.


\subsubsection{The main theorem}

We now state the main theorem of the paper:

\begin{MainTheorem}[A priori Energy Estimates] Assume that the initial data posed on $\Sigma_\delta^{u^*}$ and $C_0^1$ satisfies the conditions $\mathbf{(I_0)}$, $\mathbf{(I_2)}$ and $\mathbf{(I_\infty)}$. Therefore, for $\Ntop\geqslant 9$, there exists a constant $\varepsilon_0>0$, so that for all $\frac{1}{2}>\delta>0$, for all $\varepsilon<\varepsilon_0$, $\mathcal{D}\supset \mathcal{D}(1,u^*)$. Moreover, there exists a constant $C_0>0$, so that for all $t\in [\delta,1]$, we have
\begin{equation}\label{main energy estimates}
\begin{cases}
&\mathcal{E}(\psi)(t,u^*)+\underline{\mathcal{E}}(\psi)(t,u^*)\leqslant C_0 \varepsilon^2 t^2, \ \ \psi \in \{w,\psi_2\};\\
&\mathscr{E}_{n}(\psi)(t,u^*)\leqslant C_0 \varepsilon^2 t^2, \ \ \psi \in \{w,\wb,\psi_2\},  \ \ 1\leqslant n\leqslant \Ntop.
\end{cases}
\end{equation}
\end{MainTheorem}

\begin{remark}
The constants $C_0$ and $\varepsilon_0$ are independent of $\delta$. This will allow us to take $\delta\rightarrow 0$ so that we can construct the rarefaction waves all the way up to the singularity, see {\color{black}the second paper \cite{LuoYu2}} of this series.
\end{remark}


\subsubsection{The bootstrap argument and the ansatz}

We use the method of continuity to prove the main estimates \eqref{main energy estimates}. We propose a set of the energy ansatz and we will run a bootstrap argument to prove it on $\mathcal{D}(t^*,u^*)$. 

The ansatz $\mathbf{(B_2)}$ is as follows: we assume that there exists a constant $M>0$,  so that for all  $(t,u)\in [\delta,t^*]\times [0,u^*]$ the following inequalities hold:
\begin{equation}\label{ansatz B2}
\mathbf{(B_2)}\begin{cases}
&\mathscr{E}_{n}(\psi)(t,u)+\mathscr{F}_{n}(\psi)(t,u)\leqslant M \varepsilon^2 t^2, \ \ \psi \in \{w,\wb,\psi_2\},  \ \ 1\leqslant n\leqslant \Ntop;\\
&\mathcal{E}(\psi)(t,u)+\underline{\mathcal{E}}(\psi)(t,u)+\mathcal{F}(\psi)(t,u)+\underline{\mathcal{F}}(\psi)(t,u)\leqslant M \varepsilon^2 t^2, \ \ \psi \in \{w,\psi_2\}.
\end{cases}
\end{equation}
In the bootstrap argument, we will also need auxiliary estimates to bound the $L^\infty$ norms of lower order terms. Thus, we also assume the following set of  bootstrap assumption on the $L^\infty$ bounds.  

The ansatz $\mathbf{(B_\infty)}$ is as follows: we assume that there exists a constant $M>0$ (this is the same $M$ as in \eqref{ansatz B2}), so that for all {\color{black}$(t,u) \in [\delta, t^*]\times [0,u^*]$} and $\psi \in \{w,\wb,\psi_2\}$, the following inequalities hold:
\begin{equation}\label{ansatz Binfty}
\mathbf{(B_\infty)}
\begin{cases}
&\|L\psi\|_{L^\infty(\Sigma^{u}_t)}+\|\Xh\psi\|_{L^\infty(\Sigma^{u}_t)}\leqslant M\varepsilon;\\
&\|T(w)\|_{L^\infty(\Sigma^{u}_t)}+\|T(\psi_2)\|_{L^\infty(\Sigma^{u}_t)}+\varepsilon t \|T\wb\|_{L^\infty(\Sigma^{u}_t)}\leqslant M\varepsilon t; \\
&\|LZ^{\beta}\psi\|_{L^\infty(\Sigma^{u}_t)}+\|\Xh Z^{\beta}\psi\|_{L^\infty(\Sigma^{u}_t)}+t^{-1}\|TZ^{\beta}\psi\|_{L^\infty(\Sigma^{u}_t)}\leqslant M\varepsilon, \ \  Z \in \{\Xh, T\}, \ 1 \leqslant |\beta| \leqslant 2;\\
&\|\psi\|_{L^\infty(\Sigma^{u}_t)} \leqslant M, \ \ \|\kappa\|_{L^\infty(\Sigma^{u}_t)}+\|\Th^1+1\|_{L^\infty(\Sigma^{u}_t)}+\|\Th^2\|_{L^\infty(\Sigma^{u}_t)} \leqslant Mt.
\end{cases}
\end{equation}
In the rest of the paper, we assume the bootstrap assumptions  $\mathbf{(B_2)}$ and  $\mathbf{(B_\infty)}$ hold on $\mathcal{D}(t^*,u^*)$. We will prove that, for sufficiently small $\varepsilon$, we can improve the constant $M$ to be a universal constant $C_0$. The constant $C_0$ will be independent of $\delta, t^*$ and $u^*$. This will close the bootstrap argument {\color{black}hence proving} the main theorem of the paper.

\subsection{Heuristics for the energy ansatz}\label{subsection:heuristics-energy-ansatz}


We make the assumption that solution in the frame $\{L,T,\Xh\}$ is smooth and {\color{black}$\Th^1 \approx -1$,} $\kappa \approx t$ as $t \to 0$. By \eqref{Euler equations:form 2}, the Euler equations can be written as 
\[ c^{-1}\kappa L(V) =  A \cdot T(V) + \kappa B\cdot \Xh(V). \]
By examining the components of $w$ and $\psi_2$, it is straightforward to see
\[ T(w) = O(t\varepsilon), \ T(\psi_2) = O(t\varepsilon), \quad \text{as} \ t \to 0. \]



\section{Preparations for the energy estimates}\label{section:preparations-for-energy-estimates}
{\color{black}In the following,} we will use $\mathring{M}$ to denote a power $M^k$ of $M$. Indeed, $k\leqslant 5$. For example, we can use $\mathring{M}$ to denote $M$, $M^2$ or $M^5$.

\subsection{The control of the acoustical geometry}
\subsubsection{Preliminary estimates on connection coefficients}
We first show that $c\approx 1$. In view of $c=\frac{\gamma-1}{2}(w+\wb)$ and $(\mathbf{B_\infty})$, we have $\|L(c)\|_{L^\infty}\lesssim M \varepsilon$. Since $L=\frac{\partial}{\partial t}$, we integrate from $\Sigma_\delta$ and we obtain
\begin{align*}
|c(t,u,\vartheta)-c(\delta,u,\vartheta)|&\leqslant \int_{\delta}^t |(Lc)(t',u,\vartheta)|dt' \lesssim M\varepsilon.
\end{align*}
Since $c(\delta,u,\vartheta)\in [\frac{1}{4},2]$, we obtain that 
\[\frac{1}{8}\leqslant c \leqslant 3\] 
on $\mathcal{D}(t^*,u^*)$, provided that $M\varepsilon$ is sufficiently small.

Next, we show that $|\kappa| \leqslant \mathring{M}t$. In view of \eqref{structure eq 1: L kappa}, we have
{\color{black}
\[\kappa(t,u,\vartheta)=e^{\int_{\delta}^t e'(\tau)d\tau}\kappa(\delta,u,\vartheta)+\int_{\delta}^t e^{\int_{\delta}^\tau e'(\tau')d\tau'}m'(\tau,u,\vartheta)d\tau.\]
}
In view of \eqref{eq: m' e'} and the fact that $|\Th^1|^2+|\Th^2|^2=1$, we can use $c=\frac{\gamma-1}{2}(w+\wb)$ and $(\mathbf{B_\infty})$ to show that
\[\|m'\|_{L^\infty(\Sigma_t)}\lesssim M,  \ \ \|e'\|_{L^\infty(\Sigma_t)}\lesssim M \varepsilon.\]
Since $t\leqslant t^*\leqslant 1$, this implies the following bound:
\[\big|\kappa(t,u,\vartheta)\big|\lesssim e^{M\varepsilon}\delta+te^{M\varepsilon}M \lesssim Mt,\]
provided that $M\varepsilon$ is sufficiently small.

In view of \eqref{structure eq 2: L chi} and the fact that $h = \frac{1}{\gamma-1} c^2$, we can use $(\mathbf{B_\infty})$ to derive that
\begin{equation}\label{eq:L chi in preliminary geometric bounds}\left\|L \chi-e\chi +\chi^2\right\|_{L^\infty(\Sigma_t)} \lesssim \mathring{M}\varepsilon,
\end{equation}
for all $t\in [\delta,t^*]$. According to $(\mathbf{I_\infty})$, on the initial slice $\Sigma_\delta$, we have
\[|\chi(\delta,u,\vartheta)| =  |c(\slashed{k} - \theta)|=|\Xh^i\Xh(\psi_i)-c\theta|\lesssim \varepsilon.\]
Therefore, we can integrate \eqref{eq:L chi in preliminary geometric bounds} from to $\delta$ to $t$ to derive
\[\|\chi\|_{L^\infty(\Sigma_t)}\lesssim  \varepsilon +\mathring{M}\varepsilon t \lesssim \mathring{M}\varepsilon,\]
provided $\mathring{M}\varepsilon \leqslant 1$. 

According to the equation $L(\slashed{g})=2\slashed{g}\chi$, we can use the bound on $\chi$ to derive
\[|\slashed{g} -1|\lesssim \varepsilon t,\]
if $\mathring{M}\varepsilon$ is sufficiently small. In particular, we have $\slashed{g} \approx 1$.

We also need a bound on $\Xh(\Th^k)$ where $k=1,2$. Since $[L,\Xh]=-\chi \Xh$,  we can use \eqref{structure eq 3: L T on Ti Xi Li} and $(\mathbf{B_\infty})$ to derive that
\[\big|L(\Xh(\widehat{T}^{k}))\big|
=\big|\Xh \big[\big(\widehat{T}^j\cdot \Xh(\psi_j) + \Xh(c)\big)\Xh^k\big]-\chi \Xh\big(\Th^k\big)\big|\lesssim \mathring{M}\varepsilon \Xh(\widehat{T}^{j})+\mathring{M}\varepsilon.\]
According to $(\mathbf{I_\infty})$, we have $\|\Xh(\Th^i)\|_{L^\infty(\Sigma_\delta)}\lesssim  \delta\varepsilon$. By the standard Gronwall's inequality, if $\mathring{M}\varepsilon$ is sufficiently small, we have
\[
\|\Xh(\Th^i)\|_{L^\infty(\Sigma_t)}\leqslant \mathring{M}\varepsilon t.
\]

The same idea can be used to bound $\Xh(\kappa)$. By $[L,\Xh]=-\chi \Xh$, we have
\begin{align*}
L(\Xh\kappa)&=-\chi\Xh(\kappa)+\Xh(m')+\Xh(\kappa e')= (e'-\chi) \Xh(\kappa)+ \Xh(m')+\mu \Xh(e').
\end{align*}
We have already showed $\|e'-\chi\|_{L^\infty}\lesssim \mathring{M} \varepsilon$. By $(\mathbf{B_\infty})$, we have $\|\Xh(m')+\mu \Xh(e')\|_{L^\infty}\lesssim \mathring{M} \varepsilon$. By $(\mathbf{I_\infty})$, we also have $\|\Xh(\kappa)\|_{L^\infty(\Sigma_\delta)}\lesssim  \delta\varepsilon$. Therefore, by a direct use of Gronwall's inequality, if $\mathring{M}\varepsilon$ is sufficiently small, we have
\[
\|\Xh(\kappa)\|_{L^\infty(\Sigma_t)}\lesssim \mathring{M}\varepsilon t.
\]

In view of the commutator formula $[L,T]=-\big(\kappa\big(2c\Xh^i\cdot T(\psi_i)+2\Xh(c)\big)-\Xh(c^{-1}\kappa)\big)\Xh$, by the estimates that we have derived so far, we have
\[\big|L(T(\widehat{T}^{k}))\big|\lesssim \mathring{M}\varepsilon T(\widehat{T}^{j})+\mathring{M}\varepsilon.\]
According to $(\mathbf{I_\infty})$, we have $\|T(\Th^i)\|_{L^\infty(\Sigma_\delta)}\lesssim  \delta\varepsilon$. Thus, by Gronwall's inequality, if $\mathring{M}\varepsilon$ is sufficiently small, we have
\[\|T(\Th^i)\|_{L^\infty(\Sigma_t)}\leqslant \mathring{M}\varepsilon t.\]

In view of the above commutator formula for $[L,T]$, we can proceed exactly in the same manner to {\color{black}bound $T(\kappa)$.} Indeed, by the estimates that we have derived so far, it is straightforward to see that $\|[L,T]\kappa\|_{L^\infty}\lesssim \mathring{M}\varepsilon t$. Therefore, we have
\[|L(T\kappa)|=\left|e'T\kappa + \left(Tm' +\kappa T e'+[L,T]\kappa\right)\right| \lesssim \mathring{M}\varepsilon |T\kappa|+\mathring{M}\varepsilon t.\]
Once more, since $\|T(\kappa)\|_{L^\infty(\Sigma_\delta)}\lesssim \varepsilon \delta$,  by Gronwall's inequality, if $\mathring{M}\varepsilon$ is sufficiently small, we have
\[\|T(\kappa)\|_{L^\infty(\Sigma_t)}\lesssim \mathring{M}\varepsilon t.
\]
Finally,  in view of \eqref{structure quantities: connection coefficients}, we have $|\zeta|\leqslant  \big| -T^i \Xh(\psi_j) -  \kappa \Xh(c)\big| \lesssim \mathring{M}\varepsilon t$. Since $\eta=\zeta+\Xh(c^{-1}\kappa)$, we have
\[\|\zeta\|_{L^\infty(\Sigma_t)}+\|\eta\|_{L^\infty(\Sigma_t)}\lesssim \mathring{M} \varepsilon t.
\]

We summarize the estimates derived so far:
\begin{proposition}
{\color{black}Under the bootstrap assumptions} $(\mathbf{B}_2)$ and $(\mathbf{B}_\infty)$, if $\mathring{M}\varepsilon$ is sufficiently small, we have the following pointwise bounds on $\Sigma_t^{u^*}$ for all $t\in [\delta, t^*]$:
\begin{equation}\label{preliminary geometric bounds}
\begin{cases}
&c\approx 1, \ \slashed{g}\approx 1, \ \|\zeta\|_{L^\infty(\Sigma_t)}\lesssim \mathring{M} \varepsilon t, \ \ \|\eta\|_{L^\infty(\Sigma_t)}\lesssim \mathring{M} \varepsilon t, \  \ \|\chi\|_{L^\infty(\Sigma_t)}\lesssim \mathring{M}\varepsilon,\\
&\|\kappa\|_{L^\infty(\Sigma_t)}\lesssim Mt, \ \ \|\Xh(\kappa)\|_{L^\infty(\Sigma_t)}\lesssim \mathring{M}\varepsilon t, \ \ \|T(\kappa)\|_{L^\infty(\Sigma_t)}\lesssim \mathring{M}\varepsilon t,\\
&\|\Xh(\Th^i)\|_{L^\infty(\Sigma_t)}\leqslant \mathring{M}\varepsilon t, \|T(\Th^i)\|_{L^\infty(\Sigma_t)}\leqslant \mathring{M}\varepsilon t.
\end{cases}
\end{equation}
\end{proposition}

\subsubsection{Improved estimates on $\kappa$}

We consider the wave equation \eqref{Main Wave equation: order 0} for $\psi \in \{w,\wb,\psi\}$. Since $\Lb = 2T +c^{-1}\kappa L$, the bootstrap assumption $(\mathbf{B}_\infty)$ implies that $\|Y(\psi)\|\lesssim M\varepsilon$ for all $Y\in \mathscr{Y}=\{L,\Lb, \Xh\}$  unless $Y=\Lb$ and $\psi=\wb$. In view of Remark \ref{remark:null structure}, the righthand side of  \eqref{Main Wave equation: order 0} are bounded by $\frac{1}{\mu}\mathring{M}\varepsilon$ in $L^\infty$-norm.  Thus, by virtue of \eqref{eq:wave operator in null frame},   for $\psi \in \{w,\wb,\psi_2\}$, we have
\[ \big|\Xh^2 (\psi) - \mu^{-1}L\big(\underline{L}(\psi)\big) - \mu^{-1}\big(\frac{1}{2}\chi\cdot \underline{L}(\psi) +\frac{1}{2}\chib\cdot L(\psi)\big)- 2 \mu^{-1}\zeta \cdot \Xh(\psi)\big| \lesssim \frac{1}{\mu}\mathring{M}\varepsilon.
\]
By \eqref{preliminary geometric bounds}, we have $|\mu|\lesssim \kappa \lesssim Mt$. We multiply both sides of the above inequality by $\mu$ and we use $(\mathbf{B}_\infty)$ to derive that
\begin{equation}\label{eq: auxiliary 1}
L(\Lb\psi) =-\frac{1}{2}\chi\cdot \underline{L}(\psi)+a(t,u,\vartheta),
\end{equation}
with  $\|a(t,u,\vartheta)\|_{L^\infty(\Sigma_t)}\lesssim \mathring{M}\varepsilon$. Hence,
\begin{align*}
\Lb\psi(t,u,\vartheta)=e^{-\frac{1}{2}\int_{\delta}^t \chi(\tau,u,\vartheta) d\tau} \cdot \Lb(\psi)(\delta,u,\vartheta)+\int_{\delta}^t e^{-\frac{1}{2}\int_{0}^\tau \chi(\tau,u,\vartheta) d\tau} a(\tau,u,\vartheta)d\tau.
\end{align*}
By \eqref{preliminary geometric bounds}, if $\mathring{M} \varepsilon$ is sufficiently small, we have$\big\|e^{-\frac{1}{2}\int_{0}^t \chi(\tau,u,\vartheta) d\tau}-1\big\|_{L^\infty(\Sigma_t)}\lesssim \mathring{M}\varepsilon t$. The above formula for $\Lb\psi(t,u,\vartheta)$ thus gives a bound on $|\Lb\psi(t,u,\vartheta)-\Lb\psi(\delta,u,\vartheta)|$. Since $\Lb = 2T +c^{-1}\kappa L$, this implies
\[|T\psi(t,u,\vartheta)-T\psi(\delta,u,\vartheta)|\lesssim  \mathring{M}\varepsilon t.\]
Hence,
\begin{equation}\label{eq: auxiliary 2}
\left|Tc(t,u,\vartheta)-Tc(\delta,u,\vartheta)\right|\lesssim \mathring{M} \varepsilon t.
\end{equation}
In view of \eqref{eq: m' e'}, we conclude that
\[\left|m'(t,u,\vartheta)-m'(\delta,u,\vartheta)\right|\lesssim \mathring{M} \varepsilon t.\]
By integrating $L\kappa = m'+e'\kappa$, we have
\[\kappa(t,u,\vartheta)=e^{\int_{\delta}^t e' d\tau}\kappa(\delta,u,\vartheta)+\int_{\delta}^t e^{\int_{\delta}^\tau e' d\tau'}m'(\tau,u,\vartheta)d\tau.\]
This implies the following estimates on $\kappa$:
\begin{equation}\label{bound:precise bound on kappa}
\left|\kappa(t,u,\vartheta)-\kappa(\delta,u,\vartheta)-m'(\delta,u,\vartheta)(t-\delta)\right|\lesssim \mathring{M}\varepsilon t^2.
\end{equation}
We then use the fact that {\color{black}$\|Tc(\delta,u,\vartheta)+\frac{\gamma-1}{\gamma+1}\|_{L^\infty(\Sigma_\delta)}\lesssim \varepsilon \delta$}  in $\mathbf{I}_\infty$ to derive {\color{black}$\|m'(\delta,u,\vartheta)-1\|_{L^\infty(\Sigma_\delta)}\lesssim \varepsilon \delta$.} Therefore, we conclude that $\kappa \approx t$, i.e.,  for sufficiently small $\mathring{M}\varepsilon$, we have
\begin{equation}\label{bound on kappa}
\kappa \approx t.
\end{equation}
In fact, the above computation yields
\begin{equation}\label{bound on kappa more precise}
\big|\frac{\kappa}{t}-1\big|\lesssim \mathring{M} t \varepsilon.
\end{equation}
This also closes the bound on $\kappa$ in $\mathbf{(B_\infty)}$. In the course of the proof, we have also showed that
\begin{equation}\label{bound on L kappa}
L\kappa \approx 1.
\end{equation}
{\color{black}From \eqref{eq: auxiliary 2}} and the fact that $\|Tc(\delta,u,\vartheta)+1\|_{L^\infty(\Sigma_\delta)}\lesssim \varepsilon \delta$ in {\color{black}($\mathbf{I}_\infty$),} we obtain that 
\begin{equation}\label{bound on Tc}
{\color{black}Tc \approx -\frac{\gamma-1}{\gamma+1}.}
\end{equation}
By \eqref{def: Riemann invariants}, we also have
\begin{equation}\label{bound on Twb}
{\color{black}T\wb \approx -\frac{2}{\gamma+1}.}
\end{equation}

\subsubsection{Improved estimates on $\widehat{T}^{1}$ and $\widehat{T}^{2}$}
According to \eqref{structure eq 3: L T on Ti Xi Li}, we have
\[L(\widehat{T}^{1}+1)=\left(\widehat{T}^j\cdot \Xh(\psi_j) + \Xh(c)\right)\Th^2.
\]
By $\mathbf{(B_\infty)}$, each of the righthand terms is bounded by $\mathring{M}\varepsilon^2 t$ in $L^\infty$. Thus, {\color{black}$|L(\widehat{T}^{1}+1)|\lesssim \mathring{M}\varepsilon^2 t$.} On the other hand, by $\mathbf{(I_\infty)}$ we have $|\Th^1+1|\lesssim \varepsilon^2 \delta^2$ on $\Sigma_\delta$. Therefore, by integrating $L(\widehat{T}^{1}+1)$, we obtain that
\begin{equation}\label{bound on T1+1}
|\widehat{T}^{1}+1|\lesssim \mathring{M}\varepsilon^2 t^2.
\end{equation}
We see that $\widehat{T}^{1}+1$ has an extra $t$ power. This also closes the bound on $\widehat{T}^{1}+1$ in $\mathbf{(B_\infty)}$.

Similarly, we have
\[L(\widehat{T}^{2})= \left(\widehat{T}^j\cdot \Xh(\psi_j) + \Xh(c)\right)\Xh^2.\]
By $\mathbf{(B_\infty)}$, each of the righthand terms is bounded by $\mathring{M}\varepsilon $ in $L^\infty$. By $\mathbf{(I_\infty)}$ we have $|\Th^2|\lesssim \varepsilon  \delta$ on $\Sigma_\delta$. We then integrate the above equation to derive
\begin{equation}\label{bound on T2}
|\widehat{T}^{2}|\lesssim \mathring{M}\varepsilon t.
\end{equation}
This also closes the bound on $\widehat{T}^{2}$ in $\mathbf{(B_\infty)}$.

\subsubsection{Improved higher order pointwise estimates}

The following pointwise bounds for $\psi$ could be useful:
\begin{lemma}
	Let $\psi$ be a linear combinations of $\wb, w$ and $\psi_2$ and $c_0$ is a constant. We have
	\[\|T(\psi) + c_0\|_{L^{\infty}(\Sigma_t)} \leqslant \|T(\psi) + c_0\|_{L^{\infty}(\Sigma_{\delta})} + C\mathring{M} t \varepsilon,\]
where $C$ is a universal constant. In particular, we have
	\begin{align}\label{bound on T(v^1+c)}
		\|T(v^1+c) + 1\|_{L^{\infty}(\Sigma_t)} \lesssim \mathring{M} t \varepsilon{\color{black}.}
	\end{align}
\end{lemma}
\begin{proof}
We integrate the bound $|LT(\psi)| \lesssim M \varepsilon$ of $\mathbf{(B_\infty)}$ and we use $\mathbf{(I_\infty)}$ to bound $\|T(v^1+c) + 1\|_{L^{\infty}(\Sigma_\delta)}$. This proves the lemma.
\end{proof}
We write \eqref{structure eq 1: L kappa 2nd form} and \eqref{structure eq 3: L T on Ti Xi Li} as follows:
\begin{equation}\label{precise form of kappa and Ti}
\begin{split}
&L \kappa = -T(v^1 + c) -\underbrace{\big[ (\Th^1 + 1)T(\psi_1) + \Th^2T(\psi_2)\big] }_{\mathbf{err}_{\kappa}},\\
&L(\Th^i)=\big[\Xh(v^1+c)+ \underbrace{(\Th^1+1)\Xh(\psi_1) + \Th^2\Xh(\psi_2)}_{\mathbf{err}_{\Th}}\big]\Xh^i.
\end{split}
\end{equation}
According to the bounds \eqref{bound on T1+1}, \eqref{bound on T2} and $\mathbf{(B_\infty)}$, the error terms $\mathbf{err}_{\kappa}$ and $\mathbf{err}_{\Th}$ are bounded as follows
\begin{equation}\label{bound on err 0th order}
		\|\mathbf{err}_{\kappa}\|_{L^{\infty}(\Sigma_t)} \lesssim\mathring{M} t^2\varepsilon^2, \ \
		\|\mathbf{err}_{\Th}\|_{L^{\infty}(\Sigma_t)} \lesssim \mathring{M} t \varepsilon^2.
\end{equation}
In view of \eqref{bound on T(v^1+c)} and \eqref{precise form of kappa and Ti}, we also have the following byproduct:
\begin{equation}\label{bound on Lkappa}
	|L\kappa-1| \leqslant \mathring{M} t \varepsilon{\color{black}.}
\end{equation}
We commute $Z \in\mathscr{Z}=\{T,\Xh\}$ {\color{black}with the equation of} $L(\Th^i)$ in \eqref{precise form of kappa and Ti}. In view of \eqref{eq:commutator formulas}, we have
\begin{equation}\label{eq: LZThi}
L(Z(\widehat{T}^{i}))= (Z\Xh(v^1+c) + Z(\mathbf{err}_{\Th})) \Xh^i +  \big(\Xh(v^1+c) + \mathbf{err}_{\Th}\big)Z(\Xh^i)-\,^{(Z)}f\cdot \Xh(\Th^i),
\end{equation}
where $\,^{(\Xh)}f=\chi$, $\,^{(T)}f=\zeta+\eta$ and
\begin{align*}
		Z(\mathbf{err}_{\Th}) &= \Xh(\psi_i) \cdot Z(\Th^i)  +(\Th^1+1)Z\Xh(\psi_1) + \Th^2 Z\Xh(\psi_2). 
\end{align*}
In view of \eqref{preliminary geometric bounds}, \eqref{bound on T1+1}, \eqref{bound on T2} and $\mathbf{(B_\infty)}$,  we can bound $\Xh(v^1+c) + \mathbf{err}_{\Th}$ and $\Xh(\psi_i)$ by $\mathring{M}\varepsilon$ and bound $\,^{(Z)}f\cdot \Xh(\Th^k)$, $ (\Th^1+1)Z\Xh(\psi_1)$, $\Th^2 Z\Xh(\psi_2)$ by $\mathring{M}t\varepsilon^2$. Therefore,
\begin{equation}\label{eq: LZThi}
\big|L(Z(\widehat{T}^{i}))\big|\lesssim  \mathring{M}\varepsilon+\mathring{M}\varepsilon\big(|Z(\widehat{T}^{1})|+|Z(\widehat{T}^{2})|\big).
\end{equation}
By Gronwall's inequality and $\mathbf{(I_\infty)}$, we have $\|Z(\Th^i)\| \lesssim \mathring{M} t\varepsilon$. The bound on $Z(\Th^1)$ can be improved. In fact,
\begin{align*}
L(Z(\widehat{T}^{1}))= (Z\Xh(v^1+c) + Z(\mathbf{err}_{\Th})) \Xh^1 +  \big(\Xh(v^1+c) + \mathbf{err}_{\Th}\big)Z(\Xh^1)-\,^{(Z)}f\cdot \Xh(\Th^1){\color{black}.}
\end{align*}
We can bound $\Xh^1$ by $\mathring{M}\varepsilon t$. Therefore,
\begin{equation}\label{eq: LZTh1}
\big|L(Z(\widehat{T}^{1}))\big|\lesssim  \mathring{M}\varepsilon^2 t+\mathring{M}\varepsilon |Z(\widehat{T}^{1})|.\end{equation}
In view of the bound of $Z(\Th^1)$ on $\Sigma_\delta$ and $\|Z(\Th^i)\| \lesssim \mathring{M} t\varepsilon$, we then conclude that
\begin{equation}\label{bound on Z Ti}
		\big|Z(\Th^1)\big| \lesssim \mathring{M} t^2 \varepsilon^2, \ \ \big|Z(\Th^2)\big| \lesssim \mathring{M} t \varepsilon.
\end{equation}
In view of \eqref{eq: LZThi} and \eqref{eq: LZTh1}, we also have the following byproduct:
\begin{equation}\label{LZTi byproduct}
\big|L(Z(\widehat{T}^{1}))\big|\lesssim  \mathring{M}\varepsilon^2 t,  \ \big|L(Z(\widehat{T}^{2}))\big|\lesssim  \mathring{M}\varepsilon.
\end{equation}

We commute $Z \in\mathscr{Z}=\{T,\Xh\}$ {\color{black}with the equation of} $L(\kappa)$ in \eqref{precise form of kappa and Ti}. In view of \eqref{eq:commutator formulas}, we have
\begin{equation}\label{eq: LZkappa}
L(Z(\kappa))= -ZT(v^1+c) - Z(\mathbf{err}_{\kappa})-\,^{(Z)}f\cdot \Xh(\kappa),
\end{equation}
where $\,^{(\Xh)}f=\chi$, $\,^{(T)}f=\zeta+\eta$ and
\begin{align*}
Z(\mathbf{err}_{\kappa}) &= T(\psi_i) \cdot Z(\Th^i)  +(\Th^1+1) ZT(\psi_1) + \Th^2 ZT(\psi_2). 
\end{align*}
By  \eqref{preliminary geometric bounds}, \eqref{bound on T1+1}, \eqref{bound on T2}, $\mathbf{(B_\infty)}$ and \eqref{bound on Z Ti}, we have $|Z(\mathbf{err}_{\kappa})|\lesssim \mathring{M}\varepsilon^2t^2$ and $|ZT(v^1+c)|\lesssim \mathring{M}\varepsilon t$. Therefore,
\begin{align*}
{\color{black}|L(Z(\kappa))| \lesssim}  \mathring{M}\varepsilon \big(|\Xh(\kappa)|+|T(\kappa)|\big)+ \mathring{M}\varepsilon t.
\end{align*}
By Gronwall's inequality, we then conclude that
\begin{equation}\label{bound on Z kappa}
		\big|Z(\kappa)\big| \lesssim \mathring{M} t^2 \varepsilon.
\end{equation}
As a {\color{black}byproduct,} we have
\begin{equation}\label{bound on LZ kappa}
\big|L(Z(\kappa))\big|\lesssim \mathring{M}\varepsilon t.
\end{equation}

We now turn to the estimates on $Z(\zeta), Z(\eta)$ and $Z(\chi)$. By the explicit formula of $\eta$ and $\zeta$ in \eqref{structure quantities: connection coefficients}, we can use \eqref{preliminary geometric bounds}, \eqref{bound on T1+1}, \eqref{bound on T2}, $\mathbf{(B_\infty)}$, \eqref{bound on Z Ti} and \eqref{bound on Z kappa} to derive
\begin{equation}\label{bound on Z eta zeta}
		\big|Z(\eta)\big| +\big|Z(\zeta)\big| \lesssim \mathring{M}  \varepsilon t.
\end{equation}

To derive the bound on $Z(\chi)$, we commute $Z$ with \eqref{structure eq 2: L chi} to derive
\begin{align*}
L (Z\chi)&=Z(L \chi)-\,^{(Z)}f\cdot \Xh(\chi)
\end{align*} 
where $\,^{(\Xh)}f=\chi$, $\,^{(T)}f=\zeta+\eta$. We can apply $Z$ directly to the righthand side of \eqref{structure eq 2: L chi} to compute $Z(L \chi)$. Therefore, it requires the following explicit expressions:
\[w(\Xh,\Xh)=\Xh^i \Xh(\psi_i), \ \ w(\Xh,L)=w(L,\Xh)=L^i\Xh^i(\psi_i)+\Xh(\psi_0), \ \ w(L,L)=L^i\Xh^i(\psi_i)+L(\psi_0),\]
where $L^i = -\psi_i-c\widehat{T}^i$. Since $h = \frac{1}{\gamma-1} c^2=\psi_0 - \frac{1}{2}|\psi_1|^2-\frac{1}{2}|\psi_2|^2$, we can use \eqref{preliminary geometric bounds}, \eqref{bound on T1+1}, \eqref{bound on T2}, $\mathbf{(B_\infty)}$ and \eqref{bound on Z Ti} to {\color{black}show that,} except the terms $e\chi$ and $\chi^2$ on the righthand of \eqref{structure eq 2: L chi}, we have $|Z(L \chi)|\lesssim \mathring{M}\varepsilon$. Therefore, {\color{black}we can use} the bound on $\chi$ from \eqref{preliminary geometric bounds} to derive
\begin{equation}\label{eq:L chi in preliminary geometric bounds 2}
\left\|L (Z(\chi))-e\cdot Z(\chi) +\chi\cdot Z(\chi)\right\|_{L^\infty(\Sigma_t)} \lesssim \mathring{M}\varepsilon.
\end{equation}
According to $(\mathbf{I_\infty})$, on the initial slice $\Sigma_\delta$, we have
\[|Z\chi(\delta,u,\vartheta)| =  |Z\big(c(\slashed{k} - \theta)\big)|=|Z\big(\Xh^i\Xh(\psi_i)-c\theta\big)|\lesssim \varepsilon.\]
Therefore, we can integrate \eqref{eq:L chi in preliminary geometric bounds 2} from to $\delta$ to $t$ to derive
\begin{equation}\label{bound on Z chi}
\|Z(\chi)\|_{L^\infty(\Sigma_t)} \lesssim \mathring{M}\varepsilon,
\end{equation}
provided $\mathring{M}\varepsilon \leqslant 1$. 

To derive estimates on $Z^2(\Th^i)$, we commute $Z \in\mathscr{Z}=\{T,\Xh\}$ with \eqref{eq: LZThi}. By  \eqref{eq:commutator formulas}, for a multi-index $\alpha$ with $|\alpha|=2$, we have
\begin{align*}
L(Z^\alpha(\widehat{T}^{i}))=&\sum_{\beta+\gamma=\alpha} \big(Z^\beta(\Xh(v^1+c)) + Z^\beta(\mathbf{err}_{\Th})\big) Z^\gamma(\Xh^i) \\
&-\,^{(Z')}f\cdot \Xh(Z(\Th^i))-Z'(\,^{(Z)}f)\cdot \Xh(\Th^i)-\,^{(Z)}f\cdot Z'(\Xh(\Th^i)),
\end{align*}
where $\,^{(\Xh)}f=\chi$, $\,^{(T)}f=\zeta+\eta$ and
\begin{align*}
		Z^\beta(\mathbf{err}_{\Th}) &= \sum_{\beta'+\beta''=\beta}Z^{\beta'}(\Xh(\psi_1)) \cdot Z^{\beta''}(\Th^1+1)+  Z^{\beta'}(\Xh(\psi_2)) \cdot Z^{\beta''}(\Th^2) . 
\end{align*}
In view of \eqref{preliminary geometric bounds}, \eqref{bound on T1+1}, \eqref{bound on T2}, \eqref{bound on Z Ti} and $\mathbf{(B_\infty)}$,  we can bound the sum in the expression of $L(Z^\alpha(\widehat{T}^{i}))$ by $\mathring{M}\varepsilon+\mathring{M}\varepsilon\big(|Z^2(\widehat{T}^{1})|+|Z^2(\widehat{T}^{2})|\big)$; by \eqref{bound on Z eta zeta} and \eqref{bound on Z chi}, we can bound the terms with $\,^{(Z)}f$'s also by $\mathring{M}\varepsilon+\mathring{M}\varepsilon\big(|Z^2(\widehat{T}^{1})|+|Z^2(\widehat{T}^{2})|\big)$. Therefore,
\begin{align*}
\big|L(Z^2(\widehat{T}^{i}))\big|\lesssim  \mathring{M}\varepsilon+\mathring{M}\varepsilon\big(|Z^2(\widehat{T}^{1})|+|Z^2(\widehat{T}^{2})|\big).
\end{align*}
We then use Gronwall's inequality and $\mathbf{(I_\infty)}$ to derive $\|Z^2(\Th^i)\| \lesssim \mathring{M} t\varepsilon$. We can also improve the estimates on $Z(\Th^1)$. In fact,
\begin{align*}
L(Z^\alpha(\widehat{T}^{1}))=&\sum_{\beta+\gamma=\alpha} \big(Z^\beta(\Xh(v^1+c)) + Z^\beta(\mathbf{err}_{\Th})\big) Z^\gamma(\Xh^1) \\
&-\,^{(Z')}f\cdot \Xh(Z(\Th^1))-Z'(\,^{(Z)}f)\cdot \Xh(\Th^1)-\,^{(Z)}f\cdot Z'(\Xh(\Th^1)){\color{black}.}
\end{align*}
In the previous estimates, for $\gamma=0$ and $i=2$, we can only bound $\Xh^2$ by a constant. In the current {\color{black}scenario, the bound can be improved to} $|\Xh^2|\lesssim \mathring{M}\varepsilon t$. Therefore, using $|Z^2(\Th^i)| \lesssim \mathring{M} t\varepsilon$, we obtain that
\begin{align*}
\big|L(Z^2(\widehat{T}^{1}))\big|\lesssim  \mathring{M}\varepsilon^2 t+\mathring{M}\varepsilon |Z(\widehat{T}^{1})|.\end{align*}
Since $\big|Z^2(\Th^1)|_{\Sigma_\delta}\big|\lesssim \delta^2 \varepsilon^2$ , we then integrate the above inequality and we conclude that
\begin{equation}\label{bound on Z2 Ti}
		\big|Z^2(\Th^1)\big| \lesssim \mathring{M} t^2 \varepsilon^2, \ \ \big|Z^2(\Th^2)\big| \lesssim \mathring{M} t \varepsilon.
\end{equation}
Similar to \eqref{LZTi byproduct}, we also have the following byproduct:
\begin{equation}\label{LZTi byproduct 2}
\big|L(Z^2(\widehat{T}^{1}))\big|\lesssim  \mathring{M}\varepsilon^2 t,  \ \big|L(Z^2(\widehat{T}^{2}))\big|\lesssim  \mathring{M}\varepsilon.
\end{equation}

Finally, we derive the pointwise bound on $Z^2(\kappa)$. We commute $Z \in\mathscr{Z}=\{T,\Xh\}$ with \eqref{eq: LZkappa} to derive
\begin{equation}\label{Z2 kappa aux}
L(Z^2(\kappa))= -Z^2T(v^1+c) - Z^2(\mathbf{err}_{\kappa})-Z(\,^{(Z)}f)\cdot \Xh(\kappa)-\,^{(Z)}f\cdot Z(\Xh(\kappa))-\,^{(Z)}f\cdot \Xh(Z(\kappa)),
\end{equation}
where $\,^{(\Xh)}f=\chi$, $\,^{(T)}f=\zeta+\eta$ and
\begin{align*}
Z^2(\mathbf{err}_{\kappa}) &= \sum_{|\alpha|+|\beta|=2}Z^\alpha(\Th^1+1) Z^\beta T(\psi_1) + Z^\alpha(\Th^2) Z^\beta(T(\psi_2)). 
\end{align*}
By  \eqref{preliminary geometric bounds}, \eqref{bound on T1+1}, \eqref{bound on T2}, \eqref{bound on Z Ti}, \eqref{bound on Z kappa} and \eqref{bound on Z2 Ti}, we have $|Z^2(\mathbf{err}_{\kappa})|\lesssim \mathring{M}\varepsilon^2t^2$. The rest of the terms in \eqref{Z2 kappa aux} can be bounded in the same way. In particular, we use the ansatz $\mathbf{(B_\infty)}$ that $|Z^2T(v^1+c)| \lesssim \mathring{M}\varepsilon t$. Therefore, 
\begin{align*}
L(Z^2(\kappa))\lesssim \mathring{M}\varepsilon |Z^2(\kappa)|+ \mathring{M}\varepsilon t.
\end{align*}
By Gronwall's inequality, we then conclude that
\begin{equation}\label{bound on Z^2 kappa}
		\big|Z^2(\kappa)\big| \lesssim \mathring{M} t^2 \varepsilon.
\end{equation}
As a {\color{black}byproduct,} we also have
\begin{equation}\label{bound on LZ2 kappa}
\big|L(Z^2(\kappa))\big|\lesssim \mathring{M}\varepsilon t.
\end{equation}
We summarize the estimates derived in this subsection as follows:
\begin{proposition}
{\color{black}Under the bootstrap assumptions} $(\mathbf{B}_2)$ and $(\mathbf{B}_\infty)$, if $\mathring{M}\varepsilon$ is sufficiently small, for all multi-index $\alpha$ with $1\leqslant |\alpha|\leqslant 2$, for all $Z\in \mathscr{Z}$, we have the following pointwise bounds on $\Sigma_t^{u^*}$ for all $t\in [\delta, t^*]$:
\begin{equation}\label{preliminary geometric bounds with derivatives}
\begin{cases}
&\|Z(\zeta)\|_{L^\infty(\Sigma_t)}\lesssim \mathring{M} \varepsilon t, \ \ \|Z(\eta)\|_{L^\infty(\Sigma_t)}\lesssim \mathring{M} \varepsilon t, \  \ \|Z(\chi)\|_{L^\infty(\Sigma_t)}\lesssim \mathring{M}\varepsilon,\\
&\|Z^\alpha(\kappa)\|_{L^\infty(\Sigma_t)}\lesssim \mathring{M}\varepsilon t^2, \ \|Z^\alpha(\Th^1)\|_{L^\infty(\Sigma_t)}\leqslant \mathring{M}\varepsilon^2 t^2, \ \|Z^\alpha(\Th^2)\|_{L^\infty(\Sigma_t)}\leqslant \mathring{M}\varepsilon t.
\end{cases}
\end{equation}
\end{proposition}

\subsection{Change of coordinates and Sobolev inequalities}

\subsubsection{Control of the change of coordinates}
If one passes from the acoustical coordinates to the Cartesian coordinates on $\Sigma_t$, the transformation is controlled by the Jacobi matrix $\begin{pmatrix} \frac{\partial x_1}{\partial u} & \frac{\partial x_1}{\partial \vartheta}\\
	\frac{\partial x_2}{\partial u} & \frac{\partial x_2}{\partial \vartheta}
\end{pmatrix}$. We recall that in the acoustical coordinates $(t,u,\vartheta)$ the vector field $T$ can be written as $T = \frac{\partial}{\partial u} - \Xi \frac{\partial}{\partial \vartheta}$, see \eqref{eq: L T in terms of coordinates}.  On the other hand, $L=\frac{\partial}{\partial t}$ in the acoustical coordinates. Therefore, $L$ commutes with $ \frac{\partial}{\partial u}$ and $\frac{\partial}{\partial \vartheta}$. Hence,
\begin{align*}
	\frac{\partial \Xi}{\partial t}\frac{\partial}{\partial \vartheta} = [L,\Xi X] = -[L,T] = (\zeta + \eta)\Xh.
\end{align*}
Therefore,
\begin{equation}\label{eq:structure eq for Xi}
	L(\Xi) = \frac{1}{\sqrt{\slashed{g}}}(\zeta + \eta).
\end{equation}
Since $\Xi\big|_{\Sigma_\delta}\equiv 0$, by integrating the above equation, for all $t\in [\delta, t^*]$, we have  the following pointwise bound on $\Xi$ :
\begin{equation}\label{bound on Xi}
	\|\Xi\|_{L^\infty(\Sigma_t^{u^*})}\lesssim \mathring{M}\varepsilon t.
\end{equation}
We recall that $X=\frac{\partial}{\partial \vartheta}$ and $\slashed{g}=g(X,X)$. We can apply $L,T$ and $X$ on $x_0, x_1$ and $x_2$ to derive
\begin{align*}
	\frac{\partial x_{\nu}}{\partial t} = L^{\nu}, \quad \frac{\partial x_i}{\partial u} = T^i + \Xi X^i, \quad \frac{\partial x_i}{\partial \vartheta} = X^i = \sqrt{\slashed{g}}\Xh^i, \  i=1,2, \ \nu=0,1,2.
\end{align*}
Hence, the Jacobi matrix of the coordinates transformation $(t,u,\vartheta) \mapsto (x_0,x_1,x_2)$ is given by
\begin{align*}
	\begin{pmatrix}
		\frac{\partial x^0}{\partial t} & \frac{\partial x^0}{\partial u} & \frac{\partial x^0}{\partial \vartheta}\\
		\frac{\partial x^1}{\partial t} & \frac{\partial x^1}{\partial u} & \frac{\partial x^1}{\partial \vartheta}\\
		\frac{\partial x^2}{\partial t} & \frac{\partial x^2}{\partial u} & \frac{\partial x^2}{\partial \vartheta}
	\end{pmatrix} = \begin{pmatrix}
		1 & 0 & 0\\
		L^1 & \kappa \Th^1 + \Xi \sqrt{\slashed{g}}\Xh^1 & \sqrt{\slashed{g}}\Xh^1\\
		L^2 & \kappa \Th^2 + \Xi \sqrt{\slashed{g}}\Xh^2 & \sqrt{\slashed{g}}\Xh^2
	\end{pmatrix}.
\end{align*}
In particular, the Jacobian is given by $\Delta = -\kappa \sqrt{\slashed{g}}$ and for $k=1,2$, we have
\begin{equation}\label{eq: coordinate change}
	\begin{cases}
		&\frac{\partial x^k}{\partial u}=\kappa \Th^k + \Xi \sqrt{\slashed{g}}\Xh^k, \\ 
		&\frac{\partial x^k}{\partial \vartheta}=\sqrt{\slashed{g}}\Xh^k.
	\end{cases}
\end{equation}
We use  \eqref{structure eq 1: L kappa}, \eqref{structure eq 3: L T on Ti Xi Li}, \eqref{eq:structure eq for Xi} and $L(\slashed{g})=2 \slashed{g} \cdot \chi$ to compute the $L$-derivative of the above equations. First of all, we have
\begin{equation}\label{eq: coord aux 1}
	L\big(\frac{\partial x^k}{\partial \vartheta}\big)=\sqrt{\slashed{g}}\chi\Xh^k+\sqrt{\slashed{g}}L(\Xh^k).
\end{equation}
We can then use  \eqref{structure eq 3: L T on Ti Xi Li},\eqref{preliminary geometric bounds} and $\mathbf{(B_\infty)}$ to bound the righthand side by $\mathring{M}\varepsilon$. Similarly, we can bound $L\big(\frac{\partial x^2}{\partial u}\big)$ in the same manner. This yields
\begin{equation}\label{eq: coordinate bound 1}
	\big|L\big(\frac{\partial x^1}{\partial \vartheta}\big)\big| + \big|L\big(\frac{\partial x^2}{\partial \vartheta}\big)\big| + \big|L\big(\frac{\partial x^2}{\partial u}\big)\big| \lesssim \mathring{M} \varepsilon.
\end{equation}
The bound on $\frac{\partial x^1}{\partial u}$ is different from the previous ones. In fact, we compute that
\begin{align*}
	L\big(\frac{\partial x^1}{\partial u}\big) + 1 &= (\Th^1 + 1) + (L(\kappa)-1)\Th^1 + \kappa L(\Th^1) + L(\Xi \sqrt{\slashed{g}}\Xh^1).
\end{align*}
Thus, we use \eqref{structure eq 3: L T on Ti Xi Li},\eqref{preliminary geometric bounds},\eqref{bound on Lkappa} and $\mathbf{(B_\infty)}$ to bound the righthand side by $\mathring{M}\varepsilon$. This yields
\begin{equation}\label{eq: coordinate bound 2}
	\big|L\big(\frac{\partial x^1}{\partial u}\big)+1\big| \lesssim \mathring{M} \varepsilon.
\end{equation}
We now integrate \eqref{eq: coordinate bound 1} and\eqref{eq: coordinate bound 2}. By $\mathbf{(I_\infty)}$, we conclude that
\begin{equation}\label{eq: fundamental control on coordinate change 1}
	\big|\frac{\partial x_1}{\partial \vartheta}\big|+ \big|\frac{\partial x_1}{\partial u}+t\big|+\big|\frac{\partial x_2}{\partial \vartheta}-1\big|+\big|\frac{\partial x_2}{\partial u}\big|\lesssim \mathring{M} \varepsilon t.
\end{equation}

We can also commute $Z\in \mathscr{Z}$ with \eqref{eq: coord aux 1} and we have eight possible quantities $L(Z(f))$ where $Z\in \{\Xh,\Th\}$ and $f\in \{\frac{\partial x^k}{\partial \vartheta},\frac{\partial x^k}{\partial u}|k=1,2\}$. We treat $L\big(\Xh(\frac{\partial x^k}{\partial \vartheta})\big)$ in details and the rest can be bounded exactly in the same manner. 
In view of \eqref{eq:commutator formulas}, we have $[L,\Xh]=-\chi\cdot \Xh$. Thus, by applying $\Xh$ {\color{black}to \eqref{eq: coord aux 1},} we have
\begin{align*}
	L\big(\Xh\big(\frac{\partial x^k}{\partial \vartheta}\big)\big)=-\chi\cdot\Xh\big(\frac{\partial x^k}{\partial \vartheta}\big)+\Xh\big(\sqrt{\slashed{g}}\chi\Xh^k\big)+\Xh\big(\sqrt{\slashed{g}}L(\Xh^k)\big).
\end{align*}
We can use \eqref{structure eq 3: L T on Ti Xi Li} to replace $L(\Xh^k)$. Thus, by \eqref{preliminary geometric bounds}, \eqref{preliminary geometric bounds with derivatives} and $\mathbf{(B_\infty)}$, we can bound the second term on the righthand side of the above equation by $\mathring{M}\varepsilon$. Hence,
\begin{align*}
	\big|L\big(\Xh\big(\frac{\partial x^k}{\partial \vartheta}\big)\big)+\chi\cdot\Xh\big(\frac{\partial x^k}{\partial \vartheta}\big)\big|\lesssim \mathring{M}\varepsilon.
\end{align*}
Therefore, since $\|\chi\|_{L^\infty(\Sigma_t)}\lesssim \mathring{M}\varepsilon$, we can use Gronwall's inequality and $\mathbf{(I_\infty)}$ to conclude that
\[
\big|\Xh\big(\frac{\partial x^1}{\partial \vartheta}\big)\big|  \lesssim \mathring{M} \varepsilon t.
\] provided $\mathring{M}\varepsilon$ is sufficiently small. We proceed in a similar manner for other terms and we finally have
\begin{equation}\label{eq: fundamental control on coordinate change aux}
	\big|Z\big(\frac{\partial x_k}{\partial \vartheta}\big)\big|+ \big|Z\big(\frac{\partial x_k}{\partial u}\big)\big|\lesssim \mathring{M} \varepsilon t,  \ \ Z\in \mathscr{Z},\ k=1,2.
\end{equation}
Since $T = \frac{\partial}{\partial u} - \Xi \frac{\partial}{\partial \vartheta}$ and $\Xi=\frac{1}{\sqrt{\slashed{g}}}\frac{\partial}{\partial \vartheta}$, by \eqref{bound on Xi}, for a given $C^1$ function $f$ defined on $\Sigma_t^{u^*}$, we have
\begin{align*}
	\big|\frac{\partial f}{\partial u}\big|^2+\big|\frac{\partial f}{\partial \vartheta}\big|^2&=\big|T(f)+\Xi \sqrt{\slashed{g}}\Xh(f)\big|^2+\big|\sqrt{\slashed{g}}\Xh(f)\big|^2\\
	&\lesssim  |Tf|^2 + |\Xh f|^2.
\end{align*}
We can take $f=\frac{\partial x_k}{\partial \vartheta}$ and $\frac{\partial x_k}{\partial u}$. By \eqref{eq: fundamental control on coordinate change aux}, we derive
\begin{equation}\label{eq: fundamental control on coordinate change 2}
	\big|\frac{\partial^2 x_k}{\partial \vartheta^2}\big|+ \big|\frac{\partial^2 x_k}{\partial u\partial \vartheta}\big|+\big|\frac{\partial^2 x_k}{\partial u^2}\big|\lesssim \mathring{M} \varepsilon t, \ \ k=1,2.
\end{equation}
We summarize the estimates on the coordinates transformation $(t,u,\vartheta) \mapsto (x_0,x_1,x_2)$ as follows:
\begin{proposition}
	{\color{black}Under the bootstrap assumptions} $(\mathbf{B}_2)$ and $(\mathbf{B}_\infty)$, if $\mathring{M}\varepsilon$ is sufficiently small, we have the following pointwise bounds on $\Sigma_t^{u^*}$ for all $t\in [\delta, t^*]$:
	\begin{equation}\label{eq: fundamental control on coordinate change}
		\begin{cases}
			&\big|\frac{\partial x_1}{\partial \vartheta}\big|+ \big|\frac{\partial x_1}{\partial u}+t\big|+\big|\frac{\partial x_2}{\partial \vartheta}-1\big|+\big|\frac{\partial x_2}{\partial u}\big|\lesssim \mathring{M} \varepsilon t, \\
			&\big|\frac{\partial^2 x_k}{\partial \vartheta^2}\big|+ \big|\frac{\partial^2 x_k}{\partial u\partial \vartheta}\big|+\big|\frac{\partial^2 x_k}{\partial u^2}\big|\lesssim \mathring{M} \varepsilon t, \ \ k=1,2.
		\end{cases}
	\end{equation}
\end{proposition}

\subsubsection{Sobolev inequalities}
We recall that $\|f\|_{L^2(\Sigma_t^u)}=\sqrt{\int_{\Sigma_t^u} \left|f\right|^2}$. We have the following Sobolev inequality:
\begin{lemma}
	{\color{black}Under the bootstrap assumptions} $(\mathbf{B}_2)$ and $(\mathbf{B}_\infty)$, if $\mathring{M}\varepsilon$ is sufficiently small, for all $t\in [\delta, t^*]$, for any smooth function $f$ defined on $\Sigma_t^u$, we have 
	\begin{equation}\label{ineq:Sobolev inequalities}
		\|f\|_{L^\infty(\Sigma_t)}\lesssim \sum_{k+l\leqslant 2}\|\Xr^k \Tr^l (f)\|_{L^2(\Sigma_t)}.
	\end{equation}
\end{lemma}
\begin{proof}
	First of all, we have the usual Sobolev inequality:
	\begin{equation}\label{eq:Sobolev u vartheta to Tr Xh standard in u vartheta}
		{\color{black}\|f\|_{L^\infty(\Sigma_t)}^2}\lesssim \sum_{k+l\leqslant 2, (k,l)\neq (1,1)}\int_{0}^{u^*}\int_{0}^{2\pi}\left|\partial^k_\theta \partial^l_u  f(u',\vartheta')\right|^2du'd\vartheta'.
	\end{equation}
	
	As a consequence of this inequality, for a given $C^1$ function $f$ defined on $\Sigma_t^{u^*}$, we have
	\begin{align*}
		\big|\frac{\partial f}{\partial u}\big|^2+\big|\frac{\partial f}{\partial \vartheta}\big|^2&=\big|\frac{\partial x_1}{\partial u}\frac{\partial f}{\partial x_1}+\frac{\partial x_2}{\partial u}\frac{\partial f}{\partial x_2}\big|^2+\big|\frac{\partial x_1}{\partial \vartheta}\frac{\partial f}{\partial x_1}+\frac{\partial x_2}{\partial \vartheta} \frac{\partial f}{\partial x_2}\big|^2\\
		&\lesssim  \big(\big|\frac{\partial x_1}{\partial \vartheta}\big|^2 + \big|\frac{\partial x_1}{\partial u}\big|^2\big)|\Trh f|^2 + \big(\big|\frac{\partial x_2}{\partial \vartheta}\big|^2 + \big|\frac{\partial x_2}{\partial u}\big|^2\big)|\Xr f|^2.
	\end{align*}
	For sufficiently small $\mathring{M}\varepsilon$, \eqref{eq: fundamental control on coordinate change 1} yields
	\begin{equation}\label{eq:Sobolev u vartheta to Tr Xh 1 derivatives}
		\big|\frac{\partial f}{\partial u}\big|^2+\big|\frac{\partial f}{\partial \vartheta}\big|^2\lesssim |\Tr f|^2 + |\Xr f|^2{\color{black}.}
	\end{equation}
	We also have
	\begin{align*}
		\frac{\partial^2 f}{\partial \vartheta^2}&=-\frac{1}{t}\frac{\partial^2 x_1}{\partial \vartheta^2}\Tr(f)+\frac{\partial^2 x_2}{\partial \vartheta^2}\Xr(f)-\frac{1}{t}\frac{\partial x_1}{\partial \vartheta}\frac{\partial }{\partial \vartheta}\big(\Tr(f)\big)+\frac{\partial x_2}{\partial \vartheta}\frac{\partial }{\partial \vartheta}\big(\Xr(f)\big),\\
		\frac{\partial^2 f}{\partial u^2}&=-\frac{1}{t}\frac{\partial^2 x_1}{\partial u^2}\Tr(f)+\frac{\partial^2 x_2}{\partial u^2}\Xr(f)-\frac{1}{t}\frac{\partial x_1}{\partial u}\frac{\partial}{\partial u}\big(\Tr(f)\big)+\frac{\partial x_2}{\partial u}\frac{\partial}{\partial u}\big(\Xr(f)\big).
	\end{align*}
	We use \eqref{eq:Sobolev u vartheta to Tr Xh 1 derivatives} to bound $\frac{\partial }{\partial \vartheta}\big(\Tr(f)\big)$, $\frac{\partial }{\partial u}\big(\Tr(f)\big)$, $\frac{\partial }{\partial \vartheta}\big(\Xr(f)\big)$ and $\frac{\partial }{\partial u}\big(\Xr(f)\big)$.  This leads to
	\begin{align*}
		\big|\frac{\partial^2 f}{\partial \vartheta^2}\big|^2+\big|\frac{\partial^2 f}{\partial u^2}\big|^2 &\lesssim \frac{1}{t^2}\big(\big|\frac{\partial^2 x_1}{\partial \vartheta^2}\big|^2+\big|\frac{\partial^2 x_1}{\partial u^2}\big|^2\big)|\Tr f|^2+
		\big(\big|\frac{\partial^2 x_2}{\partial \vartheta^2}\big|^2+\big|\frac{\partial^2 x_2}{\partial u^2}\big|^2\big)|\Xr f|^2\\
		& \ \ +\big[\frac{1}{t^2}\big(\big|\frac{\partial x_1}{\partial \vartheta}\big|^2+\big|\frac{\partial x_1}{\partial u}\big|^2\big)+
		\big(\big|\frac{\partial x_2}{\partial \vartheta}\big|^2+\big|\frac{\partial x_2}{\partial u}\big|^2\big)\big]\big(|\Tr^2 f|^2+|\Tr \Xr f|^2+|\Xr^2f|^2\big)\\
		&\lesssim |\Tr f|^2+
		|\Xr f|^2+|\Tr^2 f|^2+|\Tr \Xr|^2+|\Xr^2f|^2.
	\end{align*}
	In the last step, we have used \eqref{eq: fundamental control on coordinate change}. Combined with \eqref{eq:Sobolev u vartheta to Tr Xh 1 derivatives}, the standard Sobolev inequality \eqref{eq:Sobolev u vartheta to Tr Xh standard in u vartheta} yields the derived estimate.
\end{proof}

\subsection{Comparison lemma and pointwise bounds on acoustical waves}
\subsubsection{Comparison between two null frames}
{\color{black}According to} \eqref{preliminary geometric bounds}, \eqref{bound on kappa} and $\mathbf{(B_2)}$, for all $\psi \in \{w,\wb,\psi_2\}$, for all multi-index $\alpha$ with $|\alpha|\leqslant \Ntop$, we have
\[\int_{\Sigma_t^u} t^2L(\Zr^\alpha\psi)^2+t^2\Xh(\Zr^\alpha\psi)^2+\Lb(\Zr^\alpha\psi)^2\lesssim \mathring{M}\varepsilon^2t^2\]
except for $\alpha=0$ and $\psi=\wb$. Since $\Lb=c^{-1}\kappa L+2T$, the above bounds imply that 
\begin{equation}\label{eq: B2 equi}
	\int_{\Sigma_t^u} t^2\Xh(\Zr^\alpha\psi)^2+T(\Zr^\alpha\psi)^2\lesssim \mathring{M}\varepsilon^2t^2.
\end{equation}
On the other hand, the frame $(\Tr,\Xr)$ are related to $(T,\Xh)$ by the following formulas:
\[\Tr= -\frac{\kappar}{ \kappa  (\Th^1)^2+\ (\Th^2)^2}\big(\Th^1 T+  \Th^2 X\big),\ \ \Xr= \frac{1}{ \kappa  (\Th^1)^2+\ (\Th^2)^2}\big(\Th^2 T- \kappa \Th^1 X\big).\]
In view of the improved bounds \eqref{bound on kappa}, \eqref{bound on T1+1} and \eqref{bound on T2}, \eqref{eq: B2 equi} implies that
\begin{equation}\label{eq: B2 equi bis}
	\int_{\Sigma_t^u} t^2\Xr(\Zr^\alpha\psi)^2+\Tr(\Zr^\alpha\psi)^2\lesssim \mathring{M}\varepsilon^2t^2.
\end{equation}
This bound is sufficient to bound the $L^\infty$ norms of the acoustical waves. In the rest of this subsection, we will derive a lemma to compare the new null frame $(\Lr,\Lbr,\Xr)$ with the old null frame$(L,\Lb,\Xh)$. First of all, for a smooth function $f$ defined on $\mathcal{D}(t^*,u^*)$, we have
\begin{equation}\label{change of frame 1}
	\begin{cases}
		&Lf-\Lr f=c\big(\frac{\Th^1+1}{\kappar}\Tr(f) -\Th^2\Xr(f)\big),\\
		&Tf-\Tr f=-\big[\left(\frac{\kappa}{\kappar}-1\right)\Th^1+(\Th^1+1)\big]\Tr(f) +\kappa \Th^2\Xr(f),\\
		&Xf-\Xr f=- \frac{\Th^2}{\kappar} \Tr(f) -(\Th^1+1)\Xr(f).
	\end{cases}
\end{equation}
By \eqref{change of frame 1}, we have
\[
\Tr f-T f
= \frac{-(\Th^1+1)+(1-\frac{\kappa}{\kappar})- (1-\kappa) \cdot \Th^2\cdot \frac{\Th^2}{\kappar}}{ \frac{\kappa}{\kappar} + (1-\kappa) \cdot \Th^2 \cdot \frac{\Th^2}{\kappar} } Tf+\frac{\Th^2}{ \frac{\kappa}{\kappar} + (1-\kappa) \cdot \Th^2 \cdot \frac{\Th^2}{\kappar} } \Xh f.
\]
Therefore,  \eqref{bound on kappa}, \eqref{bound on T1+1} and \eqref{bound on T2} imply that
\begin{equation}\label{eq: compa aux 1}
	\begin{split}
		|\Tr f|&\leqslant |Tf|+\mathring{M}\varepsilon |Tf|+\mathring{M}\varepsilon t|\Xh f|\lesssim |Tf|+\mathring{M}\varepsilon t|\Xh f|\\
		&\lesssim t|Lf|+|\Lb f|+\mathring{M}\varepsilon t|\Xh f|.
	\end{split}
\end{equation}
By \eqref{change of frame 1}, we also have
\[\Xr f-\Xh f=\frac{\frac{\Th^2}{\kappa} }{  (\Th^1)^2+   \Th^2 \frac{\Th^2}{\kappa}} T  (f)-\frac{\Th^1(\Th^1+1)+\Th^2 \frac{\Th^2}{\kappa}}{  (\Th^1)^2+   \Th^2 \frac{\Th^2}{\kappa}} \Xh (f).\]
Hence,
\begin{equation}\label{eq: compa aux 2}
	\begin{split}
		|\Xr f|&\leqslant |\Xh f|+\mathring{M}\varepsilon |Tf|+\mathring{M}\varepsilon t|\Xh f|\lesssim  |\Xh f|+\mathring{M}\varepsilon |Tf|\\
		&\lesssim |\Xh f|+\mathring{M}\varepsilon t |Lf|+ \mathring{M}\varepsilon |\Lb f|.
	\end{split}
\end{equation}
By virtue of \eqref{eq: compa aux 1} and \eqref{eq: compa aux 2}, the first equation of \eqref{change of frame 1} implies that
\begin{align*}
	|\Lr f|&\lesssim |Lf|+\mathring{M}\varepsilon^2 t |\Tr f|+\mathring{M}\varepsilon t |\Xr f|\lesssim |Lf|+\mathring{M}\varepsilon t|\Xh f|+\mathring{M}\varepsilon^2 t |\Lb f|.
\end{align*}
Finally,  if $\mathring{M}\varepsilon$ is sufficiently small, {\color{black}for $\Lbr= c^{-1} \kappar \Lr + 2\Tr$} we have
\begin{align*}
	|\Lbr f|&\lesssim \kappar |\Lr f|+|\Tr f|\lesssim t |\Xh f|+t|Lf|+|\Lb f|.
\end{align*}
To summarize, we have the following comparison lemma:
\begin{proposition}
	{\color{black}Under the bootstrap assumptions} $(\mathbf{B}_2)$ and $(\mathbf{B}_\infty)$, if $\mathring{M}\varepsilon$ is sufficiently small,  for all smooth functions $f$ defined on $\mathcal{D}(t^*,u^*)$, we have the following pointwise bounds:
	\begin{equation}\label{comparison bound 1: L infty}
		\begin{cases}
			|\Lr f|&\lesssim |Lf|+ \varepsilon t |\Xh f|+ \varepsilon^2 t |\Lb f|,\\
			|\Xr f|&\lesssim |\Xh f|+ \varepsilon t |Lf|+  \varepsilon |\Lb f|,\\
			|\Lbr f|&\lesssim t |\Xh f|+t|Lf|+|\Lb f|,\\
			|\Tr f|&\lesssim t|Lf|+|\Lb f|+ \varepsilon t|\Xh f|.
		\end{cases}
	\end{equation}
\end{proposition}
\begin{corollary}
	For all $\psi \in \{\wb,w,\psi_2\}$, for all multi-index $\alpha$ with $|\alpha|\leqslant \Ntop$, we have
	\begin{equation}\label{comparison bound 1: L 2}
		\begin{cases}
			&t^2 \displaystyle\int_{\Sigma_t^u}|\Lr \Zr ^\alpha \psi|^2+|\Xr \Zr ^\alpha \psi|^2\lesssim \mathscr{E}_{|\alpha|}(\psi)(t,u)+t^2\varepsilon^2\EB_{|\alpha|}(\psi)(t,u),\\
			&\displaystyle\int_{\Sigma_t^u}|\Lbr \Zr ^\alpha\psi|^2+|\Tr \Zr ^\alpha\psi|^2\lesssim\mathscr{E}_{|\alpha|}(\psi)(t,u)+\EB_{|\alpha|}(\psi)(t,u).
		\end{cases}
	\end{equation}
\end{corollary}

\subsubsection{$L^\infty$ estimates on acoustical waves}
For all multi-index $\alpha$ with $|\alpha|\leqslant \Ninf=\Ntop-1$, for all $\psi \in \{w,\wb,\psi_2\}$, except for the case $\Zr^\alpha \psi=T\wb$, we apply the Sobolev inequality \eqref{ineq:Sobolev inequalities}
to derive pointwise bound for $\mathring{Z}^\alpha \psi$:
\begin{align*}\|\mathring{Z}^\alpha\psi\|_{L^\infty(\Sigma_t)}\lesssim \sum_{k+l\leqslant 2}\|\Xr^k \Tr^l \mathring{Z}^\alpha\psi\|_{L^2(\Sigma_t)}.
\end{align*}
The righthand side is bounded by a universal constant times $\mathring{M}\varepsilon$. If at least one $\Tr$ appears in $\Zr^{\alpha}$, thus, we can rewrite the above inequality as
\begin{align*}
	\|\mathring{Z}^\alpha\psi\|_{L^\infty(\Sigma_t)}\lesssim \sum_{k+l\leqslant 2}\|\Tr\left(\Xr^k \Tr^l \mathring{Z}^{\alpha-1}\psi\right)\|_{L^2(\Sigma_t)}\lesssim \mathring{M}\varepsilon t.
\end{align*}
Therefore, we have proved the following $L^\infty$ estimates on acoustical waves:
\begin{proposition}
	For all multi-index $\alpha$ with $|\alpha|\leqslant \Ninf$, for all $\psi \in \{w,\wb,\psi_2\}$, except for the case $\Zr^\alpha \psi=T\wb$, we have
	\begin{equation}\label{ineq: L infty bound}
		\|\mathring{Z}^\alpha\psi\|_{L^\infty(\Sigma_t)}\lesssim  \begin{cases}\mathring{M}\varepsilon, \ \ & \text{if}~\Zr^\alpha=\Xr^\alpha;\\
			\mathring{M}\varepsilon t, \ \ & \text{otherwise}.
		\end{cases}
	\end{equation}
\end{proposition}

\begin{remark}[How to use the pointwise bounds]\label{remark:techical nonlinear}
	
	Given an integer $m\geqslant 2$ and functions $F_1,\cdots, F_m$ in such a way that ${\rm ord}(F_1)\leqslant {\rm ord}(F_2)\leqslant \cdots \leqslant {\rm ord}(F_m)$. For each $i\leqslant m$,  $\|F_i\|_{L^2(\Sigma_t)}$ is bounded. In addition, if ${\rm ord}(F_i)\leqslant \Ninf$, $\|F_i\|_{L^\infty(\Sigma_t)}$ is bounded. 
	
	If $\sum_{i=1}^m {\rm ord}(F_i)\leqslant N_{\rm top}+1$, we have the following two estimates:
	\begin{equation}\label{trivial multilinear estimates}
		\begin{cases}
			&\big|\int_{\Sigma_t} F_1\cdot F_2\cdots F_m\big|\leqslant  \|F_1\|_{L^\infty(\Sigma_t)}\cdots \|F_{m-2}\|_{L^\infty(\Sigma_t)}\|F_{m-1}\|_{L^2(\Sigma_t)}\|F_m\|_{L^2(\Sigma_t)},\\
			&\|F_1\cdot F_2\cdots F_m\|_{L^2(\Sigma_t)}\leqslant \|F_1\|_{L^\infty(\Sigma_t)}\cdots \|F_{m-1}\|_{L^\infty(\Sigma_t)}\|F_m\|_{L^2(\Sigma_t)}.
		\end{cases}
	\end{equation}
	The proof is trivial. It suffices to observe that for $i\leqslant m-1$, ${\rm ord}(F_i)\leqslant \Ninf$. Therefore, we can use H\"older's inequality with $L^\infty$ bounds on such $F_i$'s.
	
	In the rest of the paper, we will frequently encounter the above scenario. In most of the cases, the $F_i$'s are $\Zr^\alpha \psi$ where $\psi \in \{w,\wb,\psi_2\}$. 
\end{remark}

\section{Linear energy estimates}\label{section:linear-energy-estimates}
\subsection{Energy estimates for linear waves in rarefaction wave region}\label{section: fundamental energy estimates}
In the rest of the paper, we always assume that $\mathring{M} \varepsilon$ is sufficiently small so that the previous preliminary estimates hold. Based on these estimates, we derive the fundamental energy estimates for {\color{black}the linear wave equation} \eqref{model linear wave equation}, i.e., $ \Box_{g}\psi = \varrho$, in the rarefaction wave region in this section. To simplify the notations, we use $\mathcal{E}(t,u)$ to denote $\mathcal{E}(\psi)(t,u)$; similarly, we also use notations $\Eb(t,u)$, $\mathcal{F}(t,u)$ and $\Fb(t,u)$.

\subsubsection{Multiplier $\Lh$}
We start with identity \eqref{identity: energy with Lh} where we {\color{black}take the multiplier vector field} $J= \Lh$. We bound $Q_1,\cdots,Q_4$ one by one.

To bound $Q_1$, we notice that $T(c^{-1}\kappa)=c^{-1}T\kappa-c^{-2}\kappa Tc$. By \eqref{preliminary geometric bounds}, we have $\|T(\kappa)\|_{L^\infty(\Sigma_t)}\lesssim \mathring{M}\varepsilon t$. In view of \eqref{bound on kappa}, we conclude that $|T\kappa|\lesssim t \lesssim c^{-1}\kappa$, provided $\mathring{M}\varepsilon$ is sufficiently small.This implies that
{\color{black}\begin{align*}
	|Q_1|&= \big|\int_{\mathcal{D}(t,u)}T (c^{-1}\kappa)  |L\psi|^2 \big|\lesssim\int_{\mathcal{D}(t,u)} c^{-1}\kappa |L\psi|^2 =\int_{0}^{u} \mathcal{F}(t,u')du'.
\end{align*} }

To bound $Q_2$, in view of \eqref{bound on L kappa} and $c\approx 1$, we have $L(\kappa^2) \approx c^{-1}\kappa$.  Therefore,
{\color{black}\begin{align*}
	|Q_2|&= \int_{\mathcal{D}(t,u)} \frac{1}{2}L (\kappa^2	)  |\Xh\psi|^2 \lesssim \int_{\mathcal{D}(t,u)} c^{-1}\kappa |\Xh\psi|^2 =\int_{0}^{u} \Fb(t,u')du'.
\end{align*} }

To bound $Q_3$, in view of \eqref{preliminary geometric bounds}, we notice that $|\zeta+\eta|\lesssim \mathring{M}\varepsilon t$ and $| \Xh(c^{-1}\kappa)|\lesssim \mathring{M}\varepsilon t$. Therefore, 
\begin{align*}
	|Q_3|&=\left|\int_{\mathcal{D}(t,u)} \left(c^{-1}\kappa(\zeta+\eta)-\mu \Xh(c^{-1}\kappa)\right)L\psi\cdot \Xh\psi \right|\lesssim  \mathring{M}\varepsilon \int_{\mathcal{D}(t,u)} t^2 |L\psi| |\Xh\psi|\\
	& \lesssim  \mathring{M}\varepsilon \int_{\mathcal{D}(t,u)} c^{-1}\kappa \left(c^{-1}\kappa (L\psi)^2+\mu (\Xh\psi)^2\right)\lesssim \mathring{M}\varepsilon \int_{\delta}^{t} \E(t',u)dt'.
\end{align*}

To bound $Q_4$,in view of \eqref{preliminary geometric bounds}, we have $|\chi|\lesssim \mathring{M}\varepsilon$. Therefore, we have
\begin{align*}
	|Q_4|&= \left|\int_{\mathcal{D}(t,u)}  \frac{\kappa^2\chi}{2}(\Xh \psi)^2 +\frac{c^{-1}\kappa \chi}{2}L\psi\cdot \Lb\psi\right|\lesssim \mathring{M} \varepsilon \left|\int_{\mathcal{D}(t,u)}   \kappa^2 (\Xh \psi)^2 +\left(\kappa^2  L(\psi)^2+\Lb(\psi)^2\right)\right|\\
	& \lesssim \mathring{M}\varepsilon \int_{\delta}^{t} \E(t',u)+\Eb(t',u)dt'.
\end{align*}

Putting all the estimates in \eqref{identity: energy with Lh} , we have
\begin{equation}\label{ineq: energy ineq for Lh}
	\begin{cases}
		&\displaystyle\mathcal{E}(t,u)+\mathcal{F}(t,u)=\mathcal{E}(\delta,u) + \mathcal{F}(t,0)-\int_{\mathcal{D}(t,u)} \mu\varrho\cdot \Lh\psi +\sum_{1\leqslant j \leqslant 4}Q_i,\\
		&\displaystyle\sum_{1\leqslant j \leqslant 4}|Q_i| \lesssim \mathring{M}\varepsilon \int_{0}^{u} \mathcal{F}(t,u')+\Fb(t,u')du'+ \int_{\delta}^{t} \E(t',u)+\Eb(t',u)dt'.
	\end{cases}
\end{equation}

\subsubsection{Multiplier $\Lb$}

We turn to the identity \eqref{identity: energy with Lb} where we {\color{black}take the multiplier vector field} $J= \Lh$. We bound $\underline{Q}_1,\cdots,\underline{Q}_4$ one by one.

To bound $\underline{Q}_1$, it is straightforward to check that $L(c^{-1}\kappa)\lesssim 1$ and $\Lb(c^{-1}\kappa)\lesssim t$. Therefore,
\begin{align*}
	|\underline{Q}_1|&=\big|\int_{\mathcal{D}(t,u)} \frac{1}{2} \left( \mu  L(c^{-1}\kappa)+ \Lb(c^{-1}\kappa) \right)(\Xh\psi)^2\big|\lesssim \int_{\mathcal{D}(t,u)}t(\Xh\psi)^2\lesssim \int_{0}^{u} \Fb(t,u')du'.
\end{align*}

To bound $\underline{Q}_2$, we use $|\zeta+\eta|\lesssim \mathring{M}\varepsilon t$ and $\kappa \approx t$ to derive
\begin{align*}
	|\underline{Q}_2|&=\big|\int_{\mathcal{D}(t,u)} (\zeta+\eta) \Lb\psi\cdot \Xh\psi\big|\lesssim \mathring{M}\varepsilon \int_{\mathcal{D}(t,u)} t|\Lb\psi||\Xh\psi|\\
	&\lesssim \mathring{M}\varepsilon \int_{\mathcal{D}(t,u)} |\Lb\psi|^2+t^2|\Xh\psi|\lesssim \int_{\delta}^{t}\Eb(t',u)dt'.
\end{align*}

To bound $\underline{Q}_3$, we use $|\Xh(c^{-1}\kappa)|\lesssim \mathring{M}\varepsilon t$ to derive
\begin{align*}
	|\underline{Q}_3|&=\big|\int_{\mathcal{D}(t,u)} \mu \Xh(c^{-1}\kappa) L\psi \cdot \Xh\psi\big|\lesssim \mathring{M}\varepsilon \int_{\mathcal{D}(t,u)} t^2|L\psi||\Xh\psi|\lesssim \int_{\delta}^{t}\E(t',u)dt'.
\end{align*}

To bound $\underline{Q}_4$, we use $|\chib|\lesssim \mathring{M}\varepsilon t$ because $\chib=c^{-1}\kappa\big(-2\Xh^j\cdot \Xh(\psi_j)-\chi\big)$. Thus,
\begin{align*}
	|\underline{Q}_4|&=\left|\int_{\mathcal{D}(t,u)}\frac{1}{2}{\mu}\chib \big( (\Xh \psi)^2 + \frac{1}{\mu}L\psi\cdot \Lb\psi \big)\right|\leqslant \big|\int_{\mathcal{D}(t,u)}\frac{1}{2}{ \mu}\chib \big( (\Xh \psi)^2 + \frac{1}{2}(L\psi)^2+\frac{1}{2\mu^2}(\Lb\psi)^2\big) \big|\\
	&\lesssim \mathring{M}\varepsilon \int_{\mathcal{D}(t,u)}t^2 \left( (\Xh \psi)^2 +(L\psi)^2\right)+ (\Lb\psi)^2\lesssim \mathring{M}\varepsilon\int_{\delta}^{t} \E(t',u)+\Eb(t',u)dt'.
\end{align*}
Putting all the estimates in \eqref{identity: energy with Lb} , we have
\begin{equation}\label{ineq: energy ineq for Lb}
	\begin{cases}
		&\displaystyle\Eb(t,u)+\Fb(t,u)=\Eb(\delta,u) + \Fb(t,0)-\int_{\mathcal{D}(t,u)} \mu \varrho\cdot\Lb\psi +\sum_{1\leqslant j \leqslant 4}\underline{Q}_i,\\
		&\displaystyle\big|\sum_{1\leqslant j \leqslant 4}Q_i\big| \lesssim \int_0^u \Fb(t,u')du'+ \int_{\delta}^{t} \E(t',u)+\Eb(t',u)dt',
	\end{cases}
\end{equation}
provided $\mathring{M}\varepsilon$ is sufficiently small.

\subsubsection{The fundamental energy inequality}
We define the total energy and the total flux associated to $\psi$ as follows:
\[\mathscr{E}(\psi)(t,u)=\mathcal{E}(\psi)(t,u)+\underline{\mathcal{E}}(\psi)(t,u), \ \mathscr{F}(\psi)(t,u)=\mathcal{F}(\psi)(t,u)+\underline{\mathcal{F}}(\psi)(t,u).\]
Therefore, in view of \eqref{ineq: energy ineq for Lh} and \eqref{ineq: energy ineq for Lb}, for sufficiently small $\mathring{M}\varepsilon$, we have the following fundamental energy identity:
\begin{equation}\label{ineq: energy ineq for total}
	\begin{split}
		\mathscr{E}(\psi)(t,u)+\mathscr{F}(\psi)(t,u)&=\mathscr{E}(\psi)(\delta,u)+ \mathscr{F}(\psi)(t,0)+\mathscr{N}(\psi)(t,u)+\underline{\mathscr{N}}(\psi)(t,u)+{\mathbf{Err}}
	\end{split}
\end{equation}
where the nonlinear terms $\mathscr{N}(t,u)$ and $\underline{\mathscr{N}}(t,u)$ are defined as
\[\mathscr{N}(\psi)(t,u)=-\int_{\mathcal{D}(t,u)}\mu\varrho\cdot \Lh\psi,     \ \ \underline{\mathscr{N}}(\psi)(t,u)=-\int_{\mathcal{D}(t,u)}\mu \varrho\cdot\Lb\psi,\]
and the error term $\mathbf{Err}$ satisfies
\[\left|\mathbf{Err}\right|\lesssim \int_{0}^{u} \mathscr{F}(\psi)(t,u')du'+ \int_{\delta}^{t} \mathscr{E}(\psi)(t',u)dt'.\]

\subsection{Bilinear error integrals}\label{section: bilinear error integrals} 
We introduce three types of bilinear error integrals associated to a pair of functions $(\psi,\psi')$. 

The first one is
\begin{equation}\label{def:L1}
	\mathscr{L}_1(\psi,\psi')(t,u)=\int_{\mathcal{D}(t,u)} \kappa |\Xh \psi||L\psi'|.
\end{equation}
It is clear that
\begin{equation}\label{ineq: estimates for L1}
	\mathscr{L}_1(\psi,\psi')(t,u)\lesssim \int_{0}^{u} \mathscr{F}(\psi)(t,u')+ \mathscr{F}(\psi')(t,u')du'.
\end{equation}

The second and third bilinear error integrals are
\begin{equation}\label{def:L2 L3}
	\mathscr{L}_2(\psi,\psi')(t,u)=\int_{\mathcal{D}(t,u)} |L\psi||\Lb\psi'|, \ \ \mathscr{L}_3(\psi,\psi')(t,u)=\int_{\mathcal{D}(t,u)}  |\Xh \psi||\Lb\psi'|. 
\end{equation}
For any small positive constant $a_0$ (it will be determined later on in the energy estimates for $w, \wb$ and $\psi_2$), we have
\begin{align*}
	\mathscr{L}_2(\psi,\psi')(t,u)&\leqslant \int_{\mathcal{D}(t,u)} \frac{t}{2a_0} |L\psi|^2+\frac{a_0}{2}\frac{|\Lb\psi'|^2}{t}\\
	&\lesssim \frac{1}{a_0}\int_{0}^{u} \mathscr{F}(\psi)(t,u')du'+a_0\int_\delta^t  \frac{\mathscr{E}(\psi')(t',u)}{t'}dt'.
\end{align*}
Similar estimates also hold for {\color{black}$\mathscr{L}_3(\psi,\psi')(t,u)$.} Therefore,
\begin{equation}\label{ineq: estimates for L2 L3}
	\mathscr{L}_2(t,u)+\mathscr{L}_3(t,u)\leqslant  C_0\Big(\frac{1}{a_0}\int_{0}^{u} \mathscr{F}(t,u')du'+a_0\int_\delta^t  \frac{\mathscr{E}(t',u)}{t'}dt'\Big),
\end{equation}
where $C_0$ is a universal constant and the small positive constant $a_0$ will be determined later on.
For $\psi$ of zero order we shall also make use of another error integral (see \eqref{eq: 0 order bound on wb}):
\begin{equation}\label{def:L3r}
	\mathring{\mathscr{L}}_3(\psi,\psi')(t,u)=\int_{\mathcal{D}(t,u)}  |\Xr \psi||\Lb\psi'|{\color{black}.}
\end{equation}

\begin{remark}\label{remark:bilinear-error-integrals}
	We notice that $\mathscr{L}_i$ are of the forms $\int_{\mathcal{D}(t,u)} |Z\psi||Z'\psi'|$  but we exclude the case $\int_{\mathcal{D}(t,u)}|\Lb \psi||\Lb \psi'|$. The reason is that we can bound at least one of the factor $|Z\psi|$ in $\mathscr{L}_i$ by the flux, which provides a crucial smallness factor by integrating in $u$. This is the null structure mentioned in Section \ref{section: difficulty}. 
\end{remark}

\subsection{A refined Gronwall type inequality}
To handle the bilinear error integrals in the energy estimates, we will need a refined Gronwall type inequality:
\begin{lemma}\label{refined Gronwall}
	{\color{black}Let  $E(t,u)$ and $F(t,u)$ be two smooth non-negative functions defined on $D(t^*,u^*)$ such that 
		\[ E(t,u') \leqslant E(t,u) \ \ \text{for}  \ \  0 \leqslant u' \leqslant u \leqslant u^* \ \ \text{and} \ \ F(t',u) \leq F(t,u) \ \text{for} \ \ \delta \leq t' \leqslant t \leqslant t^*. \]
	}
	We assume that there exist positive constants $A$, $B$ and $C$ so that for all $(t,u)\in [\delta,t^*]\times [0,u^*]$, we have the following inequality:
	\begin{align*}
		{E}(t,u)+ {F}(t,u)\leqslant At^2+ B \int_0^u F (t,u') du'   + C \int_{\delta}^{t}\frac{E(t',u)}{t'}dt'.
	\end{align*}
	Then, if $e^{B u^*}C\leqslant 1$, we have the following inequality for all $(t,u)\in [\delta,t^*]\times [0,u^*]$:
	\begin{equation*}
		{E}(t,u)+F(t,u)\leqslant  3Ae^{Bu} t^{2}.
	\end{equation*}
\end{lemma}
\begin{proof}
	We define $H(t,u)=A t^2- {E}(t,u)   + C \int_{\delta}^{t}\frac{E(t',u)}{t'}dt'$. Therefore,
	\[
	F(t,u)\leqslant  {H}(t,u)+ B \int_0^u {\color{black}F}(t,u') du'.
	\]
	We use the standard Gronwall's inequality for the variable $u$ and we obtain that
	\begin{align*}
		{F}(t,u)\leqslant & {H}(t,u) + B \int_{0}^u e^{B (u-u')}{H}(t,u')du'.
	\end{align*}
	According to the definition of $H(t,u)$, this is equivalent to
	\begin{align*}
		{F}(t,u)+{E}(t,u)\leqslant At^2+ C \int_{\delta}^{t}\frac{{E}(t',u)}{t'}dt'+B \int_{0}^u e^{B (u-u')}{H}(t,u')du'.
	\end{align*}
	For $u'\leqslant u$, the definition of $H(t,u)$ also implies that 
	\begin{align*}
		{H}(t,u')&\leqslant At^2 + C \int_{\delta}^{t}\frac{{E}(t',u')}{t'}dt'\leqslant A t^2+ C \int_{\delta}^{t}\frac{{E}(t',u)}{t'}dt'{\color{black}.}
	\end{align*}
	{\color{black}Combining the above} two inequalities, we have
	\begin{align*}
		{F}(t,u)+E(t,u)&\leqslant At^2  + C \int_{\delta}^{t}\frac{E(t',u)}{t'}dt' +B \int_{0}^u e^{B (u-u')}\left[ At^2+C \int_{\delta}^{t}\frac{{E}(t',u)}{t'}dt'\right]du'\\
		&=A t^2+ C \int_{\delta}^{t}\frac{{E}(t',u)}{t'}dt' +(e^{Bu}-1)\left(At^2+ C \int_{\delta}^{t}\frac{{E}(t',u)}{t'}dt'\right).
	\end{align*}
	Therefore, 
	\begin{equation}\label{ineq: in gronwall proof 1}
		{F}(t,u)+E(t,u)\leqslant Ae^{B u}t^2+e^{B u}C \int_{\delta}^{t}\frac{E(t',u)}{t'}dt'.
	\end{equation}
	In particular,
	\begin{equation}\label{ineq: in gronwall proof 2}
		E(t,u)\leqslant Ae^{B u}t^2+e^{B u}C \int_{\delta}^{t}\frac{E(t',u)}{t'}dt'.
	\end{equation}
	For a fixed $u$, if we define $D=e^{B u}C$ and $Y(t)=\int_{\delta}^{t}\frac{E(t',u)}{t'}dt'$, then \eqref{ineq: in gronwall proof 2} is equivalent to
	\begin{align*}
		tY(t)' \leqslant  Ae^{B u} t^2+D Y(t),
	\end{align*}
	which is also equivalent to 
	\[
	\left(\frac{Y(t)}{t^D}\right)' \leqslant  Ae^{B u} t^{1-D} .\]
	We can integrate the above equation on $[\delta,t]$ and we use $D\leqslant 1$ to derive
	\begin{align*}
		\frac{Y(t)}{t^D}&\leqslant  \frac{Ae^{Bu}}{2-D}\left(t^{2-D}- \delta^{2-D}\right)\leqslant  \frac{2Ae^{Bu}}{2-D} t^{2-D}.\end{align*}
	Hence,
	\[Y(t)\leqslant  \frac{2Ae^{Bu}}{2-D} t^{2}.\]
	We put the bound on $Y(t)$ back into \eqref{ineq: in gronwall proof 1}. This implies
	\begin{align*}
		E(t,u)+F(t,u)&\leqslant  \frac{2+e^{B u}C}{2-e^{B u}C}Ae^{Bu} t^{2}.
	\end{align*}
	The result of the lemma follows immediately.
\end{proof}

\section{The lowest order energy estimates}\label{section:lowest-order-energy-estimates}

In this section, we apply the results of Section \ref{section: fundamental energy estimates} to the wave equations \eqref{Main Wave equation: order 0}, i.e., $\Box_g \Psi_0 =\varrho_0$ with $\Psi_0 \in \{\wb,w,\psi_2\}$. In view of \eqref{eq: decompose rho0 using Lr Lbr Xr}, we recall that  $\varrho_0$ is a linear combination of terms from the set $\big\{c^{-1}g(D f_1,Df_2)\big| f_1,f_2\in \{\wb,w,\psi_2\}\big\}$.

In the rest of the section, we first derive energy estimates for $w$ and $\psi_2$. We then use the Euler equations to obtain the energy bound on $\wb$.

\subsection{Energy estimates for $w$ and $\psi_2$}

We take $\Psi_0=w$ or $\psi_2$ in \eqref{Main Wave equation: order 0}. In view of the results of Section \ref{section: fundamental energy estimates}, in particular \eqref{ineq: energy ineq for total},  it suffices to bound the following error terms:
\[\begin{cases}
\mathscr{N}(\Psi_0)(t,u)&=\displaystyle-\int_{\mathcal{D}(t,u)} c^{-1}\mu g(D f_1,Df_2)\cdot \Lh\Psi_0,\\
\underline{\mathscr{N}}(\Psi_0)(t,u)&=\displaystyle-\int_{\mathcal{D}(t,u)} c^{-1}\mu  g(D f_1,Df_2) \cdot\Lb\Psi_0,
\end{cases} \ \ \ \text{with} \  f_1,f_2\in \{\wb,w,\psi_2\}.\]
According to \eqref{eq: decompose rho0 using Lr Lbr Xr}, we rewrite $\mu g(D f_1,Df_2)$ as
\begin{equation}\label{aux:lowest order 1}
\mu g(D f_1,Df_2)=-\frac{1}{2}(Lf_1 \Lb f_2+\frac{1}{2}\Lb f_1 L f_2)+\mu \Xh(f_1)\Xh(f_2).
\end{equation}
The possible error terms can be classified {\color{black}into two groups} according to either $\{f_1,f_2\}\cap \{w,\psi_2\} \neq \emptyset$ or $f_1=f_2=\wb$. We treat these two cases separately.

\medskip

\noindent{\bf Case 1} \ $\{f_1,f_2\}\cap \{w,\psi_2\} \neq \emptyset$.  Without loss of generality, we assume that $f_2\in \{w,\psi_2\}$. 

\medskip

By \eqref{preliminary geometric bounds}, \eqref{bound on Twb} and the bootstrap assumption $\mathbf{(B_\infty)}$, we have $|c|+|\Lb f_1|\lesssim 1$.  Therefore, \eqref{aux:lowest order 1} implies
\begin{align*}
\left|\mathscr{N}(\Psi_0)(t,u)\right|&\lesssim \int_{\mathcal{D}(t,u)}  \left(|Lf_1||\Lb f_2|+|\Lb f_1| |L f_2|+\mu|\Xh(f_1)||\Xh(f_2)|\right)|\Lh\Psi_0|\\
&\lesssim \int_{\mathcal{D}(t,u)}  \big(\mathring{M}\varepsilon |\Lb f_2|+  |L f_2|+\mu \mathring{M} \varepsilon |\Xh(f_2)|\big)\mu|L\Psi_0|\\
&\lesssim \int_{\mathcal{D}(t,u)}  \mathring{M}\varepsilon \mu |\Lb f_2||L\Psi_0|+ \mu |L f_2||L\Psi_0|+ \mathring{M} \varepsilon \mu^2|\Xh(f_2)||L\Psi_0|.
\end{align*}
We notice that, in the last line, both $f_2$ and $\Psi_0$ are from the set $\{w,\psi_2\}$. In view of the definition of $\mathscr{E}(\psi)(t,u)$ and $\mathscr{F}(\psi)(t,u)$, we apply Cauchy-Schwarz inequality to each of the above terms in the integrand and we obtain
\begin{align*}
\left|\mathscr{N}(\Psi_0)(t,u)\right|
&\lesssim \int_{0}^{u} \mathscr{F}_0(t,u')du'+\mathring{M}\varepsilon  \int_{\delta}^{t} \mathscr{E}_0(t',u)dt'\\
&\lesssim \mathring{M}\varepsilon^3t^2+\int_{0}^{u} \mathscr{F}_0(t,u')du'.
\end{align*}
\begin{remark}[Abuse of notations]
We have used the notations $\mathscr{E}_0(t',u)$ for $\mathscr{E}(\psi)(t',u)$ where $\psi\in \{w,\wb, \psi_2\}$. {\color{black}In the rest of the paper,} we will use notations $\mathscr{E}_n(t,u)$ to denote $\mathscr{E}(\psi)(t,u)$ where $\psi\in \{w,\wb, \psi_2\}$ if there is no confusion. Similarly, we use notations $\mathscr{F}_n(t,u)$,  $\Eb_n(t,u)$ and  $\Fb_n(t,u)$.
\end{remark}
In the previous inequality, we used $\mathbf{(B_2)}$ to bound $\mathscr{E}_0(t',u)$. It is also important to observe that the flux term $\mathscr{F}_0(t,u')$ in the above estimates is associated with $w$ and $\psi_2$. It does not include the flux of $\wb$.

Similarly, we can bound $\underline{\mathscr{N}}(\Psi_0)(t,u)$ as follows:\begin{align*}
\left|\underline{\mathscr{N}}(\Psi_0)(t,u)\right|
&\lesssim \int_{\mathcal{D}(t,u)}  \left(\mathring{M}\varepsilon |\Lb f_2|+ |L f_2|+\mu \mathring{M} \varepsilon |\Xh(f_2)|\right)|\Lb\Psi_0|\\
&\lesssim \int_{\mathcal{D}(t,u)}  \mathring{M}\varepsilon |\Lb f_2||\Lb\Psi_0|+  |L f_2||\Lb\Psi_0|+\mathring{M} \varepsilon \mu|\Xh(f_2)||\Lb\Psi_0|\\
&\lesssim \mathscr{L}_3(f_2,\Psi_0)(t,u)+\mathring{M}\varepsilon^3t^2.
\end{align*}
In the last step, we have used the notations of bilinear error integrals defined in Section \ref{section: bilinear error integrals}.

\medskip

\noindent {\bf Case 2}~$f_1=f_2=\wb$. In this case, we will bound $c$, $c^{-1}$ {\color{black}and $\Lb(\wb)$} in $L^\infty$ by a universal constant.

\medskip

For $\mathscr{N}(\Psi_0)(t,u)$,  we bound one of $\Xh(\wb)$'s in $L^\infty$ norm by $\mathring{M}\varepsilon$. This leads to
\begin{align*}
\left|\mathscr{N}(\Psi_0)(t,u)\right|&\lesssim \int_{\mathcal{D}(t,u)}  \left(|L\wb||\Lb \wb|+\mu|\Xh(\wb)|^2\right)|\Lh\Psi_0|\\
&\lesssim \int_{\mathcal{D}(t,u)}  \mu |L \wb| |L\Psi_0|+\mu^2 |\Xh(\wb)|^2|L\Psi_0|\\
&\lesssim\mathring{M}\varepsilon^3t^2+ \int_{\mathcal{D}(t,u)}   \mu |L \wb| |L\Psi_0|{\color{black}.}
\end{align*}

For $\underline{\mathscr{N}}(\Psi_0)(t,u)$, we have
\begin{align*}
\left|\underline{\mathscr{N}}(\Psi_0)(t,u)\right|&\lesssim \int_{\mathcal{D}(t,u)}  \left(|L\wb||\Lb \wb|+\mu|\Xh(\wb)|^2\right)|\Lb\Psi_0|\\
&\lesssim \mathring{M}\varepsilon^3t^2+\int_{\mathcal{D}(t,u)}  |L\wb| |\Lb\Psi_0|{\color{black}.}
\end{align*}
The appearance of $\wb$ in the integral may generate a flux term $\mathscr{F}_0(t,u')$ associated to $\wb$. To avoid it, we will use the Euler equations to replace $L(\wb)$ by derivatives of $w$ and $\psi_2$.  In fact, by \eqref{Euler equations:form 2}, we have
\begin{equation}\label{aux:lowest order 2}
\begin{split}
L (\wb) &= c \Xh(\psi_2)\Xh^2  -c \widehat{T}(\wb)(\widehat{T}^1+1) + c \widehat{T}(\psi_2)\widehat{T}^2  - c\Xh(\wb)\Xh^1 \\
&=c \Xh(\psi_2)\Xh^2+\mathring{M}t\varepsilon^2{\color{black},}
\end{split}
\end{equation}
where we bound the last two terms by $\mathbf{(B_2)}$ and we use improved estimate  \eqref{bound on T1+1} on $\widehat{T}^1+1$  to control the second term. We can bound
\[\int_{\mathcal{D}(t,u)}  \mathring{M}t\varepsilon^2 |\Lb\Psi_0|\lesssim \mathring{M}\varepsilon^3t^2,\]
because $\Psi_0\neq \wb$ so that we can use $\mathbf{(B_\infty)}$ to bound $|\Lb\Psi_0|$ by $ \mathring{M}t\varepsilon$. Therefore, 
\begin{align*}
\left|\underline{\mathscr{N}}(\Psi_0)(t,u)\right|&\lesssim \mathring{M}\varepsilon^3t^2+\int_{\mathcal{D}(t,u)}  |\Xh \psi_2 | |\Lb\Psi_0|.
\end{align*}
Therefore, we have
\begin{align*}
\left|\mathscr{N}(\Psi_0)(t,u)\right|+\left|\underline{\mathscr{N}}(\Psi_0)(t,u)\right|&\lesssim \mathring{M}\varepsilon^3t^2+ \int_{\mathcal{D}(t,u)}   \mu |\Xh \psi_2| |L\Psi_0|+|\Xh\psi_2| |\Lb\Psi_0|\\
&\lesssim \mathring{M}\varepsilon^3t^2+ \int_{0}^{u} \mathscr{F}_0(t,u')du'+\mathscr{L}_3(\psi_2,
\Phi_0)(t,u).\end{align*}

\medskip

{\color{black}Combining the above estimates} in {\bf Case 1} and {\bf Case 2}, in view of \eqref{ineq: estimates for L2 L3}, there exist universal constant $C_0$, $C_1$ and $C_2$, such that if $\mathring{M}\varepsilon$ is sufficiently small, we have
\begin{align*}
\mathscr{E}(\Psi_0)(t,u)+\mathscr{F}(\Psi_0)(t,u)&\leqslant \mathscr{E}(\Psi_0)(\delta,u)+ \mathscr{F}(\Psi_0)(t,0)+C_1\mathring{M}\varepsilon^3t^2\\
&\ \ +C_0\Big(\frac{1}{a_0}\int_{0}^{u} \mathscr{F}(\psi)(t,u')du'+a_0\int_\delta^t  \frac{\mathscr{E}(\psi)(t',u)}{t'}dt'\Big)\\
&\leqslant C_2\varepsilon^2t^2+C_0\Big(\frac{1}{a_0}\int_{0}^{u} \mathscr{F}(\psi)(t,u')du'+a_0\int_\delta^t  \frac{\mathscr{E}(\psi)(t',u)}{t'}dt'\Big).
\end{align*}
It is important to notice that  the above energy norms are associated with $w$ and $\psi_2$, i.e., $\psi\neq \wb$. Since the energy norms on the lefthand side are also associated with with $w$ and $\psi_2$, we apply the refine Gronwall's inequality, i.e., Lemma \ref{refined Gronwall}. We may take $a_0 = \frac{1}{2C_0}$ and $u_0^* = \frac{\log 2}{2 C_0^2}$ so that $e^{B u^*}C\leqslant 1$. Therefore, the refined Gronwall's inequality yields that,  
for all $(t,u)\in [\delta,t^*]\times [0,u_0^*]$,
\begin{equation}\label{eq: 0 order bound on w and psi2}
\mathscr{E}(\psi)(t,u)+\mathscr{F}(\psi)(t,u)\lesssim  \varepsilon^2t^2,
\end{equation}
where $\psi= w$ or $\psi_2$. This closes the second estimate of the bootstrap assumption $\mathbf{(B_2)}$, see \eqref{ansatz B2}. We notice that $u_0^*$ is a universal constant. As we shall see, by iteration we can improve $u_0^*$ to $u^*$ as long as we have a lower bound on $c$; see Section \ref{subsection:conclusion-higher-order-energy-estimates}.

\subsection{Energy bounds for $\wb$}\label{subsection:zero-order-bounds-wb}

This section is devoted to bound $L\wb$ and $\XR\wb$. We point out that these estimates are not included in the  bootstrap assumption $\mathbf{(B_2)}$.

According to \eqref{aux:lowest order 2}, we have the following pointwise bound:
\begin{align*}
|\kappa L (\wb) -\mu \Xh(\psi_2)\Xh^2| \lesssim \mathring{M}t^2 \varepsilon^2.
\end{align*}
Therefore, we can bound $L\wb$ in terms of $\Xh(\psi_2)$. Indeed, by the bound \eqref{eq: 0 order bound on w and psi2} on $\psi_2$, we have
\[\int_{\Sigma_t^u}c^{-2}\kappa^2 |L (\wb)|^2  \lesssim \int_{\Sigma_t^u}\mu^2 |\Xh(\psi_2)|^2  +\mathring{M}t^2\varepsilon^3\lesssim t^2 \varepsilon^2,\]
provided $\varepsilon$ is sufficiently small. The contribution of $L\wb$ in the flux term can be bounded in the same manner. Therefore, for all $(t,u)\in [\delta,t^*]\times [0,u^*]$,  we obtain that
 \begin{equation}\label{eq: 0 order bound on L wb}
\int_{\Sigma_t^u}c^{-2}\kappa^2 |L (\wb)|^2  +\int_{C_u^t}  c^{-1}\kappa |L(\wb)|^2{\color{black}\lesssim t^2 \varepsilon^2.}
\end{equation}
In view of \eqref{change of frame 1}, \eqref{bound on T1+1}, \eqref{bound on T2}, we also have
\begin{align*}
	\int_{\Sigma_t^u}c^{-2}\kappa^2 |\Lr(\wb)|^2 \lesssim t^2 \varepsilon^2{\color{black}.}
\end{align*}

For $\XR(\wb)$ we use the equation \eqref{Euler equations:form 3} to obtain
\begin{align*}
	\int_{\Sigma_t^u}\kappa^2|\XR(\wb)|^2 \lesssim \int_{\Sigma_t^u}c^{-2}\kappa^2|\Lr(\psi_2)|^2 + (\kappar^{-1}\kappa)^2|\Tr(\psi_2)|^2 + \kappa^2|\Xr(w)|^2 \lesssim t^2 \varepsilon^2.
\end{align*}
By \eqref{eq: 0 order bound on w and psi2} and \eqref{comparison bound 1: L infty}, we have
\begin{align*}
	\int_{\Sigma_t^u}\kappa^2|\XR(\wb)|^2 \lesssim t^2 \varepsilon^2.
\end{align*}
The contribution of $\XR\wb$ in the flux term can be bounded in the same manner.
Recalling the definitions  in \eqref{def:Er_0-Fr_0}, we  have the following energy bounds for $\wb$:
\begin{equation}\label{eq: 0 order bound on wb}
	\mathring{\mathscr{E}}_0(\wb)(t,u) + \mathring{\mathscr{F}}_0(\wb)(t,u)  \lesssim t^2 \varepsilon^2{\color{black}.}
\end{equation}
We summarize the zero order energy estimates as 
\begin{equation}\label{eq: 0 order bound}
	\underbrace{\sum_{\psi \in \{w,\psi_2\}}\mathscr{E}_{0}(\psi)(t,u) + \mathring{\mathscr{E}}_0(\wb)(t,u)}_{:= \mathscr{E}_0(t,u)} + \underbrace{\sum_{\psi \in \{w,\psi_2\}}\mathscr{F}_{0}(\psi)(t,u) + \mathring{\mathscr{F}}_0(\wb)(t,u)}_{:= \mathscr{F}_0(t,u)} \lesssim t^2 \varepsilon^2{\color{black}.}
\end{equation}
\begin{remark}
		It seems that the above approach can not provide energy bounds on  $\Xh(\wb)$.	In fact, since we use $\mathring{\mathscr{Z}}$ as commutator and the second null frame to decompose $g(D\psi,D\psi')$, it is $\XR(\wb)$  that will appear in the error terms, instead of $\Xh(\wb)$.
	\end{remark}

\section{Lower order estimates and extra vanishing}\label{section:lower-order-estimates}

\subsection{The $L^2$ and pointwise bounds on objects of $\Lambda$}
We recall that $\Lambda = \{\yr, \zr, \chir,\etar\}$. We use $\lambda$ to denote a generic object from $\Lambda$.

\subsubsection{Bounds on $\lambda = \chir, \etar$}
Since $\chir=-\Xr(\psi_2)$ and $\etar=-\Tr(\psi_2)$, the estimates on $\chir$ and $\etar$ are easy. In fact, according to \eqref{eq: B2 equi bis} and \eqref{ineq: L infty bound},  for all multi-indices $\alpha$ with $|\alpha|\leqslant \Ntop$ and $\beta$ with $|\beta|\leqslant \Ninf-1$ , for all $\lambda \in \{\chir,\etar\}$, we have
\begin{equation*}
\|\mathring{Z}^\alpha(\lambda)\|_{L^2(\Sigma_t)}\lesssim  \begin{cases}\mathring{M}\varepsilon, \ \ & \text{if}~\Zr^\beta=\Xr^\alpha ~\text{and}~\lambda =\chir;\\
\mathring{M}\varepsilon t, \ \ & \text{otherwise}.
\end{cases} \ \ |\alpha|\leqslant \Ntop.
\end{equation*}
and
\begin{equation*}
\|\mathring{Z}^\beta(\lambda)\|_{L^\infty(\Sigma_t)}\lesssim  \begin{cases}\mathring{M}\varepsilon, \ \ & \text{if}~\Zr^\beta=\Xr^\beta ~\text{and}~\lambda =\chir;\\
\mathring{M}\varepsilon t, \ \ & \text{otherwise}.
\end{cases} \ \ |\beta|\leqslant \Ninf-1.
\end{equation*}

\subsubsection{Bounds on $\lambda = \yr,\zr$}\label{section: yr zr}

When $\lambda=\yr$ or $\zr$, the estimates are much more involved. We will frequently compute commutators of the shape $[\Lr, \Zr^\alpha]$. In view of $[\Lr, \Xr]=\yr\cdot \Tr-\chir\cdot\Xr$ and $[\Lr, \Tr]=\zr\cdot \Tr-\etar\Xr$, for any multi-index $\alpha$, we have the following schematic commutation formula:
\begin{equation}\label{eq: commutation formular for Lr Zr}
[\Lr, \Zr^\alpha]=\sum_{\substack{\alpha_1+\alpha_2=\alpha \\ |\alpha_1|\leqslant |\alpha|-1}}\Zr^{\alpha_1}(\lambda) \Zr^{\alpha_2}, \  \ \lambda \in \{\yr, \zr, \chir,\etar\}.
\end{equation}
\begin{remark}[A key structure in the commutator]\label{remark: commutator structure}
We observe that if the $\lambda$ appearing in a single term $\Zr^{\alpha_1}(\lambda) \Zr^{\alpha_2}$ in \eqref{eq: commutation formular for Lr Zr} happens to be $\yr$ or $\zr$, then at least one of the $\Zr$'s in $\Zr^{\alpha_2}$ is $\Tr$.

Similarly, if the $\lambda$ appearing in a single term $\Zr^{\alpha_1}(\lambda) \Zr^{\alpha_2}$ in \eqref{eq: commutation formular for Lr Zr} happens to be $\chir$ or $\etar$, then at least one of the $\Zr$'s in $\Zr^{\alpha_2}$ is $\Xr$.
\end{remark}
Since $\Tr$ commutes with all $\Zr\in \mathring{\mathscr{Z}}$, we also have
\begin{equation}\label{eq: commutation formular for Lbr Zr}
[\Lbr, \Zr^\alpha]=\Zr^\alpha(c^{-1}) \kappar \Lr+c^{-1}\kappar\sum_{\substack{\alpha_1+\alpha_2=\alpha\\ |\alpha_1|\leqslant |\alpha|-1}}\Zr^{\alpha_1}(\lambda) \Zr^{\alpha_2}, \  \ \lambda \in \{\yr, \zr, \chir,\etar\}.
\end{equation}
\begin{remark}\label{remark: how to get bound on yr zr nondegeneration condition}
To derive the estimates on $\yr$ and $\zr$, we will combine the commutator formulas $[\Lr, \Xr]=\yr\cdot \Tr-\chir\cdot\Xr$ and $[\Lr, \Tr]=\zr\cdot \Tr-\etar\Xr$ with the following key fact for the rarefaction waves: $\Tr(\wb)\approx -1$. Indeed, from \eqref{bound on Twb}, we have $T(\wb)\approx -1$. Since $\Tr= -\frac{\kappar}{ \kappa  (\Th^1)^2+\ (\Th^2)^2}\big(\Th^1 T+  \Th^2 X\big)$, we can use \eqref{bound on kappa}, \eqref{bound on T1+1} and \eqref{bound on T2} to get $\Tr(\wb)\approx -1$.
\end{remark}
To obtain the estimates on $\yr$, we apply $[\Lr, \Xr]=\yr\cdot \Tr-\chir\cdot\Xr$ to $\wb$ and use \eqref{Euler equations:form 3} to replace $\Lr (\wb)$. This leads to
\begin{align*} 
\yr\cdot \Tr \wb&=\Lr\Xr(\wb) -\Xr\Lr (\wb)+\chir\cdot\Xr\wb=\Lr\Xr(\wb) -\frac{1}{2}\Xr \big(c \Xr(\psi_2)\big)+\chir\cdot\Xr\wb.
\end{align*}
Since $\chir=-\Xr(\psi_2)$, we obtain the following schematic formula:
\begin{equation}\label{eq: yr 0 order equation}
\yr\cdot \Tr \wb=\Lr\Xr(\wb) -\frac{1}{2}c\Xr^2(\psi_2)+\Xr(\psi)\Xr(\psi),
\end{equation}
where $\psi\in \{w,\wb,\psi_2\}$. We remark that in the expression $\Xr(\psi)\Xr(\psi)$ we ignore the numerical constants. We apply $\Zr^\alpha$ to \eqref{eq: yr 0 order equation} and we {\color{black}keep track of all} the top order terms as follows:
\begin{align*} 
&\Zr^{\alpha}(\yr)\cdot \Tr \wb+\sum_{\substack{\alpha_1+\alpha_2=\alpha \\ |\alpha_1|\leqslant |\alpha|-1}} \Zr^{\alpha_1}(\yr) \Zr^{\alpha_2}(\Tr \wb)\\
=&\Zr^{\alpha}\big(\Lr\Xr(\wb)\big) -\frac{1}{2}c\Zr^\alpha\big(\Xr^2(\psi_2)\big)+\sum_{\substack{\alpha_1+\alpha_2=\alpha \\ |\alpha_2|\leqslant |\alpha|-1}} \Zr^{\alpha_1}(c) \Zr^{\alpha_2}(\Xr^2(\psi_2))+\sum_{\alpha_1+\alpha_2=\alpha } \Zr^{\alpha_1}(\Xr\psi)\Zr^{\alpha_2}(\Xr\psi).
\end{align*}
We use \eqref{eq: commutation formular for Lr Zr} to commute $\Zr^\alpha$ and $\Lr$ for the first term on the righthand side to derive
\begin{align*} 
\Zr^{\alpha}(\yr)\cdot \Tr \wb+\sum_{\substack{\alpha_1+\alpha_2=\alpha \\ |\alpha_1|\leqslant |\alpha|-1}} \Zr^{\alpha_1}(\yr) \Zr^{\alpha_2}(\Tr \wb)
=&\Lr\Zr^{\alpha}\Xr (\wb) +\sum_{\substack{\alpha_1+\alpha_2=\alpha, \\ |\alpha_1|\leqslant |\alpha|-1}}\Zr^{\alpha_1}(\lambda) \Zr^{\alpha_2}\Xr(\wb)-\frac{1}{2}c\Zr^\alpha\big(\Xr^2(\psi_2)\big)\\
&\ +\sum_{\substack{\alpha_1+\alpha_2=\alpha \\ |\alpha_2|\leqslant |\alpha|-1}} \!\!\!\Zr^{\alpha_1}(c) \Zr^{\alpha_2}(\Xr^2(\psi_2))+\sum_{\alpha_1+\alpha_2=\alpha } \!\!\!\Zr^{\alpha_1}(\Xr\psi)\Zr^{\alpha_2}(\Xr\psi).
\end{align*}
Therefore, we obtain the following schematic expression:
\begin{align*} 
\Zr^{\alpha}(\yr)\cdot \Tr \wb=&\Lr\Zr^{\alpha}\Xr (\wb) +\sum_{\substack{\alpha_1+\alpha_2=\alpha, \\ |\alpha_1|\leqslant |\alpha|-1}}\Zr^{\alpha_1}(\lambda) \Zr^{\alpha_2+1}(\wb)-\frac{1}{2}c\Zr^\alpha\big(\Xr^2(\psi_2)\big)\\
&+\sum_{\substack{\alpha_1+\alpha_2=\alpha \\ |\alpha_2|\leqslant |\alpha|-1}} \Zr^{\alpha_1}(c) \Zr^{\alpha_2}(\Xr^2(\psi_2))+\sum_{\alpha_1+\alpha_2=\alpha } \Zr^{\alpha_1}(\Xr\psi)\Zr^{\alpha_2}(\Xr\psi)\\
=&\Lr\Zr^{\alpha}\Xr (\wb) +\sum_{\substack{\alpha_1+\alpha_2=\alpha, \\ |\alpha_1|\leqslant |\alpha|-1}}\Zr^{\alpha_1}(\lambda) \Zr^{\alpha_2+1}(\wb)-\frac{1}{2}c\Zr^\alpha\big(\Xr^2(\psi_2)\big)+\sum_{\substack{\alpha_1+\alpha_2=\alpha \\ |\alpha_1|=1}} \Zr^{\alpha_1}(c) \Zr^{\alpha_2}(\Xr^2(\psi_2))\\
&+\sum_{\substack{\alpha_1+\alpha_2=\alpha \\ |\alpha_1|\geqslant 2}} \Zr^{\alpha_1}(c) \Zr^{\alpha_2}(\Xr^2(\psi_2))+\sum_{\alpha_1+\alpha_2=\alpha } \Zr^{\alpha_1}(\Xr\psi)\Zr^{\alpha_2}(\Xr\psi).
\end{align*}
Thus, 
\begin{equation}\label{eq:precise formula for Z alpha yr}
\begin{split}
\Zr^{\alpha}(\yr)\cdot \Tr \wb=&\Lr\Zr^{\alpha}\Xr (\wb) +\sum_{\substack{\alpha_1+\alpha_2=\alpha \\ |\alpha_1|\leqslant 1}} 
\Zr^{\alpha_1}(c) \Zr^{\alpha_2}(\Xr^2(\psi_2))+\sum_{\substack{\alpha_1+\alpha_2=\alpha, \\ |\alpha_1|\leqslant |\alpha|-1}}\Zr^{\alpha_1}(\lambda) \Zr^{\alpha_2+1}(\wb)\\
&+\sum_{\substack{\alpha_1+\alpha_2=\alpha \\ |\alpha_1|\geqslant 2}} \Zr^{\alpha_1}(c) \Zr^{\alpha_2}(\Xr^2(\psi_2))+\sum_{\alpha_1+\alpha_2=\alpha } \Zr^{\alpha_1}(\Xr\psi)\Zr^{\alpha_2}(\Xr\psi).
\end{split}
\end{equation}
We now compute the $L^2(\Sigma_t)$ norm on each term appeared in \eqref{eq:precise formula for Z alpha yr}. In view of Remark \ref{remark:techical nonlinear}, \eqref{eq: B2 equi bis} and \eqref{ineq: L infty bound}, the last two sums are bounded by $\mathring{M}\varepsilon^2$ in $L^2(\Sigma_t)$. Since $|\Tr(\wb)|\approx 1$, $|c|\approx 1$ and $|\Zr(c)|\lesssim 1$ , we have
\begin{equation}\label{ineq: inductive bound on yr}
\begin{split}
\|\Zr^{\alpha}(\yr)\|_{L^2(\Sigma_t)}\lesssim&\|\Lr\Zr^{\alpha}\Xr (\wb)\|_{L^2(\Sigma_t)} +\|\Zr^\alpha\Xr^2(\psi_2)\|_{L^2(\Sigma_t)}+\sum_{|\alpha_2|=|\alpha|-1} \|\Zr^{\alpha_2}(\Xr^2(\psi_2))\|_{L^2(\Sigma_t)}\\
&+\sum_{\substack{\alpha_1+\alpha_2=\alpha, \\ |\alpha_1|\leqslant |\alpha|-1}}\|\Zr^{\alpha_1}(\lambda) \Zr^{\alpha_2+1}(\wb)\|_{L^2(\Sigma_t)}+\mathring{M}\varepsilon^2.
\end{split}
\end{equation}

To obtain the estimates on $\zr$, we apply $[\Lr, \Tr]=\zr\cdot \Tr-\etar\Xr$ to $\wb$ and use \eqref{Euler equations:form 3} to replace $\Lr (\wb)$:
\begin{equation}\label{eq: zr 0 order equation}
\begin{split}
\zr\cdot \Tr \wb&=\Lr\Tr(\wb) -\Tr\Lr (\wb)+\etar\cdot\Xr\wb=\Lr\Tr(\wb) -\frac{1}{2}\Tr \big(c \Xr(\psi_2)\big)+\etar\cdot\Xr\wb\\
&=\Lr\Tr(\wb) -\frac{1}{2}c \Tr \Xr(\psi_2)-\frac{1}{2}\Tr (c)\Xr(\psi_2)-\Tr(\psi_2)\Xr\wb.
\end{split}
\end{equation}
We apply $\Zr^\alpha$ to the above equation and we {\color{black}keep track of all} the top order terms as follows:
\begin{align*} 
&\Zr^{\alpha}(\zr)\cdot \Tr \wb+\sum_{\substack{\alpha_1+\alpha_2=\alpha \\ |\alpha_1|\leqslant |\alpha|-1}} \Zr^{\alpha_1}(\zr) \Zr^{\alpha_2}(\Tr \wb)\\
=&\Zr^{\alpha}\big(\Lr\Tr(\wb)\big) -\frac{1}{2}c\Zr^\alpha\big(\Tr\Xr(\psi_2)\big)+\sum_{|\alpha_2|=|\alpha|-1 } \Zr(c) \Zr^{\alpha_2}(\Tr\Xr(\psi_2))-\frac{1}{2}\Tr (c)\Zr^{\alpha}\Xr(\psi_2)\\
&+\sum_{\substack{\alpha_1+\alpha_2=\alpha \\ |\alpha_1|\geqslant 2}} \Zr^{\alpha_1}(c) \Zr^{\alpha_2}(\Tr\Xr(\psi_2))+\sum_{\substack{\alpha_1+\alpha_2=\alpha \\ |\alpha_1|\geqslant 1}} \Zr^{\alpha_1}\Tr (c)\Zr^{\alpha_2}\Xr(\psi_2)+\sum_{\alpha_1+\alpha_2=\alpha}\Zr^{\alpha_1}\Tr(\psi_2)\Zr^{\alpha_2}\Xr\wb.
\end{align*}
Similar to the calculations for $\yr$, when we compute the $L^2(\Sigma_t)$ norm for $\zr$, by \eqref{eq: B2 equi bis}, \eqref{ineq: L infty bound} and Remark \ref{remark:techical nonlinear}, we can bound the last three sums by $\mathring{M}\varepsilon^2$. Therefore, by abusing the notations, we rewrite the above formula as
\begin{align*} 
&\Zr^{\alpha}(\zr)\cdot \Tr \wb+\sum_{\substack{\alpha_1+\alpha_2=\alpha \\ |\alpha_1|\leqslant |\alpha|-1}} \Zr^{\alpha_1}(\zr) \Zr^{\alpha_2}(\Tr \wb)\\
=&\Zr^{\alpha}\big(\Lr\Tr(\wb)\big) -\frac{1}{2}c\Zr^\alpha\left(\Tr\Xr(\psi_2)\right)+\sum_{|\alpha_2|=|\alpha|-1 } \Zr(c) \Zr^{\alpha_2}(\Tr\Xr(\psi_2))-\frac{1}{2}\Tr (c)\Zr^{\alpha}\Xr(\psi_2)+\mathring{M}\varepsilon^2.
\end{align*}
We then use \eqref{eq: commutation formular for Lr Zr}  to commute $\Zr^\alpha$ and $\Lr$ for the first term on the righthand side to derive
\begin{align*} 
&\Zr^{\alpha}(\zr)\cdot \Tr \wb+\sum_{\substack{\alpha_1+\alpha_2=\alpha \\ |\alpha_1|\leqslant |\alpha|-1}} \Zr^{\alpha_1}(\zr) \Zr^{\alpha_2}(\Tr \wb)\\
=&\Lr\Zr^{\alpha}\Tr(\wb)+\sum_{\substack{\alpha_1+\alpha_2=\alpha, \\ |\alpha_1|\leqslant |\alpha|-1}}\Zr^{\alpha_1}(\lambda) \Zr^{\alpha_2}\Tr(\wb)\\
&-\frac{1}{2}c\Zr^\alpha\big(\Tr\Xr(\psi_2)\big)+\sum_{|\alpha_2|=|\alpha|-1 } \Zr(c) \Zr^{\alpha_2}(\Tr\Xr(\psi_2))-\frac{1}{2}\Tr (c)\Zr^{\alpha}\Xr(\psi_2)+\mathring{M}\varepsilon^2.
\end{align*}
Hence,
\begin{align*} 
\Zr^{\alpha}(\zr)\cdot \Tr \wb=&\Lr\Zr^{\alpha}\Tr(\wb)-\frac{1}{2}c\Zr^\alpha\big(\Tr\Xr(\psi_2)\big)+\sum_{|\alpha_2|=|\alpha|-1 } \Zr(c) \Zr^{\alpha_2}(\Tr\Xr(\psi_2))-\frac{1}{2}\Tr (c)\Zr^{\alpha}\Xr(\psi_2)\\
&+\sum_{\substack{\alpha_1+\alpha_2=\alpha, \\ |\alpha_1|\leqslant |\alpha|-1}}\Zr^{\alpha_1}(\lambda) \Zr^{\alpha_2+1}(\wb)+\mathring{M}\varepsilon^2.
\end{align*}
We remark that if we trace all the previous calculations, similar to \eqref{eq:precise formula for Z alpha yr}, we have
\begin{equation}\label{eq:precise formula for Z alpha zr}
\begin{split}
\Zr^{\alpha}(\zr)\cdot \Tr \wb
=&\Lr\Zr^{\alpha}\Tr(\wb)+\!\!\!\!\!\!\sum_{\substack{\alpha_1+\alpha_2=\alpha\\ |\alpha_1|\leqslant 1} } \Zr^{\alpha_1}(c) \Xr\Zr^{\alpha_2}(\Tr(\psi_2))+\Tr (c)\Xr\Zr^{\alpha}(\psi_2)+\!\!\!\!\!\!\sum_{\substack{\alpha_1+\alpha_2=\alpha, \\ |\alpha_1|\leqslant |\alpha|-1}}\!\!\!\!\!\!\Zr^{\alpha_1}(\lambda) \Zr^{\alpha_2+1}(\wb)\\
+&\!\!\!\!\!\!\sum_{\substack{\alpha_1+\alpha_2=\alpha \\ |\alpha_1|\geqslant 2}} \Zr^{\alpha_1}(c) \Zr^{\alpha_2}(\Tr\Xr(\psi_2))+\!\!\!\!\!\!\sum_{\substack{\alpha_1+\alpha_2=\alpha \\ |\alpha_1|\geqslant 1}} \Zr^{\alpha_1}\Tr (c)\Zr^{\alpha_2}\Xr(\psi_2)+\!\!\!\!\!\!\sum_{\alpha_1+\alpha_2=\alpha}\!\!\!\Zr^{\alpha_1}\Tr(\psi_2)\Zr^{\alpha_2}\Xr\wb.
\end{split}
\end{equation}
We then compute the $L^2(\Sigma_t)$ bound on each term appeared in the above formula.  By using {\color{black}$\Tr(\wb)\approx -1$,} $|c|\approx 1$ and $|\Zr(c)|\lesssim 1$ , we have
\begin{equation}\label{ineq: inductive bound on zr}
\begin{split}
\|\Zr^{\alpha}(\zr)\|_{L^2(\Sigma_t)}\lesssim&\|\Lr\Zr^{\alpha}\Tr (\wb)\|_{L^2(\Sigma_t)} +\|\Zr^\alpha\Tr\Xr(\psi_2)\|_{L^2(\Sigma_t)}+\sum_{|\alpha_2|=|\alpha|-1} \|\Zr^{\alpha_2}(\Tr\Xr(\psi_2))\|_{L^2(\Sigma_t)}\\
&+\|\Zr^{\alpha}\Xr(\psi_2)\|_{L^2(\Sigma_t)}+\!\!\!\!\!\!\sum_{\substack{\alpha_1+\alpha_2=\alpha, \\ |\alpha_1|\leqslant |\alpha|-1}}\!\!\!\!\!\!\|\Zr^{\alpha_1}(\lambda) \Zr^{\alpha_2+1}(\wb)\|_{L^2(\Sigma_t)}+\mathring{M}\varepsilon^2.
\end{split}
\end{equation}

\medskip

With the help of \eqref{ineq: inductive bound on yr} and \eqref{ineq: inductive bound on zr}, we perform an induction argument on $|\alpha|$ to derive  $L^2$ bounds on $\yr$ and $\zr$. More precisely, for all $|\alpha|\leqslant \Ntop-1$, we will show that
\begin{equation}\label{ineq: L^2 bound on lambda}
\|\Zr^{\alpha}(\yr)\|_{L^2(\Sigma_t)}+\|\Zr^{\alpha}(\zr)\|_{L^2(\Sigma_t)}\lesssim\frac{1}{t}\sqrt{\mathscr{E}_{\leqslant |\alpha|}(t)}+\mathring{M}\varepsilon^2.
\end{equation}
In the above expression, $\mathscr{E}_{\leqslant |\alpha|}(t)$ is {\color{black}the sum of energies} for all $\psi \in \{w,\wb,\psi_2\}$.

First of all, we notice that every linear term on the righthand sides of \eqref{ineq: inductive bound on yr} and \eqref{ineq: inductive bound on zr} contains either an $\Xr$ or an $\Lr$ derivative. By \eqref{comparison bound 1: L 2} and the ansatz $\mathbf{(B_2)}$, we have
\begin{equation}\label{ineq: inductive bound on lambda}
\|\Zr^{\alpha}(\yr)\|_{L^2(\Sigma_t)}+\|\Zr^{\alpha}(\zr)\|_{L^2(\Sigma_t)}\lesssim\frac{1}{t}\sqrt{\mathscr{E}_{\leqslant |\alpha|+1}(t)}+\!\!\!\!\sum_{\substack{\alpha_1+\alpha_2=\alpha, \\ |\alpha_1|\leqslant |\alpha|-1}}\!\!\!\!\|\Zr^{\alpha_1}(\lambda) \Zr^{\alpha_2+1}(\wb)\|_{L^2(\Sigma_t)}+\mathring{M}\varepsilon^2.
\end{equation}

We start to run the induction argument.  For $|\alpha|=0$, according to \eqref{eq: yr 0 order equation} and \eqref{eq: zr 0 order equation}, we have
\begin{equation*}
\| \yr\|_{L^2(\Sigma_t)}+\|\zr\|_{L^2(\Sigma_t)}\lesssim\frac{1}{t}\sqrt{\mathscr{E}_{\leqslant 1}(t)}+\mathring{M}\varepsilon^2.
\end{equation*}
Hence, \eqref{ineq: inductive bound on lambda} holds for $|\alpha|=0$.

We now make another assumption that $|\alpha|\leqslant \Ntop-2$. The induction hypothesis is that \eqref{ineq: L^2 bound on lambda} holds for all indices of length at most $|\alpha|-1$. In this case, the $\Zr^{\alpha_2+1}(\wb)$ term in \eqref{ineq: inductive bound on lambda} can be bounded in $L^\infty$ norm. This is because $|\alpha_2|+1\leqslant \Ninf$, see \eqref{ineq: L infty bound}. Hence, \eqref{ineq: inductive bound on lambda} and the induction hypothesis yield
\begin{align*}
\|\Zr^{\alpha}(\yr)\|_{L^2(\Sigma_t)}+\|\Zr^{\alpha}(\zr)\|_{L^2(\Sigma_t)}&\lesssim\frac{1}{t}\sqrt{\mathscr{E}_{\leqslant |\alpha|+1}(t)}+\!\!\!\!\sum_{\substack{\alpha_1+\alpha_2=\alpha, \\ |\alpha_1|\leqslant |\alpha|-1}}\!\!\!\!\|\Zr^{\alpha_1}(\lambda) \|_{L^2(\Sigma_t)}\|\Zr^{\alpha_2+1}(\wb)\|_{L^\infty(\Sigma_t)}+\mathring{M}\varepsilon^2\\
&\lesssim\frac{1}{t}\sqrt{\mathscr{E}_{\leqslant |\alpha|+1}(t)}+\sum_{|\alpha_1|\leqslant |\alpha|-1}\|\Zr^{\alpha_1}(\lambda) \|_{L^2(\Sigma_t)}+\mathring{M}\varepsilon^2\\
&\lesssim\frac{1}{t}\sqrt{\mathscr{E}_{\leqslant |\alpha|+1}(t)}+\mathring{M}\varepsilon^2.
\end{align*}
This proves \eqref{ineq: inductive bound on lambda} for all $\alpha$ with $|\alpha|\leqslant \Ntop-2$.

To verify the case where $|\alpha|=\Ntop-1$,  it requires the $L^\infty$ bounds on lower order derivatives of $\yr$ and $\zr$. For all multi-index $\alpha$ with $|\alpha|\leqslant \Ntop-4$ and $\lambda \in \{\yr,\zr\}$,  since $|\alpha|+2\leqslant \Ntop-2$, we apply \eqref{ineq:Sobolev inequalities}:
\begin{equation}\label{eq:yr zr aux l infty}
\begin{split}
\|\mathring{Z}^\alpha(\lambda)\|_{L^\infty(\Sigma_t)}&\lesssim \sum_{k+l\leqslant 2}\|\Xr^k \Tr^l \mathring{Z}^\alpha(\lambda)\|_{L^2(\Sigma_t)}\lesssim \sum_{|\beta|\leqslant |\alpha|+2}\|\mathring{Z}^\beta(\lambda)\|
_{L^2(\Sigma_t)}\\
&\lesssim \frac{1}{t}\sqrt{\mathscr{E}_{\leqslant N_{\rm top}}(t)}+\mathring{M}\varepsilon^2\lesssim \mathring{M}\varepsilon.
\end{split}
\end{equation}
In the last step, we have used $\mathbf{(B_2)}$.

Let $N'=\Ntop-4$. To prove \eqref{ineq: inductive bound on lambda} for $|\alpha|=\Ntop-1$, we write \eqref{ineq: inductive bound on lambda} as
\begin{align*}
&\|\Zr^{\alpha}(\yr)\|_{L^2(\Sigma_t)}+\|\Zr^{\alpha}(\zr)\|_{L^2(\Sigma_t)}\\
\lesssim&\frac{1}{t}\sqrt{\mathscr{E}_{\leqslant |\alpha|+1}(t)}+\big(\!\!\!\!\!\!\sum_{\substack{\alpha_1+\alpha_2=\alpha, \\ |\alpha_1|\leqslant N'}}\!\!\!\!\!\!+\!\!\!\!\!\!\sum_{\substack{\alpha_1+\alpha_2=\alpha, \\1\leqslant |\alpha_2|\leqslant N'}}\!\!\!\!\!\!\big) \! \|\Zr^{\alpha_1}(\lambda) \Zr^{\alpha_2+1}(\wb)\|_{L^2(\Sigma_t)}+\mathring{M}\varepsilon^2\\
\lesssim&\frac{1}{t}\sqrt{\mathscr{E}_{\leqslant |\alpha|+1}(t)}+\!\!\!\sum_{\substack{\alpha_1+\alpha_2=\alpha, \\ |\alpha_1|\leqslant N_{\infty}}}\!\!\!\!\!\!\mathring{M}\varepsilon \|\Zr^{\alpha_2+1}(\wb)\|_{L^2(\Sigma_t)}+\!\!\!\sum_{\substack{\alpha_1+\alpha_2=\alpha, \\1\leqslant |\alpha_2|\leqslant N_{\infty}}} \!\!\!\!\!\!\mathring{M}\varepsilon\|\Zr^{\alpha_1}(\lambda)\|_{L^2(\Sigma_t)}+\mathring{M}\varepsilon^2.
\end{align*}
Hence, \eqref{ineq: inductive bound on lambda} follows from the case $|\alpha|\leqslant \Ntop-2$ and \eqref{comparison bound 1: L 2}. Moreover, by repeating the argument in \eqref{eq:yr zr aux l infty}, we also proved that, for all multi-index $\alpha$ with $|\alpha|\leqslant \Ntop-3$ and $\lambda \in \{\yr,\zr\}$,  we have 
\begin{equation}\label{ineq: L infty bound on lambda}
\|\mathring{Z}^\alpha(\lambda)\|_{L^\infty(\Sigma_t)}\lesssim \mathring{M}\varepsilon.
\end{equation}

\subsubsection{Summary}
We summarize the results of the section as follows:
\begin{proposition} For all $|\alpha|\leqslant \Ntop-1$, for all $\lambda \in \{\yr,\zr,\chir,\etar\}$, we have
\begin{equation}\label{ineq: L^2 bound on lambda final}
\|\Zr^{\alpha}(\lambda)\|_{L^2(\Sigma_t)}\lesssim\frac{1}{t}\sqrt{\mathscr{E}_{\leqslant |\alpha|+1}(t)}+\mathring{M}\varepsilon^2.
\end{equation}
Moreover, for all multi-index $\alpha$ with $|\alpha|\leqslant \Ninf-1$, for $\lambda \in \{\yr,\zr,\chir,\etar\}$,  we have 
\begin{equation}\label{ineq: L infty bound on lambda final}
\|\mathring{Z}^\alpha(\lambda)\|_{L^\infty(\Sigma_t)}\lesssim \mathring{M}\varepsilon.
\end{equation}
\end{proposition}

\begin{remark}\label{remark:estimates-Lambda}
The estimates on objects of $\Lambda$ lose one derivative, i.e., the order of the righthand side of \eqref{ineq: L^2 bound on lambda final} is higher compared to the lefthand side.
\end{remark}

\subsection{Other auxiliary formulas and bounds}
\subsubsection{Other auxiliary formulas}
We recall that  $\Xr$ and $\Trh$ commute with vectors in $\mathring{\mathscr{Z}}$. For all multi-index $\alpha$,  we apply $\Zr^\alpha\in  \mathring{\mathscr{Z}}$ to \eqref{Euler equations:form 3} and we ignore {\color{black}the irrelevant constants in coefficients.} This leads to the following formulas:
\begin{equation}\label{Euler equations:form 3 with derivatives}
\begin{cases}
\Zr^\alpha \Lr (\wb) &= \displaystyle\sum_{\alpha_1+\alpha_2=\alpha} \Zr^{\alpha_1}(c) \cdot \Xr(\Zr^{\alpha_2}(\psi_2)),\\
\Zr^\alpha \Lr (w) &= 	\displaystyle\sum_{\alpha_1+\alpha_2=\alpha}\big[\Zr^{\alpha_1}(c)  \Trh(\Zr^{\alpha_2}(w))+ \Zr^{\alpha_1}(c) \Xr(\Zr^{\alpha_2}(\psi_2))\big],\\
\Zr^\alpha\Lr (\psi_2) &= \displaystyle\sum_{\alpha_1+\alpha_2=\alpha} \big[\Zr^{\alpha_1}(c)\Trh(\Zr^{\alpha_2}(\psi_2))+ \Zr^{\alpha_1}(c) \Xr(\Zr^{\alpha_2}(w+\wb))\big].
\end{cases}
\end{equation}
By dividing multiplying both sides by $c^{-1}$, we can also put \eqref{Euler equations:form 3} in the following form:
\begin{equation}\label{Euler equations:form 3 with derivatives prime}
\begin{cases}
\Zr^\alpha \big(c^{-1}\Lr (\wb)\big)&=   \frac{1}{2}\Xr(\Zr^{\alpha}(\psi_2)),\\
\Zr^\alpha  \big(c^{-1}\Lr (w) \big)&= 	 -2\Trh(\Zr^{\alpha}(w))+ \frac{1}{2}\Xr(\Zr^{\alpha}(\psi_2)),\\
\Zr^\alpha \big(c^{-1}\Lr (\psi_2)\big) &= - \Trh(\Zr^{\alpha}(\psi_2))+ \Xr(\Zr^{\alpha}(w+\wb)).
\end{cases}
\end{equation}
We can also use $\Lb$ as the main direction to write \eqref{Euler equations:form 3} as follows:
\begin{equation}\label{Euler equations:form 4}
\begin{cases}
\Lbr (\wb) &=2\Tr(\wb) +\frac{1}{2}\kappar \Xr(\psi_2),\\
\Lbr (w) &= 	\frac{1}{2}\kappar\Xr(\psi_2),\\
\Lbr (\psi_2) &= \Tr(\psi_2)+\kappar\Xr( w+\wb).
\end{cases}
\end{equation}
We apply $\Zr^\alpha\in  \mathring{\mathscr{Z}}$ to the above equations to derive
\begin{equation}\label{Euler equations:form 4 with derivatives}
\begin{cases}
\Zr^\alpha \Lbr (\wb) &= 2\Tr(\Zr^\alpha(\wb)) +\frac{1}{2}\kappar \Xr(\Zr^\alpha(\psi_2)),\\
\Zr^\alpha \Lbr (w) &= 		\frac{1}{2}\kappar\Xr(\Zr^\alpha(\psi_2)),\\
\Zr^\alpha\Lbr (\psi_2) &=  \Tr(\Zr^\alpha(\psi_2))+\kappar\Xr(\Zr^\alpha(w+\wb)).
\end{cases}
\end{equation}

\subsubsection{Other auxiliary bounds}

We collect some estimates on waves of the form $\Zr^\alpha \Lr \Zr^\beta \psi$ where $\psi\in \{w,\wb,\psi_2\}$. They will appear in the higher order energy estimates.

First of all, we notice that the $\Trh$ derivative only acts on $w$ or $\psi_2$. Therefore, by \eqref{ineq: L infty bound}, for all multi-index $\alpha$ with $|\alpha|\leqslant \Ninf-1$, for all $\psi \in \{w,\wb,\psi_2\}$,  we have
\begin{align*}
\|\mathring{Z}^\alpha \Lr \psi\|_{L^\infty(\Sigma_t)}\lesssim  \mathring{M}\varepsilon.
\end{align*}
We now commute $\Lr$ with $\Zr^\alpha$ to derive bounds on $\Lr\mathring{Z}^\alpha  \psi$.  In view of \eqref{eq: commutation formular for Lr Zr}, we can apply extra $\Zr^\beta$ derivatives and we obtain
\[\Zr^\beta\Lr \Zr^\alpha \psi = \Zr^{\alpha+\beta} \Lr  \psi +\sum_{\substack{\alpha_1+\alpha_2=\alpha,|\alpha_1|\leqslant |\alpha|-1\\
  \beta_1+\beta_2=\beta}}\Zr^{\alpha_1+\beta_1}(\lambda) \Zr^{\alpha_2+\beta_2}\psi.\]
Hence, by \eqref{ineq: L infty bound} and \eqref{ineq: L infty bound on lambda final}, for multi-indices $\alpha$ and $\beta$ with $|\alpha|+|\beta|\leqslant \Ninf-1$, we have
\begin{equation*}
\left\|\Zr^\beta\Lr \mathring{Z}^\alpha  \psi \right\|_{L^\infty(\Sigma_t)}\lesssim  \mathring{M}\varepsilon.
\end{equation*}
We can also apply \eqref{eq: commutation formular for Lr Zr} to $\Zr^\beta \psi$. Therefore, we have
\[\Lr \Zr^{\alpha+\beta} \psi = \Zr^\alpha \Lr \Zr^\beta \psi +\sum_{\substack{\alpha_1+\alpha_2=\alpha \\ |\alpha_1|\leqslant |\alpha|-1}}\Zr^{\alpha_1}(\lambda) \Zr^{\alpha_2+\beta}\psi{\color{black}.}\]
If  $|\alpha|+|\beta|\leqslant \Ntop$, the possible top order derivatives of $\lambda$ in this formula is at most $\Zr^{\Ntop-1}(\lambda)$. Hence, the inequality \eqref{ineq: L^2 bound on lambda final} can be applied. 
Therefore,
\begin{align*}
\|\Zr^\alpha \Lr \Zr^\beta \psi \|_{L^2(\Sigma_t)} &\lesssim \|\Lr \Zr^{\alpha+\beta} \psi\|_{L^2(\Sigma_t)} +\sum_{\substack{\alpha_1+\alpha_2=\alpha \\ |\alpha_1|\leqslant |\alpha|-1}}\|\Zr^{\alpha_1}(\lambda) \Zr^{\alpha_2+\beta}\psi\|_{L^2(\Sigma_t) }\lesssim \mathring{M}\varepsilon.
\end{align*}
In the last step, we have used  Remark \ref{remark:techical nonlinear}. We summarize the above estimates as follows:
\begin{proposition}
{\color{black}Under the bootstrap assumptions} $(\mathbf{B}_2)$ and $(\mathbf{B}_\infty)$, if $\mathring{M}\varepsilon$ is sufficiently small,  for all $t\in [\delta, t^*]$, we have the following bounds:
\begin{itemize}
\item For multi-indices $\alpha$ and $\beta$ with $|\alpha|+|\beta|\leqslant \Ninf-1$,  for all $\psi \in \{w,\wb,\psi_2\}$, we have
\begin{equation}\label{ineq: L infty bound for Zbeta L Zalpha}
\left\|\Zr^\beta\Lr \mathring{Z}^\alpha  \psi \right\|_{L^\infty(\Sigma_t)}\lesssim  \mathring{M}\varepsilon.
\end{equation}
\item For multi-indices $\alpha$ and $\beta$ with $|\alpha|+|\beta|\leqslant \Ntop$,  for all $\psi \in \{w,\wb,\psi_2\}$, we have
\begin{equation}\label{ineq: L 2 bound for Zbeta L Zalpha}
\| \Zr^\alpha \Lr \Zr^\beta \psi \|_{L^2(\Sigma_t)}\lesssim\mathring{M}\varepsilon.
\end{equation}
\end{itemize}
\end{proposition}

\section{Higher order energy estimates}\label{section:higher-order-energy-estimates}

We now apply the identities in Section \ref{section:Energy identities for higher order terms} to derive the higher order energy estimates for acoustical waves.

We recall that for $\Psi_0 \in \{\wb,w,\psi_2\}$, the equation \eqref{Main Wave equation: order 0} can be written as $\Box_g \Psi_0 =\varrho_0$. For a multi-index $\alpha$ with $|\alpha|=n$, we use $\Psi_n$ to denote $\Zr^\alpha(\Psi_0)$. When one applies \eqref{identity: energy with Lh} and \eqref{identity: energy with Lb} to $\Box_g \Psi_n =\varrho_n$, the corresponding error integrals ${Q}_0$ and $\underline{Q}_0$ are given by $-\int_{D(t,u)}  \frac{\mu}{\mur}\cdot \varrhor_n \cdot \Lh\Psi_n$ and $-\int_{D(t,u)} \frac{\mu}{\mur}\cdot \varrhor_n\cdot \Lb\Psi_n$ {\color{black}where $ \varrhor_n=  \mur \varrho_n$, respectively.} For $\Psir_{n}:=\Zr_{n}\big(\Zr_{n-1}\big(\cdots \big(\Zr_{1}(\Psir_0)\big)\cdots\big)\big)$, we have
\[\varrhor_{n}=\Zr_{n}\big(\cdots \big(\Zr_{1}(\varrhor_{0})\big)\cdots\big)+\sum_{i=0}^{n-1} \Zr_{n} \big( \cdots \big( \Zr_{i+2}\big(\,^{(\Zr_{i+1})}{\sigma}_{i}\big)\big)\cdots \big).
\]
Therefore, schematically, $\varrhor_n$ is a sum of the following  two types of terms:
\[\bullet \ \text{\bf Type I} :~  \Zr^{\beta}\left( \varrhor_{0}\right), \ \ |\beta| =n ; \ \ \ \ \ \ \bullet \ \text{\bf Type II} :~  \Zr^{\beta}\left(\,^{(\Zr)}{\sigma}\right), \ \ |\beta|\leqslant n-1.
\]
The {\bf Type II} terms in $\varrhor_n$ are of the form $\Zr^{\beta}\left(\,^{(\Zr_{i+1})}{\sigma_i}\right)$ where $|\beta|= n-i-1$. By \eqref{eq: decompose sigma}, we have $\,^{(\Zr_{i+1})}{\sigma_i}=\,^{(\Zr_{i+1})}{\sigma_{i,1}}+\,^{(\Zr_{i+1})}{\sigma_{i,2}}+\,^{(\Zr_{i+1})}{\sigma_{i,3}}$.  Thus, we can further decompose  {\bf Type II} terms as a sum of {\color{black}the three types of terms:} the {\bf Type $\mathbf{II}_k$} terms correspond to the contribution of $\,^{(\Zr_{i+1})}{\sigma_{i,k}}$ terms respectively, where $k=1,2,3$. 

In the rest of the paper, $n\leqslant \Ntop$.

\subsection{Energy estimates on {\bf Type I} terms}
Since $\frac{\mu}{\mur}\lesssim 1$, it suffices to bound $\mathscr{N}(\Psi_n)(t,u)$ and $\underline{\mathscr{N}}(\Psi_n)(t,u)$ in the following form
\[\mathscr{N}_n(t,u)= \int_{\mathcal{D}(t,u)}  \big| \Zr^{\beta}\big( \varrhor_{0}\big)\big| \big|\Lh\Psi_n\big|,  \ \  \underline{\mathscr{N}}_n(t,u)= \int_{\mathcal{D}(t,u)} \big|\Zr^{\beta}\big(\varrhor_{0}\big)\big| \big|\Lb\Psi_n\big|,\ \ \text{with}~|\beta|\leqslant n.\]
We remark that we used the simplified notations $\Psi_n $ for $\Zr^\alpha(\Psi_0)$, $\mathscr{N}_n(t,u)$ for $\mathscr{N}(\Psi_n)(t,u)$ and $\underline{\mathscr{N}}_n(t,u)$ for $\underline{\mathscr{N}}(\Psi_n)(t,u)$.  We will also use $\mathscr{E}_{n}(t,u)$ for $\mathscr{E}_{n}(\psi)(t,u)$ where $\psi\in \{w,\wb,\psi_2\}$. Similarly, we also use notations like $\mathscr{E}_{\leqslant n}(t,u)$, $\mathscr{F}_{n}(t,u)$ etc.

We recall that $\varrhor_0=  \mur \varrho_0$ and $\varrho_0$ is a linear combination of terms from the set $\big\{c^{-1}g(D f_1,Df_2)\big| f_1,f_2\in \{\wb,w,\psi_2\}\big\}$ where 
\[g(D f_1,Df_2)=-\frac{1}{2\mur}\Lr(f_1) \Lbr(f_2)-\frac{1}{2\mur}\Lbr(f_1) \Lr(f_2)+\Xr(f_1)\Xr(f_2).\]
By applying $\Zr^{\beta}$ to $\varrhor_0$,  we can write $\Zr^{\beta}\left( \varrhor_{0}\right)$ as a linear combination of the following terms:
\begin{equation}\label{bbb2}
\Zr^{\beta_1}\big(c^{-1} \Lr (f_1)\big)\Zr^{\beta_2}\big(\Lbr( f_2)\big),   \ \  \kappar \Zr^{\beta_1}\big(\Xr(f_1)\big)\Zr^{\beta_2}\big(\Xr(f_2)\big), 
\end{equation}
where $f_1,f_2\in \{\wb,w,\psi_2\}$ and  $|\beta_1|+|\beta_2|=|\beta|$.

\subsubsection{The first case: $\Lh$ as the multiplier}\label{subsection top order source term Lh}
The contribution of \eqref{bbb2} in $\mathscr{N}_n(t,u)$ split into the sum (over $\beta_1$ and $\beta_2$) of {\color{black}the following terms:}
\[
\begin{cases}
\mathscr{N}_{f_1,f_2;1}(t,u)&= \displaystyle\int_{\mathcal{D}(t,u)}   \big| \Zr^{\beta_1}\big(c^{-1}\Lr (f_1)\big)\big|\big|\Zr^{\beta_2}\big(\Lbr (f_2)\big)\big| \big|\Lh\Psi_n \big|, \\
\mathscr{N}_{f_1,f_2;2}(t,u)&=  \displaystyle \int_{\mathcal{D}(t,u)}    \kappar \big|\Zr^{\beta_1}\big(\Xr(f_1)\big)\big|\big|\Zr^{\beta_2}\big(\Xr(f_2)\big)\big| \big|\Lh\Psi_n 	\big|, 
\end{cases}
\]
where  $|\beta_1|+|\beta_2|=|\beta|=n\leqslant \Ntop$. 

\medskip

We start with the estimate on $\mathscr{N}_{f_1,f_2;2}(t,u)$. Since $\Xr$ commute with all $\Zr\in \mathring{\mathscr{Z}}$, we have
\begin{align*}
\mathscr{N}_{f_1,f_2;2}(t,u)\lesssim & \int_{\mathcal{D}(t,u)}   \kappar^2 \big| \Xr\big(\Zr^{\beta_1} f_1\big)\big|\big| \Xr\big(\Zr^{\beta_2} f_2\big)\big| \big|  L\Psi_n\big|.
\end{align*}
Without loss of generality, we assume $|\beta_1|+1\leqslant \Ninf$. Thus, by \eqref{ineq: L infty bound}, $\|\Xr \big(\Zr^{\beta_1} (f_1)\big) \|_{L^\infty}\lesssim \mathring{M} \varepsilon$. Hence,
\begin{align*}
\mathscr{N}_{f_1,f_2;2}(t,u)
\lesssim & \mathring{M}\varepsilon \int_{\mathcal{D}(t,u)}  \kappar^2  \big| \Xr\big(Z^{\beta_2} f_2\big)\big| \big|  L\Psi_n\big|.
\end{align*}
We can use Cauchy-Schwarz inequality as well as \eqref{comparison bound 1: L 2} and $\mathbf{(B_2)}$ to derive
\begin{align*}
\mathscr{N}_{f_1,f_2;2}(t,u)
\lesssim & \mathring{M} \varepsilon \int_\delta^t \mathscr{E}_{\leqslant n}(\tau)+\varepsilon\EB_{\leqslant n}(\tau)d\tau\lesssim \mathring{M} \varepsilon^3t^3.
\end{align*}

\medskip

We turn to $\mathscr{N}_{f_1,f_2;1}(t,u)$. It consists of the following two cases:
\begin{itemize}[noitemsep,wide=0pt, leftmargin=\dimexpr\labelwidth + 2\labelsep\relax]

\item[(a)] The case where $|\beta_2|\geqslant \Ninf$. 

In this case, since $|\beta_1|+1\leqslant \Ninf$,  by \eqref{Euler equations:form 3 with derivatives prime} and \eqref{ineq: L infty bound}, we have $\big|\Zr^{\beta_1}\big(c^{-1}\Lr (f_1)\big)\big|\lesssim \mathring{M}\varepsilon$. Therefore,  $\mathscr{N}_{f_1,f_2;1}(t,u)$ is bounded as follows:
\begin{align*}
\mathscr{N}_{f_1,f_2;1}(t,u)&\lesssim  \mathring{M}\varepsilon\int_{\mathcal{D}(t,u)}   \kappar \big| \Zr^{\beta_2}\big(\Lbr (f_2)\big)\big| \big|L\Psi_n\big|.
\end{align*}
We then consider three cases where $f_2=\wb, w$ or $\psi_2$ respectively. For $f_2=\wb$, we use \eqref{Euler equations:form 4 with derivatives} to replace $\Zr^{\beta_2}\big(\Lbr (f_2)\big)$ and this leads to
\begin{align*}
\mathscr{N}_{f_1,f_2;1}(t,u)&\lesssim  \varepsilon\int_{\mathcal{D}(t,u)}  \kappar \big(\big|\Tr(\Zr^{\beta_2}(\wb))\big|+\big|\kappar \Xr(\Zr^{\beta_2}(\psi_2))\big| \big)  \big|L\Psi_n\big| \\
&\lesssim  \varepsilon\int_\delta^t \big(\int_{\Sigma_\tau^u} \big|\Tr(\Zr^{\beta_2}(\wb))\big|^2+\kappar^2 \big|\Xr(\Zr^{\beta_2}(\psi_2))\big|^2 \big)^\frac{1}{2}\big(\int_{\Sigma_\tau^u} \kappar \big|L\Psi_n\big|^2 \big)^\frac{1}{2}d\tau.
\end{align*}
We can use \eqref{comparison bound 1: L 2} and $\mathbf{(B_2)}$ and this leads to $\mathscr{N}_{f_1,f_2;1}(t,u)\lesssim \mathring{M}\varepsilon^3t^2$. The estimates for $f_2=w$  or $\psi_2$ can be derived exactly in the same manner. Hence,
\begin{align*}
\mathscr{N}_{f_1,f_2;1}(t,u)& \lesssim  \mathring{M}\varepsilon^3t^2.
\end{align*}

\item[(b)] The case where $|\beta_2|\leqslant \Ninf-1$. 

In this case, we can use \eqref{Euler equations:form 4 with derivatives} to replace $\Zr^{\beta_2}\big(\Lbr (f_2)\big)$. Thus, \eqref{ineq: L infty bound} implies that $\big|\Zr^{\beta_2}\big(\Lbr (f_2)\big)\big|\lesssim 1$, provided $\mathring{M}\varepsilon$ is sufficiently small. Hence,
\begin{align*}
\mathscr{N}_{f_1,f_2;1}(t,u)&\lesssim \int_{\mathcal{D}(t,u)} \kappar \big| \Zr^{\beta_1} \big(c^{-1}\Lr (f_1) \big) \big|  \big| L\Psi_n  \big|.
\end{align*}
We will use formula \eqref{Euler equations:form 3 with derivatives prime} to replace $\Zr^{\beta_1}\big(c^{-1}\Lr (f_1)\big)$ in the integrand. {\color{black}We consider two cases where $f_1=\wb$, and $f_1=w$ or $\psi_2$  separately.}  
\begin{itemize}
\item[(b-1)] For $f_1=\wb$,   by \eqref{Euler equations:form 3 with derivatives prime}, we have
\begin{align*}
\mathscr{N}_{f_1,f_2;1}(t,u)&\lesssim  \int_{\mathcal{D}(t,u)}  \kappar \big|\Xr(\Zr^{\beta_1}(\psi_2))\big| \big|L\Psi_n\big| \\
&\lesssim  \int_{\mathcal{D}(t,u)}  \kappar \big(\big|\Xh(\Zr^{\beta_1}(\psi_2))\big|+\varepsilon t \big|L(\Zr^{\beta_1}(\psi_2))\big|+\varepsilon \big|\Lb(\Zr^{\beta_1}(\psi_2))\big|\big)\big|L\Psi_n\big|.
\end{align*}
In the last step, we used \eqref{comparison bound 1: L infty} to bound $\Xr(\Zr^{\beta_1}(\psi_2))$. We can proceed in the same manner as for $\mathscr{N}_{f_1,f_2;2}(t,u)$ to bound the contribution of the second and third terms by $\mathring{M}\varepsilon^3t^2$. Thus,
\begin{align*}
\mathscr{N}_{f_1,f_2;1}(t,u)&\lesssim \mathring{M}\varepsilon^3t^2+ \int_{\mathcal{D}(t,u)}  \kappar \big|\Xh(\Zr^{\beta_1}(\psi_2))\big| \big|L\Psi_n\big|\\
&\lesssim \mathring{M}\varepsilon^3t^2+\int_0^u \mathscr{F}_{\leqslant n}(t,u')du'.
\end{align*}

\item[(b-2)] For $f_1=w$,  by \eqref{Euler equations:form 3 with derivatives prime} and \eqref{comparison bound 1: L infty}, we have
\begin{align*}
\mathscr{N}_{f_1,f_2;1}(t,u)&\lesssim  \int_{\mathcal{D}(t,u)}  \kappar \big( \big|\Trh(\Zr^{\beta_1}(w)) \big|+ \big|\Xr(\Zr^{\beta_1}(\psi_2)) \big| \big)  \big|L\Psi_n \big|\\
&\lesssim  \int_{\mathcal{D}(t,u)}   \big( \big|\Lb(\Zr^{\beta_1}(\psi)) \big|+\kappar \big|L(\Zr^{\beta_1}(\psi)) \big|+\kappar \big|\Xh(\Zr^{\beta_1}(\psi)) \big| \big)  \big|L\Psi_n \big|.
\end{align*}
We can bound the last two terms by $\displaystyle\int_0^u \mathscr{F}_{\leqslant n}(t,u')du'$. Therefore, according to \eqref{def:L2 L3}, we obtain that
\begin{align*}
\mathscr{N}_{f_1,f_2;1}(t,u) \lesssim \mathscr{L}_2(\Zr^{\beta_1}(\psi),\Psi_n)(t,u) +\int_0^u \mathscr{F}_{\leqslant n}(u')du'.
\end{align*}
The estimates for $f_1=\psi_2$ can be derived exactly in the same manner.
\end{itemize}

\end{itemize}
Combining all the above estimates, we obtain that
\begin{equation}\label{est: N_n for I_1}
\mathscr{N}_{n}(t,u)\lesssim  \mathring{M}\varepsilon^3t^2+\sum_{1\leqslant |\beta|\leqslant n}\mathscr{L}_2(\Zr^{\beta_1}(\psi),\Psi_n)(t,u)+\int_0^u \mathscr{F}_{\leqslant n}(u')du'.
\end{equation}

\subsubsection{The second case: $\Lb$ as the multiplier}\label{subsection top order source term Lb}

We turn to $\underline{\mathscr{N}}_{n}(t,u)$. The contribution of \eqref{bbb2} in $\underline{\mathscr{N}}_{n}(t,u)$ {\color{black}splits} into two types of terms:
\[
\begin{cases}
\underline{\mathscr{N}}_{f_1,f_2;1}(t,u)&= \displaystyle\int_{\mathcal{D}(t,u)}   \big| \Zr^{\beta_1} \big(c^{-1}\Lr (f_1) \big) \big| \big|\Zr^{\beta_2} \big(\Lbr (f_2) \big) \big| \big|\Lb\Psi_n  \big|, \\
\underline{\mathscr{N}}_{f_1,f_2;2}(t,u)&= \displaystyle \int_{\mathcal{D}(t,u)}    \kappar  \big|\Zr^{\beta_1} \big(\Xr(f_1) \big) \big| \big|\Zr^{\beta_2} \big(\Xr(f_2) \big) \big|  \big|\Lb\Psi_n 	 \big|, 
\end{cases}\]
where  $|\beta_1|+|\beta_2|=|\beta|=n\leqslant \Ntop$.

\medskip

We start with $\underline{\mathscr{N}}_{f_1,f_2;2}(t,u)$.  Without loss of generality, we assume that  $|\beta_1|+1\leqslant \Ninf$. Therefore, we can use \eqref{ineq: L infty bound} to derive $\|\Zr^{\beta_1} \big( \Xr (f_1)\big) \|_{L^\infty}\lesssim  \mathring{M}\varepsilon$. By virtue of $\mathbf{(B_2)}$, we have
\begin{align*}
\underline{\mathscr{N}}_{f_1,f_2;2}(t,u)
\lesssim & \mathring{M}\varepsilon \int_{\mathcal{D}(t,u)}  \kappar  \big| \Xr\big(\Zr^{\beta_2} f_2\big)\big| \big|  \Lb\Psi_n\big| \lesssim    \mathring{M} \varepsilon^3t^2.
\end{align*}

\medskip

We turn to  $\underline{\mathscr{N}}_{f_1,f_2;1}(t,u)$ and we consider the following two cases:
\begin{itemize}[noitemsep,wide=0pt, leftmargin=\dimexpr\labelwidth + 2\labelsep\relax]

\item[(a)] The case where $|\beta_2|\geqslant \Ninf$. 

Similar to the case (a)  in Section \ref{subsection top order source term Lh}, we use \eqref{Euler equations:form 3 with derivatives prime} to derive $\big|\Zr^{\beta_2}\big(c^{-1}\Lr (f_1)\big)\big|\lesssim \mathring{M}\varepsilon$. Hence,
\begin{align*}
\underline{\mathscr{N}}_{f_1,f_2;1}(t,u)&\lesssim \mathring{M} \varepsilon\int_{\mathcal{D}(t,u)}     \big|\Zr^{\beta_2}\big(\Lbr (f_2)\big)\big| \big|\Lb\Psi_n\big|.
\end{align*}
If $f_2=\wb$, we use \eqref{Euler equations:form 4 with derivatives} to replace $\Zr^{\beta_3}\big(\Lbr (f_2)\big)$ and we also use \eqref{comparison bound 1: L infty} to derive
\begin{equation}\label{section 9.2 aux 1}
\begin{split}
\underline{\mathscr{N}}_{f_1,f_2;1}(t,u)&\lesssim \mathring{M}  \varepsilon\int_{\mathcal{D}(t,u)} \big(\big|\Tr(\Zr^{\beta_2}(\wb))\big|+\big|\kappar \Xr(\Zr^{\beta_2}(\psi_2))\big|\big)  \big|\Lb\Psi_n\big| \\
&\lesssim \mathring{M}  \varepsilon\int_{\mathcal{D}(t,u)} \big(\big|\Lb(\Zr^{\beta_2}(\wb))\big|+\kappar\big|L(\Zr^{\beta_2}(\wb))\big|+\kappar\big|\Xh(\Zr^{\beta_2}(\psi_2))\big| \big) \big|\Lb\Psi_n\big|.
\end{split}
\end{equation}
By $\mathbf{(B_2)}$, the above is bounded by $ \mathring{M}\varepsilon^3t^2$. The estimates for $f_2=w$ or $\psi_2$ can be derived exactly in the same way. Hence,
\begin{align*}
\underline{\mathscr{N}}_{f_1,f_2;1}(t,u)& \lesssim \mathring{M}\varepsilon^3t^2.
\end{align*}

\item[(b)] The case where $|\beta_2|< \Ninf$. 

Similar to the case (b)  in Section \ref{subsection top order source term Lh}, we use \eqref{Euler equations:form 4 with derivatives} to derive $\big|\Zr^{\beta_2}\big(\Lbr (f_2)\big)\big|\lesssim 1$. Hence,
\begin{align*}
\underline{\mathscr{N}}_{f_1,f_2;1}(t,u)&\lesssim \int_{\mathcal{D}(t,u)} \big| \Zr^{\beta_1}\big(c^{-1}\Lr (f_1)\big)\big| \big| \Lb\Psi_n \big|.
\end{align*}
We consider two cases where $f_1=\wb$, and $f_1=w$ or $\psi_2$, and they will be treated differently. 
\begin{itemize}
\item[(b-1)] For $f_1=\wb$,  we use \eqref{Euler equations:form 3 with derivatives prime} to replace $\Zr^{\beta_1}\big(c^{-1}\Lr (\wb)\big)$. Similar to the case (b-1)  in Section \ref{subsection top order source term Lh}, by combining with \eqref{comparison bound 1: L infty}, this leads to
\begin{equation}\label{section 9.2 aux 2}
\begin{split}
\underline{\mathscr{N}}_{f_1,f_2;1}(t,u)&\lesssim  \int_{\mathcal{D}(t,u)} \big|\Xr(\Zr^{\beta_1}(\psi_2))\big| \big|\Lb\Psi_n\big|  \\
&\lesssim \int_{\mathcal{D}(t,u)}  \big(\big|\Xh(\Zr^{\beta_1}(\psi_2))\big|+\varepsilon t \big|L(\Zr^{\beta_1}(\psi_2))\big| +\varepsilon\big|\Lb(\Zr^{\beta_1}(\psi_2))\big|\big)\big|\Lb\Psi_n\big| \\
&\lesssim \mathring{M}\varepsilon^3 t^3+\mathscr{L}_3(\Zr^{\beta_1}(\psi_2),\Psi_n)(t,u){\color{black}.}
\end{split}
\end{equation}

\item[(b-2)] For $f_1=w$ or $\psi_2$,  the direct use of the second equation of \eqref{Euler equations:form 3 with derivatives prime} will generate a $\Trh$ direction and it causes a loss in $\kappar$. We will commute $\Lr$ with $Z^{\beta_1}$ to avoid the loss. 

We further decompose the integral into two sums. Schematically, we have
\begin{align*}
\underline{\mathscr{N}}_{f_1,f_2;1}(t,u)&\lesssim  \big(\sum_{\substack{\beta_1'+\beta_1''=\beta_1\\ \beta_1''\geqslant \Ninf}} +\sum_{\substack{\beta_1'+\beta_1''=\beta_1\\ \beta_1''< \Ninf}}\big)\int_{\mathcal{D}(t,u)} \big|\Zr^{\beta'_1}\big(c^{-1}\big)\big| \big|\Zr^{\beta''_1}\big(\Lr (f_1)\big)\big| \big| \Lb\Psi_n \big|=\mathbf{S}_1+\mathbf{S}_2.
\end{align*}
We use $\mathbf{S}_1$ and $\mathbf{S}_2$ to denote the first and the second sum.

In $\mathbf{S}_1$,  since $\beta_1'< \Ninf$, we have $\big|\Zr^{\beta'_1}\big(c^{-1}\big)\big| \lesssim 1$. Thus,
\[\mathbf{S}_1\lesssim \sum_{\beta_1''\leqslant \beta_1}\int_{\mathcal{D}(t,u)} \big|\Zr^{\beta''_1}\big(\Lr (f_1)\big)\big| \big| \Lb\Psi_n \big|.\]
We apply \eqref{eq: commutation formular for Lr Zr} to $\Zr^{\beta''_1}\big(\Lr(f_1)\big)$ and we derive 
\begin{align*}
\mathbf{S}_1&\lesssim \sum_{\beta_1''\leqslant \beta_1}\big(\int_{\mathcal{D}(t,u)} \big|\Lr \Zr^{\beta''_1}\big(f_1\big)\big| \big| \Lb\Psi_n \big|+\sum_{\substack{\alpha_1+\alpha_2=\beta''_1\\|\alpha_2|\geqslant 1}} \int_{\mathcal{D}(t,u)} \big|\Zr^{\alpha_1}(\lambda)\big| \big|\Zr^{\alpha_2}(f_1)\big| \big| \Lb\Psi_n \big|\big)\\
&\lesssim \sum_{|\beta_1''|\leqslant n}\mathscr{L}_2(\Zr^{\beta''_1}(\psi),\Psi_n)(t,u) +\sum_{\substack{|\alpha_1|+|\alpha_2|\leqslant n\\ |\alpha_2|\geqslant 1}}\underbrace{\int_{\mathcal{D}(t,u)} \big|\Zr^{\alpha_1}(\lambda)\big| \big|\Zr^{\alpha_2}(f_1)\big| \big| \Lb\Psi_n \big|}_{\mathbf{I}_{\alpha_1,\alpha_2}}.
\end{align*}
We recall that the geometric quantities $\lambda \in \Lambda=\{\yr, \zr, \chir,\etar\}$. According to Remark \ref{remark: commutator structure}, if $\lambda=\yr$ or $\zr$, we have $\Zr^{\alpha_2}(f_1)=\Tr(\Zr^{\alpha'_2}(f_1))$. 

It remains to bound the integrals $\mathbf{I}_{\alpha_1,\alpha_2}$ where $|\alpha_1|+|\alpha_2|\leqslant n$ and $|\alpha_2|\geqslant 1$.  According to the size of $\alpha_1$, we have two different cases :
\begin{itemize}
\item[1)]$|\alpha_1|\leqslant \Ninf-1$. By \eqref{ineq: L infty bound on lambda final}, we have
\begin{equation}\label{ineq: case 1}
\mathbf{I}_{\alpha_1,\alpha_2}\lesssim \mathring{M}\varepsilon\int_{\mathcal{D}(t,u)} \!\!\!\!\!\!\big|\Zr^{\alpha_2}(f_1)\big| \big| \Lb\Psi_n \big|\lesssim \mathring{M}\varepsilon\int_{\mathcal{D}(t,u)} \!\!\!\!\!\!\big(\big|\Tr\Zr^{\alpha_2-1}(f_1)\big|+\big|\Xr\Zr^{\alpha_2-1}(f_1)\big|\big) \big| \Lb\Psi_n\big|.
\end{equation}
We then apply \eqref{comparison bound 1: L infty} and we derive
\begin{align*}
\mathbf{I}_{\alpha_1,\alpha_2}&\lesssim \mathring{M}\varepsilon \int_{\mathcal{D}(t,u)} \big(\big|\Lb\Zr^{\alpha_2-1}(f_1)\big|+\big|\Xh\Zr^{\alpha_2-1}(f_1)\big|+t\big|L\Zr^{\alpha_2-1}(f_1)\big|\big) \big| \Lb\Psi_n \big|\\
&\lesssim \mathring{M}\varepsilon^3 t^2+\mathring{M}\varepsilon \mathscr{L}_3(\Zr^{\alpha_2-1}(f_1),\Psi_n)(t,u).
\end{align*}
In view of the inequality \eqref{ineq: estimates for L2 L3} and $\mathbf{(B_2)}$, we have
\begin{align*}
\mathbf{I}_{\alpha_1,\alpha_2}&\lesssim \mathring{M}\varepsilon^3 t^2.
\end{align*}

\item[2)] $|\alpha_1|\geqslant \Ninf$. In this case, we use the bound  $\big|\Zr^{\alpha_2}(f_1)\big|\lesssim \varepsilon$. Hence,
\begin{align*}
\mathbf{I}_{\alpha_1,\alpha_2}&\lesssim \mathring{M}\varepsilon\int_{\mathcal{D}(t,u)}  \big|\Zr^{\alpha_1}(\lambda)\big| \big| \Lb\Psi_n \big|\lesssim \mathring{M}\varepsilon\int_{\delta}^t \|\Zr^{\alpha_1}(\lambda)\|_{L^2(\Sigma_\tau)} \left\| \Lb\Psi_n \right\|_{L^2(\Sigma_\tau)}d\tau.
\end{align*}
Therefore, we can apply the bound \eqref{ineq: L^2 bound on lambda final} and $\mathbf{(B_2)}$ to derive
\begin{align*}
\mathbf{I}_{\alpha_1,\alpha_2}&\lesssim \mathring{M}\varepsilon^3t^2+\mathring{M}\varepsilon\int_{\delta}^t \frac{1}{
\tau}\mathscr{E}_{\leqslant n}(\tau,u)d\tau\lesssim \mathring{M}\varepsilon^3t^2.
\end{align*}
\end{itemize} 
Combining the case 1) and 2), we obtain that
\begin{align*}
\mathbf{S}_1&\lesssim \mathring{M}\varepsilon^3 t^2+\sum_{|\gamma|\leqslant n}\mathscr{L}_2(\Zr^{\gamma}(\psi),\Psi_n)(t,u).
\end{align*}

\medskip

In $\mathbf{S}_2$, we have $\beta_1''< \Ninf$.  We have to first deal with $\Zr^{\beta'_1}\big(c^{-1}\big)$. It can be expanded as a linear combination of terms of the shape $c^{-m}\Zr^{\beta'_{1;i_1}}(c)\Zr^{\beta'_{1;i_2}}(c)\cdots \Zr^{\beta'_{1;i_k}}(c)$ with $\sum_{j=1}^k\beta'_{1;i_j}=\beta'_1$. Without loss of generality, we assume that $|\beta'_{1;i_1}|=\displaystyle\max_{j\leqslant k}|\beta'_{1;i_j}|$. Hence, 
\begin{equation}\label{eq: a bound on Zbeta c -1}
\big|\Zr^{\beta'_1}\big(c^{-1}\big)\big|\lesssim \big|\Zr^{\beta'_{1;i_1}}\big(c\big)\big|.
\end{equation}
Therefore, we have
\begin{align*}
\mathbf{S}_2&\lesssim \varepsilon \sum_{\beta'_{1;i_1}\leqslant \beta_1} \int_{\mathcal{D}(t,u)} \big|\Zr^{\beta'_{1;i_1}}\big(c\big)\big| \big|\Zr^{\beta''_1}\big(\Lr (w)\big)\big|\big| \Lb\Psi_n\big|.
\end{align*}
We may assume that $|\beta'_{1;i_1}|\geqslant 2$. Otherwise, we use the bound $\big|\Zr^{\beta'_{1;i_1}}\big(c\big)\big| \lesssim 1$ and this term has already been controlled in $\mathbf{S}_1$. We then write $\Zr^{\beta'_{1;i_1}}\left(c\right)$ as $\Zr\big(\Zr^{\widetilde{\beta}'_{1;i_1}}\big)(c)$  where $\Zr^{\beta'_{1;i_1}}=\mathring{Z} \Zr^{\widetilde{\beta}'_{1;i_1}}$ and $|\widetilde{\beta}'_{1;i_1}|\geqslant 1$. By $\beta_1''<\Ninf$, we use \eqref{ineq: L infty bound for Zbeta L Zalpha} to bound $\big|\Zr^{\beta''_1}\big(\Lr (w)\big)\big| \lesssim \mathring{M} \varepsilon$.  Therefore, by rewriting $c$ in terms of $w$ and $\wb$, we have
\begin{equation}\label{type I S1 aux 1}
\mathbf{S}_2\lesssim \mathring{M}\varepsilon   \int_{\mathcal{D}(t,u)}\big(\big|\Tr \Zr^{\widetilde{\beta}'_{1;i_1}}\psi\big|+\big|\Xr \Zr^{\widetilde{\beta}'_{1;i_1}}\psi\big|\big) \big| \Lb\Psi_n \big|.
\end{equation}
We have already handled a similar bound in \eqref{ineq: case 1}. This leads to
\begin{align*}
\mathbf{S}_2&\lesssim \mathring{M}\varepsilon^3 t^2.
\end{align*}

\end{itemize}

\end{itemize}

Combining all the above estimates, we conclude that
\begin{equation}\label{est: N_nb for I_1}
\underline{\mathscr{N}}_{n}(t,u)\lesssim  \mathring{M}\varepsilon^3 t^2+\sum_{|\gamma|\leqslant n}\mathscr{L}_2(\Zr^{\gamma}(\psi),\Psi_n)(t,u)+\sum_{|\gamma|\leqslant n}\mathscr{L}_3(\Zr^{\gamma}(\psi_2),\Psi_n)(t,u).
\end{equation}

\subsubsection{Summary}

In view of \eqref{est: N_n for I_1} and \eqref{est: N_nb for I_1}, the error terms of {\bf Type I} can be bounded as follows:
\begin{align*}
{\mathscr{N}}_n(t,u)+\underline{\mathscr{N}}_n(t,u)
&\lesssim    \mathring{M}\varepsilon^3 t^2+\sum_{|\gamma|\leqslant n}\mathscr{L}_2(\Zr^{\gamma}(\psi),\Psi_n)(t,u)+\sum_{|\gamma|\leqslant n}\mathscr{L}_3( \Zr^{\gamma}(\psi_2),\Psi_n)(t,u)+\int_0^u \mathscr{F}_{\leqslant n}(u')du'.
\end{align*}

\subsection{Estimates on {\bf Type $\mathbf{II}_1$} terms}

For the sake of simplicity, we use ${}^{(\mathring{Z})} \sigma_{k}$ to denote $\,^{(\Zr_{i+1})}{\sigma_{i,k}}$ where $k=1,2,3$. Since $\frac{\mu}{\mur}\lesssim 1$, it is suffices to bound the contribution of  $\,^{(\Zr_{i+1})}{\sigma_{i,1}}$'s in $\mathscr{N}_n(t,u)$ and $\underline{\mathscr{N}}_n(t,u)$ in the following form
\[\mathscr{N}_n(t,u)= \int_{\mathcal{D}(t,u)}  \big| \Zr^\beta \big({}^{(\mathring{Z})} \sigma_{1}\big)\big| \big|\Lh\Psi_n\big|,  \ \  \underline{\mathscr{N}}_n(t,u)= \int_{\mathcal{D}(t,u)} \big| \Zr^\beta\big({}^{(\mathring{Z})} \sigma_{1}\big)\big| \big|\Lb\Psi_n\big|,\]
where $|\beta|\leqslant n-1$.
For $\Zr=\Xr$ or $\Tr$, we have $\,^{(\mathring{Z})}{\pi}_{\Lbr\Xr}= c^{-1}\kappa\,^{(\mathring{Z})}{\pi}_{\Lr\Xr}$. Therefore, we rewrite \eqref{eq: sigma 1 circle} as
\begin{equation*}
\begin{split}
{}^{(\Zr)} \sigma_{1}&=-\frac{1}{2}\underbrace{\big(\Lr(c^{-1}\kappar)+2\chibr-2c^{-1}z\big)\cdot \pi_{\Lr\Xr} \cdot \Xr(\Psi_{m})}_{\sigma_{1,1}} +\frac{1}{4}\underbrace{(c^{-1}\chibr-c^{-2}z) \cdot \frac{\pi_{\Lbr\Lbr}}{c^{-1}\kappar^2} \cdot\Lr(\Psi_{m})}_{\sigma_{1,2}}\\
& \ \ +\underbrace{\frac{1}{4}\big[\frac{1}{\kappar}\big(\Lr(c^{-1}\kappar)+\chibr-c^{-1}z\big)+\Lbr\big(\frac{1}{\kappar}\big)\big]\cdot c^{-1}\pi_{\Lr\Lr}\cdot \Lbr(\Psi_{m})}_{\sigma_{1,3}}.\end{split}
\end{equation*}
In the above expression, we used ${\pi}$ to denote the deformation tensor $\,^{(\mathring{Z})}{\pi}$. In view of \eqref{eq: varphor_n expression}, it is important to observe that $|\beta|+m+1\leqslant \Ntop$.

We will first derives estimates on $\sigma_{1,1}$ and $\sigma_{1,2}$ and then on $\sigma_{1,3}$.
  
\subsubsection{Estimates on $\sigma_{1,1}$ and $\sigma_{1,2}$}\label{section sigma11 12}
The terms in $\sigma_{1,1}$ and $\sigma_{1,2}$ can be schematically represented as $G \times D\times W$ where
\begin{equation}\label{sigma 1 1 G D W}
G\in \{ \Lr(c^{-1}\kappar), \chibr,  c^{-1}\chibr,  c^{-1}z, c^{-2}z\}, \ D\in \{\,^{(\mathring{Z})}{\pi}_{\Lr\Xr}, c\kappar^{-2}\,^{(\mathring{Z})}{\pi}_{\Lbr\Lbr} \},\ W\in \{\Lr(\Psi_{m}),\Xr(\Psi_{m})\}.
\end{equation}
We will bound these terms one by one. In the following, we bound the derivative of $G\times D\times W$ by
\[\left|\Zr^{\beta}\left( G\times D\times W\right)\right|\lesssim \sum_{\beta_1+\beta_2+\beta_3=\beta}|\Zr^{\beta_1}(G)| |\Zr^{\beta_2}(D)||\Zr^{\beta_3}(W)|.\]
According to size of the multi-indices $\beta_i$'s, it suffices to consider three cases:

\begin{itemize} [noitemsep,wide=0pt, leftmargin=\dimexpr\labelwidth + 2\labelsep\relax]

\item[(a)] $|\beta_1|\leqslant \Ninf-1$ and $|\beta_2|\leqslant \Ninf-1$. 

We first use \eqref{ineq: L infty bound} and \eqref{ineq: L infty bound on lambda final} to show that  $|\Zr^{\beta_1}(G)| |\Zr^{\beta_2}(D)|\lesssim \varepsilon$.  Indeed, for $G=\Lr(c^{-1}\kappar)$, we only have $|\Zr^{\beta_1}(G)|\lesssim 1$. But {\color{black}in that case} we must have $D=\,^{(\mathring{Z})}{\pi}_{\Lr\Xr}$.Therefore, by the tables of deformation tensors in Section \ref{section:commutators and their  deformation tensors}, it is straightforward to check that $\big|\Zr^{\beta_2}\big(\,^{(\mathring{Z})}{\pi}_{\Lr\Xr}\big)\big|\lesssim \varepsilon$. Hence,  $|\Zr^{\beta_1}(G)| |\Zr^{\beta_2}(D)|\lesssim \varepsilon$; For $D=c\kappar^{-2}\,^{(\mathring{Z})}{\pi}_{\Lbr\Lbr}$, we only have $|\Zr^{\beta_1}(D)|\lesssim 1$. But {\color{black}in that case} we must have $D= c^{-1}\chibr$ or $c^{-2}z$. Therefore, $\big|\Zr^{\beta_2}(D)\big|\lesssim \varepsilon$. Hence,  $|\Zr^{\beta_1}(G)| |\Zr^{\beta_2}(D)|\lesssim \varepsilon$. The other cases are much easier and they can derived in the same manner. As a conclusion, we have
\begin{align*}
\left|\Zr^{\beta}\left(G \times D\times W\right)\right|&\lesssim \mathring{M}\varepsilon\sum_{\beta_3 \leqslant \beta}\left(|\Zr^{\beta_3}(\Lr(\Psi_{m}))|+|\Zr^{\beta_3}(\Xr(\Psi_{m}))|\right).
\end{align*}
Since $\Xr$ commutes with $\Zr^{\beta_3}$, the contribution of $\Zr^{\beta_3}(\Xr(\Psi_{m}))$ to $\mathscr{N}_n(t,u)$ and $\underline{\mathscr{N}}_n(t,u)$ can be bounded similarly as in \eqref{ineq: case 1}:
\begin{equation}\label{bound auxxx 1}
\mathscr{N}_n(t,u)+\underline{\mathscr{N}}_n(t,u)\lesssim \mathring{M}\varepsilon\sum_{\beta_3 \leqslant \beta}\int_{\mathcal{D}(t,u)}  |\Xr \Zr^{\beta_3} (\Psi_{m})|\big(\big|\Lh\Psi_n\big|+\big|\Lb\Psi_n\big|\big)\lesssim \mathring{M}\varepsilon^3 t^2.
\end{equation}
It remains to bound the contribution from $\Zr^{\beta_3}(\Lr(\Psi_{m}))$, i.e.,
\begin{equation}\label{eq:sigma 1 1 aux 1}
\mathscr{N}_n(t,u)+\underline{\mathscr{N}}_n(t,u)\lesssim\mathring{M} \varepsilon\sum_{\beta_3 \leqslant \beta}\int_{\mathcal{D}(t,u)}  |\Zr^{\beta_3}(\Lr(\Psi_{m}))| \big(\big|\Lh\Psi_n\big|+\big|\Lb\Psi_n\big|\big).
\end{equation}
We notice that $|\beta_3|+m\leqslant \Ntop-1$.  We apply \eqref{eq: commutation formular for Lr Zr} to bound the righthand side of \eqref{eq:sigma 1 1 aux 1} by
\begin{align*}
&  \mathring{M} \varepsilon\int_{\mathcal{D}(t,u)} \big(\big|\Lr \Zr^{\beta_3}\big( \Psi_m\big)\big| +\sum_{\substack{|\alpha_1|+|\alpha_2|\leqslant |\beta_3|\\ |\alpha_2|\geqslant 1}}\big|\Zr^{\alpha_1}(\lambda)\big| \big|\Zr^{\alpha_2}(\Psi_m)\big|\big) \big(\big|\Lh\Psi_n\big|+\big|\Lb\Psi_n\big|\big)\\
\lesssim& \mathring{M}\varepsilon^3t^2+\mathring{M} \varepsilon\sum_{\substack{|\alpha_1|+|\alpha_2|\leqslant |\beta_3|\\ |\alpha_2|\geqslant 1}}\int_{\mathcal{D}(t,u)}\underbrace{ \big|\Zr^{\alpha_1}(\lambda)\big| \big|\Zr^{\alpha_2}(\Psi_m)\big| \big(\big|\Lh\Psi_n\big|+\big|\Lb\Psi_n\big|\big)}_{\mathbf{I}_{\alpha_1,\alpha_2}},
\end{align*}
where $\lambda \in \{\yr, \zr, \chir,\etar\}$. We have used \eqref{ineq: estimates for L2 L3} and $\mathbf{(B_2)}$ in the last step. To deal with $\mathbf{I}_{\alpha_1,\alpha_2}$, we can proceed exactly as for case (b-2) of Section \ref{subsection top order source term Lb}. Together with \eqref{bound auxxx 1} and  this finally leads to
\[\mathscr{N}_n(t,u)+\underline{\mathscr{N}}_n(t,u)\lesssim \mathring{M}\varepsilon^3t^2.
\]

\item[(b)] $|\beta_2|\geqslant \Ninf$. 

Similar to case (a), by \eqref{ineq: L infty bound}, \eqref{ineq: L infty bound on lambda final} and \eqref{ineq: L infty bound for Zbeta L Zalpha}, we have $|\Zr^{\beta_1}(G)| |\Zr^{\beta_3}(W)|\lesssim \mathring{M}\varepsilon$. Since $D\in \{\,^{(\mathring{Z})}{\pi}_{\Lr\Xr}, c\kappar^{-2}\,^{(\mathring{Z})}{\pi}_{\Lbr\Lbr} \}$ where $\Zr =\Xr$ or $\Tr$, it is straightforward to check that 
\[D\in \{  y-2\Xr(c),  z-2\Tr(c), \Xr(\psi_2), \Tr(\psi_2)\}.\] 
Therefore, schematically, we have
$\Zr^{\beta_2}(D)=\mathring{Z}(\Zr^{\beta_2'}(\Psi_0))=\mathring{Z}(\Psi_{|\beta_2|})$. Hence, the contribution of those terms in ${\mathscr{N}}(t,u)$ and $\underline{\mathscr{N}}(t,u)$ can bounded as follows
\begin{align*}
{\mathscr{N}}_n(t,u)+\underline{\mathscr{N}}_n(t,u)&\lesssim   \mathring{M}\varepsilon \int_{\mathcal{D}(t,u)}\big(|\Xr(\Psi_{|\beta_2|})+|\Tr(\Psi_{|\beta_2|}) \big)\left(|\Lh\Psi_n|+|\Lb\Psi_n|\right)\lesssim  \mathring{M}\varepsilon^3t^2.
\end{align*}
where we proceeded exactly as in \eqref{section 9.2 aux 1} or \eqref{section 9.2 aux 2} and we also used \eqref{ineq: estimates for L2 L3} and \eqref{comparison bound 1: L 2}.

\item[(c)] $|\beta_1|\geqslant \Ninf$. 

 Since $\Lr(c^{-1}\kappar)=c^{-1}-c^{-2}\kappar \Lr(c)$, in view of \eqref{sigma 1 1 G D W}, we may assume that $G= c^{-1}$, $c^{-2}\kappar \Lr(c)$, $\chibr$,  $c^{-1}\chibr$,  $c^{-1}\kappar \zr$ or $c^{-2}\kappar \zr$.

If $G=c^{-1}$, the corresponding terms in $\mathscr{N}_n(t,u)$ and $\underline{\mathscr{N}}_n(t,u)$ are bounded by
\[\int_{\mathcal{D}(t,u)}\big|\Zr^{\beta_1}(c^{-1})\big|\big|\Zr^{\beta_2}\big(\pi_{\Lr\Xr}\big)\big|\big|\Zr^{\beta_3}\big(\Xr(\Psi_{m})\big)\big| \big(|\Lh\Psi_n|+|\Lb\Psi_n|\big){\color{black}.}\]
We can bound this term by $\mathring{M}\varepsilon^3t^2$ exactly in the same manner as for the $\mathbf{S}_2$ terms in  case (b-2) of Section \ref{subsection top order source term Lb}.

If $G\neq c^{-1}$, it can be written as $c^{-k}\kappar \mathring{G}$ where $k=1,2$ and $\mathring{G}=\Xr(\psi_2), \zr$ or $\Lr(c)$. Thus, by \eqref{comparison bound 1: L 2},
\begin{equation}\label{auxxx x}
\|\Zr^{\beta_1}G\|_{L^2(\Sigma_\tau)}\leqslant \sum_{\beta_1'+\beta_2''=\beta_2}\|\Zr^{\beta_1}(c^{-k})\cdot \Zr^{\beta_1}\mathring{G}\|_{L^2(\Sigma_\tau)}\lesssim \mathring{M}\varepsilon t.
\end{equation}
On the other hand, similar to case (a), by \eqref{ineq: L infty bound} and  \eqref{ineq: L infty bound for Zbeta L Zalpha} we have $|\Zr^{\beta_2}(D)||\Zr^{\beta_3}(W)|\lesssim \mathring{M}\varepsilon$. Therefore, we can apply $\mathbf{(B_2)}$ to derive
\begin{align*}
{\mathscr{N}}_n(t,u)+\underline{\mathscr{N}}_n(t,u)&\lesssim   \mathring{M}\varepsilon \int_{\delta}^t\|\Zr^{\beta_1}(G)\|_{L^2(\Sigma_\tau)} \big(\|\Lh\Psi_n\|_{L^2(\Sigma_\tau)}+\|\Lb\Psi_n\|_{L^2(\Sigma_\tau)}\big)d\tau\lesssim \mathring{M}\varepsilon^3t^2.
\end{align*}
\end{itemize}
By combining all the above estimates, the total contribution of $\sigma_{1,1}$ and $\sigma_{1,2}$ in $\mathscr{N}_n(t,u)$ and $\underline{\mathscr{N}}_n(t,u)$ are bounded as
\[\mathscr{N}_n(t,u)+\underline{\mathscr{N}}_n(t,u)\lesssim \mathring{M}\varepsilon^3t^2.\]

\subsubsection{Estimates  on $\sigma_{1,3}$}

A direct computation shows
\[\sigma_{1,3}=\frac{1}{4}\left(-c^{-2}\Lr(c)+c^{-1}\chir-c^{-1}\zr\right)\cdot c^{-1}\pi_{\Lr\Lr}\cdot \Lbr(\Psi_{m}).\]
For $Z=\Xr$ or $\Tr$, by the tables in Section \ref{section:commutators and their  deformation tensors}, $c^{-1}\,^{(Z)}\pi_{\Lr\Lr}=-2y$ or $-2z$. Hence, the terms in $\sigma_{1,3}$ can be schematically written as $G \times D\times \Lbr(\Psi_{m})$ with
\[G\in \{-c^{-2}\Lr(c), c^{-1}\chir,c^{-1}\zr\}, \ \ D\in \{y, z \}.\]
Thus,
\[\big|\Zr^{\beta}\big( G\times D\times  \Lbr({\color{black}\Psi_{m}})\big)\big|\lesssim \sum_{\beta_1+\beta_2+\beta_3=\beta}\big|\Zr^{\beta_1}(G)\big| \big|\Zr^{\beta_2}(D)\big|\big|\Zr^{\beta_3}\big( \Lbr(\Psi_{m})\big)\big|.\]
It suffices to consider the following three cases:
\begin{itemize}[noitemsep,wide=0pt, leftmargin=\dimexpr\labelwidth + 2\labelsep\relax]

\item[(i)] $|\beta_1|\leqslant \Ninf-1$ and $|\beta_2|\leqslant \Ninf-1$. 

By \eqref{ineq: L infty bound} and \eqref{ineq: L infty bound on lambda final}, {\color{black}we have} $|\Zr^{\beta_1}(G)| |\Zr^{\beta_2}(D)|\lesssim\mathring{M}\varepsilon$. Hence,
\begin{align*}
\mathscr{N}_n(t,u)+\underline{\mathscr{N}}_n(t,u)&\lesssim \mathring{M}\varepsilon\sum_{\beta_3 \leqslant \beta}\int_{\mathcal{D}(t,u)} \big|\Zr^{\beta_3}(\Lbr(\Psi_{m}))\big| \big(\big|\Lh\Psi_n\big|+\big|\Lb\Psi_n\big|\big)\\
&\lesssim \mathring{M}\varepsilon\sum_{\beta_3 \leqslant \beta}\int_{\mathcal{D}(t,u)} \big(\kappa|\Xr\Zr^{\beta_3}(\Psi_{m})|+|\Tr\Zr^{\beta_3}(\Psi_{m})| \big)\big(\big|\Lh\Psi_n\big|+\big|\Lb\Psi_n\big|\big),
\end{align*}
where we used \eqref{Euler equations:form 4 with derivatives} to replace $\Zr^{\beta_3}(\Lbr(\Psi_{m}))$ in the last step. Similar to  \eqref{section 9.2 aux 2}, we then use \eqref{comparison bound 1: L infty} to replace $\Tr$ and $\Xr$ derivatives by $L,\Lb$ and $\Xh$ derivatives. This shows that
\begin{align*}
\mathscr{N}_n(t,u)+\underline{\mathscr{N}}_n(t,u)&\lesssim \mathring{M}\varepsilon^3t^2.
\end{align*}

\item[(ii)] $|\beta_2|\geqslant \Ninf$. 

By \eqref{ineq: L infty bound}, \eqref{ineq: L infty bound on lambda final} and  \eqref{Euler equations:form 4 with derivatives}, we have  $|\Zr^{\beta_1}(G)|\lesssim \mathring{M} \varepsilon$ and $|\Zr^{\beta_3}(\Lbr(\Psi_{m}))|\lesssim 1$.  Since $D\in \{y, z \}$, we write $Z^{\beta_2}(D)=\mathring{Z}(\Zr^{\beta_2'}(\Psi_0))=\mathring{Z}(\Psi_{|\beta_2|})$ with $\mathring{Z}=\Tr$ or $\Xr$. Hence, similar to \eqref{type I S1 aux 1} or \eqref{section 9.2 aux 1}, we have
\begin{align*}
{\mathscr{N}}_n(t,u)+\underline{\mathscr{N}}_n(t,u)&\lesssim   \mathring{M}\varepsilon \int_{\mathcal{D}(t,u)}|\mathring{Z}(\Psi_{|\beta_2|}) \big(|\Lh\Psi_n|+|\Lb\Psi_n|\big)\lesssim \varepsilon^3t^2.
\end{align*}

\item[(iii)] $|\beta_1|\geqslant \Ninf$. 

According to \eqref{ineq: L infty bound}, \eqref{ineq: L infty bound on lambda final} and  \eqref{Euler equations:form 4 with derivatives}, we have $|\Zr^{\beta_3}(\Lbr(\Psi_{m}))|\lesssim 1$ and $|\Zr^{\beta_2}(D)|\lesssim \varepsilon \kappar$. Therefore,
\begin{align*}
{\mathscr{N}}_n(t,u)+\underline{\mathscr{N}}_n(t,u)&\lesssim   \mathring{M}\varepsilon \int_{\mathcal{D}(t,u)}|{\Zr}^{\beta_1}\big(\kappar \cdot G\big) |\big(|\Lh\Psi_n|+|\Lb\Psi_n|\big).
\end{align*}
Since $G\in \{-c^{-2}\Lr(c), c^{-1}\chir,c^{-1}\zr\}$, similar to \eqref{auxxx x}, we have $\|{\Zr}^{\beta_1}\left(\kappar \cdot G\right) \|_{L^2}\lesssim \mathring{M}\varepsilon$. Hence,
\begin{align*}
{\mathscr{N}}_n(t,u)+\underline{\mathscr{N}}_n(t,u)&\lesssim \mathring{M}\varepsilon \int_{\delta}^t\|{\Zr}^{\beta_1}\big(\kappar \cdot G\big) \|_{L^2(\Sigma_\tau)}\big(\|\Lh\Psi_n\|_{L^2(\Sigma_\tau)}+\|\Lb\Psi_n\|_{L^2(\Sigma_\tau)}\big)d\tau\lesssim  \mathring{M}\varepsilon^3t^2.
\end{align*}
\end{itemize}
Therefore, the total contribution of $\sigma_{1,3}$ in $\mathscr{N}_n(t,u)$ and $\underline{\mathscr{N}}_n(t,u)$ {\color{black}is bounded} as
\[\mathscr{N}_n(t,u)+\underline{\mathscr{N}}_n(t,u)\lesssim \mathring{M}\varepsilon^3t^2.\]

\subsubsection{Summary}

Combining the estimates for $\sigma_{1,1}, \sigma_{1,2}$ and $\sigma_{1,3}$, the error terms of {\bf Type $\mathbf{II}_1$} can be bounded as follows:
\[\mathscr{N}_n(t,u)+\underline{\mathscr{N}}_n(t,u)\lesssim \mathring{M}\varepsilon^3t^2.\]

\subsection{Estimates on {\bf Type $\mathbf{II}_2$} terms}

For {\bf Type $\mathbf{II}_2$} terms, since $\frac{\mu}{\mur}\lesssim 1$, it suffices to bound  the following integrals:
\[\mathscr{N}_n(t,u)= \int_{\mathcal{D}(t,u)}  \big| \Zr^\beta \big({}^{(\mathring{Z})} \sigma_{2} \big) \big|  \big|\Lh\Psi_n \big|,  \ \  \underline{\mathscr{N}}_n(t,u)= \int_{\mathcal{D}(t,u)}   \big| \Zr^\beta \big({}^{(\mathring{Z})} \sigma_{2} \big) \big|  \big|\Lb\Psi_n \big|,\]
where $|\beta|\leqslant n-1$. We can rewrite \eqref{eq: sigma 2 circle} as
\begin{equation*}
\begin{split}
{}^{(\Zr)} \sigma_{2}&=-\frac{1}{2} \underbrace{\big(\pi_{\Lbr \Xr} \cdot \Lr \Xr(\Psi_{m})+\pi_{\Lr \Xr} \cdot \Lbr \Xr(\Psi_{m})+ \pi_{\Lbr \Xr} \cdot \Xr \Lr(\Psi_{m})+\pi_{\Lr \Xr} \cdot \Xr \Lbr(\Psi_{m})\big)}_{\sigma_{2,1}},\\
&\ \ +\frac{1}{2} \underbrace{\pi_{\Lr \Lbr} \cdot \Xr \Xr(\Psi_{m})}_{\sigma_{2,2}}+\underbrace{\frac{1}{4\mur}  \pi_{\Lbr \Lbr} \Lr \Lr(\Psi_{m})+\frac{1}{4\mur} \pi_{\Lr \Lr} \Lbr \Lbr(\Psi_{m})}_{\sigma_{2,3}},
\end{split}
\end{equation*}
where  ${\pi}$ stands for $\,^{(\mathring{Z})}{\pi}$. In view of \eqref{eq: varphor_n expression}, we have $|\beta|+m+1\leqslant \Ntop$.

\subsubsection{Estimates on $\sigma_{2,1}$}\label{section sigma21}
By the tables in Section \ref{section:commutators and their  deformation tensors},  for $\Zr=\Xr$ or $\Tr$, we have $\,^{(\mathring{Z})}{\pi}_{\Lbr\Xr}= c^{-1}\kappar\,^{(\mathring{Z})}{\pi}_{\Lr\Xr}$. Thus, we can replace $\Lbr$ by $c^{-1}\kappar\Lr+2\Tr$ to derive
\begin{align*}
\sigma_{2,1}&=-\pi_{\Lbr \Xr}\big({\color{black}\Lr \Xr(\Psi_{m})+ \Xr\Lr(\Psi_{m})}\big)-2\pi_{\Lr \Xr}\Xr\Tr(\Psi_{m})-\frac{1}{2}\pi_{\Lr \Xr}\Xr(c^{-1})\kappar \Lr(\Psi_{m})\\
&\stackrel{\Xr\Lr=[\Xr,\Lr]+\Lr\Xr}{=}-\pi_{\Lbr \Xr}\big(2\Lr \Xr(\Psi_{m})+\chir \Xr(\Psi_{m})-\yr \Tr(\Psi_{m})\big)-2\pi_{\Lr \Xr}\Xr\Tr(\Psi_{m})-\frac{1}{2}\pi_{\Lr \Xr}\kappar\Xr(c^{-1}) \Lr(\Psi_{m})\\
&=-\pi_{\Lbr \Xr}\big(2\Lr \Xr(\Psi_{m})+\chir \Xr(\Psi_{m})\big)+\pi_{\Lr \Xr}\big(c^{-1} y \Tr(\Psi_{m})-2\Xr\Tr(\Psi_{m})\big)-\frac{1}{2}\pi_{\Lr \Xr}\kappar\Xr(c^{-1}) \Lr(\Psi_{m}).
\end{align*}
The terms in $\sigma_{2,1}$  can be schematically represented as $D\times W$ with
\[D\in \{\,^{(\mathring{Z})} \pi_{\Lbr \Xr} ,\,^{(\mathring{Z})} \pi_{\Lr \Xr}\},  \ \ W\in \{\Lr\Xr(\Psi_{m}), \chir \Xr(\Psi_{m}), c^{-1} y\Tr(\Psi_{m}),\Xr\Tr(\Psi_{m}),\kappar\Xr(c^{-1}) \Lr(\Psi_{m})\}.
\]
We show that for all possible $F=D$ or $W$, for all multi-index $\alpha$, we have
\begin{equation}\label{eq: auxxx bounds}
\begin{cases}
&\|\Zr^{\alpha}(F)\|_{L^2(\Sigma_t)}\lesssim \mathring{M}\varepsilon, \ \ \ {\rm ord}\left(\Zr^{\alpha}(F)\right)\leqslant \Ntop+1;\\
&\|\Zr^{\alpha}(F)\|_{L^\infty(\Sigma_t)}\lesssim \mathring{M}\varepsilon, \ \ \  {\rm ord}\left(\Zr^{\alpha}(F)\right)\leqslant \Ninf.
\end{cases}
\end{equation}
{\color{black}We} check case by case to prove \eqref{eq: auxxx bounds}:
\begin{itemize}
\item $F=\chir \Xr(\Psi_{m})$. 

In this case, $F=-\Xr(\psi_2)\Xr(\Psi_{m})$. Hence, 
\[\Zr^{\alpha}F=\sum_{\alpha_1+\alpha_2=\alpha}\Zr^{\alpha_1}(\Xr(\psi_2))\Zr^{\alpha_2}(\Xr(\Psi_m)){\color{black}.}\]
Therefore,  $\eqref{eq: auxxx bounds}$ follows immediately from \eqref{comparison bound 1: L 2},\eqref{ineq: L infty bound} and \eqref{ineq: L 2 bound for Zbeta L Zalpha}.

\item $F=\Lr\Xr(\Psi_{m}), \Xr\Tr(\Psi_{m})$ or $\,^{(\mathring{Z})}\pi_{\Lr \Xr}$. 

We recall that $\,^{(\mathring{X})}{\pi}_{\Lr\Xr}=-\chir=\Xr(\psi_2)$ and  $\,^{(\mathring{T})}{\pi}_{\Lr\Xr}=-\etar=\Tr(\psi_2)$. Therefore, \eqref{eq: auxxx bounds} is a direct consequence of \eqref{comparison bound 1: L 2},\eqref{ineq: L infty bound} and \eqref{ineq: L 2 bound for Zbeta L Zalpha}.

\item $F= \,^{(\mathring{X})}{\pi}_{\Lbr\Xr}$ or  $\kappar\Xr(c^{-1}) \Lr(\Psi_{m})$. 

Because $\,^{(\mathring{X})}{\pi}_{\Lbr\Xr}=c^{-1}\kappar \Xr(\psi_2)$ and  $\,^{(\mathring{T})}{\pi}_{\Lbr\Xr}=c^{-1}\kappar\Tr(\psi_2)$, $F$ can be written schematically as $\kappar^a\Zr^{\alpha}(c^{-1})E$, where $(a,|\alpha|)\in \{0,1\}$ and $F'=\Xr(\Psi_{m}), \Xr(\psi_2), \Tr(\psi_2)$ or $\Lr(\Psi_{n-1})$. According to \eqref{comparison bound 1: L 2}, \eqref{ineq: L infty bound} and \eqref{ineq: L 2 bound for Zbeta L Zalpha}, $F'$ satisfies \eqref{eq: auxxx bounds}. Therefore, we can use Remark \ref{remark:techical nonlinear} to conclude that $F$ also satisfies  \eqref{eq: auxxx bounds}. 

\item $F=c^{-1} y\Tr(\Psi_{m})$.

The worst scenario for $\Psi_m$ is that $\Psi_{m}=\wb$ because the $L^\infty$ or $L^2$ estimates of $\Tr(\Psi_{m})$ is only bounded by a universal constant. On the other hand, $y=\Xr(-\psi_1+c)$, it satisfies \eqref{eq: auxxx bounds}. Therefore, by  Remark \ref{remark:techical nonlinear}, $F$ also satisfies  \eqref{eq: auxxx bounds}. 
\end{itemize}

According to \eqref{eq: auxxx bounds} and Remark \ref{remark:techical nonlinear}, in view of $\left|\Zr^{\beta}\big(  D\times W\big)\right|\lesssim \sum_{\beta_1+\beta_2=\beta} |\Zr^{\beta_1}(D)||\Zr^{\beta_2}(W)|$, we conclude that
\begin{equation}\label{eq: auxxx bounds 111}
\|\Zr^{\beta}(  D\times W)\|_{L^2(\Sigma_t)} \lesssim \mathring{M}\varepsilon^2.
\end{equation}

The extra $\varepsilon$ in \eqref{eq: auxxx bounds 111} shows that the  contribution of $\sigma_{2,1}$ in $\mathscr{N}_n(t,u)$ and $\underline{\mathscr{N}}_n(t,u)$ can be bounded as follows:
\begin{align*}
{\mathscr{N}}_n(t,u)+\underline{\mathscr{N}}_n(t,u)&\lesssim  \int_{\delta}^t \|\Zr^{\beta}(  D\times W)\|_{L^2(\Sigma_\tau)} \left(\|\Lh\Psi_n\|_{L^2(\Sigma_\tau)}+\|\Lb\Psi_n\|_{L^2(\Sigma_\tau)}\right)d\tau\lesssim \mathring{M}\varepsilon^3t^2.
\end{align*}

\subsubsection{Estimates on $\sigma_{2,2}$}\label{section sigma22}
By the tables in Section \ref{section:commutators and their  deformation tensors}, we have $\,^{(\Zr)}\pi_{\Lr \Lbr}=-2\kappar \Zr(c)$. Hence,
\[\big|\Zr^{\beta}\big( \sigma_{2,2}\big)\big|\lesssim \sum_{\beta_1+\beta_2=\beta} \kappar\big|\Zr^{\beta_1} \mathring{Z}(c)\big||\Xr^2\Zr^{\beta_2}(\Psi_{m})|.\]
Unless $|\alpha|=0$ or $\mathring{Z}=\Tr$, for $F=\mathring{Z}(c)$ or $\Xr^2(\Psi_{m})$,  just as for \eqref{eq: auxxx bounds}, it is straightforward to see that
\begin{equation*}
\begin{cases}
&\|\Zr^{\alpha}(F)\|_{L^2(\Sigma_t)}\lesssim \mathring{M}\varepsilon, \ \ \ {\rm ord}\left(\Zr^{\alpha}(F)\right)\leqslant \Ntop+1;\\
&\|\Zr^{\alpha}(F)\|_{L^\infty(\Sigma_t)}\lesssim \mathring{M}\varepsilon, \ \ \  {\rm ord}\left(\Zr^{\alpha}(F)\right)\leqslant \Ninf.
\end{cases}
\end{equation*}
Therefore, similar to\eqref{eq: auxxx bounds 111}, unless $|\beta_1|=0$ and $\mathring{Z}=\Tr$,  we have
\[\|\kappar \cdot \Zr^{\beta_1} \mathring{Z}(c)\cdot \Xr^2\Zr^{\beta_2}({\color{black}\Psi_{m}})\|_{L^2(\Sigma_t)} \lesssim \mathring{M}\varepsilon^2.\]
The corresponding contribution in $\sigma_{2,1}$ in $\mathscr{N}_n(t,u)$ and $\underline{\mathscr{N}}_n(t,u)$ can be bounded by $\mathring{M}\varepsilon^3t^2$. It remains to treat the case where $|\beta_1|=0$ and $\mathring{Z}=\Tr$. In fact, we have
\begin{align*}
{\mathscr{N}}_n(t,u)+\underline{\mathscr{N}}_n(t,u)&\lesssim  \int_{\mathcal{D}(t,u)}  \kappar\big| \mathring{T}(c)\big||\Xr^2\Zr^{\beta}({\color{black}\Psi_{m}})|\big(\big|\Lh\Psi_n\big|+\big|\Lb\Psi_n\big|\big)\\
&\lesssim \int_{\delta}^t \mathscr{E}_{\leqslant n}(\tau,u)d\tau.
\end{align*}

{\color{black}Combining} all the estimates, the contribution of $\sigma_{2,2}$ are bounded as follows:
\begin{align*}
{\mathscr{N}}_n(t,u)+\underline{\mathscr{N}}_n(t,u)&\lesssim \mathring{M}\varepsilon^3t^2+\int_{\delta}^t \mathscr{E}_{\leqslant n}(\tau,u)d\tau.
\end{align*}

\subsubsection{Estimates  on $\sigma_{2,3}$}\label{section sigma23}

The term $\sigma_{2,3}$ is much harder than the previous terms due to the presence of $\Lr^2(\Psi_m)$. We can use $\Lbr= c^{-1} \kappar \Lr + 2\Tr$ to expand $\Lbr \Lbr$ in terms of $\Lr$ and $\Tr$. This gives
\begin{align*}
\pi_{\Lbr \Lbr} \Lr \Lr(\Psi_{m})+\pi_{\Lr \Lr} \Lbr \Lbr(\Psi_{m})=&\big(\pi_{\Lbr \Lbr}+c^{-2}\kappar^2 \pi_{\Lr\Lr}\big) \Lr \Lr(\Psi_{m})+2c^{-1}\kappar \pi_{\Lr \Lr} \big(\Lr \Tr (\Psi_{m})+
\Tr \Lr(\Psi_{m})\big)\\
&+4\pi_{\Lr \Lr}\Tr\Tr(\Psi_{m})+\Lbr(c^{-1}\kappar )\pi_{\Lr \Lr}\Lr(\Psi_{m}).
\end{align*}
For $\mathring{Z}\in \mathring{\mathscr{Z}}$, by the tables in Section \ref{section:commutators and their  deformation tensors}, we have $\,^{(\mathring{Z})}\pi_{\Lbr \Lbr}+c^{-2}\kappar^2 \,^{(\mathring{Z})}\pi_{\Lr\Lr}=-4c^{-2}\kappar \mur \mathring{Z}(c)$. Therefore, we can decompose $\sigma_{2,3}$ as $\sigma'_{2,3}+\sigma''_{2,3}$:
\begin{align*}
-\underbrace{c^{-2}\kappar \mathring{Z}(c)\Lr^2(\Psi_{m})}_{\sigma'_{2,3}}
 \ \ +\underbrace{\frac{1}{2}c^{-2} \pi_{\Lr \Lr}  \big(\Lr \Tr (\Psi_{m})+\Tr \Lr(\Psi_{m})\big)+c^{-1}\pi_{\Lr \Lr} \frac{\Tr^2(\Psi_{m})}{\kappar}+\frac{1}{4}\frac{\Lbr(c^{-1}\kappar )}{\kappar}  c^{-1}\pi_{\Lr \Lr}\Lr(\Psi_{m})}_{\sigma''_{2,3}}.
\end{align*}
The terms in $\sigma''_{2,3}$  can be schematically represented as  $C\times D\times W$ where 
\[C\in \big\{1, \frac{\Lbr(c^{-1}\kappar )}{\kappar}\big\}, \ \ D\in \big\{c^{-a}\,^{(\mathring{Z})} \pi_{\Lr \Lr} | a=1,2\big\},  \ \ W\in \big\{\Lr\Tr(\Psi_{m}),\Tr\Lr(\Psi_{m}),\frac{\Tr^2(\Psi_{m})}{\kappar}, \Lr(\Psi_{m})\big\}.
\]
We prove that for all $F= D$ or $W$, for all multi-index $\alpha$, we have
\begin{equation}\label{eq:111}
\begin{cases}
&\|\Zr^{\alpha}(F)\|_{L^2(\Sigma_t)}\lesssim \mathring{M}\varepsilon, \ \ \ {\rm ord} \big(\Zr^{\alpha}(F)\big)\leqslant \Ntop+1;\\
&\|\Zr^{\alpha}(F)\|_{L^\infty(\Sigma_t)}\lesssim \mathring{M}\varepsilon, \ \ \  {\rm ord}\big(\Zr^{\alpha}(F)\big)\leqslant \Ninf.
\end{cases}
\end{equation}
\begin{remark}
If $F$ satisfies \eqref{eq:111}, then $c^{-a}F$ also satisfies \eqref{eq:111},  where $a=0, \pm 1$. This is direct from \eqref{eq:111} and Remark \ref{remark:techical nonlinear}.
\end{remark}
We check case by case to prove \eqref{eq:111} as follows: 
\begin{itemize}
\item $F=\Lr\Tr(\Psi_{m}),\Lr\Tr(\Psi_{m}),\Lr(\Psi_{m})$ or $c^{-a}\,^{(\mathring{Z})} \pi_{\Lr \Lr}$.

We recall that $\,^{(\mathring{X})} \pi_{\Lr \Lr}=2c\Xr(-\psi_1+c)$ and $\,^{(\mathring{X})} \pi_{\Lr \Lr}=2c\left(1+\Tr(-\psi_1+c)\right)$. Therefore, \eqref{eq:111} follows from \eqref{bound on T(v^1+c)}, \eqref{ineq: L infty bound}, \eqref{ineq: L infty bound for Zbeta L Zalpha} and \eqref{ineq: L 2 bound for Zbeta L Zalpha}.

\item $F=\frac{\Tr^2(\Psi_{m})}{\kappar}$.

We have $\Zr^{\alpha}(F)=\frac{1}{\kappar}\Tr\Zr^\alpha\Tr(\Psi_{m})$. Hence, the $L^\infty$ bounds in \eqref{eq:111} is directly from \eqref{ineq: L infty bound}. On the other hand, by \eqref{comparison bound 1: L infty}, we have
\[|\Zr^{\alpha}(F)|\lesssim |L(\Zr^\alpha\Tr(\Psi_{m}))|+|\Xh(\Zr^\alpha\Tr(\Psi_{m}))|+\frac{1}{\kappar}|\Lb(\Zr^\alpha\Tr(\Psi_{m}))|.\]
Therefore, The  $L^2$ bound  in \eqref{eq:111} is a consequence of $\mathbf{(B_2)}$.
\end{itemize}
For $C=\frac{\Lbr(c^{-1}\kappar )}{\kappar}$, we write it as $C=2\Tr(c^{-1})+c^{-1}\kappar \Lr(c^{-1})+c^{-2}$.
Therefore, by the same argument for \eqref{eq: a bound on Zbeta c -1}, for multi-indices $\alpha$ and $\beta$ with $|\alpha|\leqslant \Ntop$ and $|\beta|\leqslant \Ninf-1$, we have
\begin{equation}\label{eq: auxxx bounds sigma 23 C}
 \|\Zr^{\alpha}(C)\|_{L^2(\Sigma_t)}+\|\Zr^{\beta}(C)\|_{L^\infty(\Sigma_t)}\lesssim 1.
 \end{equation}
By writing $\Zr^{\beta}(C\times  D\times W)$ as $\sum_{\beta_1+\beta_2+\beta_3=\beta} \Zr^{\beta_1}(C)\Zr^{\beta_2}(D) \Zr^{\beta_3}(W)$, we can use Remark \ref{remark:techical nonlinear}, \eqref{eq: auxxx bounds sigma 23 C} and \eqref{eq:111} to conclude that
\begin{equation}\label{eq: auxxx bounds 111  sigma 23}
\| \Zr^{\beta}(C\times  D\times W)\|_{L^2(\Sigma_t)} \lesssim \mathring{M}\varepsilon^2.
\end{equation}
Therefore, the contribution of $\sigma''_{2,3}$ in $\mathscr{N}_n(t,u)$ and $\underline{\mathscr{N}}_n(t,u)$ can be bounded by
\begin{align*}
\int_{\delta}^t \|\Zr^{\beta}(C\times  D\times W)\|_{L^2(\Sigma_\tau)} \big(\|\Lh\Psi_n\|_{L^2(\Sigma_\tau)}+\|\Lb\Psi_n\|_{L^2(\Sigma_\tau)}\big)\lesssim \mathring{M}\varepsilon^3t^2.
\end{align*}

\medskip

It remains to bound the most difficult term $\sigma'_{2,3}$. We split it into two terms:
\[\sigma'_{2,3}=c^{-2}\kappar \mathring{Z}(c)\Lr^2(\Psi_{m})=-\mathring{Z}(c^{-1})\underbrace{\Lr (\kappar\Lr(\Psi_{m}))}_{\sigma'_{2,3;1}}+\mathring{Z}(c^{-1}) \underbrace{\Lr(\Psi_{m}) }_{\sigma'_{2,3;2}}.\] 
In view of \eqref{eq: commutation formular for Lr Zr}, for $\Psi_{m}=\Zr^{\alpha'}( \psi)$ where $\psi \in \{w,\wb, \psi_2\}$ and $|\alpha'|=m$, we have
\[\Lr\Psi_{m} =\Zr^{\alpha'}(\Lr\psi)+\sum_{\alpha_1+\alpha_2=\alpha', \atop \\ |\alpha_1|\leqslant |\alpha|-1}\Zr^{\alpha_1}(\lambda)\Zr^{\alpha_2}(\psi).\]

Let $\Phi_{m+1}=\kappar\Lr\Psi_{m}$. We can use \eqref{Euler equations:form 3 with derivatives} to replace $\Lr\psi$ and we derive
\begin{equation}\label{eq: computation for Phi m+1}
\begin{split}
 \Phi_{m+1}&:=\kappar\Lr\Psi_{m} =\Zr^\alpha(\kappar\Lr\psi)+\sum_{\substack{\alpha_1+\alpha_2=\alpha'\\ |\alpha_1|\leqslant |\alpha'|-1}}\Zr^{\alpha_1}(\kappar\lambda)\Zr^{\alpha_2}(\psi)\\
&=\sum_{\alpha_1+\alpha_2=\alpha'} \big[\Zr^{\alpha_1}(c)\Tr(\Zr^{\alpha_2}(\psi'))+ \Zr^{\alpha_1}(c) \kappar\Xr(\Zr^{\alpha_2}(\psi''))\big]+\sum_{\substack{\alpha_1+\alpha_2=\alpha'\\ |\alpha_1|\leqslant |\alpha'|-1}}\Zr^{\alpha_1}(\kappar\lambda)\Zr^{\alpha_2}(\psi)\\
&=\sum_{\alpha_1+\alpha_2=\alpha'} \Zr^{\alpha_1}(c)\overline{Z}(\Zr^{\alpha_2}(\psi'))+\sum_{\substack{\alpha_1+\alpha_2=\alpha'\\ |\alpha_1|\leqslant |\alpha'|-1}}\Zr^{\alpha_1}(\kappar\lambda)\Zr^{\alpha_2}(\psi),
\end{split}
\end{equation}
where $\overline{Z}=\kappar\Xr$ or $\Tr$. According to \eqref{Euler equations:form 3 with derivatives}, we observe that if $\overline{Z}=\Tr$, then $\psi'\neq \wb$.  Hence,
\begin{equation}\label{eq: Phi m+1}
 \Zr^\gamma\big( \Phi_{m+1}\big)=\sum_{\alpha_1+\alpha_2=\alpha'+\gamma}  \Zr^{\alpha_1}(c)\overline{Z}(\Zr^{\alpha_2}(\psi'))+\sum_{\substack{\alpha_1+\alpha_2=\alpha'+\gamma\\ |\alpha_1|\leqslant |\alpha'|+|\gamma|-1}}\Zr^{\alpha_1}(\kappar\lambda)\Zr^{\alpha_2}(\psi).
\end{equation}

We claim that for $\Zr^\gamma\big( \Phi_{m+1}\big)$, we have
\begin{equation}\label{eq: auxxx bounds sigma 23aaaaa  sss}
\begin{cases}
&\|\Zr^\gamma\big( \Phi_{m+1}\big)\|_{L^2(\Sigma_t)}\lesssim \mathring{M}\varepsilon, \ \ |\alpha'|+|\gamma|\leqslant \Ntop;\\
&\|\Zr^{\gamma'}\big( \Phi_{m+1}\big)\|_{L^\infty(\Sigma_t)}\lesssim \mathring{M}\varepsilon, \ \ \  |\alpha'|+|\gamma'|\leqslant \Ninf-1.
\end{cases}
\end{equation}
We prove \eqref{eq: auxxx bounds sigma 23aaaaa  sss}  by {\color{black}checking each term} of the righthand side of \eqref{eq: Phi m+1}. Because $\psi'\in \{w,\psi_2\}$ for $\overline{Z}=\Tr$,  the terms in the first sum of  \eqref{eq: Phi m+1} are bounded by $\mathring{M}\varepsilon$ by $\mathbf{(B_2)}$ and \eqref{ineq: L infty bound}. For the terms in the second sum, the index restriction $|\alpha'|+|\gamma|\leqslant \Ntop$ implies the total order of $\lambda$ appearing in  \eqref{eq: Phi m+1} is at most $\Ntop$. Thus, we can apply \eqref{ineq: L^2 bound on lambda final}, \eqref{ineq: L infty bound on lambda final}, $\mathbf{(B_2)}$, \eqref{ineq: L infty bound} as well as Remark \ref{remark:techical nonlinear} to bound these terms. This completes the proof of \eqref{eq: auxxx bounds sigma 23aaaaa  sss}.

By \eqref{eq: commutation formular for Lr Zr} and \eqref{eq: Phi m+1}, we can further compute
\begin{align*}
\Zr^\beta\big(\sigma'_{2,3;1}\big)&=\Zr^\beta\big(\Lr\big(\Phi_{m+1}\big)\big)=  \Lr\Zr^\beta\big(\Phi_{m+1}\big)+\sum_{\substack{\beta_1+\beta_2=\beta \\ |\beta_1|\leqslant |\beta|-1}}\Zr^{\beta_1}(\lambda)\Zr^{\beta_2}\big(\Phi_{m+1}\big)\\
&=\sum_{\alpha_1+\alpha_2=\alpha'+\beta} \big[ \Lr\Zr^{\alpha_1}(c)\overline{Z}(\Zr^{\alpha_2}(\psi'))+\Zr^{\alpha_1}(c) \Lr\overline{Z}(\Zr^{\alpha_2}(\psi'))\big]\\
&\ \ +\sum_{\substack{\alpha_1+\alpha_2=\alpha'+\beta\\ |\alpha_1|\leqslant |\alpha'|+|\beta|-1}}\big[\Lr\Zr^{\alpha_1}(\kappar\lambda)\Zr^{\alpha_2}(\psi)+\Zr^{\alpha_1}(\kappar\lambda)\Lr\Zr^{\alpha_2}(\psi)\big]+\sum_{\substack{\beta_1+\beta_2=\beta \\ |\beta_1|\leqslant |\beta|-1}}\Zr^{\beta_1}(\lambda)\Zr^{\beta_2}\big(\Phi_{m+1}\big).
\end{align*}
In the last step, we proceeded exactly as for \eqref{eq: computation for Phi m+1}. By regrouping the above terms, we arrive at the following expression
\begin{align*}
\Zr^\beta\big(\sigma'_{2,3;1}\big)
&=\Lr\Zr^{\alpha'+\beta-1}(\kappar\lambda)\Zr(\psi)+\!\!\!\!\!\!\sum_{\substack{\alpha_1+\alpha_2=\alpha'+\beta, \\ |\alpha_1|\leqslant 1}}\Zr^{\alpha_1}(c) \Lr\overline{Z}(\Zr^{\alpha_2}(\psi')) +\!\!\!\!\!\!\sum_{\alpha_1+\alpha_2=\alpha'+\beta} \underbrace{\Lr\Zr^{\alpha_1}(c)\overline{Z}(\Zr^{\alpha_2}(\psi'))}_{\mathbf{Err}_1}\\
&+\sum_{\substack{\alpha_1+\alpha_2=\alpha'+\beta \\ |\alpha_1|\geqslant 2}} \underbrace{\Zr^{\alpha_1}(c) \Lr\overline{Z}(\Zr^{\alpha_2}(\psi'))}_{\mathbf{Err}_2} +\sum_{\substack{\alpha_1+\alpha_2=\alpha'+\beta\\ |\alpha_2|\geqslant 2}}\underbrace{\Lr\Zr^{\alpha_1}(\kappar\lambda)\Zr^{\alpha_2}(\psi)}_{\mathbf{Err}_3}\\
&+\sum_{\substack{\alpha_1+\alpha_2=\alpha'+\beta\\ |\alpha_1|\leqslant |\alpha'|+|\beta|-1}} \underbrace{\Zr^{\alpha_1}(\kappar\lambda)\Lr\Zr^{\alpha_2}(\psi) }_{\mathbf{Err}_4} +\sum_{\substack{\beta_1+\beta_2=\beta \\ |\beta_1|\leqslant |\beta|-1}}\underbrace{\Zr^{\beta_1}(\lambda)\Zr^{\beta_2}\big(\Phi_{m+1}\big)}_{\mathbf{Err}_5}{\color{black}.}
\end{align*}
For any $k\leqslant 5$, each single term in $\mathbf{Err}_k$ can be written as a product of two functions $F_1\cdot F_2$ in the obvious way. We apply $\mathbf{(B_2)}$, \eqref{ineq: L infty bound}, \eqref{ineq: L^2 bound on lambda final}, \eqref{ineq: L infty bound on lambda final}, \eqref{eq: auxxx bounds sigma 23aaaaa  sss} and Remark \ref{remark:techical nonlinear} to $F_1$ and $F_2$. This shows that, for $j=1$ and $2$, we have
\[\begin{cases}
&\| F_j\|_{L^2(\Sigma_t)}\lesssim \mathring{M}\varepsilon, \ \ |\alpha'|+|\beta|\leqslant \Ntop;\\
&\| F_j\|_{L^\infty(\Sigma_t)}\lesssim \mathring{M}\varepsilon, \ \ |\alpha'|+|\beta|\leqslant \Ninf-1.
\end{cases}
\]
Let $\mathbf{Err}=\sum_{1\leqslant k\leqslant 5} \mathbf{Err}_k$. The above discussion shows that
\begin{equation}\label{sigma' 231}
\Zr^\beta\big(\sigma'_{2,3;1}\big)=\Lr\Zr^{\alpha'+\beta-1}(\kappar\lambda)\Zr(\psi)+\!\!\!\!\!\!\sum_{\substack{\alpha_1+\alpha_2=\alpha'+\beta, \\ |\alpha_1|\leqslant 1}}\Zr^{\alpha_1}(c) \Lr\overline{Z}(\Zr^{\alpha_2}(\psi')) +\mathbf{Err},
\end{equation}
with
\begin{equation}\label{eq: auxxx bounds on Errors}
\begin{cases}
&\| \mathbf{Err}\|_{L^2(\Sigma_t)}\lesssim \mathring{M}\varepsilon^2, \ \ |\alpha'|+|\beta|\leqslant \Ntop;\\
&\| \mathbf{Err}\|_{L^\infty(\Sigma_t)}\lesssim \mathring{M}\varepsilon^2, \ \ |\alpha'|+|\beta|\leqslant \Ninf-1.
\end{cases}
\end{equation}

We come back to $\sigma'_{2,3}$. By definition, we have $\sigma'_{2,3}=-\mathring{Z}(c^{-1})\sigma'_{2,3;1}+ \mathring{Z}(c^{-1}) \sigma'_{2,3;2}$. Therefore, the contribution of $\sigma'_{2,3}$ to $\mathscr{N}_n(t,u)$ and $\underline{\mathscr{N}}_n(t,u)$ are bounded by
\begin{align*}
\int_{\mathcal{D}(t,u)}  \big| \Zr^\beta ({}\sigma'_{2,3} ) \big|  \big( |\Lh\Psi_n|+ |\Lb\Psi_n| \big)\leqslant \sum_{l=1,2}\sum_{\beta'+\beta''=\beta}\int_{\mathcal{D}(t,u)} \big|\Zr^{\beta'+1}(c^{-1}) \big| \big|\Zr^{\beta''}(\sigma'_{2,3;l}) \big| \big( |\Lh\Psi_n|+ |\Lb\Psi_n| \big).
\end{align*}
We first deal with $\sigma'_{2,3;2}$. By definition, $\sigma'_{2,3;2}=\Lr(\Psi_{m})$. Therefore, by \eqref{eq: commutation formular for Lr Zr}, we have 
\begin{align*}
\Zr^{\beta''}(\sigma'_{2,3;2}) &=\Lr\Zr^{\beta''}\Psi_{m}+\sum_{\substack{\beta''_1+\beta''_2=\beta''  \\ |\beta''_1|\leqslant |\beta''|-1}}\Zr^{\beta''_1}(\lambda)\Zr^{\beta''_2}(\Psi_{m})\\
&=\Lr\Zr^{\beta''}\Psi_{m}+\Zr^{\beta''-1}(\lambda)\Zr(\Psi_{m})+\sum_{\substack{\beta''_1+\beta''_2=\beta''  \\ |\beta''_2|\geqslant 2}}\underbrace{\Zr^{\beta''_1}(\lambda)\Zr^{\beta''_2}(\Psi_{m})}_{\mathbf{Err}_6}.
\end{align*}
Similar to the previously defined error terms $\mathbf{Err}_k$ with $1\leqslant k \leqslant 5$, $\mathbf{Err}_6$ enjoys the same estimates as \eqref{eq: auxxx bounds on Errors}. In view of $\mathbf{(B_2)}$, \eqref{ineq: L infty bound}, \eqref{ineq: L^2 bound on lambda final} and \eqref{ineq: L infty bound on lambda final}, unless $\Zr(\Psi_{m})=\Tr(\wb)$ (and this forces $\lambda=\yr$ or $\zr$ and $n=1$), the second term $\Zr^{\beta''-1}(\lambda)\Zr(\Psi_{m})$ also enjoys the same estimates as $\mathbf{Err}_6$ , i.e., \eqref{eq: auxxx bounds on Errors}.  Therefore, it suffices to regard $\Zr^{\beta''}(\sigma'_{2,3;2})$ as
\begin{equation}\label{eq:Z beta'' L Psi}
\Zr^{\beta''}(\sigma'_{2,3;2})  = \Lr\Zr^{\beta''}\Psi_{m}+\Zr^{\beta''-1}(\lambda)\Tr(\wb)+\mathbf{Err},
\end{equation}
where $\mathbf{Err}$ satisfies the estimates \eqref{eq: auxxx bounds on Errors}. We notice that $\lambda$ must be $\yr$ or $\zr$ in \eqref{eq:Z beta'' L Psi}.

Similarly, for the first term in \eqref{sigma' 231}, i.e., $\Lr\Zr^{\alpha+\beta-1}(\kappar\lambda)\Zr(\psi)$, unless $\psi=\wb$ and $\Zr=\Tr$ (this forces $\lambda = \yr$ or $\zr$),  it also enjoys the same estimates as $\mathbf{Err}$. It suffices to regard $\Zr^{\beta''}(\sigma'_{2,3;1})$ as
\[\Zr^{\beta''}(\sigma'_{2,3;1})=\sum_{\substack{\alpha_1+\alpha_2=\alpha'+\beta'', \\ |\alpha_1|\leqslant 1}}\Zr^{\alpha_1}(c) \Lr\overline{Z}(\Zr^{\alpha_2}(\psi'))+\Lr\Zr^{\alpha'+\beta-1}(\kappar\lambda)\Tr(\wb)+\mathbf{Err}.\]
Hence, we can bound $\int_{\mathcal{D}(t,u)} \big| \Zr^\beta({}\sigma'_{2,3})\big|\big(|\Lh\Psi_n|+ |\Lb\Psi_n|\big)$ by the sum of the following five terms:
\begin{align*}
\mathbf{A}_0&=\sum_{\beta'+\beta''=\beta}\int_{\mathcal{D}(t,u)} \big|\Zr^{\beta'+1}(c^{-1})\big|\cdot\mathbf{Err}\cdot\big(\big|\Lh\Psi_n\big|+ \big|\Lb\Psi_n\big|\big),\\
\mathbf{A}_1&=\sum_{\substack{\beta'+\beta''=\beta\\ \alpha_1+\alpha_2=\alpha'+\beta'', |\alpha_1|\leqslant 1}}\int_{\mathcal{D}(t,u)}  \big|\Zr^{\beta'+1}(c^{-1})\big|\big|\Zr^{\alpha_1}(c)\big|\big|\Lr\overline{Z}(\Zr^{\alpha_2}(\psi'))\big| \big(\big|\Lh\Psi_n\big|+ \big|\Lb\Psi_n\big|\big),\\
\mathbf{A}_2&=\sum_{\beta'+\beta''=\beta}\int_{\mathcal{D}(t,u)}  \big|\Zr^{\beta'+1}(c^{-1})\big|\big|\Lr\Zr^{\alpha+\beta-1}(\kappar\lambda)\big| \big|\Tr(\wb)\big|  \big(\big|\Lh\Psi_n\big|+ \big|\Lb\Psi_n\big|\big),\\
\mathbf{A}_3&=\sum_{\beta'+\beta''=\beta}\int_{\mathcal{D}(t,u)} \big|\Zr^{\beta'+1}(c^{-1})\big|\big|\Lr\Zr^{\beta''}(\Psi_{m})\big|\big(\big|\Lh\Psi_n\big|+ \big|\Lb\Psi_n\big|\big),\\
\mathbf{A}_4&=\sum_{\beta'+\beta''=\beta}\int_{\mathcal{D}(t,u)} \big|\Zr^{\beta'+1}(c^{-1})\big|\big|\Zr^{\beta''-1}(\lambda)\big| \big|\Tr(\wb)\big| \big(\big|\Lh\Psi_n\big|+ \big|\Lb\Psi_n\big|\big){\color{black}.}
\end{align*}

First of all, by $\big|\Tr(\wb)\big|\lesssim 1$, we can remove $\big|\Tr(\wb)\big|$ from $\mathbf{A}_2$ and $\mathbf{A}_4$. 

Next, we will remove the factor $\big|\Zr^{\beta'+1}(c^{-1})\big|$ from all $\mathbf{A}_i$'s. We also notice that the $\Zr^{\alpha_1}(c)$ term in $\mathbf{A}_1$ can also be removed in the same way. In fact, similar to \eqref{eq: a bound on Zbeta c -1}, $\Zr^{\beta'+1}(c^{-1})$ can be written as a linear combination of terms of the type $c^{-m'}\Zr^{\gamma_{1}}(c)\Zr^{\gamma_2}(c)\cdots \Zr^{\gamma_k}(c)$. Without loss of generality, let $|\gamma_1|=\max_{j\leqslant k}|\gamma_k|$. According to the size of $|\gamma_1|$, we have the following three cases:
\begin{itemize}
\item $|\gamma_1|\leqslant 1$. 

In this case, we have
\[c^{-m}\Zr^{\gamma_{1}}(c)\Zr^{\gamma_2}(c)\cdots \Zr^{\gamma_k}(c)=c^{-m}\Zr(c)\Zr(c)\cdots \Zr(c),\] 
where is bounded by $1$. Hence, we simply replace this term by $1$ in  $\mathbf{A}_i$'s.

\item $2\leqslant |\gamma_1|\leqslant \Ntop$. 

We can apply \eqref{ineq: L infty bound} to $c^{-1}$, $\Zr^{\gamma_{1}}(c),\cdots \Zr^{\gamma_k}(c)$ to derive
\[\| c^{-m'}\Zr^{\gamma_{1}}(c)\Zr^{\gamma_2}(c)\cdots \Zr^{\gamma_k}(c)\|_{L^\infty(\Sigma_t)}\lesssim \mathring{M}\varepsilon.\]
For sufficiently small $\varepsilon$, we can still replace this term by $1$ in  $\mathbf{A}_i$'s. 

\item $|\gamma_1|>\Ntop$. 

According to $\mathbf{(B_2)}$, \eqref{comparison bound 1: L 2} and \eqref{ineq: L infty bound}, we apply Remark \ref{remark:techical nonlinear} to $c^{-1}$, $\Zr^{\gamma_{1}}(c),\cdots \Zr^{\gamma_k}(c)$ to derive
\[ \|c^{-m'}\Zr^{\gamma_{1}}(c)\Zr^{\gamma_2}(c)\cdots \Zr^{\gamma_k}(c)\|_{L^2(\Sigma_t)}\lesssim \mathring{M}\varepsilon.\]
Since $|\gamma_1|>\Ntop$, the orders of $\Lr\overline{Z}(\Zr^{\alpha_2}(\psi'))$, $\Lr\Zr^{\alpha+\beta-1}(\kappar\lambda)$, $\Lr\Zr^{\beta''}(\Psi_{m})$ and $\Zr^{\beta''-1}(\lambda)$ are all less than $\Ntop$. In view of \eqref{ineq: L infty bound on lambda final} and \eqref{ineq: L infty bound for Zbeta L Zalpha}, the $L^\infty$ norm of these four functions are bounded by $\mathring{M}\varepsilon$. Therefore, in each of the $\mathbf{A}_i$'s, we can use $\mathring{M}\varepsilon$ to bound the terms involving $c$'s in $L^2(\Sigma_t)$, use $\mathring{M}\varepsilon$ to bound the terms $\mathbf{Err}$, $\Lr\overline{Z}(\Zr^{\alpha_2}(\psi'))$, $\Lr\Zr^{\alpha+\beta-1}(\kappar\lambda)$, $\Lr\Zr^{\beta''}(\Psi_{m})$ and $\Zr^{\beta''-1}(\lambda)$ in $L^\infty(\Sigma_1)$  and use the ansatz $\mathbf{(B_2)}$ to bound $\big|\Lh\Psi_n\big|+ \big|\Lb\Psi_n\big|$ in $L^2(\Sigma_t)$. As a conclusion, the corresponding contribution from the $\mathbf{A}_i$'s are bounded by $\mathring{M}\varepsilon^3t^2$. 
\end{itemize}
From the previous discussion, we conclude that
\begin{align*}
\int_{\mathcal{D}(t,u)}  \big| \Zr^\beta ({}\sigma'_{2,3} ) \big|  \big( |\Lh\Psi_n|+ |\Lb\Psi_n| \big)\lesssim \mathring{M}\varepsilon^3t^2+\sum_{j=0}^4 \mathbf{A'}_j,
\end{align*}
where $\mathbf{A'}_j$ are the $\mathbf{A}_j$'s without the terms of $c$:
\begin{align*}
&\mathbf{A}'_0=\sum_{\beta'+\beta''=\beta}\int_{\mathcal{D}(t,u)} \mathbf{Err}\cdot\big(\big|\Lh\Psi_n\big|+ \big|\Lb\Psi_n\big|\big),\ \ \mathbf{A}'_1=\!\!\!\!\!\!\!\!\!\!\!\!\sum_{\substack{\beta'+\beta''=\beta\\ \alpha_1+\alpha_2=\alpha'+\beta'', |\alpha_1|\leqslant 1}}\!\!\!\!\!\!\!\!\!\!\!\!\int_{\mathcal{D}(t,u)}    \big|\Lr\overline{Z}(\Zr^{\alpha_2}(\psi'))\big| \big(\big|\Lh\Psi_n\big|+ \big|\Lb\Psi_n\big|\big),\\
&\mathbf{A}'_2=\sum_{\beta'+\beta''=\beta}\int_{\mathcal{D}(t,u)}  \big|\Lr\Zr^{\alpha+\beta-1}(\kappar\lambda)\big| \big(\big|\Lh\Psi_n\big|+ \big|\Lb\Psi_n\big|\big),\ \ \mathbf{A}'_3=\sum_{\beta'+\beta''=\beta}\int_{\mathcal{D}(t,u)}  \big|\Lr\Zr^{\beta''}(\Psi_{m})\big|\big(\big|\Lh\Psi_n\big|+ \big|\Lb\Psi_n\big|\big),\\
&\mathbf{A}'_4=\sum_{\beta'+\beta''=\beta}\int_{\mathcal{D}(t,u)}  \big|\Zr^{\beta''-1}(\lambda)\big|   \big(\big|\Lh\Psi_n\big|+ \big|\Lb\Psi_n\big|\big).
\end{align*}
Notice that $\lambda=\yr$  or $\zr$  in $\mathbf{A}'_4$ and $\mathbf{A}'_4$,  since we must have $\lambda=\yr$ or $\zr$ in \eqref{eq:Z beta'' L Psi}.    

We bound the $\mathbf{A}'_i$'s. We start with $\mathbf{A}'_2$. First of all, we recall that 
\[ \kappar\lambda \in \Lambda = \{\kappar\yr, \kappar\zr\}=\{\Xr(v^1+c), \Tr(v^1+c)\}.\] 
Therefore,  we have
\begin{align*}
\big|\Lr\Zr^{\alpha+\beta-1}(\kappar\lambda)\big| &\lesssim \big|\Lr\Zr^{\alpha+\beta}(\psi)\big|,
\end{align*}
where $\psi \in \{w,\wb,\psi_2\}$. Therefore,
\begin{equation}\label{eq:A prime 2}
\mathbf{A}'_2\lesssim\sum_{|\alpha|+|\beta|\leqslant n} \int_{\mathcal{D}(t,u)}\big|\Lr\Zr^{\alpha+\beta}(\psi)\big| \big(|\Lh\Psi_n|+ |\Lb\Psi_n|\big).
\end{equation}
We notice that $\mathbf{A}'_3$ can also be bounded by the righthand side of the above inequality. Therefore, by \eqref{def:L1}, \eqref{ineq: estimates for L1}, \eqref{def:L2 L3} and \eqref{comparison bound 1: L infty}, 
\begin{equation}\label{bounds on A prime 2 3}
\mathbf{A}'_2+\mathbf{A}'_3\lesssim \mathring{M}\varepsilon^3 t^2+\sum_{|\beta|\leqslant n}\mathscr{L}_2(\Zr^{\beta}(\psi),\Psi_n)(t,u)+\int_0^u \mathscr{F}_{\leqslant n}(u')du'.
\end{equation}

For $\mathbf{A}'_1$,  we recall from \eqref{eq: computation for Phi m+1} that $\overline{Z}=\kappar\Xr$ or $\Tr$. If $\overline{Z}= \Tr$, {\color{black}the corresponding integrands in $\mathbf{A}'_1$ have} already appeared in \eqref{eq:A prime 2}. If $\overline{Z}= \kappa \Xr$, {\color{black}the corresponding integrands} in $\mathbf{A}'_1$ are computed by
\[\Lr\overline{Z}(\Zr^{\alpha_2}(\psi'))=\kappar \Lr \Xr(\Zr^{\alpha_2}(\psi'))+ \Xr(\Zr^{\alpha_2}(\psi')).\]
Therefore, their contributions in $\mathbf{A}'_1$ are given by
\begin{align}
	\mathbf{A}'_1 &\lesssim \int_{\mathcal{D}(t,u)}\big|\kappar \Lr \Xr(\Zr^{\alpha_2}(\psi'))+ \Xr(\Zr^{\alpha_2}(\psi'))\big| \big(\big|\Lh\Psi_n\big|+ \big|\Lb\Psi_n\big|\big) \label{eq:A prime 3}\\
	&\lesssim \int_{\delta}^t \mathscr{E}_{\leqslant n}(\tau,u)d\tau + {\color{black}\int_{0}^u \mathscr{F}_{\leqslant n}(\tau,u')du'} + \sum_{1\leq|\beta|\leqslant n}\mathscr{L}_3(\Zr^{\beta}(\psi),\Psi_n)(t,u) + \mathring{\mathscr{L}}_3(\psi,\Psi_n)(t,u), \nonumber
\end{align}
where $\mathring{\mathscr{L}}_3$ is defined in \eqref{def:L3r} and we sum over $\psi \in \{\wb,w,\psi_2\}$.

We turn to the most difficult term $\mathbf{A}'_4$. Recall that we must have $\lambda=\yr$  or $\zr$  in  $\mathbf{A}'_4$. 
By \eqref{eq:precise formula for Z alpha yr} and \eqref{eq:precise formula for Z alpha zr}, we have the following schematic expression as 
\begin{equation}\label{eq: Z beta prime prim -1}
\begin{split}
\Zr^{\beta''-1}(\lambda)\cdot \Tr \wb
&=\Lr\Zr^{\beta''}(\psi)+\sum_{\substack{\beta''_1+\beta''_2=\beta''-1\\ |\beta''_1|\leqslant 1} } \Zr^{\beta''_1}(c) \Xr\Zr^{\beta''_2+1}(\psi)+\sum_{\substack{\beta''_1+\beta''_2=\beta''-1, \\ |\beta''_1|\leqslant |\beta''|-2}}\underbrace{\Zr^{\beta''_1}(\lambda) \Zr^{\beta''_2+1}(\wb)}_{\mathbf{Err}_7}\\
&+\sum_{\substack{\beta''_1+\beta''_2=\beta''-1 \\ |\beta''_1|\geqslant 2}}\underbrace{ \Zr^{\beta''_1}(c) \Zr^{\beta''_2+2}( \psi)}_{\mathbf{Err}_8}+\sum_{\beta''_1+\beta''_2=\beta''-1}\underbrace{\left(\Zr^{\beta''_1}\Tr(\psi_2)+\Zr^{\beta''_1}(\Xr\psi)\right)\Zr^{\beta''_2}\Xr\psi}_{\mathbf{Err}_9}.
\end{split}
\end{equation}
Similar to $\mathbf{Err}_k$ with $k=1,\cdots, 6$, we can use \eqref{comparison bound 1: L 2}, \eqref{ineq: L infty bound}, \eqref{ineq: L^2 bound on lambda final}, \eqref{ineq: L infty bound on lambda final} and Remark \ref{remark:techical nonlinear} to show that, for $k=7,8,9$, we have
\begin{equation}\label{eq: auxxx bounds on other Errors}
\begin{cases}
&\| \mathbf{Err}_k\|_{L^2(\Sigma_t)}\lesssim \mathring{M}\varepsilon^2;\\
&\| \mathbf{Err}_k\|_{L^\infty(\Sigma_t)}\lesssim \mathring{M}\varepsilon^2, \ \ \  \text{if}~|\beta'|\leqslant \Ninf.
\end{cases}
\end{equation}
Therefore, we can regroup $\mathbf{Err}_7$, $\mathbf{Err}_8$ and $\mathbf{Err}_9$ into the  $\mathbf{Err}$ term in $\mathbf{A}'_0$. Therefore, in order to bound the contribution of \eqref{eq: Z beta prime prim -1} in  $\mathbf{A}'_4$, we can equivalently rewrite it as
\begin{equation}\label{eq: lambda aux aaa}
\Zr^{\beta''-1}(\lambda)\cdot \Tr \wb
=\Lr\Zr^{\beta''}(\psi)+\sum_{\alpha_1+\alpha_2=\alpha,|\alpha_1|\leqslant 1} \Zr^{\alpha_1}(c) \Xr\Zr^{\alpha_2+1}(\psi).
\end{equation}
Since $\Tr\wb \approx 1$, we can replace $\Zr^{\beta''-1}(\lambda)$ in $\mathbf{A}'_4$ by the righthand side of the above equation. We notice that the first term on the righthand side of \eqref{eq: lambda aux aaa}, i.e.,$\Lr\Zr^{\beta''}(\psi)$, has already appeared in bounds for $\mathbf{A}'_1$. For the sum on the righthand side of \eqref{eq: lambda aux aaa}, we can repeat the argument for $\mathbf{A}_0, \cdots, \mathbf{A}_5$ to remove $\Zr^{\alpha_1}(c)$. On the other hand, $\Xr\Zr^{\alpha_2+1}(\psi)$ has also appeared in \eqref{eq:A prime 3}. Therefore, $\mathbf{A}'_4$ {\color{black}can be estimated exactly in the same way} as $\mathbf{A}'_1$, $\mathbf{A}'_2$ and $\mathbf{A}'_3$ in \eqref{bounds on A prime 2 3} and \eqref{eq:A prime 3}.

Finally, by \eqref{eq: auxxx bounds on Errors}, it is straightforward to show that $\mathbf{A}'_0\lesssim \mathring{M}\varepsilon^3 t^2$. By putting all the estimates together, the  contribution of $\sigma_{2,3}=\sigma'_{2,3}+\sigma''_{2,3}$ in $\mathscr{N}_n(t,u)$ and $\underline{\mathscr{N}}_n(t,u)$ are bounded by
\begin{align*}
{\mathscr{N}}(t,u)+\underline{\mathscr{N}}(t,u)\lesssim&\mathring{M}\varepsilon^3 t^2 + \int_{\delta}^t \mathscr{E}_{\leqslant n}(\tau,u)d\tau + \int_0^u \mathscr{F}_{\leqslant n}(u')du'\\
& + \sum_{|\beta|\leqslant n}\mathscr{L}_2(\Zr^{\beta}(\psi),\Psi_n)(t,u)+\sum_{{\color{black}1\leqslant|\beta|\leqslant n}}\mathscr{L}_3(\Zr^{\beta}(\psi),\Psi_n)(t,u) + \mathring{\mathscr{L}}_3(\psi,\Psi_n)(t,u).
\end{align*}

\subsubsection{Summary}

Combining the estimates for $\sigma_{2,1}, \sigma_{2,2}$ and $\sigma_{2,3}$, the error terms of {\bf Type $\mathbf{II}_2$} can be bounded as follows:
\begin{align*}
	{\mathscr{N}}(t,u)+\underline{\mathscr{N}}(t,u)\lesssim&\mathring{M}\varepsilon^3 t^2 + \int_{\delta}^t \mathscr{E}_{\leqslant n}(\tau,u)d\tau + \int_0^u \mathscr{F}_{\leqslant n}(u')du'\\
	& + \sum_{|\beta|\leqslant n}\mathscr{L}_2(\Zr^{\beta}(\psi),\Psi_n)(t,u)+\sum_{{\color{black}1\leqslant|\beta|\leqslant n}}\mathscr{L}_3(\Zr^{\beta}(\psi),\Psi_n)(t,u) + \mathring{\mathscr{L}}_3(\psi,\Psi_n)(t,u).
\end{align*}

\subsection{Estimates on {\bf  Type $\mathbf{II}_3$} terms}

For {\bf Type $\mathbf{II}_3$} type terms, we have to bound the following integrals:
\[\mathscr{N}_n(t,u)=-\int_{D(t,w)}  \frac{\mu}{\mur} \cdot  \Zr^\beta\big({}^{(\mathring{Z})} \sigma_{3}\big) \cdot \Lh\Psi_n,  \ \ \underline{\mathscr{N}}_n(t,u)=-\int_{D(t,u)} \frac{\mu}{\mur} \cdot  \Zr^\beta\big({}^{(\mathring{Z})} \sigma_{3}\big)  \cdot \Lb\Psi_n.\]
We remark that the estimates on the Type $\mathbf{II}_3$ terms are different from the previous ones. The negative sign in the above expression for $\underline{\mathscr{N}}_n(t,u)$ is crucial, see  {\color{black}Section \ref{section sigma34}} for the bounds on $\sigma_{3,4}$.

We regroup the terms of \eqref{eq: sigma 3 circle} as follows:
\begin{equation} 
\begin{split}
{}^{(\Zr)} \sigma_{3}&=\underbrace{-\frac{1}{2}\left( \Lr \left(\pi_{\Lbr \Xr} \right)+ \Lbr \left(\pi_{\Lr \Xr} \right)- \Xr(\pi_{\Lr \Lbr})\right) \Xr(\Psi_{m})- \frac{1}{2}  \Xr (\pi_{\Lbr \Xr} )\cdot\Lr(\Psi_{m})}_{\sigma_{3,1}}\\
&\underbrace{- \frac{1}{2} \Xr ( \pi_{\Lr \Xr} )\cdot \Lbr(\Psi_{m})}_{\sigma_{3,2}}+\underbrace{\Lr\left(\frac{1}{4\mur} \pi_{\Lbr \Lbr}\right)  \Lr(\Psi_{m})}_{\sigma_{3,3}}+\underbrace{\frac{1}{4\kappar}\Lbr\left( c^{-1}\pi_{\Lr \Lr}\right)  \Lbr(\Psi_{m})}_{\sigma_{3,4}}{\color{black},}
\end{split}
\end{equation}
where  ${\pi}$ stands for $\,^{(\mathring{Z})}{\pi}$. In view of \eqref{eq: varphor_n expression}, we have $|\beta|+m+1\leqslant \Ntop$.

\subsubsection{The bounds on $\sigma_{3,1}$}\label{section sigma31}
The terms in {\color{black}$\sigma_{3,1}$}  can be schematically represented as $D\times W$ where
\[D\in \big\{\Lr\big(\,^{(\mathring{Z})} \pi_{\Lbr \Xr}\big) ,\Lbr\big(\,^{(\mathring{Z})} \pi_{\Lr \Xr}\big), \Xr\big(\,^{(\mathring{Z})}\pi_{\Lr \Lbr}\big),  \Xr \big(\,^{(\mathring{Z})}\pi_{\Lbr \Xr} \big)\big\},  \ \ W\in \{\Xr(\Psi_{m}),\Lr(\Psi_{m})\}.\]

We prove that for all $F= D$ or $W$, for all multi-index $\alpha$, we have
\begin{equation}\label{eq:111sigma31}
\begin{cases}
&\|\Zr^{\alpha}(F)\|_{L^2(\Sigma_t)}\lesssim \mathring{M}\varepsilon, \ \ \ {\rm ord} \big(\Zr^{\alpha}(F)\big)\leqslant \Ntop+1;\\
&\|\Zr^{\alpha}(F)\|_{L^\infty(\Sigma_t)}\lesssim \mathring{M}\varepsilon, \ \ \  {\rm ord}\big(\Zr^{\alpha}(F)\big)\leqslant \Ninf.
\end{cases}
\end{equation}
In view of \eqref{comparison bound 1: L 2} and \eqref{ineq: L infty bound}, \eqref{eq:111sigma31} automatically holds for $F=W$. If $F=D$, in view of the tables of deformation tensors in Section \ref{section:commutators and their  deformation tensors}, {\color{black}the set  of $D$'s can be written as}  
\[\big\{\Lr\big(c^{-1}\kappar\Zr(\psi_2)\big), \Xr\big(c^{-1}\kappar\Zr(\psi_2)\big), \Lbr\big(\Zr(\psi_2)\big), \kappar\Xr\Zr(\psi_2)\big\}\]
where we have ignored the irrelevant constants. Since 
\[\mathring{Y}\big(c^{-1}\kappar\Zr(\psi_2)\big)=-c^{-2}\mathring{Y}(c)\kappar\Zr(\psi_2) +c^{-1}\mathring{Y}(\kappar)\Zr(\psi_2)+c^{-1}\kappar\mathring{Y}\Zr(\psi_2),  \ \ \mathring{Y}= \Lr, \Xr,\]
and $\Lbr\big(\Zr(\psi_2)\big)=c^{-1}\kappa \Lr\big(\Zr(\psi_2)\big)+2\Tr\Zr(\psi_2)$, it suffices to check \eqref{eq:111sigma31} for
\[D\in \big\{c^{-2}\kappar\mathring{Y}(c)\Zr(\psi_2), c^{-1}\Zr(\psi_2), c^{-1}\kappar\mathring{Y}\Zr(\psi_2), \Tr\Zr(\psi_2), \kappar\Xr\Zr(\psi_2)\big|\mathring{Y}=\Lr, \Xr\big\}.\]
This is straightforward from \eqref{comparison bound 1: L 2}, \eqref{ineq: L infty bound}, \eqref{ineq: L infty bound for Zbeta L Zalpha} and \eqref{ineq: L 2 bound for Zbeta L Zalpha}.

We apply \eqref{eq:111sigma31} and Remark \ref{remark:techical nonlinear} to each single term of $\big|\Zr^{\beta}\big(  D\times W\big)\big|\lesssim \sum_{\beta_1+\beta_2=\beta} |\Zr^{\beta_1}(D)||\Zr^{\beta_2}(W)|$. This shows that
\begin{equation}\label{eq: sigma 31 aux DW}
\|\Zr^{\beta}(D\times W)\|_{L^2(\Sigma_t)} \lesssim \mathring{M}\varepsilon^2.
\end{equation}
We still bound $\big|\frac{\mu}{\mur}\big|$ in  $\mathscr{N}_n(t,u)$ and $\underline{\mathscr{N}}_n(t,u)$ by $1$. Hence, the contribution of $\sigma_{3,1}$ in the error integrals {\color{black}is bounded} as follows:
\begin{equation}\label{eq: sigma 31 final}
{\mathscr{N}}(t,u)+\underline{\mathscr{N}}(t,u)\lesssim  \int_{\delta}^t\big|\frac{\mu}{\mur}\big| \|\Zr^{\beta}(D\times W)\|_{L^2(\Sigma_\tau)} \big(\|\Lh\Psi_n\|_{L^2(\Sigma_\tau)}+\|\Lb\Psi_n\|_{L^2(\Sigma_\tau)}\big)d\tau \lesssim \mathring{M}\varepsilon^3t^2.
\end{equation}

\subsubsection{The bounds on $\sigma_{3,2}$}\label{section sigma32}
We still ignore the irrelevant constants in this subsection. By the tables in Section \ref{section:commutators and their  deformation tensors}, we have ${}^{(\Zr)} \sigma_{3,2}=\Xr\Zr(\psi_2)\cdot\Lbr(\Psi_{m})$. Therefore,
\begin{align*} \Zr^\beta\big({}^{(\Zr)} \sigma_{3,2}\big)&=\Xr\Zr^{\beta+1}(\psi_2) \Lbr(\Psi_{m})+\underbrace{\sum_{\beta_1+\beta_2=\beta, |\beta_2|\geqslant 1} \Xr\Zr^{\beta_1+1}(\psi_2) \Zr^{\beta_2}\Lbr(\Psi_{m})}_{\sigma'_{3,2}}.
\end{align*}
We first consider the contribution of $\sigma'_{3,2}$. It is similar to $\sigma_{3,1}$. We notice that for  $\beta_1+\beta_2=\beta$ and $|\beta_1|\geqslant 1$, we have
\begin{equation}\label{eq:111sigma32}
\begin{cases}
&\|\Xr\Zr^{\beta_1+1}(\psi_2)\|_{L^2(\Sigma_t)}+\| \Zr^{\beta_2}\Lbr(\Psi_{m})\|_{L^2(\Sigma_t)}\lesssim \mathring{M}\varepsilon , \ \ \    |\beta_1|+1\leqslant \Ntop, \  |\beta_2|+m\leqslant \Ntop;\\
&\|\Xr\Zr^{\beta_1+1}(\psi_2)\|_{L^\infty(\Sigma_t)}+\| \Zr^{\beta_2}\Lbr(\Psi_{m})\|_{L^\infty(\Sigma_t)}\lesssim \mathring{M}\varepsilon , \ \ \   |\beta_1|+2\leqslant \Ninf, \ |\beta_2|+m+1\leqslant \Ninf.
\end{cases}
\end{equation}
In view of \eqref{comparison bound 1: L 2} and \eqref{ineq: L infty bound}, the estimates for $\Xr\Zr^{\beta_1+1}(\psi_2)$ {\color{black}are trivial.} For $\Zr^{\beta_2}\Lbr(\Psi_{m})$, we have
\begin{align*}
\Zr^{\beta_2}\Lbr(\Psi_{m})&=2\Zr^{\beta_m}\Tr(\Psi_{m})+\kappar\Zr^{\beta_m}\big(c^{-1}\Lr \Psi_{m}\big)=2\Zr^{\beta_2}\Tr(\Psi_{m})+\kappar\!\!\!\!\!\!\sum_{\beta_2'+\beta_2''=\beta_2}\!\!\!\!\!\!\Zr^{\beta'_2}\big(c^{-1}\big)\Zr^{\beta''_2}\Lr\Psi_{m}.
\end{align*}
Therefore, by applying  \eqref{comparison bound 1: L 2}, \eqref{ineq: L infty bound}, \eqref{ineq: L infty bound for Zbeta L Zalpha} and \eqref{ineq: L 2 bound for Zbeta L Zalpha}, the bounds \eqref{eq:111sigma32} are proved. Just as the proof of \eqref{eq: sigma 31 aux DW}, we conclude that $\|\sigma'_{3,2}\|_{L^2(\Sigma_t)} \lesssim \mathring{M}\varepsilon^2$.
Hence, by the same argument for \eqref{eq: sigma 31 final},  the  contribution of $\sigma'_{3,2}$ in $\mathscr{N}_n(t,u)$ and $\underline{\mathscr{N}}_n(t,u)$ are bounded by $\mathring{M}\varepsilon^3t^2$.  With this bound, the contribution of $\sigma_{3,2}$ are estimated as follows:
\begin{align*}
\int_{\mathcal{D}(t,u)} \big|\frac{\mu}{\mur}\big|  \big| \Zr^\beta ({}\sigma_{3,2})\big|\big( |\Lh\Psi_n|+ |\Lb\Psi_n|\big)&\lesssim \mathring{M}\varepsilon^3 t^2+\int_{\mathcal{D}(t,u)} \big| \Xr\Zr^{\beta+1}(\psi_2) \Lbr(\Psi_{m})\big|\big(|\Lh\Psi_n|+|\Lb\Psi_n|\big).
\end{align*}
If $m\geqslant 1$, by applying  \eqref{comparison bound 1: L 2}, \eqref{ineq: L infty bound}, \eqref{ineq: L infty bound for Zbeta L Zalpha} and \eqref{ineq: L 2 bound for Zbeta L Zalpha}, we still have $\| \Xr\Zr^{\beta+1}(\psi_2) \Lbr(\Psi_{m})\|_{L^2(\Sigma_t)} \lesssim \mathring{M}\varepsilon^2$. Hence, the corresponding terms in the integral are bounded by $\mathring{M}\varepsilon^3t^2$. It remains to consider the case where $m=0$, i.e.,  $\Psi_{0}=\psi \in \{w,\wb,\psi_2\}$. We bound $\Lbr(\Psi_{m})$ in $L^\infty$ by $1$ in this case. Therefore, it suffices to bound the following integral:
\begin{align*}
\int_{\mathcal{D}(t,u)} \big| \Xr\Zr^{\beta+1}(\psi_2) \big| \big(|\Lh\Psi_n|+|\Lb\Psi_n|\big)
&\lesssim \mathring{M}\varepsilon^3 t^2+{\color{black}\mathscr{L}_3}(\Psi_n,\Zr^n(\psi_2))(t,u) +\int_{0}^u  \mathscr{F}_{\leqslant n}(t,u')du'.
\end{align*}
In the last step, we have used \eqref{def:L1}, \eqref{ineq: estimates for L1}, \eqref{def:L2 L3}, \eqref{comparison bound 1: L infty} and \eqref{comparison bound 1: L 2}. 

To summarize, the contribution of $\sigma_{3,2}$ {\color{black}in the error integrals is bounded} as follows:
\begin{align*}
{\mathscr{N}}(t,u)+\underline{\mathscr{N}}(t,u)&\lesssim \mathring{M}\varepsilon^3 t^2+\mathscr{L}_3(\Psi_n,\Zr^n(\psi_2))(t,u) +\int_{0}^u  \mathscr{F}_{\leqslant n}(t,u')du'.
\end{align*}

\subsubsection{The bounds on $\sigma_{3,3}$}\label{section sigma33}
By the tables in Section \ref{section:commutators and their  deformation tensors}, we have
\[ \Lr\big(\frac{1}{2\mur} \pi_{\Lbr \Lbr}\big)  =\begin{cases} c^{-2}(y-2\Xr(c))+\kappar\Lr\big(c^{-2}(y-2\Xr(c))\big), \ \ &\Zr=\Xr;\\
\boxed{-2c^{-2}\Tr(c)}+c^{-2}z+\Lr\big(c^{-2}(z-2\Tr(c))\big),\ \ &\Zr=\Tr.\end{cases}\]
In view of the definition of $y$ and $z$, \eqref{comparison bound 1: L 2}, \eqref{ineq: L infty bound}, \eqref{ineq: L infty bound for Zbeta L Zalpha} and \eqref{ineq: L 2 bound for Zbeta L Zalpha}, we can repeat the proof of \eqref{eq:111sigma31} to show that each single term $F$ in the above formula, except for the one in the box, satisfies the following estimates: 
\begin{equation*}
\begin{cases}
&\|\Zr^{\alpha}(F)\|_{L^2(\Sigma_t)}\lesssim \mathring{M}\varepsilon, \ \ \ {\rm ord}\big(\Zr^{\alpha}(F)\big)\leqslant \Ntop+1;\\
&\|\Zr^{\alpha}(F)\|_{L^\infty(\Sigma_t)}\lesssim \mathring{M}\varepsilon, \ \ \  {\rm ord}\big(\Zr^{\alpha}(F)\big)\leqslant \Ninf.
\end{cases}
\end{equation*}
By \eqref{comparison bound 1: L 2}, \eqref{ineq: L infty bound for Zbeta L Zalpha} and \eqref{ineq: L 2 bound for Zbeta L Zalpha}, these estimates also hold for $\Lr(\Psi_{m})$.  Therefore, we apply Remark \ref{remark:techical nonlinear} to each single term of $\big|\Zr^{\beta}\big(F\cdot \Lr(\Psi_{m})\big)\big|\lesssim \sum_{\beta_1+\beta_2=\beta} |\Zr^{\beta_1}(F)||\Zr^{\beta_2}(\Lr(\Psi_{m}))|$ to derive
\[\|\Zr^{\beta}(F\cdot\Lr(\Psi_{m}))\|_{L^2(\Sigma_t)} \lesssim \mathring{M}\varepsilon^2.\]
Therefore, except for the boxed term, the contribution of $\sigma_{3,3}$ to the error integral are bounded by
\begin{equation}\label{bound sigma 3 3 extra epsilon}
{\mathscr{N}}(t,u)+\underline{\mathscr{N}}(t,u)\lesssim  \int_{\delta}^t\big|\frac{\mu}{\mur}\big| \|\Zr^{\beta}(F\cdot \Lr(\Psi_{m}))\|_{L^2(\Sigma_\tau)} \big(\|\Lh\Psi_n\|_{L^2(\Sigma_\tau)}+\|\Lb\Psi_n\|_{L^2(\Sigma_\tau)}\big)d\tau \lesssim \mathring{M}\varepsilon^3t^2.
\end{equation}
We {\color{black}write the boxed term} $c^{-2}\Tr(c)$ as $-\Tr(c^{-1})$. Hence, {\color{black}for the error integrals} of $\sigma_{3,3}$, it remains to control the contribution from the boxed term:
\begin{align*}
\int_{\mathcal{D}(t,u)} \big| \Zr^{\beta} \big(\Tr(c^{-1}) \Lr(\Psi_{m}) \big) \big| \big(|\Lh\Psi_n|+ |\Lb\Psi_n| \big)\leqslant \sum_{\beta'+\beta''=\beta}\int_{\mathcal{D}(t,u)} \big| \Zr^{\beta'}\Tr(c^{-1}) \big| \big|\Zr^{\beta''} \big( \Lr(\Psi_{m}) \big) \big| \big(|\Lh\Psi_n|+ |\Lb\Psi_n| \big).
\end{align*}
This term have already been controlled in Section \ref{section sigma23}, see the term $\mathbf{A}_3$ after the equation \eqref{eq:Z beta'' L Psi}. As a conclusion, it is bounded by the righthand side of \eqref{bounds on A prime 2 3}.

Putting all the bounds together, the contribution of $\sigma_{3,3}$ in the error integrals are bounded as follows:
\begin{align*}
{\mathscr{N}}(t,u)+\underline{\mathscr{N}}(t,u)&\lesssim \mathring{M}\varepsilon^3 t^2+\sum_{|\beta|\leqslant n}\mathscr{L}_2(\Zr^{\beta}(\psi),\Psi_n)(t,u)+\int_0^u \mathscr{F}_{\leqslant n}(u')du'.
\end{align*}

\subsubsection{The bounds on $\sigma_{3,4}$}\label{section sigma34}


It remains to bound the following integrals:
\begin{equation}\label{precise error integrals}
\mathscr{N}_n(t,u)=-\int_{D(t,u)}  \frac{\mu}{\mur} \cdot  \Zr^\beta\big({}^{(\mathring{Z})} \sigma_{3,4}\big) \cdot \Lh\Psi_n,  \ \ \underline{\mathscr{N}}_n(t,u)=-\int_{D(t,u)} \frac{\mu}{\mur} \cdot  \Zr^\beta\big({}^{(\mathring{Z})} \sigma_{3,4}\big)  \cdot \Lb\Psi_n.
\end{equation}
By the tables in Section \ref{section:commutators and their  deformation tensors}, we have $\frac{1}{4\kappar}\Lbr\big( c^{-1}\,^{\Zr_0}\pi_{\Lr \Lr}\big) =-\frac{1}{2\kappar}\Lbr \Zr_0(v^1+c)$. Hence,
\begin{equation}\label{eq:decomp of sigma 34}
\begin{split}
\Zr^\beta\big({}^{(\Zr_0)} \sigma_{3,4}\big)&= -\sum_{\beta_1+\beta_2=\beta}\frac{1}{2\kappar}\Zr^{\beta_1} \Lbr \Zr_0(v^1+c)  \cdot \Zr^{\beta_2} \Lbr(\Psi_{m})\\
&=\underbrace{-\frac{1}{2\kappar}\Zr^{\beta} \Lbr \Zr_0(v^1+c)  \Lbr(\Psi_{m})}_{\sigma'^{(\beta)}_{3,4}}- \!\!\!\!\underbrace{\sum_{\beta_1+\beta_2=\beta, |\beta_2|\geqslant 1}\!\!\!\frac{1}{2\kappar} \Zr^{\beta_1} \Lbr \Zr_0(v^1+c)  \Zr^{\beta_2}\big(\Lbr(\Psi_{m})\big) }_{\sigma''^{(\beta)}_{3,4}}.
\end{split}
\end{equation}
For an arbitrary smooth function $f$, by writing {\color{black}$\Lbr$} in terms of $\Lr$ and $\Tr$, we have
\begin{equation}\label{eq:Zr alpha Lbr psi}
\begin{split}
\Zr^\gamma \Lbr f&=2\Tr\Zr^\gamma f+\sum_{\gamma'+\gamma''=\gamma}\kappar\Zr^{\gamma'}\left(c^{-1}\right)\Zr^{\gamma''}\Lr f.
\end{split}
\end{equation}
We use this formula to study each term appeared in $\sigma''^{(\beta)}_{3,4}$. The first case is for $F=\Zr^{\beta_1} \Lbr \Zr_0(v^1+c)$ where $f=\Zr_0(v^1+c)$ and $\gamma=\beta_1$. The second case is for $F=\Zr^{\beta_2}\big(\Lbr(\Psi_{m})\big)$ where $\gamma=\beta_2$ and $f=\Psi_{m}$. Based on the assumption that $|\beta_2|\geqslant 1$, we can repeat the proof of \eqref{eq:111sigma31} and use \eqref{comparison bound 1: L 2}, \eqref{ineq: L infty bound}, \eqref{ineq: L infty bound for Zbeta L Zalpha} and \eqref{ineq: L 2 bound for Zbeta L Zalpha} to show that each single $F$ satisfies
\begin{equation}\label{eq: sigma 3 4 aux 123}
\begin{cases}
&\| F\|_{L^2(\Sigma_t)}\lesssim \mathring{M}\varepsilon t, \ \  {\rm ord}\left( F\right)\leqslant \Ntop+1;\\
&\| F\|_{L^\infty(\Sigma_t)}\lesssim \mathring{M}\varepsilon t, \ \   {\rm ord}\left( F\right)\leqslant \Ninf.
\end{cases}
\end{equation}
 We then use Remark \ref{remark:techical nonlinear} and this gives $\|\sigma''^{(\beta)}_{3,4}\|_{L^2(\Sigma_t)} \lesssim \mathring{M}\varepsilon^2 t$. Therefore,  the contributions of $\sigma''^{(\beta)}_{3,4}$ to the error integral are bounded by
\begin{equation*}
{\mathscr{N}}(t,u)+\underline{\mathscr{N}}(t,u)\lesssim  \int_{\delta}^t\big|\frac{\mu}{\mur}\big| \|\sigma''^{(\beta)}_{3,4}\|_{L^2(\Sigma_\tau)} \big(\|\Lh\Psi_n\|_{L^2(\Sigma_\tau)}+\|\Lb\Psi_n\|_{L^2(\Sigma_\tau)}\big)d\tau \lesssim \mathring{M}\varepsilon^3t^2.
\end{equation*}

Finally, we turn to $\sigma'^{(\beta)}_{3,4}$. We observe that the same argument also shows that \eqref{eq: sigma 3 4 aux 123} holds for $F=\Lbr(\Psi_{m})$ unless $m=0$ and $\Psi_{m}=\wb$. It also holds for $F=\Zr^{\beta} \Lbr \Zr_0(v^1+c)$. Therefore, unless $m=0$ and $\Psi_{m}=\wb$, the contribution of $\sigma'^{(\beta)}_{3,4}$ can be bounded exactly in the same way as $\sigma''^{(\beta)}_{3,4}$. Hence, it suffices to assume that  
\[\sigma'^{(\beta)}_{3,4}=-\frac{1}{2\kappar}\Zr^{\beta} \Lbr \Zr_0(v^1+c)  \Lbr(\wb),\]
where we keep the precise constant $-\frac{1}{2}$. We apply \eqref{eq:Zr alpha Lbr psi} for $f=\Zr_0(v^1+c)$ and $\gamma=\beta$. This shows that
\begin{align*}
\sigma'^{(\beta)}_{3,4}=\underbrace{-\frac{\Lbr(\wb)}{\kappar}\Tr\Zr^{\beta} \Zr_0(v^1+c)}_{\sigma'^{(\beta)}_{3,4;1}}  + \underbrace{\Lbr(\wb)\sum_{\beta'+\beta''=\beta}\Zr^{\beta'}\big(c^{-1}\big)\Zr^{\beta''}\Lr\Zr_0(v^1+c)}_{\sigma'^{(\beta)}_{3,4;2}}.
\end{align*}
We remark {\color{black}that the constants of the last sum $\sigma'^{(\beta)}_{3,4;2}$ are irrelevant.}

To bound $\sigma'^{(\beta)}_{3,4;2}$, we first use \eqref{eq: commutation formular for Lr Zr} to commute derivatives.  Therefore
\begin{align*}
\sigma'^{(\beta)}_{3,4;2}=\Lbr(\wb)\sum_{\beta'+\beta''=\beta}\Zr^{\beta'}\big(c^{-1}\big)\big[\Lr\Zr^{\beta''}\Zr_0(v^1+c)+\sum_{\substack{\beta''_1+\beta''_2=\beta'' \\ |\beta''_1|\leqslant |\beta''|-1}}\Zr^{\beta''_1}(\lambda) \Zr^{\beta''_2}\Zr_0(v^1+c)\big],
\end{align*}
where $\lambda \in \{\yr, \zr, \chir,\etar\}$. We notice that the contribution of the term $\Lr\Zr^{\beta''}\Zr_0(v^1+c)$ have already been controlled in the $\mathbf{A}_3$ term of Section \ref{section sigma23} so that it is bounded by the righthand side of \eqref{bounds on A prime 2 3}. To estimate the contribution of the term $\Zr^{\beta''_1}(\lambda) \Zr^{\beta''_2}\Zr_0(v^1+c)$, in view of the fact that $|\beta''_1|\leqslant |\beta''|-1\leqslant \Ntop-1$ and $|\beta''_2|\geqslant 1$, for $F=\Zr^{\beta''_1}(\lambda)$ or $\Zr^{\beta''_2}\Zr_0(v^1+c)$, we can use \eqref{comparison bound 1: L 2}, \eqref{ineq: L infty bound}, \eqref{ineq: L^2 bound on lambda final} and \eqref{ineq: L infty bound on lambda final} to show that 
\begin{equation*}
\begin{cases}
&\| F\|_{L^2(\Sigma_t)}\lesssim \mathring{M}\varepsilon, \ \  {\rm ord}\left( F\right)\leqslant \Ntop+1;\\
&\| F\|_{L^\infty(\Sigma_t)}\lesssim \mathring{M}\varepsilon, \ \   {\rm ord}\left( F\right)\leqslant \Ninf.
\end{cases}
\end{equation*}
Therefore, $\|\Zr^{\beta''_1}(\lambda) \Zr^{\beta''_2}\Zr_0(v^1+c)\|_{L^2(\Sigma_t)}\lesssim \mathring{M}\varepsilon^2$. Similar to \eqref{bound sigma 3 3 extra epsilon}, its contribution in the error integrals are bounded by $\mathring{M}\varepsilon^3t^2$. Combing these cases, we derive that
\begin{align*}
\Big|\int_{D(t,u)}   \frac{\mu}{\mur} \cdot  \sigma'^{(\beta)}_{3,4;2} \cdot\big(|\Lh\Psi_n|+ |\Lb\Psi_n| \big)]\Big| \lesssim \mathring{M}\varepsilon^3 t^2+\sum_{|\beta|\leqslant n}\mathscr{L}_2(\Zr^{\beta}(\psi),\Psi_n)(t,u)+\int_0^u \mathscr{F}_{\leqslant n}(u')du'.
\end{align*}

For $\sigma'^{(\beta)}_{3,4;1}$, by \eqref{precise error integrals}, its contribution in the error integrals are exactly 
\[\mathbf{I}= \int_{D(t,u)}  \frac{\mu}{\mur} \cdot  \frac{\Lbr(\wb)}{\kappar}\Tr\Zr^{\beta} \Zr_0(v^1+c)  \cdot\big( \Lh\Zr^\beta\Zr_0\wb+ \Lb\Zr^\beta\Zr_0\wb\big),\]
where we also used the fact that $\Psi_n= \Zr^\beta\Zr_0\wb$. We use the formula $v^1+c=\frac{\gamma+1}{2}\wb+\frac{\gamma-3}{2}w$ to replace $v^1+c$ in $\mathbf{I}$. This leads to
\begin{align*}
\mathbf{I}&=\int_{D(t,u)}  \frac{\mu}{\mur} \cdot  \frac{\Lbr(\wb)}{\kappar}\Tr\Zr^{\beta} \Zr_0\left(\frac{\gamma+1}{2}\wb+\frac{\gamma-3}{2}w\right)  \cdot\left( \Lh\Zr^\beta\Zr_0\wb+ \Lb\Zr^\beta\Zr_0\wb\right)\\
&=\underbrace{\frac{\gamma+1}{2}\int_{D(t,u)}  \frac{\mu}{\mur} \cdot  \frac{\Lbr(\wb)}{\kappar}\Tr\Zr^{\beta} \Zr_0\wb \cdot   \Lb\Zr^\beta\Zr_0\wb}_{\mathbf{I}_1}+\underbrace{\frac{\gamma-3}{2}\int_{D(t,u)}  \frac{\mu}{\mur} \cdot  \frac{\Lbr(\wb)}{\kappar}\Tr\Zr^{\beta} \Zr_0 w  \cdot  \Lb\Zr^\beta\Zr_0\wb}_{\mathbf{I}_2}\\
&\ \ +\underbrace{\int_{D(t,u)}   \frac{\mu}{\mur} \cdot  \Lbr(\wb)\cdot\Tr\Zr^{\beta} \Zr_0\left(\frac{\gamma+1}{2}\wb+\frac{\gamma-3}{2}w\right)  \cdot \frac{1}{\kappar}\Lh\Zr^\beta\Zr_0\wb}_{\mathbf{I}_3}.
\end{align*}
We bound $\mathbf{I}_1,\mathbf{I}_2$ and $\mathbf{I}_3$ in different ways. 
\begin{itemize}[noitemsep,wide=0pt, leftmargin=\dimexpr\labelwidth + 2\labelsep\relax]
\item For $\mathbf{I}_3$, {\color{black}using \eqref{comparison bound 1: L infty} to convert the $\Tr$ derivatives into $L,\Lb$ and $\Xh$ leads to}
\[\big|\Tr\big[\Zr^{\beta} \Zr_0\big(\frac{\gamma+1}{2}\wb+\frac{\gamma-3}{2}w\big) \big]\big|\lesssim t|L\Psi_n|+|\Lb \Psi_n|+ \varepsilon t|\Xh \Psi_n|,\]
where we recall that $|\beta|=n-1$. Therefore, by bounding $\frac{\mu}{\mur}$ and $\Lbr(\wb)$ by a universal constant, we derive that
\begin{align*}
|\mathbf{I}_3|&\lesssim \int_{D(t,u)}   \big(t|L\Psi_n|+|\Lb \Psi_n|+ \varepsilon t|\Xh \Psi_n|\big)\cdot \frac{1}{\kappar}\Lh\Zr^\beta\Zr_0\wb\\
&\lesssim \mathring{M}\varepsilon^3 t^2+ \mathscr{L}_2\big(\Zr^\beta\Zr_0(\wb),\Psi_n\big)(t,u) +\int_{0}^u  \mathscr{F}_{\leqslant n}(t,u')du'.
\end{align*}

\item For $\mathbf{I}_2$, we use the fact that the Riemann invariant $w$ is almost invariant along the null direction $\Lb$. In fact, we have
\begin{align*}
\Tr\Zr^{\beta} \Zr_0(\wb)=\Zr^{\beta+1}\Tr(\wb)=\frac{1}{2}\Zr^{\beta+1}(\Lbr\wb-c^{-1}\kappar \Lr\wb)=\frac{1}{2}\Zr^{\beta+1}\big(\frac{1}{2}\kappar\Xr(\psi_2)-c^{-1}\kappar \Lr(\wb)\big).
\end{align*}
In the last step, we have used the second equation of \eqref{Euler equations:form 4}. Thus, we can regard $\Tr\Zr^{\beta} \Zr_0(\wb)$ as a sum of $F_1=\kappar\Xr\Zr^{\beta+1} (\psi_2)$ and $F_2=\kappar\Zr^{\beta+1}\big( c^{-1} \Lr(\wb)\big)$. According to \eqref{comparison bound 1: L infty}, the contribution of $F_1$ to $\mathbf{I}_2$ is obviously bounded by $\mathring{M}\varepsilon^3 t^2+\int_{0}^u  \mathscr{F}_{\leqslant n}(t,u')du'$. The contribution of $F_2$ to $\mathbf{I}_2$ can be bounded by
\begin{align*}
\int_{D(t,u)}   \big|\Zr^{\beta+1}\big( c^{-1} \Lr(\wb)\big)\big| \big|\Lb\Zr^\beta\Zr_0(\wb)\big|=\sum_{\beta'+\beta''=\beta+1}\int_{D(t,u)}   \big|\Zr^{\beta'}\big( c^{-1}\big)\big|\big|\Zr^{\beta''}\big( \Lr(\wb)\big)\big| \big|\Lb\Zr^\beta\Zr_0(\wb)\big|{\color{black}.}
\end{align*}
We notice that the terms in the sum have already been controlled in the $\sigma'_{2,3;2}$ term of Section \ref{section sigma23}. As a conclusion, we have
\begin{equation*}
\mathbf{I}_2\lesssim \mathring{M}\varepsilon^3 t^2+\sum_{|\beta|\leqslant n}\mathscr{L}_2(\Zr^{\beta}(\psi),\Psi_n)(t,u)+\int_0^u \mathscr{F}_{\leqslant n}(u')du'.
\end{equation*}

\item For $\mathbf{I}_1$, by writing $\Lb=c^{-1}\kappa L+2T$, it can be decomposed as follows:
\begin{align*}
\mathbf{I}_{1}
&=\underbrace{{\color{black}(\gamma+1)\int_{D(t,u)}  \frac{\mu}{\mur} \cdot  \frac{\Lbr(\wb)}{\kappar}|\Tr\Zr^{\beta}\wb|^2}}_{\mathbf{I}_{1,1}}+\underbrace{\frac{\gamma+1}{2}\int_{D(t,u)}   \frac{\mu}{\mur} \frac{\kappa}{\kappar}\cdot  \frac{\Lbr(\wb)}{c}\Tr\Zr^{\beta} \Zr_0\wb \cdot    L\Zr^\beta\Zr_0\wb}_{\mathbf{I}_{1,2}}\\
&+\underbrace{\int_{D(t,u)}  \frac{\mu}{\mur} \cdot  \frac{\Lbr(\wb)}{\kappar}\Tr\Zr^{\beta} \Zr_0\wb \cdot   (T-\Tr)\Zr^\beta\Zr_0\wb}_{\mathbf{I}_{1,3}}.
\end{align*}
We notice that $\mathbf{I}_{1,2}$ can be bounded exactly in the same way as  $\mathbf{I}_3$.

To bound  $\mathbf{I}_{1,3}$, by \eqref{change of frame 1}, \eqref{bound on kappa more precise}, \eqref{bound on T1+1}, \eqref{bound on T2}, we have $|(\Tr -T)f|\lesssim \mathring{M}\varepsilon |Tf|+\mathring{M}\varepsilon t|\Xh f|$. In view of $\eqref{comparison bound 1: L 2}$, we obtain that $\mathbf{I}_{1,3}\lesssim \mathring{M}\varepsilon^3 t^2$.

For $\mathbf{I}_{1,1}$, its absolute value can not be bounded through the Gronwall type inequalities. We observe that 
\[\Lbr(\wb)=2\Tr(\wb)+c^{-1}\kappar \Lr(\wb)=2\Tr(\wb)+\mathring{M}\varepsilon t.\]
Therefore,  $\Lbr(\wb)$ is negative provided $\mathring{M}\varepsilon$ is sufficiently small. The negative sign reflects the fundamental physical nature of rarefaction wave: the density of the gas decreases along the transversal direction. Thus, $\mathbf{I}_{1,1}$ is a negative quantity   so that it can be ignored. 
\end{itemize}

Putting all the bounds together, the contribution of $\sigma_{3,4}$ {\color{black}in the error integrals is bounded} as follows:
\begin{align*}
{\mathscr{N}}(t,u)+\underline{\mathscr{N}}(t,u)&\lesssim \mathring{M}\varepsilon^3 t^2+\sum_{|\beta|\leqslant n}\mathscr{L}_2(\Zr^{\beta}(\psi),\Psi_n)(t,u)+\int_0^u \mathscr{F}_{\leqslant n}(u')du'.
\end{align*}

\subsubsection{Summary}

Combining the estimates for $\sigma_{3,1}, \sigma_{3,2}$, $\sigma_{3,3}$ and $\sigma_{3,4}$, the error terms of {\bf Type $\mathbf{II}_3$} can be bounded as follows:
\begin{align*}
{\mathscr{N}}(t,u)+\underline{\mathscr{N}}(t,u)&\lesssim \mathring{M}\varepsilon^3 t^2+\sum_{|\beta|\leqslant n}\mathscr{L}_2(\Zr^{\beta}(\psi),\Psi_n)(t,u)+\mathscr{L}_3(\Psi_n,\Zr^n(\psi_2))(t,u)+\int_0^u \mathscr{F}_{\leqslant n}(u')du'.
\end{align*}

\subsection{Conclusion of higher order energy estimates}\label{subsection:conclusion-higher-order-energy-estimates}
Combining the estimates for  {\bf Type $\mathbf{I}$} and {\bf Type $\mathbf{II}$}, the contributions of nonlinear terms can be bounded as follows:
\begin{equation}\label{bounds-N-Nbar}
	\begin{split}
		{\mathscr{N}}(t,u)+\underline{\mathscr{N}}(t,u)&\lesssim \mathring{M}\varepsilon^3 t^2 + \int_{\delta}^t \mathscr{E}_{\leqslant n}(\tau,u)d\tau + \int_0^u \mathscr{F}_{\leqslant n}(t,u')du'\\
		&  + \sum_{|\beta|\leqslant n}\mathscr{L}_2(\Zr^{\beta}(\psi),\Psi_n)(t,u)+\sum_{1\leq|\beta|\leqslant n}\mathscr{L}_3(\Zr^{\beta}(\psi),\Psi_n)(t,u) + \mathring{\mathscr{L}}_3(\psi,\Psi_n)(t,u).
	\end{split}
\end{equation}
For convenience, we introduce the following notations:
	\begin{equation}\label{def:Edot-Fdot}
		\dot{\mathscr{E}}_{\leqslant n}(\psi)(t,u) = \sum_{1 \leqslant |\alpha| \leqslant n} \mathscr{E}_{\alpha}(\psi)(t,u), \quad \dot{\mathscr{F}}_{\leqslant n}(\psi)(t,u) = \sum_{1 \leqslant |\alpha| \leqslant n} \mathscr{F}_{\alpha}(\psi)(t,u).
	\end{equation}
Therefore, by \eqref{ineq: estimates for L2 L3}, we have
\begin{align*}
	\sum_{j=2}^3\sum_{1\leqslant|\beta|\leqslant n}\mathscr{L}_j(\Zr^{\beta}(\psi),\Psi_n)(t,u) \lesssim \frac{1}{a_0}\int_{0}^{u}\dot{\mathscr{F}}_{\leqslant n}(t,u')du'+a_0\int_\delta^t\frac{\dot{\mathscr{E}}_{\leqslant n}(t',u)}{t'}dt',
\end{align*}
where $a_0 > 0$ is a constant to be determined. 
In view of the zeroth order energy estimates \eqref{eq: 0 order bound},  we have
\begin{align*}
	\int_{\delta}^t \mathscr{E}_{0}(\tau,u)d\tau + \int_0^u \mathscr{F}_{0}(t,u')du' \lesssim \varepsilon^2 t^2.
\end{align*}
Also, similar to the proof of \eqref{ineq: estimates for L2 L3}, we have
\begin{align*}
	\mathscr{L}_2(\psi,\Psi_n)(t,u) + \mathring{\mathscr{L}}_3(\psi,\Psi_n)(t,u) &\lesssim \frac{1}{a_0}\int_{0}^{u} \mathscr{F}_0(t,u')du'+a_0\int_\delta^t  \frac{\mathscr{E}_n(t',u)}{t'}dt' \\
	&\lesssim \varepsilon^2 t^2 + a_0\int_\delta^t  \frac{\mathscr{E}_n(t',u)}{t'}dt'.
\end{align*}
Therefore, the righthand side of \eqref{bounds-N-Nbar} can be bounded as follows:
\begin{align*}
	{\mathscr{N}}(t,u)+\underline{\mathscr{N}}(t,u)&\lesssim \varepsilon^2 t^2 + \frac{1}{a_0}\int_{0}^{u}\dot{\mathscr{F}}_{\leqslant n}(t,u')du'+a_0\int_\delta^t\frac{\dot{\mathscr{E}}_{\leqslant n}(t',u)}{t'}dt'.
\end{align*}
In view of the fundamental energy inequality \eqref{ineq: energy ineq for total}, there exist \emph{universal constants} $C_0$, $C_1$ and $C_2$, such that if $\mathring{M}\varepsilon$ is sufficiently small, for $1 \leqslant n \leqslant \Ntop$, we have
\begin{align*}
	\mathscr{E}_n(t,u) + \mathscr{F}_n(t,u) &\leqslant \mathscr{E}_n(\delta,u)+ \mathscr{F}_n(t,0)+C_1\varepsilon^2t^2 + C_0\Big(\frac{1}{a_0}\int_{0}^{u}\dot{\mathscr{F}}_{\leqslant n}(t,u')du'+a_0\int_\delta^t\frac{\dot{\mathscr{E}}_{\leqslant n}(t',u)}{t'}dt'\Big).
\end{align*}
Summing for $1 \leqslant n \leqslant \Ntop$, we have
\begin{align*}
	\dot{\mathscr{E}}_{\leqslant n}(t,u) + \dot{\mathscr{F}}_{\leqslant n}(t,u) &\leqslant C_2\varepsilon^2t^2 + \frac{C_0}{a_0}\int_{0}^{u}\dot{\mathscr{F}}_{\leqslant n}(t,u')du' + a_0C_0\int_\delta^t\frac{\dot{\mathscr{E}}_{\leqslant n}(t',u)}{t'}dt'.
\end{align*}
{\color{black}We apply Lemma \ref{refined Gronwall} by setting} $a_0 = \frac{1}{2C_0}$ and $u_0^* = \frac{\log 2}{2 C_0^2}$. Then we have
\begin{align*}
	\dot{\mathscr{E}}_{\leqslant n}(t,u) + \dot{\mathscr{F}}_{\leqslant n}(t,u) &\leqslant \underbrace{C_2\varepsilon^2}_{:=A}t^2 + \underbrace{2C_0^2}_{:=B}\int_{0}^{u}\dot{\mathscr{F}}_{\leqslant n}(t,u')du' + \underbrace{\frac{1}{2}}_{:=C}\int_\delta^t\frac{\dot{\mathscr{E}}_{\leqslant n}(t',u)}{t'}dt'.
\end{align*}
where $A, B$ and $C$ are the constants in Lemma \ref{refined Gronwall}. Moreover, $e^{B u^*}C\leqslant 1$.  Therefore, for all $(t,u)\in [\delta,t^*]\times [0,u_0^*]$, we have
\begin{equation*}
	\dot{\mathscr{E}}_{\leqslant n}(t,u) + \dot{\mathscr{F}}_{\leqslant n}(t,u) \leqslant  6 C_2 t^{2} \varepsilon^2.
\end{equation*}
We can repeat the above argument a finite number of times on intervals $[u_0^*, u_1^*], \cdots, [u_N^*,u^*]$. Notice that the only growth comes from the flux $\dot{\mathscr{F}}_{\leqslant n}(t,u_j^*)$, enlarging by a power of $6$. 

Therefore, for all $(t,u)\in [\delta,t^*]\times [0,u^*]$, we have
\begin{equation}
	\dot{\mathscr{E}}_{\leqslant n}(t,u) + \dot{\mathscr{F}}_{\leqslant n}(t,u) \lesssim  \varepsilon^2t^2{\color{black}.}
\end{equation}
{\color{black}This closes the bootstrap assumption} $\mathbf{(B_2)}$ in \eqref{ansatz B2}.

\section{Closing the bootstrap ansatz on the pointwise bounds}\label{section:closing-pointwise-bootstrap-ansatz}

\subsection{Preparations}
We recall that $(\Xh,T)$ and $(\Xr, \Tr)$ are related by
\begin{equation}\label{eq: transformation from Xh T to Xr Tr}
\begin{cases}
\Xh &=-\Th^1 \Xr-\frac{1}{\kappar}\Th^2\Tr,\\
T &= \kappa \Th^2 \Xr-\frac{\kappa}{\kappar}\Th^1 \Tr.
\end{cases}
\end{equation}
For a vector $Y$ defined on $\Sigma_t$, using the frame $(\Xr,\Tr)$, we can decompose it as $Y=Y^{\Xr}\Xr+Y^{\Tr}\Tr$. Therefore, we have
\begin{equation}\label{the coefficeints}
\Xh^{\Xr}=-\Th^1, \ \ \Xh^{\Tr}=-\frac{1}{\kappar}\Th^2, \ \ T^{\Xr}= \kappa \Th^2, \ \ T^{\Tr}= -\frac{\kappa}{\kappar}\Th^1.
\end{equation}

According to \eqref{bound on kappa more precise} and \eqref{preliminary geometric bounds},  we have the following bound on $\Sigma_t$:
\begin{equation}\label{compare Xh T and Xr Tr 0th order}
|\Xh^{\Xr}-1|\lesssim \mathring{M}t^2\varepsilon^2, \ |\Xh^{\Tr}|\lesssim \mathring{M}\varepsilon, \ |T^{\Xr}|\lesssim \mathring{M}t^2\varepsilon, \ |T^{\Tr}-1|\lesssim \mathring{M}t\varepsilon.
\end{equation}
In fact, in view of the fact that $Z(\kappar)=0$ for $Z\in \mathscr{Z}=\{\Xh,T\}$, we can apply \eqref{preliminary geometric bounds with derivatives} and we conclude that, for all multi-index $\alpha$ with $1\leqslant |\alpha|\leqslant 2$, we have the following estimates on $\Sigma_t$:
\begin{equation}\label{compare Xh T and Xr Tr 1st and 2nd order}
|Z^\alpha\big(\Xh^{\Xr}\big)|\lesssim \mathring{M}t^2\varepsilon^2, \ |Z^\alpha\big(\Xh^{\Tr}\big)|\lesssim \mathring{M}\varepsilon, \ |Z^\alpha\big(T^{\Xr}\big)|\lesssim \mathring{M}t^2\varepsilon, \ |Z^\alpha\big(T^{\Tr}\big)|\lesssim \mathring{M}t\varepsilon.
\end{equation}
We remark that, {\color{black}compared to the others,} the bounds on $Z^\alpha(\Xh^{\Tr})$'s lack the decay factor $t$.

We also recall the bounds from \eqref{LZTi byproduct}, \eqref{LZTi byproduct 2}, \eqref{bound on LZ kappa} and \eqref{bound on LZ2 kappa} that, for all multi-index $\alpha$ with $1\leqslant |\alpha|\leqslant 2$ and for all $Z\in \mathscr{Z}$:
\begin{equation}\label{bounds on LZalpha  Thi and kappa}
\big|L(Z^\alpha(\widehat{T}^{1}))\big|\lesssim  \mathring{M}\varepsilon^2 t,  \ \big|L(Z^\alpha(\widehat{T}^{2}))\big|\lesssim  \mathring{M}\varepsilon, \  \big|L(Z^\alpha(\kappa))\big|\lesssim \mathring{M}\varepsilon t.
\end{equation}
In view of \eqref{preliminary geometric bounds with derivatives}, \eqref{the coefficeints} and \eqref{compare Xh T and Xr Tr 1st and 2nd order}, we also have
\begin{equation}\label{bounds on LZalpha  coefficients}
\big|LZ^\alpha\big(\Xh^{\Xr}\big)\big|\lesssim \mathring{M}\varepsilon^2 t, \ \ \big|LZ^\alpha\big(\Xh^{\Tr}\big)\big|\lesssim \mathring{M}t^{-1}\varepsilon, \ \ \big|LZ^\alpha\big(T^{\Xr}\big)\big|\lesssim \mathring{M}\varepsilon t, \ \ \big|LZ^\alpha\big(T^{\Tr}\big)\big|\lesssim \mathring{M}\varepsilon t.
\end{equation}

In view of the expression \eqref{the coefficeints}, estimates on the coefficients $\Xh^{\Xr}$ and $T^{\Xr}$ can be derived directly from those of $\kappa$, $\Th^1$ and $\Th^2$. In the next lemma, we will connect the pointwise bounds of $\Xh^{\Tr}$ and $T^{\Tr}$ to the maximal characteristic speed $v^1+c=-\psi_1+c$.

\begin{lemma}
For  all $Z \in \mathscr{Z}$ and all multi-index $\alpha$ with $1 \leqslant |\alpha| \leqslant 2$, for all $t\in [\delta,t^*]$, we have
\begin{equation}\label{eq: connect Xh Tr coefficient to v1+c}
		\left|Z^{\alpha}\big(\Xh^{\Tr}\big)(t,u,\vartheta)+\frac{Z^{\alpha}(\Th^2)(\delta,u,\vartheta)}{t} +\frac{1}{t}\int_{\delta}^t Z^{\alpha}\Xh(v^1+c)(\tau,u,\vartheta)d\tau \right| \lesssim \mathring{M} t \varepsilon^2,
\end{equation}
	and 
\begin{equation}\label{eq: connect T Tr coefficient to v1+c}
		\left|Z^{\alpha}\big(T^{\Tr}\big)(t,u,\vartheta)-\frac{Z^{\alpha}(\kappa)(\delta,u,\vartheta)}{t} -\frac{1}{t}\int_{\delta}^t Z^{\alpha}T(v^1+c)(\tau,u,\vartheta) d\tau \right|\lesssim \mathring{M} t \varepsilon^2.
\end{equation}
\end{lemma}
\begin{proof}
{\color{black}We start with the second equation} of \eqref{precise form of kappa and Ti}. Since $\Xh^2=-\Th^1$, we have
\begin{equation*}
L(\Th^2)=\Xh(v^1+c) -\Xh(v^1+c)(\Th^1+1)+ \mathbf{err}_{\Th} \cdot \Xh^2=\Xh(v^1+c)+\mathbf{Err}_{\Th} .
\end{equation*}
We commute the equation first with $Z\in \mathscr{Z}$ and then with $Z'\in \mathscr{Z}$. Therefore,
\[
L(Z(\Th^2))=Z\Xh(v^1+c)+Z(\mathbf{Err}_{\Th}) -\,^{(Z)}f\cdot \Xh(\Th^2),
\]
and
\[
L(Z'Z(\Th^2))=Z'Z\Xh(v^1+c)+Z'Z(\mathbf{Err}_{\Th}) -Z'\big(\,^{(Z)}f\cdot \Xh(\Th^2)\big)-\,^{(Z')}f\cdot \Xh Z(\Th^2).
\]
where $\,^{(\Xh)}f=\chi$ and $\,^{(T)}f=\zeta+\eta$. In view of \eqref{bound on T1+1}, \eqref{bound on T2}, \eqref{preliminary geometric bounds with derivatives} and $\mathbf{(B_\infty)}$, it is straightforward to check that $Z(\mathbf{Err}_{\Th})$,  $\,^{(Z)}f\cdot \Xh(\Th^2)$, $Z'Z(\mathbf{Err}_{\Th})$, $Z'\big(\,^{(Z)}f\cdot \Xh(\Th^2)\big)$ and $\,^{(Z')}f\cdot \Xh Z(\Th^2)$ are bounded pointwisely by $\mathring{M}\varepsilon^2t$. Therefore, for all multi-index $\alpha$ with $1 \leqslant |\alpha| \leqslant 2$, we have
\[\big|L(Z^\alpha(\Th^2))(\tau, u,\vartheta)-Z^\alpha\Xh(v^1+c)(\tau, u,\vartheta)\big|\lesssim \mathring{M}\varepsilon^2t.\]
We integrate this inequality from $\delta$ to $t$ and we obtain that
\begin{equation}\label{eq: connect Xh Tr coefficient to v1+c aux 1}
\big|Z^\alpha(\Th^2)(t, u,\vartheta)-Z^\alpha(\Th^2)(\delta, u,\vartheta)-\int_{\delta}^tZ^\alpha\Xh(v^1+c)(\tau, u,\vartheta)d\tau\big|\lesssim \mathring{M}\varepsilon^2t.
\end{equation}
We divide both sides by $-\kappar=-t$. This yields the first inequality of the lemma.

To prove the second inequality, we first notice the following schematic formula:
\[Z^{\alpha}\big(T^{\Tr}\big)=-Z^\alpha\big(\frac{\kappa}{\kappar}\Th^1\big)=\frac{Z^\alpha(\kappa)}{\kappar}-\frac{Z^\alpha(\kappa)}{\kappar}(\Th^1+1)+\sum_{\beta+\gamma=\alpha, |\beta|\geqslant 1}\frac{Z^\gamma(\kappa)}{\kappar}Z^\beta(\Th^1).\]
The last two terms are bounded by $\mathring{M}\varepsilon^2t$. Therefore, it suffices to compute $Z^\alpha(\kappa)$. This is based on the second equation of \eqref{precise form of kappa and Ti}. It can be derived exactly in the same way as  for \eqref{eq: connect Xh Tr coefficient to v1+c}. This completes the proof of the lemma.
\end{proof}

Since we have already closed the energy ansatz $\mathbf{(B_2)}$. Therefore, the constant $\mathring{M}$ in \eqref{ineq: L infty bound}, \eqref{ineq: L^2 bound on lambda final}, \eqref{ineq: L infty bound on lambda final}, \eqref{ineq: L infty bound for Zbeta L Zalpha} and \eqref{ineq: L 2 bound for Zbeta L Zalpha} can be improved to be a universal constant. Therefore, we have the following bounds:

\begin{lemma}
For all multi-index $\alpha,\beta, \gamma$ with $1\leqslant |\alpha|\leqslant 3$, $|\beta|\leqslant 2$ and $|\gamma|\leqslant 2$, for all $\psi \in \{w,\wb,\psi_2\}$, for $\lambda \in \{\yr,\zr,\chir,\etar\}$, except for the case $\Zr^\alpha \psi=T\wb$, we have
\begin{equation*}
\|\mathring{Z}^\alpha(\psi)\|_{L^\infty(\Sigma_t)}\lesssim  \begin{cases} \varepsilon, \ \ & \text{if}~\Zr^\alpha=\Xr^\alpha;\\
\varepsilon t, \ \ & \text{otherwise};   
\end{cases}, \ \   \|\mathring{Z}^\beta(\lambda)\|_{L^\infty(\Sigma_t)}\lesssim \varepsilon, \ \ \| \Lr \Zr^\gamma \psi \|_{L^2(\Sigma_t)}\lesssim \varepsilon.
\end{equation*}
\end{lemma}
In view of \eqref{eq: transformation from Xh T to Xr Tr}, \eqref{compare Xh T and Xr Tr 0th order} and $L-\Lr =c\big(\frac{\Th^1+1}{\kappar}\Tr -\Th^2\Xr\big)$, we also have
\begin{corollary}
For all multi-index $\alpha, \gamma$ with $0\leqslant |\alpha|\leqslant 2$, $|\beta|\leqslant 1$ and $|\gamma|\leqslant 2$, for all $\psi \in \{w,\wb,\psi_2\}$, for $\lambda \in \{\yr,\zr,\chir,\etar\}$, except for the case $Z\Zr^\alpha \psi=T\wb$, we have
\begin{equation}\label{circled Linfty bound}
\|Z\mathring{Z}^\alpha(\psi)\|_{L^\infty(\Sigma_t)}\lesssim  \begin{cases} \varepsilon, \ \ & \text{if}~\Zr^\alpha=\Xr^\alpha \ \text{and} \ Z=\Xh;\\
\varepsilon t, \ \ & \text{otherwise};   
\end{cases}, \ \   \| L \Zr^\gamma \psi \|_{L^2(\Sigma_t)}\lesssim \varepsilon,
\end{equation}
and
\begin{equation}\label{circled Linfty bound on lambda}
\|Z\mathring{Z}^\beta(\lambda)\|_{L^\infty(\Sigma_t)}\lesssim \varepsilon.
\end{equation}
\end{corollary}

We have the following useful Gronwall type lemma:
\begin{lemma}\label{lemma: another new Gronwall}
Let $F(t)$ and $G(t)$ be two non-negative continuous functions defined on $[\delta,t^*]$. We assume that, for all $t\in [\delta,t^*]$, 
\begin{align*}
		F(t) \leqslant \frac{F_0(\delta)}{t} + \frac{1}{t}\int_{\delta}^{t}F(\tau)d\tau + G(t),
\end{align*}
where $F_0(\delta)$ is a constant. Then, for all $t>\delta$, we have
\begin{align*}
		F(t) \leqslant \frac{F_0(\delta)}{\delta} + \int_{\delta}^{t}\frac{G(\tau)}{\tau}d\tau + G(t).
\end{align*}
\end{lemma}
\begin{proof} We define $f(t) = t^{-1}\int_{\delta}^{t}F(\tau)d\tau$. We rewrite the inequality as 
\begin{equation}\label{eq: Gronwall aux}
F(t) \leqslant \frac{F_0(\delta)}{t} + f(t) + G(t){\color{black}.}
\end{equation}
By the definition of $f$, it is straightforward to check that $tf'(t) +f(t)= F(t)$. Plugging into the above equation, we obtain that	
\begin{align*}
f'(t) \leqslant \frac{F_0(\delta)}{t^2} + \frac{G(t)}{t}.
\end{align*}
In view of the fact that $f(\delta) =0$, we integrate the above equation from $\delta$ to $t$ to derive
\begin{align*}
f(t) \leqslant \big(\frac{1}{\delta} - \frac{1}{t}\big)F_0(\delta) + \int_{\delta}^{t}\frac{G(\tau)}{\tau}d\tau.
\end{align*}
Combined with \eqref{eq: Gronwall aux}, this completes the proof of the lemma.
\end{proof}

\subsection{Estimates on the second derivatives}
In the rest of the paper, we assume that $Z, Z_0\in \mathscr{Z}$.  In this subsection,  we will bound $\|YZ_0(\psi)\|_{L^{\infty}(\Sigma_t)}$ for all $t\in [\delta,t^*]$, where $Y=\Xh, T$ or $L$. 
Since $Z_0(\psi) = Z^{\mathring{T}}_0\mathring{T}(\psi) + Z^{\mathring{X}}_0\mathring{X}(\psi)$, we have
\begin{equation}\label{formula for YZ0 psi}
	|YZ_0(\psi)|\leqslant |Z^{\Tr}_0| |Y \mathring{T}(\psi)| + |Z^{\Xr}_0| |Y\mathring{X}(\psi)| + |\mathring{T}(\psi)| |Y(Z^{\mathring{T}}_0)|  + |\mathring{X}(\psi)| |Y(Z^{\mathring{X}}_0) |.
\end{equation}

\subsubsection{The case $\psi \in \{w, \psi_2\}$}\label{section: closing Binfty 2nd order w psi2}

For $Y=T$, since $\psi \in \{w, \psi_2\}$, in view of \eqref{compare Xh T and Xr Tr 0th order}, \eqref{compare Xh T and Xr Tr 1st and 2nd order}, \eqref{bounds on LZalpha  Thi and kappa} and \eqref{circled Linfty bound}, we derive that \begin{align*}
	|TZ_0(\psi)| &\lesssim |T\Zr_0(\psi)| + \mathring{M} t \varepsilon^2 \lesssim  t \varepsilon+ \mathring{M} t \varepsilon^2.
\end{align*}
where for $Z_0=\Xh$ and $T$, $\Zr_0$ represents $\Xr$ and $\Tr$ respectively. For sufficiently small $\varepsilon$, this shows that 
\begin{align*}
	|TZ_0(\psi)| \lesssim  t \varepsilon.
\end{align*}
For $Y=\Xh$, by applying \eqref{compare Xh T and Xr Tr 0th order}, \eqref{compare Xh T and Xr Tr 1st and 2nd order}, \eqref{bounds on LZalpha  Thi and kappa} and \eqref{circled Linfty bound}, we have two cases:
\begin{itemize}
\item $Z_0=\Xh$, for sufficiently small $\varepsilon$, we have 
\begin{align*}
	|\Xh^2(\psi)| &\lesssim |\Xh\Xh_0\psi| + \mathring{M}\varepsilon^2 \lesssim  \varepsilon+\mathring{M}\varepsilon^2\lesssim \varepsilon.
\end{align*}
\item $Z_0=T$. According to \eqref{eq:commutator formulas}, $\underline{\chi} =  \kappa(\slashed{k} + \theta)$, we have
\[|[T,\Xh](\psi)|\leqslant |\kappa\theta\cdot \Xh(\psi)| \lesssim \mathring{M}\varepsilon^2 t.\]
We have already proved that $|T\Xh(\psi)| \lesssim  t \varepsilon$. Therefore, 
\begin{equation}\label{the commutator trick}
|\Xh T(\psi)|\leqslant |T\Xh(\psi)|+|[T,\Xh](\psi)|\lesssim \varepsilon t,
\end{equation}
for sufficiently small $\varepsilon$.
\end{itemize}
Finally, we take $Y=L$ in \eqref{formula for YZ0 psi} to derive
\begin{align*}
	|LZ_0(\psi)|\leqslant |Z^{\Tr}_0| |L \mathring{T}(\psi)| + |Z^{\Xr}_0| |L\mathring{X}(\psi)| + |\mathring{T}(\psi)| |L(Z^{\mathring{T}}_0)|  + |\mathring{X}(\psi)| |L(Z^{\mathring{X}}_0) |.
\end{align*}
By applying \eqref{compare Xh T and Xr Tr 0th order}, \eqref{compare Xh T and Xr Tr 1st and 2nd order}, \eqref{bounds on LZalpha  Thi and kappa} and \eqref{circled Linfty bound}, it is straightforward to check that, for sufficiently small $\varepsilon$, 
\begin{align*}
	|LZ_0(\psi)|\lesssim |L \Zr_0(\psi)| + \mathring{M}\varepsilon^2\lesssim \varepsilon.
\end{align*}
We have closed the bootstrap assumption $\mathbf{(B_\infty)}$ for $YZ_0(\psi)$ where $\psi \in \{w, \psi_2\}$.

\subsubsection{The case $\psi = \wb$}\label{section: closing B2 for alpha2 psi=wb}
Since $v^1+c=\frac{\gamma-3}{2}w+\frac{\gamma+1}{2}\wb$, in view of the bounds on $w$ derived in Section \ref{section: closing Binfty 2nd order w psi2}, in order to close the part of $YZ_0(\wb)$ in $\mathbf{(B_\infty)}$, it suffices to bound $v^1+c$ in the place of $\wb$.  We remark that the maximal characteristic speed $v^1+c$ appears naturally as the main term for evolution equations of geometric quantities such as $\Th^i$ and $\kappa$.

We first bound $YZ_0(v^1+c)$ for $Y =T$ or $\Xh$. By $Z_0 = Z^{\mathring{T}}_0\cdot\mathring{T} + Z^{\mathring{X}}_0\cdot\mathring{X}$, we have
\begin{align*}
YZ_0(v^1+c)&= Z^{\mathring{T}}_0  Y \mathring{T}(v^1+c) + Z^{\mathring{X}}_0   Y\mathring{X}(v^1+c) + \mathring{T}(v^1+c)   Y(Z^{\mathring{T}}_0)  + \mathring{X}(v^1+c)   Y(Z^{\mathring{X}}_0)\\
&= Z^{\mathring{T}}_0   Y \mathring{T}(v^1+c) + Z^{\mathring{X}}_0   Y\mathring{X}(v^1+c) + (\mathring{T}(v^1+c)+1)   Y(Z^{\mathring{T}}_0)  + \mathring{X}(v^1+c)   Y(Z^{\mathring{X}}_0)- Y(Z^{\mathring{T}}_0).
\end{align*}

We notice that the presence of $\mathring{T}(v^1+c)$ formally cause a loss in $t$ and $\varepsilon$. This difficulty can be resolved by applying Lemma \ref{lemma: another new Gronwall}, provided the source term $G(t)$ vanishes as $t \to 0^+$.  By applying \eqref{compare Xh T and Xr Tr 0th order}, \eqref{compare Xh T and Xr Tr 1st and 2nd order}, \eqref{bounds on LZalpha  Thi and kappa} and $\mathbf{(B_\infty)}$, it is straightforward to check that
\begin{equation}\label{eq:YZ0 v1+c}
	|YZ_0(v^1+c)| \leqslant |Y\ZR_0(v^1+c)| + |Y(Z^{\mathring{T}}_0)| + \mathring{M} t \varepsilon^2.
\end{equation}
According to $Y,Z_0\in \{\Xh, T\}$, it suffices to check the following four cases:
\begin{itemize}
\item $Y=T$ and $Z_0=T$.

We can use \eqref{eq: connect T Tr coefficient to v1+c} to replace $T(T^{\Tr})$ in \eqref{eq:YZ0 v1+c}. Hence,
\begin{align*}
	|T^2(v^1+c)| &\leqslant \frac{\big|T(\kappa)|_{t=\delta}\big|}{t} + \frac{1}{t}\int_{\delta}^t|T^2(v^1+c)|d\tau + |T\Tr(v^1+c)| + \mathring{M} t \varepsilon^2.
\end{align*}
We notice that, by \eqref{circled Linfty bound}, $|T\Tr(v^1+c)|\lesssim \varepsilon t$ and it is merely linear in $\varepsilon$. Therefore, we can rewrite the above equation as
\begin{align*}
	|T^2(v^1+c)| &\leqslant \frac{\big|T(\kappa)|_{t=\delta}\big|}{t} + \frac{1}{t}\int_{\delta}^t|T^2(v^1+c)|d\tau +G(t),
\end{align*}
with $|G(t)|\lesssim  |T\Tr(v^1+c)| +\mathring{M}t\varepsilon^2$. We can apply Lemma \ref{lemma: another new Gronwall} and this leads to
\begin{align*}
	|T^2(v^1+c)| \leqslant \frac{\big|T(\kappa)|_{t=\delta}\big|}{\delta} + \int_{\delta}^{t}\frac{|T\Tr(v^1+c)|}{\tau}d\tau + |T\Tr(v^1+c)|+ \mathring{M} t \varepsilon^2{\color{black}.}
\end{align*}
Once again, by \eqref{circled Linfty bound},  we have $|T\Tr(v^1+c)|\lesssim \varepsilon t$. The key fact about this inequality is the $t$ factor on the righthand side. {\color{black}Plugging this bound} in the above inequality, in view of the $T(\kappa)$ in $\mathbf{(I_\infty)}$, for sufficiently small $\varepsilon$, we obtain that
\begin{align*}
	|T^2(v^1+c)| \lesssim \varepsilon t.
\end{align*}

\item $Y=T$ and $Z_0=\Xh$.

We can use \eqref{eq: connect Xh Tr coefficient to v1+c} to replace $T(X^{\Tr})$ in \eqref{eq:YZ0 v1+c}.  We proceed exactly as in the previous case and we obtain that
\begin{align*}
	|T\Xh(v^1+c)| &\leqslant \frac{\big|T(\Th^2)|_{t=\delta}\big|}{t} + \frac{1}{t}\int_{\delta}^t|T\Xh(v^1+c)|d\tau +G(t),
\end{align*}
with $|G(t)|\lesssim  |T\Xr(v^1+c)| +\mathring{M}t\varepsilon^2$. By \eqref{circled Linfty bound},  we have $|T\Xr(v^1+c)|\lesssim \varepsilon t$ and this estimate has a decay factor $t$ on the righthand side. Therefore, we can repeat the previous proof to use  Lemma \ref{lemma: another new Gronwall} to show that
\begin{align*}
	|T\Xh(v^1+c)| \lesssim \varepsilon t.
\end{align*}

\item $Y=\Xh$ and $Z_0=T$.

We use the commutator $[T,\Xh]$ from \eqref{eq:commutator formulas} as in \eqref{the commutator trick}. In fact,
\begin{align*}
	|\Xh T(v^1+c)| &\leqslant |T\Xh(v^1+c)| + |[T,\Xh](v^1+c)| \lesssim \varepsilon t.
\end{align*}

\item $Y=\Xh$ and $Z_0=\Xh$.

This is the most difficult case and it uses the full strength of the estimates on $\yr$.  We can use \eqref{eq: connect Xh Tr coefficient to v1+c} to replace $\Xh(X^{\Tr})$ in \eqref{eq:YZ0 v1+c}.  We proceed exactly as in the previous case and we obtain that
\begin{align*}
	|\Xh^2(v^1+c)| &\leqslant \frac{\big|\Xh(\Th^2)|_{t=\delta}\big|}{t} + \frac{1}{t}\int_{\delta}^t|\Xh^2(v^1+c)|d\tau +G(t),
\end{align*}
with $|G(t)|\lesssim  |\Xh\Xr(v^1+c)| +\mathring{M}t\varepsilon^2$. Since $\Xr(v^1+c)=\kappar \yr$, $|G(t)|\lesssim  t|\Xh(\yr)| +\mathring{M}t\varepsilon^2$. Therefore, this estimate has a decay factor $t$ on the righthand side thanks to {\color{black}the \emph{extra decay} of $y$.} Therefore, since $\frac{\big|\Xh(\Th^2)|_{t=\delta}\big|}{\delta}\approx \varepsilon$. 
we can repeat the previous proof to use  Lemma \ref{lemma: another new Gronwall} to show that
\begin{align*}
	|\Xh^2(v^1+c)| \lesssim \varepsilon.
\end{align*}
\end{itemize}

It remains to consider the case for $Y=L$. We commute $Z_0$ with the first equation of \eqref{Euler equations:form 2} and we obtain the following schematic formula:
\begin{equation}\label{commute with Z_0}
L (Z_0(\wb)) = Z_0\big[c\frac{T(\wb)}{\kappa}(\widehat{T}^1+1)\big]+Z_0\big[c \frac{T(\psi_2)}{\kappa}\widehat{T}^2\big] +Z_0\big(c \Xh(\psi_2)\Xh^2\big)+Z_0\big(c\Xh(\wb)\Xh^1\big)+\,^{(Z_0)}f\cdot \Xh(\wb),
\end{equation}
where $\,^{(\Xh)}f=\chi$ and $\,^{(T)}f=\zeta+\eta$. We can use Leibniz rule to write the $Z_0$ derivative of the product into a sum of terms. It is straightforward to see that all the terms have been controlled in the previous steps. It follows that 
\begin{align*}
	|L(Z_0(\wb))| \lesssim \varepsilon.
\end{align*}

We now have closed the bootstrap assumption $\mathbf{(B_\infty)}$ for $YZ_0(\wb)$.

\subsection{Estimates on the third derivatives}

In this subsection,  we will bound $\|YZ_1Z_0(\psi)\|_{L^{\infty}(\Sigma_t)}$ for all $t\in [\delta,t^*]$ where $Y=\Xh, T$ or $L$ and 
 $Z_1,Z_0\in \{T,\Xh\}$.
 
 We expand $Z_1$ and $Z_0$ in terms of $\Xr$ and $\Tr$. First of all, we write $Z_0$ as $ Z_0^{\Tr}\Tr+Z_0^{\Xr}\Xr$. This yields
\begin{equation}\label{eq: Y Z1 Z0 psi}
\begin{split}
	Y Z_1 Z_0(\psi) =&  Z_0^{\TR} \cdot Y Z_1\TR(\psi) + Z_0^{\XR} \cdot Y Z_1\XR(\psi)\\
	&+ \underbrace{Y(Z_0^{\TR}) Z_1\TR(\psi) + Y(Z_0^{\XR}) Z_1\XR(\psi) + Y\big[Z_1(Z_0^{\TR})\TR(\psi) + Z_1(Z_0^{\XR})\XR(\psi)\big]}_{\mathbf{err}_1}.
\end{split}
\end{equation}
The first two terms on the righthand side are the main terms. They can be represented as $YZ_1\Zr(\psi)$ in the schematic way. Next, for $\ZR \in \{\Tr,\Xr\}$, we write $Z_1$ as $ Z_1^{\Tr}\Tr+Z_1^{\Xr}\Xr$. This yields
\begin{align*}
YZ_1\ZR(\psi) &= Z_1^{\TR} \cdot Y\TR\ZR(\psi) + Z_1^{\XR} \cdot Y\XR\ZR(\psi) + \underbrace{Y(Z_1^{\TR}) \cdot \TR\ZR(\psi) + Y(Z_1^{\XR}) \cdot \XR\ZR(\psi)}_{\mathbf{err}_{2,\Zr}}.
\end{align*}
We plug this results into \eqref{eq: Y Z1 Z0 psi} and we obtain that
\begin{equation}\label{eq: key equation for Y Z1 Z0 psi}
	Y Z_1 Z_0(\psi) = Z_1^{\ZR_1}Z_0^{\ZR_0} \cdot Y\ZR_1\ZR_0(\psi) + \!\!\!\underbrace{\sum_{(\mathring{Y}_1,\mathring{Y}_2) \neq (\ZR_1,\ZR_0)}\!\!\! Z_1^{\mathring{Y}_1}Z_0^{\mathring{Y}_0} \cdot Y\mathring{Y}_1\mathring{Y}_0(\psi)}_{\mathbf{err}_3}  + \mathbf{err}_1 + \mathbf{err}_2,
\end{equation}
where $\mathbf{err}_2 =  Z_0^{\Tr}\cdot \mathbf{err}_{2,\Tr} +  Z_0^{\Xr}\cdot \mathbf{err}_{2,\Xr}$.

\subsubsection{The case $\psi \in \{w, \psi_2\}$}\label{section: closing Binfty 3rd order w psi2}
We first consider the case where $Y=\Xh$ or $T$.  We control the error terms in \eqref{eq: key equation for Y Z1 Z0 psi}.

Since $\psi \neq \wb$, we have $|\Tr(\psi)| \lesssim \varepsilon t$. By applying \eqref{compare Xh T and Xr Tr 0th order}, \eqref{compare Xh T and Xr Tr 1st and 2nd order}, \eqref{bounds on LZalpha  Thi and kappa} and $\mathbf{(B_\infty)}$, it is straightforward to check that $|\mathbf{err}_1| \lesssim  \mathring{M}  \varepsilon^2 t$. 

For $\mathbf{err}_3$,  since $(\mathring{Y}_1,\mathring{Y}_2) \neq (\ZR_1,\ZR_0)$, according to  \eqref{compare Xh T and Xr Tr 0th order}, $\big|Z_1^{\mathring{Y}_1}Z_0^{\mathring{Y}_0}\big|\lesssim \mathring{M} \varepsilon$. Therefore, unless $\mathring{Y}_1=\mathring{Y}_2=\Xr$, by \eqref{circled Linfty bound}, $|Y\mathring{Y}_1\mathring{Y}_0(\psi)|\lesssim \varepsilon t$. Therefore, except for $\mathring{Y}_1=\mathring{Y}_2=\Xr$, the other terms of $\mathbf{err}_3$ are all bounded by $\mathring{M}\varepsilon^2 t$. If $\mathring{Y}_1=\mathring{Y}_2=\Xr$, since $(\mathring{Y}_1,\mathring{Y}_2) \neq (\ZR_1,\ZR_0)$, therefore, by \eqref{compare Xh T and Xr Tr 0th order}, at least one of $\big|Z_1^{\mathring{Y}_1}\big|$ and $\big|Z_0^{\mathring{Y}_0}\big|$ are bounded by $\mathring{M}\varepsilon t$. Hence, this term is also bounded by $\mathring{M}\varepsilon^2 t$. As a conclusion, we have $|\mathbf{err}_3| \lesssim  \mathring{M}  \varepsilon^2 t$. 

Similarly, we have $|\mathbf{err}_2| \lesssim  \mathring{M}  \varepsilon^2 t$. Hence, \eqref{eq: key equation for Y Z1 Z0 psi} implies that
\[|Y Z_1 Z_0(\psi)| \lesssim |Y\ZR_1\ZR_0(\psi)|+  \mathring{M}  \varepsilon^2 t.\]
In view of \eqref{circled Linfty bound}, for sufficiently small $\varepsilon$, this gives the desired bound for $YZ^\alpha(\psi)$ where $Y\in \{\Xh,T\}$, $|\alpha|=2$ and $\psi\in \{w,\psi_2\}$.

For $Y=L$,  we use \eqref{bounds on LZalpha  coefficients} to bound $\mathbf{err}_1$, $\mathbf{err}_2$ and $\mathbf{err}_3$. In fact, $LZ^\alpha\big(\Xh^{\Tr}\big)$ is the worst possible terms appearing in $\mathbf{err}_i$'s. The other terms can be bounded immediately by  $\mathring{M}  \varepsilon^2$. On the other side, $LZ^\alpha\big(\Xh^{\Tr}\big)$'s only appear in $\mathbf{err}_1$ and $\mathbf{err}_2$ through the following two possible forms: $LZ_1(\Xh^{\Tr})\Tr(\psi)$ and $LZ_1^{\TR} \cdot \Tr\Zr(\psi)$. Since $\psi\neq \wb$, we have $|\Tr(\psi)|\lesssim \mathring{M}\varepsilon t$ and $|\Tr\Zr(\psi)|\lesssim \mathring{M}\varepsilon t$. This extra factor $t$ shows that 
\[|\mathbf{err}_1|+|\mathbf{err}_2|+|\mathbf{err}_3| \lesssim  \mathring{M}  \varepsilon^2.\]
Thus,
\[|L Z_1 Z_0(\psi)| \lesssim |L\ZR_1\ZR_0(\psi)|+  \mathring{M}  \varepsilon^2.\]
In view of \eqref{circled Linfty bound}, for sufficiently small $\varepsilon$, this gives the desired bound for $LZ^\alpha(\psi)$ where $Y\in \{\Xh,T\}$, $|\alpha|=2$ and $\psi\in \{w,\psi_2\}$.

We have closed the bootstrap assumption $\mathbf{(B_\infty)}$ for $YZ^\alpha(\psi)$ where $\psi \in \{w, \psi_2\}$ and $|\alpha|=2$.

\subsubsection{The case $\psi = \wb$}

We proceed in a similar way as in Section \ref{section: closing B2 for alpha2 psi=wb}. To close the corresponding parts in $\mathbf{(B_\infty)}$, it suffices to bound $v^1+c$ in the place of $\wb$. Therefore, we set $\psi=v^1+c$ in \eqref{eq: key equation for Y Z1 Z0 psi}.

We start with the case where $Y=\Xh$ or $T$.

First of all, we can repeat the same argument for the terms $\mathbf{err}_3$ and $\mathbf{err}_2$ in Section \ref{section: closing Binfty 3rd order w psi2}. This gives immediately that 
\[|\mathbf{err}_2|+|\mathbf{err}_3| \lesssim  \mathring{M}  \varepsilon^2t.\]
Next, to bound $\mathbf{err}_1$, we notice that except for $YZ_1\big(Z_0^{\Tr}\big)\cdot \Tr(v^1+c)$, the rest of the terms in $\mathbf{err}_1$ can also be bounded exactly in the same way as in Section \ref{section: closing Binfty 3rd order w psi2}. Hence, we can rewrite \eqref{eq: key equation for Y Z1 Z0 psi} as 
\begin{equation*}
	Y Z_1 Z_0(v^1+c) = Z_1^{\ZR_1}Z_0^{\ZR_0} \cdot Y\ZR_1\ZR_0(v^1+c) + YZ_1(Z_0^{\TR})\cdot \TR(v^1+c)+ \mathbf{err},
\end{equation*}
where $|\mathbf{err}| \lesssim  \mathring{M}  \varepsilon^2t$. Since $|\TR(v^1+c)+1|\lesssim \varepsilon t$, we can rewrite this equation as 
\begin{equation}\label{eq: key equation for Y Z1 Z0 v1 c}
	Y Z_1 Z_0(v^1+c) = Z_1^{\ZR_1}Z_0^{\ZR_0} \cdot Y\ZR_1\ZR_0(v^1+c) -YZ_1(Z_0^{\TR})+ \mathbf{err},
\end{equation}
where $|\mathbf{err}| \lesssim  \mathring{M}  \varepsilon^2t$.

According to $Y$, we consider the following two cases:
\begin{itemize}
\item $Y=T$.

If $Z_0 = T$, in view of \eqref{eq: connect T Tr coefficient to v1+c}, \eqref{eq: key equation for Y Z1 Z0 v1 c} shows that
\begin{align*}
	|TZ_1T(v^1+c)| \leqslant \frac{\big|TZ_1(\kappa)|_{t=\delta}\big|}{t} + \frac{1}{t}\int_{\delta}^t|T\ZR_1\TR(v^1+c)|d\tau + |T\ZR_1\TR(v^1+c)| + {\mathbf{err}'},
\end{align*}
with $|{\mathbf{err}'}|\lesssim \mathring{M}\varepsilon^2 t$. Let  $G(t)= |T\ZR_1\TR(v^1+c)| + {\mathbf{err}'}$.  We can apply Lemma \ref{lemma: another new Gronwall} and this leads to
\begin{align*}
	|TZ_1T(v^1+c)| \leqslant \frac{\big|TZ_1(\kappa)|_{t=\delta}\big|}{\delta} + \int_{\delta}^{t}\frac{|T\ZR_1\TR(v^1+c)|}{\tau}d\tau + |T\Zr_1\Tr(v^1+c)|+ \mathring{M} t \varepsilon^2{\color{black}.}
\end{align*}
By \eqref{circled Linfty bound},  we have $|T\ZR_1\TR(v^1+c)|\lesssim \varepsilon t$. In view of $TZ_1(\kappa)$ in $\mathbf{(I_\infty)}$, for sufficiently small $\varepsilon$, we obtain that
\begin{align*}
	|TZ_1T(v^1+c)| \lesssim \varepsilon t.
\end{align*}

If $Z_0 = \Xh$, in view of \eqref{eq: connect Xh Tr coefficient to v1+c}, \eqref{eq: key equation for Y Z1 Z0 v1 c} shows that
\begin{align*}
	|TZ_1\Xh(v^1+c)| \leqslant \frac{\big|TZ_1(\Th^2)|_{t=\delta}\big|}{t} + \frac{1}{t}\int_{\delta}^t|T\ZR_1\Xr(v^1+c)|d\tau + |T\ZR_1\Xr(v^1+c)| + {\mathbf{err}'},
\end{align*}
with $|{\mathbf{err}'}|\lesssim \mathring{M}\varepsilon^2 t$. By \eqref{circled Linfty bound},  we have $|T\ZR_1\Xr(v^1+c)|\lesssim \varepsilon t$. We then repeat the previous computations to derive
\begin{align*}
	|TZ_1\Xh(v^1+c)| \lesssim \varepsilon t.
\end{align*}

\item $Y=\Xh$.

If at least one of $Z_0$ and $Z_1$ is $T$, we can repeat the proof for $Y=T$ to show that, for sufficiently small $\varepsilon$, we have
\begin{align*}
	|\Xh T^2(v^1+c)| +\Xh T\Xh(v^1+c)|+\Xh \Xh T(v^1+c)|\lesssim \varepsilon t.
\end{align*}

It remains to bound the most difficult term $\Xh^3(v^1+c)$.  We can use \eqref{eq: connect Xh Tr coefficient to v1+c} to proceed exactly as in the previous case and we obtain that
\begin{align*}
	|\Xh^3(v^1+c)| &\leqslant \frac{\big|\Xh^2(\Th^2)|_{t=\delta}\big|}{t} + \frac{1}{t}\int_{\delta}^t|\Xh^3(v^1+c)|d\tau +G(t),
\end{align*}
with $|G(t)|\lesssim  |\Xh\Xr^2(v^1+c)| +\mathring{M}t\varepsilon^2$. Since $\Xr^2(v^1+c)=\kappar \Xr(\yr)$, we have $|G(t)|\lesssim  t|\Xh\Xr(\yr)| +\mathring{M}t\varepsilon^2$. The better decay of {\color{black}$y$} from \eqref{circled Linfty bound on lambda} allows us to use  Lemma \ref{lemma: another new Gronwall} to show that
\begin{align*}
	|\Xh^3(v^1+c)| \lesssim \varepsilon.
\end{align*}
\end{itemize}

It remains to bound  $LZ_1Z_0(\wb)$. We commute $Z_1$ with the first equation of \eqref{commute with Z_0} to derive
\begin{align*}
L (Z_0(\wb)) =& Z_1Z_0\big[c\frac{T(\wb)}{\kappa}(\widehat{T}^1+1)\big]+Z_1Z_0\big[c \frac{T(\psi_2)}{\kappa}\widehat{T}^2\big] +Z_1Z_0\big(c \Xh(\psi_2)\Xh^2\big)+Z_1Z_0\big(c\Xh(\wb)\Xh^1\big)\\
&+Z_1(\,^{(Z_0)}f) \Xh(\wb)+\,^{(Z_0)}f\cdot Z_1\Xh(\wb)+\,^{(Z_1)}f\cdot \Xh Z_0(\wb),
\end{align*}
where $\,^{(\Xh)}f=\chi$ and $\,^{(T)}f=\zeta+\eta$. We can use Leibniz rule to write the derivatives of the product into a sum of terms. It is straightforward to see that all the terms have been controlled in the previous steps. It follows that 
\begin{align*}
	|L(Z_0(\wb))| \lesssim \varepsilon.
\end{align*}

We now have closed the bootstrap assumption $\mathbf{(B_\infty)}$ for $YZ_1Z_0(\wb)$. Hence, {\color{black}we have closed} the bootstrap assumption $\mathbf{(B_\infty)}$. This completes the proof of the {\bf Main Theorem}.

\section*{Acknowledgment}
The authors are grateful to the anonymous referees, who suggested many valuable improvements and corrections. PY is supported by NSFC11825103, NSFC12141102, New Cornerstone Investigator Program and Xiao-Mi Professorship. TWL is supported by NSFC 11971464.

\end{document}